\documentclass[a4paper,12pt,reqno]{amsart}
\usepackage{amssymb}
\usepackage[T1]{fontenc}
\newtheorem{thm}{Theorem}[section]
\newtheorem*{thm*}{Theorem}
\newtheorem{lem}[thm]{Lemma}
\newtheorem{cor}[thm]{Corollary}
\newtheorem{prop}[thm]{Proposition}
\theoremstyle{definition}
\newtheorem{defin}{Definition}[section]
\newtheorem*{exp*}{Example}

\newtheorem{rem}{Remark}[section]
\newtheorem*{rem*}{Remark}
\newtheorem{cond}{Condition}[section]
\def\C{{\mathbb C}}
\def\N{{\mathbb N}}

\def\R{{\mathbb R}}
\def\Z{{\mathbb Z}}
\def\E{{\mathbb E}}
\def\V{{\mathbb V}}
\def\J{{\mathbb J}}
\def\G{{\mathcal G}}
\def\T{{\mathcal T}}

\def\D{{\mathcal D}}
\def\A{{\mathcal A}}
\newcommand{\cO}{{\mathcal O}}
\newcommand{\cC}{{\mathcal C}}
\newcommand{\cE}{{\mathcal E}}
\newcommand{\cL}{{\mathcal L}}
\newcommand{\cV}{{\mathcal V}}
\newcommand{\cW}{{\mathcal W}}
\def\g{{\mathfrak g}}

\def\p{{\mathfrak p}}
\def\q{{\mathfrak q}}

\def\m{{\mathfrak m}}

\newcommand{\gl}{{\mathfrak{gl}}}
\newcommand{\sgl}{{\mathfrak{sl}}}
\DeclareSymbolFont{script}{U}{eus}{m}{n}
\DeclareMathSymbol{\Wedge}{0}{script}{"5E}
\newcommand{\RP}{{\mathbb{RP}}}
\newcommand{\CP}{{\mathbb{CP}}}
\newcommand{\wt}{\widetilde}

\newcommand{\spann}{\operatorname{span}}
\newcommand{\coker}{\operatorname{coker}}
\newcommand{\Id}{\operatorname{Id}}
\newcommand{\Rho}{{\mathsf P}}
\newcommand{\mapsfrom}
{\mathrel{\DOTSB\leftarrow\hspace{-4pt}\rule[1.1pt]{.4pt}{3.6pt}}}
\newcommand{\mapsisoto}{
   \makebox[0pt][l]{\raisebox{5pt}{\hspace*{3.8pt}$\simeq$}}\longrightarrow}
\newcommand{\Lam}{{\mathit\Lambda}}
\newcommand{\vol}{{\mathrm{vol}}}
\newcommand{\scale}{\tau}
\newcommand{\Kf}{\Omega}
\newcommand{\Bt}{U}
\newcommand{\mob}{\m}
\newcommand{\Ns}{{\mathcal N}}
\newcommand{\ts}{\eta}
\newcommand{\ms}{\eta}
\newcommand{\sms}{\tilde\ms}
\newcommand{\hv}{{\mathit\Lambda}}
\newcommand{\hp}{\lambda}
\newcommand{\tor}{\tau}
\newcommand{\Ric}{\mathrm{Ric}}
\newcommand{\Scal}{\mathrm{Scal}}
\newcommand{\cy}{\chi}
\newcommand{\wc}{\psi}
\newcommand{\adend}{\phi}
\newcommand{\epst}{\varepsilon}
\def\bps/{\textup(pseudo\mbox{-}\textup)\nobreak\hspace{0pt}}
\newcommand{\weight}[2]{\def\topwt{#1}\def\botwt{#2}}
\newcommand{\pweight}[9][\times]{\begin{picture}(70,36)(0,-7)
\put(5,3){\makebox(0,0){$#1$}}
\put(20,3){\makebox(0,0){$\bullet$}}\put(35,3){\makebox(0,0){$\bullet$}}
\put(50,3){\makebox(0,0){$\bullet$}}\put(65,3){\makebox(0,0){$\bullet$}}
\put(5,3){\line(1,0){60}}
\put(5,10.2){\makebox(0,0){\tiny$\updownarrow$}}
\put(20,10.2){\makebox(0,0){\tiny$\updownarrow$}}
\put(35,10.2){\makebox(0,0){\tiny$\updownarrow$}}
\put(50,10.2){\makebox(0,0){\tiny$\updownarrow$}}
\put(65,10.2){\makebox(0,0){\tiny$\updownarrow$}}
\put(5,17){\makebox(0,0){$#1$}}
\put(20,17){\makebox(0,0){$\bullet$}}\put(35,17){\makebox(0,0){$\bullet$}}
\put(50,17){\makebox(0,0){$\bullet$}}\put(65,17){\makebox(0,0){$\bullet$}}
\put(5,17){\line(1,0){60}}
\put(5,24){\makebox(0,0){$\scriptstyle\vphantom{pd}\topwt$}}
\put(20,24){\makebox(0,0){$\scriptstyle\vphantom{pd}#2$}}
\put(35,24){\makebox(0,0){$\scriptstyle\vphantom{pd}#3$}}
\put(50,24){\makebox(0,0){$\scriptstyle\vphantom{pd}#4$}}
\put(65,24){\makebox(0,0){$\scriptstyle\vphantom{pd}#5$}}
\put(5,-4){\makebox(0,0){$\scriptstyle\vphantom{pd}\botwt$}}
\put(20,-4){\makebox(0,0){$\scriptstyle\vphantom{pd}#6$}}
\put(35,-4){\makebox(0,0){$\scriptstyle\vphantom{pd}#7$}}
\put(50,-4){\makebox(0,0){$\scriptstyle\vphantom{pd}#8$}}
\put(65,-4){\makebox(0,0){$\scriptstyle\vphantom{pd}#9$}}
\end{picture}}
\begin{document}

\title{C-projective geometry} 
\author[D.M.J. Calderbank]{David M.J. Calderbank}
\address{Department of Mathematical Sciences\\ University of Bath\\
Bath BA2 7AY\\ UK}
\email{D.M.J.Calderbank@bath.ac.uk}
\author[M.G. Eastwood]{Michael G. Eastwood}
\address{School of Mathematical Sciences\\ University of Adelaide\\
SA 5005 Australia}
\email{meastwoo@member.ams.org}
\author[V.S. Matveev]{Vladimir S. Matveev}
\address{Institute of Mathematics\\ FSU Jena\\ 07737 Jena, Germany}
\email{vladimir.matveev@uni-jena.de}
\author[K. Neusser]{Katharina Neusser}
\address{Mathematical Institute\\ Charles University\\ Sokolovsk\'a 83\\
Praha, Czech Republic}
\email{kath.neusser@gmail.com}

\renewcommand{\thefootnote}{}
\footnotetext{This article was initiated when its authors participated in a
workshop at the Kioloa campus of the Australian National University in March
2013\@. We would like to thank the The Edith and Joy London Foundation for
providing the excellent facilities at Kioloa. We would also like to thank the
Group of Eight, Deutscher Akademischer Austausch Dienst, Australia-Germany
Joint Research Cooperation Scheme for financially supporting the workshop in
2013 and a subsequent Kioloa workshop in 2014; for the latter, we thank, in
addition, FSU Jena and the Deutsche Forschungsgemeinschaft (GK 1523/2)
for their financial support. The fourth author was also supported during part
of this  project by the Eduard \v Cech Institute. We would also like to thank
the latter for supporting a meeting of the first, second and fourth author 
at the Mathematical Institute at Charles University in July 2014.}

\begin{abstract}
We develop in detail the theory of c-projective geometry, a natural analogue of
projective differential geometry adapted to complex manifolds. We realise it as
a type of parabolic geometry and describe the associated Cartan or tractor
connection.  A K\"ahler manifold gives rise to a c-projective structure and
this is one of the primary motivations for its study. The existence of two or
more K\"ahler metrics underlying a given c-projective structure has many
ramifications, which we explore in depth. As a consequence of this analysis, we
prove the Yano--Obata Conjecture for complete K\"ahler manifolds: if such a
manifold admits a one parameter group of c-projective transformations that are
not affine, then it is complex projective space, equipped with a multiple of
the Fubini--Study metric.
\end{abstract}

\maketitle

\section*{Introduction}

C-projective geometry is a natural analogue of real projective differential
geometry for complex manifolds.  Like projective geometry, it has many facets,
which have been discovered and explored independently and repeatedly over the
past sixty years. Our aim in this work is to develop in detail a unified theory
of c-projective geometry, which highlights its relation with real projective
geometry as well as its differences.

\smallbreak
Projective geometry is a classical subject concerned with the behaviour of
straight lines, or more generally, (unparametrised) geodesic curves of a metric
or affine connection. It has been known for some time~\cite{Lagrange,T} that
two non-proportional metrics can have the same geodesic curves: central
projection takes great circles on the $n$-sphere, namely the geodesics for the
round metric, to straight lines in Euclidean $n$-space, namely geodesics for
the flat metric. The quotient of the round $n$-sphere under the antipodal
identification may be identified with the \emph{flat model} for $n$-dimensional
projective geometry: the real projective $n$-space $\RP^n$, viewed as
a homogeneous space under the group $\mathrm{PSL}(n+1,\R)$ of projective
transformations, which preserve the family of (linearly embedded) projective
lines $\RP^1\hookrightarrow \RP^n$. More generally, a
\emph{projective structure} on a smooth $n$-manifold (for $n\geq 2$) is an
equivalence class of torsion-free affine connections having the same geodesic
curves. In this setting, it is a nontrivial and interesting question whether
these curves are the geodesic curves of a \bps/Riemannian metric, i.e.~whether
any connection in the projective equivalence class preserves a nondegenerate
metric, possibly of indefinite signature. Such projective structures are called
\emph{metrisable} and the corresponding metrics \emph{compatible}. Rather
surprisingly, the partial differential equations controlling the metrisability
of a given projective structure can be set up as a \emph{linear}
system~\cite{benenti,Dini,Liouville,Sinjukov}. More precisely, there is a
projectively invariant linear differential operator acting on symmetric
contravariant $2$-tensors such that the nondegenerate elements of its kernel
correspond to compatible metrics.

In modern language, a projective structure determines a canonical \emph{Cartan
  connection}~\cite{Cartan} modelled on $\RP^n$, and hence
projective geometry is a \emph{parabolic
  geometry}~\cite{csbook,Eprojectivenotes}. In these terms, the metrisability
operator is a first \emph{BGG \textup(Bernstein--Gelfand--Gelfand\textup)
  operator}, which is a differential operator of finite type~\cite{EM}. Its
solutions correspond to parallel sections of a bundle with connection, which
is, up to curvature corrections, a linear representation of the Cartan
connection. The kernel is thus finite-dimensional; it is zero for generic
projective structures, with the maximal dimension attained on the flat model
$\RP^n$. The parabolic viewpoint on projective geometry has proven
to be very useful, for example in understanding projective compactifications
of Einstein metrics~\cite{CG2,CGH2}, the geometry of holonomy reductions of
projective structures~\cite{Arm}, and (solving problems posed by Sophus Lie in
1882) projective vector fields on surfaces~\cite{bryant,Matveev}.

Projective geometry has been linked to the theory of finite dimensional
integrable systems with great success: the equation for symmetric Killing
tensors is projectively invariant~\cite{Eprojectivenotes}, and (consequently)
the existence of two projectively equivalent metrics on a manifold implies the
existence of nontrivial integrals for the geodesic flows of both metrics. This
method has been effectively employed when the manifold is closed or complete
(see e.g.~\cite{hyperbolic,Matveev2007}). Moreover, the integrability of many
classically studied integrable geodesic flows (e.g., on ellipsoids) is closely
related to the existence of a projectively compatible metric, and many
geometric structures that lead to such integrable geodesic flows have been
directly related to the existence of a projectively compatible metric, see
e.g.~\cite{ben,benenti}.

\smallbreak
C-projective geometry arises when one retells this story, mutatis mutandis, for
complex or, indeed, almost complex manifolds, i.e.~smooth manifolds equipped
with an almost complex structure $J$, which is a smooth endomorphism of the
tangent bundle such that $J^2=-\Id$.  On such a manifold $M$, the
relevant \bps/Riemannian metrics are \emph{Hermitian} with respect to $J$,
i.e.~$J$-invariant, and the relevant affine connections are those which
preserve $J$, called \emph{complex connections}.  Such connections cannot be
torsion-free unless the almost complex structure is \emph{integrable}, i.e.~its
Nijenhuis tensor vanishes identically~\cite{NN}. This holds in particular if
the Levi-Civita connection of a Hermitian metric $g$ preserves $J$, in which
case $g$ is called a \emph{\bps/K\"ahler} metric.

In 1947, Bochner~\cite[Theorem 2]{Bochner} observed that any two metrics that
are K\"ahler with respect to the same complex structure cannot be projectively
equivalent (i.e.~have the same geodesic curves) unless they are affinely
equivalent (i.e.~have the same Levi-Civita connection). This led Otsuki and
Tashiro~\cite{OtsTash} to introduce a broader class of curves, which they
called ``holomorphically flat'', and which depend on both the connection and
the (almost) complex structure. We refer to these curves as $J$-\emph{planar}:
whereas a geodesic curve for an affine connection $\nabla$ is a curve $c$ whose
acceleration $\nabla_{\dot{c}}\dot{c}$ is proportional to its
velocity~$\dot{c}$, a $J$-planar curve is one whose acceleration is in the
linear span of $\dot{c}$ and~$J\dot{c}$.  On a Riemann surface, therefore, all
curves are $J$-planar. The other common manifold where it is possible to see
all $J$-planar curves without computation is complex projective space with its
Fubini--Study connection. The point here is that the linearly embedded complex
lines $\CP^1\hookrightarrow\CP^n$ are totally geodesic.
Therefore, the $J$-planar curves on $\CP^n$ are precisely the smooth
curves lying within such complex lines. Viewed in a standard affine coordinate
patch $\C^n\hookrightarrow\CP^n$, the $J$-planar curves are again the
smooth curves lying inside an arbitrary complex line $\{az+b\}\subset\C^n$ but
otherwise unconstrained. Evidently, these are the intrinsic $J$-planar curves
for the flat connection on~$\C^n$.

The $J$-planar curves provide a nontrivial notion of projective equivalence in
complex differential geometry, due to Otsuki and Tashiro~\cite{OtsTash} in the
K\"ahler setting, and Tashiro~\cite{Tashiro} for almost complex manifolds in
general. Two complex connections on an almost complex manifold $(M,J)$ are
\emph{c-projectively equivalent} if they have the same torsion and the same
$J$-planar curves. An \emph{almost c-projective manifold} is a complex manifold
$(M,J)$ equipped with a c-projective equivalence class of such connections. If
$J$ is integrable, we follow the usual convention and drop the word ``almost''
to arrive at the notion of a \emph{c-projective manifold}. We caution the
reader that Otsuki and Tashiro~\cite{OtsTash}, and many later researchers,
refer to ``holomorphically projective correspondences'', rather than
c-projective equivalences, and many authors use the terminology
``h-projective'' or ``holomorphic(ally) projective'' instead of
``c-projective''. We avoid their terminology because the connections in a
c-projective class are typically not holomorphic, even if the complex structure
is integrable; similarly, we avoid the term ``complex projective structure'',
which is often used for the holomorphic analogue of a real projective
structure, or related concepts.

During the decades following Otsuki and Tashiro's 1954 paper, c-projective
structures provided a prominent research direction in Japanese and Soviet
differential geometry. Many of the researchers involved had some background in
projective geometry, and the dominant line of investigation sought to
generalise methods and results from projective geometry to the c-projective
setting. This was a very productive direction, with more than 300 publications
appearing in the relatively short period from 1960 to 1990. One can compare,
for example, the surveys by Mike\v{s}~\cite{mikes1,mikes}, or the papers of
Hiramatu~\cite{hir,hir1}, to see how successfully c-projective analogues of
results in projective geometry were found. In particular, the linear system
for c-projectively equivalent K\"ahler metrics was obtained by Domashev and
Mike\v{s}~\cite{DM}, and its finite type prolongation to a connection was given
by Mike\v{s}~\cite{Mikes}.

Relatively recently, the linear system for c-projectively equivalent K\"ahler
metrics has been rediscovered, under different names and with different
motivations. On a fixed complex manifold, a compatible \bps/K\"ahler metric is
determined uniquely by its K\"ahler form (a compatible symplectic form), and
under this correspondence, c-projectively equivalent K\"ahler metrics are
essentially the same as \emph{Hamiltonian $2$-forms} defined and investigated
in Apostolov et al.~\cite{ACG,ApostolovII,ApostolovIII,ApostolovIV}: the
defining equation \cite[(12)]{ACG} for a Hamiltonian $2$-form is actually
algebraically equivalent to the metrisability equation \eqref{metriequ1}. In
dimension $\ge 6$, c-projectively equivalent metrics are also essentially the
same as conformal Killing (or twistor) $(1,1)$-forms studied in
\cite{Moroianu,Semmelmann0,Semmelmann1}, see \cite[Appendix~A]{ACG} or
\cite[\S{1.3}]{MatRos} for details.

The work of~\cite{ACG,ApostolovII} provides, \emph{a postiori}, local and
global classification results for c-projectively equivalent K\"ahler metrics,
although the authors were unaware of this interpretation, nor the pre-existing
literature. Instead, as explained in~\cite{ACG,ApostolovII} and~\cite{CMR}, the
notion and study of Hamiltonian $2$-forms was motivated by their natural
appearance in many interesting problems in K\"ahler geometry, and the unifying
role they play in the construction of explicit K\"ahler metrics on projective
bundles. In subsequent papers, e.g.~\cite{ApostolovIII,ApostolovIV},
Hamiltonian $2$-form methods were used to construct many new examples of
K\"ahler manifolds and orbifolds with interesting properties.

Another independent line of research closely related to c-projectively
equivalent metrics (and perhaps underpinning the utility of Hamiltonian
$2$-forms) appeared within the theory of finitely dimensional integrable
systems.  C-projectively equivalent metrics are closely related (see
e.g.~\cite{Kiyohara2010}) to the so-called K\"ahler--Liouville integrable
systems of type $A$ introduced and studied by Kiyohara in \cite{Kiyo1997}. In
fact, Topalov~\cite{Top} (see also \cite{Kiyo}) shows that generic
c-projectively equivalent K\"ahler metrics have integrable geodesic flows,
cf.~\cite{TM} for the analogous result in the projective case. On the one hand,
integrability provides, as in projective geometry, a number of new methods that
can be used in c-projective geometry. On the other hand, examples from
c-projective geometry turn out to be interesting for the theory of integrable
systems, since there are only a few known examples of K\"ahler metrics with
integrable geodesic flows.
 
Despite the many analogies between results in projective and c-projective
geometry, there seem to be very few attempts in the literature to explain why
these two subjects are so similar. In 1978, it was noted by Yoshimatsu~\cite{Y}
that c-projective manifolds admit canonical Cartan connections, and this was
generalised to almost c-projective manifolds by Hrdina~\cite{Hrdina} in 2009.
Thus c-projective geometry, like projective geometry, is a parabolic geometry;
its flat model is $\CP^n$, viewed as a homogeneous space under the
group ${\mathrm{PSL}}(n+1,\C)$ of projective transformations, which preserve
the $J$-planar curves described above. Despite this, c-projective structures
have received very little attention in the parabolic geometry literature: apart
from the work of Hrdina, and some work in dimension $4$~\cite{CorrCap,Mettler},
they have only been studied in~\cite{Arm}, where they appear as holonomy
reductions of projective geometries. A possible explanation for this oversight
is that ${\mathrm{PSL}}(n+1,\C)$ appears in c-projective geometry as a real Lie
group and, as such, its complexification is semisimple, but not simple. This is
related to the subtle point that most interesting c-projective structures are
not holomorphic.

\smallbreak
The development of c-projective geometry, as described above, has been rather
nonlinear until relatively recently, when a number of independent threads have
converged on a coherent set of ideas. However, a thorough description of almost
c-projective manifolds in the framework of parabolic geometries is lacking in
the literature. We therefore believe it is timely to lay down the fundamentals
of such a theory.

The article is organised as follows. In Section~\ref{almostcomplexmanifolds},
we survey the background on almost complex manifolds and complex connections.
As we review in Section~\ref{complexconnections}, the torsion of any complex
connection on an almost complex manifold, of real dimension $2n\geq 4$,
naturally decomposes into five irreducible pieces, one of which is invariantly
defined and can be identified as the Nijenhuis tensor. All other pieces can be
eliminated by a suitable choice of complex connection, which we call
\emph{minimal}. In first four sections of the article we carry along the
Nijenhuis tensor in almost all calculations and discussions.

Section~\ref{sec:elements} begins with the classical viewpoint on almost
c-projective structures, based on $J$-planar curves and equivalence classes of
minimal complex connections~\cite{OtsTash}. We then recall the notion of
parabolic geometries and establish, in Theorem~\ref{EquivCat}, an equivalence
of categories between almost c-projective manifolds and parabolic geometries
with a normal Cartan connection, modelled on $\CP^n$.

As a consequence of this parabolic viewpoint, we can associate a fundamental
local invariant to an almost c-projective manifold, namely the curvature
$\kappa$ of its normal Cartan connection; furthermore, $\kappa\equiv0$ if and
only if the almost c-projective manifold is locally isomorphic to
$\CP^n$ equipped with its standard c-projective structure. Since the
Cartan connection is normal (for this we need the complex connections to be
minimal), its curvature is a $2$-cycle for Lie algebra homology, and is
uniquely determined by its homology class, also known as the \emph{harmonic
curvature}. We construct and discuss this curvature in
section~\ref{sectionharmcurv}. For almost c-projective structures there are
three irreducible parts to the harmonic curvature. One of the pieces is the
Nijenhuis tensor, which is precisely the obstruction to the underlying almost
complex manifold actually being complex. One of the other two parts is
precisely the obstruction to there being a holomorphic connection in the
c-projective class. When it vanishes we end up with \emph{holomorphic
projective geometry}, i.e.~ordinary projective differential geometry but in the
holomorphic category. The remaining piece can then be identified with the
classical projective Weyl curvature (for $n\geq 3$) or Liouville curvature (for
$n=2$).

Another consequence of the parabolic perspective is that representation theory
is brought to the fore, both as the appropriate language for discussing natural
bundles on almost c-projective manifolds, and also as the correct tool for
understanding invariant differential operators on the flat model, and their
curved analogues. The various \emph{BGG complexes} on $\CP^n$ and
their curved analogues are systematically introduced and discussed in
Section~\ref{sec:tractorBGG}.

In particular, there is a BGG operator that controls the \emph{metrisability}
of a c-projective structure just as happens in the projective setting. A large
part of this article is devoted to the metrisability equation, which we
introduce in Section~\ref{sec:metrisability}, where we also obtain its
prolongation to a connection, not only for compatible \bps/K\"ahler metric, but
also in the non-integrable case of quasi-K\"ahler or (2,1)-symplectic
structures. For the remainder of the article, we suppose that the
Nijenhuis tensor vanishes, in other words that we are starting with a complex
manifold. In this case, a compatible metric is exactly a \bps/K\"ahler metric
(and a \emph{normal solution} of the metrisability equation corresponds to a
\bps/K\"ahler--Einstein metric). We shall also restrict our attention to
\emph{metric c-projective structures}, i.e.~to the metrisable case where the
c-projective structure arises from a \bps/K\"ahler metric. Borrowing
terminology from the projective case, we refer to the dimension of the solution
space of the metrisability equation as the \emph{\textup(degree of\textup)
mobility} of the metric c-projective structure (or of any compatible
\bps/K\"ahler metric). We are mainly interested in understanding when the
metric c-projective structure has mobility at least two, and the consequences
this has for the geometry and topology of the manifold.

In Section~\ref{sec:integrability}, we develop the consequences of mobility
for integrability, by showing that a pencil (two dimensional family) of
solutions to the metrisability equation generates a family of holomorphic
Killing vector fields and Hermitian symmetric Killing tensors, which together
provide commuting linear and quadratic integrals for the geodesic flow of any
metric in the pencil. In Section~\ref{metricstructures}, we study an
important, but somewhat mysterious, phenomenon in which tractor bundles for
metric c-projective geometries are naturally equipped with congenial
connections, which are neither induced by the normal Cartan connection nor
equal to the prolongation connection, but which have the property that their
covariant constant sections nevertheless correspond to solutions of the
corresponding first BGG operator.

We bring these tools together in Section~\ref{sec:global}, where we establish
the Yano--Obata Conjecture for complete K\"ahler manifolds, namely that the
identity component of the group of c-projective transformations of the
manifold consists of affine transformations unless the manifold is complex
projective space equipped with a multiple of the Fubini--Study metric. This
result is an analogue of the \emph{the Projective Lichnerowicz Conjecture}
obtained in~\cite{CMH,Matveev2007}, but the proof given there does not
generalise directly to the c-projective situation. Our proof also differs from
the proof for closed manifolds given in~\cite{MR}, and makes use of many
preliminary results obtained by the methods of parabolic geometry, which also
apply in the projective case.

Here, and throughout the article, we see that not only results from projective
geometry, but also methods and proofs, can be generalised to the c-projective
case, and we explain why and how. We hope that this article will set the scene
for what promises to be an interesting series of further developments in
c-projective geometry. In fact, several such developments already appeared
during our work on this article, which we discuss in Section~\ref{s:outlook}.

\tableofcontents

\section{Almost complex manifolds}\label{almostcomplexmanifolds}
Recall that an \emph{almost complex structure} on a smooth manifold $M$ is a
smooth endomorphism $J$ of the tangent bundle $TM$ of $M$ that satisfies
$J^2=-\Id$. Equivalently, an almost complex structure makes $TM$ into
a complex vector bundle in which multiplication by $i$ is decreed to be the
real endomorphism~$J$. In particular, the dimension of $M$ is necessarily even,
say~$2n$, and an almost complex structure is yet equivalently a reduction of
structure group to ${\mathrm{GL}}(n,\C)\subset{\mathrm{GL}}(2n,\R)$.

\subsection{Real and complex viewpoints}\label{barredandunbarred}

If $M$ is a complex manifold in the usual sense of being equipped with
holomorphic transition functions, then $TM$ is a complex vector bundle and
multiplication by $i$ defines a real endomorphism $TM\to TM$, which we write
as~$J$. This is enough to define the holomorphic structure on $M$: holomorphic
functions may be characterised amongst all smooth complex-valued functions
$f=u+iv$ as satisfying $Xu=(JX)v$ for all vector fields $X$ (the
\emph{Cauchy--Riemann equations}).

Thus, complex manifolds may be regarded as a subclass of almost complex
manifolds and the celebrated Newlander--Nirenberg Theorem tells us how to
recognise them:

\begin{thm}[Newlander--Nirenberg,~\cite{NN}] 
An almost complex manifold $(M,J)$ is a complex manifold if and only if the
tensor
\begin{equation}\label{nijenhuis}
N^J(X,Y):= [X,Y]-[JX,JY]+J([JX,Y]+[X,JY])
\end{equation}
vanishes for all vector fields $X$ and $Y$ on $M$, where $[\enskip,\enskip]$
denotes the Lie bracket of vector fields.
\end{thm}

Note that $N^J\colon TM\times TM\to TM$ is a $2$-form with values in $TM$,
which satisfies $N^J(JX,Y)=-JN^J(X,Y)$. It is called the \emph{Nijenhuis
  tensor} of~$J$. When $N^J$ vanishes we say that the almost complex structure
$J$ is \emph{integrable}. This viewpoint on complex manifolds, as
even-dimensional smooth manifolds equipped with integrable almost complex
structures, turns out to be very useful especially from the differential
geometric viewpoint.

It is useful to complexify the tangent bundle of $M$ and decompose the result
into eigenbundles under the action of~$J$. Specifically,
\begin{equation}\label{types}
\C TM=T^{1,0}M\oplus T^{0,1}M
=\{X\mbox{ s.t.\ }JX=iX\}\oplus\{X\mbox{ s.t.\ }JX=-iX\}.
\end{equation}
Notice that $T^{0,1}M=\overline{T^{1,0}M}$. There is a corresponding
decomposition of the complexified cotangent bundle, which we write as
$\Wedge^1M$ or simply $\Wedge^1$ if $M$ is understood. Specifically,
\begin{equation}\label{one-forms}
\Wedge^1=\Wedge^{0,1}\oplus\Wedge^{1,0}
=\{\omega\mbox{ s.t.\ }J\omega=-i\omega\}\oplus
\{\omega\mbox{ s.t.\ }J\omega=i\omega\},
\end{equation}
where sections of $\Wedge^{1,0}$ respectively of $\Wedge^{0,1}$ are known as
$1$-forms of \emph{type}~$(1,0)$ respectively $(0,1)$, see e.g.~\cite{KobNim}.
Notice that the canonical complex linear pairing between $\C TM$ and
$\Wedge^1M$ induces natural isomorphisms $\Wedge^{0,1}=(T^{0,1})^*$ and
$\Wedge^{1,0}=(T^{1,0})^*$ of complex vector bundles.

It is convenient to introduce abstract indices~\cite{PR} for real or complex
tensors on $M$ and also for sections of the bundles $T^{1,0}M$, $\Wedge^{0,1}$,
and so on. Let us write $X^\alpha$ for real or complex fields and
$\omega_\alpha$ for real or complex $1$-forms on~$M$. In local coordinates
$\alpha$ would range over $1,2,\ldots,2n$ where $2n$ is the dimension of~$M$.
Let us denote by $X^a$ a section of $T^{1,0}M$. In any frame, the index $a$
would then range over $1,2,\ldots,n$. Similarly, let us write $X^{\bar{a}}$ for
a section of $T^{0,1}M$ and the conjugate linear isomorphism
$T^{0,1}M=\overline{T^{1,0}M}$ as
$X^a\mapsto\overline{X^a}=\overline{X}{}^{\bar{a}}$. Accordingly, sections of
$\Wedge^{1,0}$ and $\Wedge^{0,1}$ will be denoted by $\omega_a$ and
$\omega_{\bar a}$ respectively, and the canonical pairings between $T^{1,0}M$
and $\Wedge^{1,0}$, respectively $T^{0,1}M$ and $\Wedge^{0,1}$, written as
$X^a\omega_a$, respectively $X^{\bar a}\omega_{\bar a}$, an abstract index
counterpart to the \emph{Einstein summation convention}.

We shall need the complex linear homomorphism $\C TM\to T^{1,0}M$ defined as
projection onto the first summand in the decomposition (\ref{types}) and given
explicitly as $X\mapsto\frac12(X-iJX)$. It is useful to write it in abstract
indices as
$$X^\alpha\mapsto\Pi_\alpha^aX^\alpha.$$
It follows that the dual homomorphism $\Wedge^{1,0}\hookrightarrow\Wedge^1$ is
given in abstract indices by
$$\omega_a\mapsto\Pi^a_\alpha\omega_a$$
and also that the homomorphisms $\C TM\to T^{0,1}M$ and 
$\Wedge^{0,1}\hookrightarrow\Wedge^1$ are given by 
$$X^\alpha\mapsto\overline\Pi{}_\alpha^{\bar{a}}X^\alpha\quad\mbox{and}\quad
\omega_{\bar{a}}\mapsto\overline\Pi{}_\alpha^{\bar{a}}\omega_{\bar{a}},$$
respectively.

Let us denote by $X^a\mapsto\Pi_a^\alpha X^a$, the inclusion
$T^{1,0}M\hookrightarrow\C TM$, paying attention to the distinction in their
indices between $\Pi_\alpha^a$ and~$\Pi_a^\alpha$. Various identities follow,
such as $\Pi_a^\alpha\Pi_\alpha^b=\delta_a{}^b,$ where the Kronecker delta
$\delta_a{}^b$ denotes the identity transformation on~$T^{1,0}M$. The symbol
$\Pi_a^\alpha$ also gives us access to the dual and conjugate homomorphisms.
Thus,
$$\omega_\alpha\mapsto\overline\Pi{}_{\bar{a}}^\alpha\omega_\alpha$$
extracts the $(0,1)$-part of a complex-valued $1$-form $\omega_\alpha$ on~$M$. 
The following identities are immediate from the definitions
\begin{align}\notag
\Pi_\alpha^a\Pi_a^\beta&=\tfrac12(\delta_\alpha{}^\beta-iJ_\alpha{}^\beta)&
\overline\Pi{}_\alpha^{\bar{a}}\overline\Pi{}_{\bar{a}}^\beta&=
\tfrac12(\delta_\alpha{}^\beta+iJ_\alpha{}^\beta)\\
\label{usefulidentities}
\Pi_a^\alpha J_\alpha{}^\beta&=i\Pi_a^\beta&
\overline\Pi{}_{\bar{a}}^\alpha J_\alpha{}^\beta&=
-i\overline\Pi{}_{\bar{a}}^\beta\\
\notag
J_\alpha{}^\beta\Pi_\beta^a&=i\Pi_\alpha^a&
J_\alpha{}^\beta\overline\Pi{}_\beta^{\bar{a}}&=
-i\overline\Pi{}_\alpha^{\bar{a}}.
\end{align}
They are indispensable for the calculations in the following sections. Further
useful abstract index conventions are as follows. Quantities endowed with
several indices denote sections of the tensor product of the corresponding
vector bundles. Thus, a section of $TM\otimes TM$ would be denoted
$X^{\alpha\beta}$ whilst $\Phi_\alpha{}^\beta$ is necessarily a section
of~$T^*M\otimes TM$ or, equivalently, an endomorphism of $TM$, namely
$X^\beta\mapsto\Phi_\alpha{}^\beta X^\alpha$, yet equivalently an endomorphism
of~$T^*M$, namely $\omega_\alpha\mapsto\Phi_\alpha{}^\beta\omega_\beta$. We
have already seen this notation for an almost complex
structure~$J_\alpha{}^\beta$. But it is unnecessary notationally to distinguish
between real-{} and complex-valued tensors. Thus, by $\omega_\alpha$ we can
mean a section of $T^*M$ or of $\Wedge^1M:=\C T^*M$ and if a distinction is
warranted, then it can be made in words or by context. For example, an almost
complex structure $J_\alpha{}^\beta$ is a real endomorphism whereas
$\Pi_\alpha^a$ is necessarily complex.

Symmetry operations can also be written in abstract index notation. For
example, the skew part of a covariant $2$-tensor $\phi_{\alpha\beta}$ is
$\frac12(\phi_{\alpha\beta}-\phi_{\beta\alpha})$, which we write as
$\phi_{[\alpha\beta]}$. Similarly, we write
$\phi_{(\alpha\beta)}=\frac12(\phi_{\alpha\beta}+\phi_{\beta\alpha})$ for the
symmetric part and then
$\phi_{\alpha\beta}=\phi_{(\alpha\beta)}+\phi_{[\alpha\beta]}$ realises the
decomposition of vector bundles
$\Wedge^1\otimes\Wedge^1=S^2\Wedge^1\oplus\Wedge^2$. In general, round brackets
symmetrise over the indices they enclose whilst square brackets take the skew
part, e.g.\
\[
R_{\alpha\beta}{}^\gamma{}_\delta\mapsto R_{[\alpha\beta}{}^\gamma{}_{\delta]}=
\tfrac16(R_{\alpha\beta}{}^\gamma{}_\delta+R_{\beta\delta}{}^\gamma{}_\alpha
+R_{\delta\alpha}{}^\gamma{}_\beta-R_{\beta\alpha}{}^\gamma{}_\delta
-R_{\alpha\delta}{}^\gamma{}_\beta-R_{\delta\beta}{}^\gamma{}_\alpha).
\]
By (\ref{one-forms}) differential forms on almost complex manifolds can be
naturally decomposed according to \emph{type} (see e.g. \cite{KobNim}).  We
pause to examine the decomposition of $2$-forms, especially from the abstract
index point of view. {From} (\ref{one-forms}) it follows that the bundle
$\Wedge^2$ of complex-valued $2$-forms decomposes into types according to
\begin{equation}\label{two-forms}
\Wedge^2=\Wedge^2(\Wedge^{0,1}\oplus\Wedge^{1,0})=
\Wedge^{0,2}\oplus\Wedge^{1,1}\oplus\Wedge^{2,0}\end{equation}
and, as we shall make precise in Section~\ref{BGG}, there is no finer
decomposition available (it is a decomposition into \emph{irreducibles}). Using
the projectors $\Pi_\alpha^a$ and~$\Pi_a^\alpha$, we can explicitly execute
this decomposition:
\begin{align*}\omega_{\alpha\beta}&\mapsto
\bigl(\overline{\Pi}{}_{\bar{a}}^\alpha\overline{\Pi}{}_{\bar{b}}^\beta
\omega_{\alpha\beta},
\Pi_a^\alpha\overline{\Pi}{}_{\bar{b}}^\beta\omega_{\alpha\beta},
\Pi_a^\alpha\Pi_b^\beta\omega_{\alpha\beta}\bigr)\\
\overline{\Pi}{}_\alpha^{\bar{a}}\overline{\Pi}{}_\beta^{\bar{b}}
\omega_{\bar{a}\bar{b}}
+2\Pi_{[\alpha}^a\overline{\Pi}_{\beta]}^{\bar{b}}\omega_{a\bar{b}}+
\Pi_\alpha^a\Pi_\beta^b\omega_{ab}
&\mapsfrom (\omega_{\bar{a}\bar{b}},\omega_{a\bar{b}},\omega_{ab})
\end{align*}
in accordance with~(\ref{usefulidentities}). Notice that we made a choice here,
namely to identify $\Wedge^{1,1}$ as $\Wedge^{1,0}\otimes\Wedge^{0,1}$ \emph{in
this order} and, consequently, write forms of type $(1,1)$
as~$\omega_{a\bar{b}}$. We could equally well choose the opposite convention
or, indeed, use both conventions simultaneously representing a $(1,1)$ form as
$\omega_{a\bar{b}}$ and/or $\omega_{\bar{a}b}$ but now subject to
$\omega_{a\bar{b}}=-\omega_{\bar{b}a}$. Strictly speaking, this goes against
the conventions of the abstract index notation~\cite{PR} but we shall allow
ourselves this extra leeway when it is useful. For example, the reconstructed
form $\omega_{\alpha\beta}$ may then be written as
\begin{equation*}
\omega_{\alpha\beta} =\Pi_\alpha^a\overline{\Pi}{}_\beta^{\bar{b}}\omega_{a\bar{b}}
+\overline{\Pi}{}_\alpha^{\bar{a}}\Pi_\beta^b\omega_{\bar{a}b}.
\end{equation*}
Two-forms of various types may be characterised as
\begin{equation}\label{types_of_two_form} \begin{split}
\omega_{\alpha\beta}\mbox{ is type }(0,2)&\iff
J_\alpha{}^\gamma\omega_{\beta\gamma}=i\omega_{\alpha\beta}\\
\omega_{\alpha\gamma}\mbox{ is type }(1,1)&\iff
J_{[\alpha}{}^\gamma\omega_{\beta]\gamma}=0\\
\omega_{\alpha\beta}\mbox{ is type }(2,0)&\iff
J_\alpha{}^\gamma\omega_{\beta\gamma}=-i\omega_{\alpha\beta}
\end{split}\end{equation}
but already this is a little awkward and becomes more so for higher forms and
extremely awkward when attempting to decompose more general tensors as we shall
have cause to do when considering torsion and curvature. Notice that forms of
type $(1,1)$ in (\ref{types_of_two_form}) are characterised by a real
condition. Indeed, the complex bundle $\Wedge^{1,1}$ is the complexification of
a real irreducible bundle whose sections are the real $2$-forms satisfying
$J_{[\alpha}{}^\gamma\omega_{\beta]\gamma}=0$. As for forms of types $(0,2)$
and $(2,0)$, there is a real bundle whose sections satisfy
$$J_\alpha{}^\gamma J_\beta{}^\delta\omega_{\gamma\delta}=-\omega_{\alpha\beta}$$
(as opposed to $J_\alpha{}^\gamma J_\beta{}^\delta\omega_{\gamma\delta}=
\omega_{\alpha\beta}$ for sections of $\Wedge^{1,1}$) and whose
complexification is $\Wedge^{0,2}\oplus\Wedge^{2,0}$. Thus, the real $2$-forms
split irreducibly into just two kinds but the complex $2$-forms split into
three types~(\ref{two-forms}).

Notice that if $E$ is a complex vector bundle on~$M$, then we can decompose
$2$-forms with values in $E$ into types by using the same
formulae~(\ref{types_of_two_form}). In particular, we can do this on an almost
complex manifold when $E=TM$, regarded as a complex bundle via the action
of~$J$. Writing this out explicitly, a real tensor
$T_{\alpha\beta}{}^\gamma=T_{[\alpha\beta]}{}^\gamma$ is said to be
\begin{equation}\label{torsion_types}\begin{split}
\mbox{of type }(0,2)&\iff
J_\alpha{}^\gamma T_{\beta\gamma}{}^\delta=T_{\alpha\beta}{}^\gamma J_\gamma{}^\delta\\
\mbox{of type }(1,1)&\iff J_{[\alpha}{}^\gamma T_{\beta]\gamma}{}^\delta=0\\
\mbox{of type }(2,0)&\iff
J_\alpha{}^\gamma T_{\beta\gamma}{}^\delta=-T_{\alpha\beta}{}^\gamma J_\gamma{}^\delta.
\end{split}\end{equation}
For example, as the Nijenhuis tensor (\ref{nijenhuis}) satisfies
$N^J(Y,JX)=JN^J(X,Y)$, it is of type $(0,2)$. Further to investigate this
decomposition~(\ref{torsion_types}), it is useful to apply the projectors
$\Pi_\alpha^a$ and $\Pi_a^\alpha$ to obtain 
\begin{align*}
T_{\bar{a}\bar{b}}{}^{\bar{c}}&\equiv
\overline{\Pi}{}_{\bar{a}}^\alpha\overline{\Pi}{}_{\bar{b}}^\beta
\overline{\Pi}{}_\gamma^{\bar{c}}T_{\alpha\beta}{}^\gamma&
T_{a\bar{b}}{}^{\bar{c}}&\equiv \Pi_a^\alpha\overline{\Pi}{}_{\bar{b}}^\beta
\overline{\Pi}{}_\gamma^{\bar{c}}T_{\alpha\beta}{}^\gamma&
T_{ab}{}^{\bar{c}}&\equiv \Pi_a^\alpha\Pi_b^\beta
\overline{\Pi}{}_\gamma^{\bar{c}}T_{\alpha\beta}{}^\gamma\\
T_{\bar{a}\bar{b}}{}^c&\equiv
\overline{\Pi}{}_{\bar{a}}^\alpha\overline{\Pi}{}_{\bar{b}}^\beta
\Pi_\gamma^cT_{\alpha\beta}{}^\gamma&
T_{a\bar{b}}{}^c&\equiv
\Pi_a^\alpha\overline{\Pi}{}_{\bar{b}}^\beta
\Pi_\gamma^cT_{\alpha\beta}{}^\gamma&
T_{ab}{}^c&\equiv
\Pi_a^\alpha\Pi_b^\beta \Pi_\gamma^cT_{\alpha\beta}{}^\gamma
\end{align*}
satisfying
\begin{equation*}
\begin{array}{ccccccc}T_{ab}{}^c=T_{[ab]}{}^c&& 
T_{\bar{a}\bar{b}}{}^c=T_{[\bar{a}\bar{b}]}{}^c&&
T_{ab}{}^{\bar{c}}=T_{[ab]}{}^{\bar{c}}&&
T_{\bar{a}\bar{b}}{}^{\bar{c}}=T_{[\bar{a}\bar{b}]}{}^{\bar{c}}\\[4pt]
&\makebox[0pt]{$\overline{T_{ab}{}^c}=T_{\bar{a}\bar{b}}{}^{\bar{c}}$}&&
\makebox[0pt]{$\overline{T_{a\bar{b}}{}^c}=-T_{b\bar{a}}{}^{\bar{c}}$}&&
\makebox[0pt]{$\overline{T_{ab}{}^{\bar{c}}}=T_{\bar{a}\bar{b}}{}^c$}
\end{array}
\end{equation*}
and from which we can recover $T_{\alpha\beta}{}^\gamma$ according to 
\begin{align*}
T_{\alpha\beta}{}^\gamma&=
\overline{\Pi}{}_\alpha^{\bar{a}}\overline{\Pi}{}_\beta^{\bar{b}}
\overline{\Pi}{}_{\bar{c}}^\gamma T_{\bar{a}\bar{b}}{}^{\bar{c}}
+2\Pi_{[\alpha}^a\overline{\Pi}_{\beta]}^{\bar{b}}
\overline{\Pi}{}_{\bar{c}}^\gamma T_{a\bar{b}}{}^{\bar{c}}
+\Pi_\alpha^a\Pi_\beta^b\overline{\Pi}{}_{\bar{c}}^\gamma T_{ab}{}^{\bar{c}}\\
&\qquad+\overline{\Pi}{}_\alpha^{\bar{a}}\overline{\Pi}{}_\beta^{\bar{b}}
\Pi_c^\gamma T_{\bar{a}\bar{b}}{}^c
+2\Pi_{[\alpha}^a\overline{\Pi}_{\beta]}^{\bar{b}}
\Pi_c^\gamma T_{a\bar{b}}{}^c
+\Pi_\alpha^a\Pi_\beta^b\Pi_c^\gamma T_{ab}{}^c.
\end{align*}
{From} (\ref{usefulidentities}) and (\ref{torsion_types}), the splitting of
$T_{\alpha\beta}{}^\gamma$ into types corresponds exactly to components
\begin{equation}\label{torsion_types_with_indices}\begin{split}
\mbox{type }(0,2)&\leftrightarrow
(T_{ab}{}^{\bar{c}},T_{\bar{a}\bar{b}}{}^c)\\
\mbox{type }(1,1)&\leftrightarrow
(T_{a\bar{b}}{}^c,T_{a\bar{b}}{}^{\bar{c}})\\
\mbox{type }(2,0)&\leftrightarrow
(T_{ab}{}^c,T_{\bar{a}\bar{b}}{}^{\bar{c}}).
\end{split}\end{equation}
Notice that, for each of types $(1,1)$ and $(2,0)$, a complex-valued $1$-form
can be invariantly extracted: 
\begin{equation*}\begin{split}
\mbox{type }(1,1):&
\phi_\alpha\equiv\Pi_\alpha^aT_{a\bar{b}}{}^{\bar{b}}=
\tfrac12\bigl(T_{\alpha\beta}{}^\beta
+iT_{\alpha\beta}{}^\gamma J_\gamma{}^\beta\bigr)\\
\mbox{type }(2,0):&
\psi_\alpha\equiv\Pi_\alpha^aT_{ab}{}^b=
\tfrac12\bigl(T_{\alpha\beta}{}^\beta
-iT_{\alpha\beta}{}^\gamma J_\gamma{}^\beta\bigr).
\end{split}\end{equation*}
On the other hand, just from the index structure, tensors
$T_{\alpha\beta}{}^\gamma=T_{[\alpha\beta]}{}^\gamma$ of type $(0,2)$ seemingly
cannot be further decomposed (and this is confirmed in Section~\ref{BGG}).  In
any case, it follows easily from $J_\alpha{}^\gamma
T_{\beta\gamma}{}^\delta=T_{\alpha\beta}{}^\gamma J_\gamma{}^\delta$ that
$T_{\alpha\beta}{}^\gamma$ of type $(0,2)$ satisfy
$T_{\alpha\beta}{}^\beta=0=T_{\alpha\beta}{}^\gamma J_\gamma{}^\beta$.

\subsection{Complex connections}\label{complexconnections}

The geometrically useful affine connections $\nabla$ on an almost complex
manifold $(M,J)$ are those that preserve $J_\alpha{}^\beta$,
i.e.~$\nabla_\alpha J_\beta{}^\gamma=0$. We call them \emph{complex
  connections}. The space of complex connections is an affine space over the
vector space that consists of $1$-forms with values in the complex
endomorphisms $\gl(TM,J)$ of $TM$. A complex connection $\nabla$ naturally
extends to a linear connection on $\C TM$ that preserves the decomposition into
types~(\ref{types}). Indeed, preservation of type is also a sufficient
condition for an affine connection to be complex.

Given a complex connection $\nabla$, we denote by $T_{\alpha\beta}{}^\gamma$
its torsion, which is a $2$-form with values in $TM$. As such
$T_{\alpha\beta}{}^\gamma$ naturally splits according to type into a direct sum
of three components as in~(\ref{torsion_types}). A straightforward computation
shows that the $(0,2)$-component of the torsion of any complex connection
equals $-\frac{1}{4}N^J$. In particular, this component is an invariant of the
almost complex structure and cannot be eliminated by a suitable choice of
complex connection. However, all other components can be removed. To see this,
suppose $\hat\nabla$ is another complex connection. Then there is an element
$\upsilon\in T^*M\otimes\gl(TM,J)$ such that $\hat\nabla=\nabla+\upsilon$. It
follows that their torsions are related by the formula $\hat T
=T+\partial\upsilon$, where $\partial$ is the composition
\begin{align*}
T^*M\otimes\gl(TM,J)\hookrightarrow
T^*M\otimes T^*M\otimes TM&\to \Wedge^2T^*M\otimes TM\\
\upsilon_{\alpha\beta}{}^\gamma&\mapsto 2\upsilon_{[\alpha\beta]}{}^\gamma.
\end{align*}
Notice that the image of $\partial$ is spanned by $2$-forms of type $(2,0)$ and
$(1,1)$. Consequently, its cokernel can be identified with forms of type
$(0,2)$. Hence, any complex connection can be deformed in such a way that its
torsion is of type $(0,2)$. We have shown the following classical result:
 
\begin{prop}[\cite{KobNim,lich:connections}]\label{lichnerowicz}
On any almost complex manifold $(M,J)$ there is a complex connection such
that $T=-\frac{1}{4}N^J$.
\end{prop}

Since $\partial$ is not injective such a complex connection is not unique.
Complex connections $\nabla$ with $T=-\frac{1}{4}N^J$ form an affine space over
\begin{equation}\label{kerpartial}
\ker\partial=(S^2T^*M\otimes TM)\cap (T^*M\otimes\gl(TM,J))
\end{equation} 
and are called \emph{minimal connections}.

{From} Proposition \ref{lichnerowicz} and the above discussion one also
deduces immediately:

\begin{cor}\label{torsion_free_complex_connection} 
There exists a complex torsion-free connection on an almost complex
manifold $(M,J)$ if and only if $N^J\equiv 0$.
\end{cor}

\begin{rem}\label{almost_complex_natural_bundles}
We have already noted that the cokernel of $\partial$ can be identified with
tensors $T_{\alpha\beta}{}^\gamma$ such that
\[
T_{(\alpha\beta)}{}^\gamma=0\qquad J_\alpha{}^\gamma T_{\beta\gamma}{}^\delta
=T_{\alpha\beta}{}^\gamma J_\gamma{}^\delta.
\] 
Consequently, $T_{\alpha\beta}{}^\beta=T_{\alpha\beta}{}^\gamma J_\gamma{}^\beta=0$.
As we shall see in Section~\ref{BGG}, such tensors~are irreducible.
More precisely, the natural vector bundles on an almost complex manifold
$(M,J)$ correspond to representations of ${\mathrm{GL}}(n,\C)$ and we
shall see that $\coker\partial$ corresponds to an irreducible representation
of ${\mathrm{GL}}(n,\C)$. On the other hand, its kernel (\ref{kerpartial})
decomposes into two irreducible components, namely a trace-free part and a
trace part. We shall see in the next section that deforming a complex
connection by an element from the latter space exactly corresponds to changing
a connection c-projectively.
\end{rem}

\section{Elements of c-projective geometry}\label{sec:elements}

We now introduce almost c-projective structures, first from the classical
perspective of $J$-planar curves and equivalence classes of complex affine
connections~\cite{OtsTash}, then from the modern viewpoint of parabolic
geometries~\cite{csbook,Hrdina,Y}. The (categorical) equivalence between these
approaches is established in Theorem~\ref{EquivCat}. This leads us to study the
intrinsic curvature of an almost c-projective manifold, namely the harmonic
curvature of its canonical normal Cartan connection.

\subsection{Almost c-projective structures}\label{almost_c-projective}

Recall that affine connections $\nabla$ and $\hat\nabla$ on a manifold $M$ are
projectively equivalent if there is a $1$-form $\Upsilon_\alpha$ on $M$ such
that
\begin{equation}\label{projchange}
\hat\nabla_\alpha X^\gamma=\nabla_\alpha X^\gamma
+\Upsilon_\alpha X^\gamma+\delta_\alpha{}^\gamma\Upsilon_\beta X^\beta.
\end{equation}
Suppose now that $(M,J)$ is an almost complex manifold. Then $\nabla$ and
$\hat\nabla$ are called \emph{c-projectively equivalent}, if there is a (real)
$1$-form $\Upsilon_\alpha$ on $M$ such that
\begin{flalign}\label{cprojchange}
&&\hat\nabla_\alpha X^\gamma &=
\nabla_\alpha X^\gamma+\upsilon_{\alpha\beta}{}^\gamma X^\beta,&&\\
\text{where}&&
\upsilon_{\alpha\beta}{}^\gamma &:=
\tfrac{1}{2}(\Upsilon_\alpha \delta_\beta{}^\gamma+
\delta_\alpha{}^\gamma\Upsilon_\beta
-J_\alpha{}^\delta\Upsilon_\delta J_\beta{}^\gamma 
-J_\alpha{}^\gamma\Upsilon_\delta J_\beta{}^\delta).&&\nonumber
\end{flalign}
Note that $\upsilon_{\alpha\beta}{}^\gamma
J_\gamma{}^\delta=\upsilon_{\alpha\gamma}{}^\delta J_{\beta}{}^\gamma$. In
other words $\upsilon_{\alpha\beta}{}^\gamma$ is a $1$-form on $M$ with values
in $\gl(TM,J)$, which implies that if $\nabla$ is a complex connection, 
then so is $\hat\nabla$. Moreover
$\upsilon_{\alpha\beta}{}^\gamma=\upsilon_{(\alpha\beta)}{}^\gamma$ and so
c-projectively equivalent connections have the same torsion. In particular, if
$\nabla$ is minimal, then so is~$\hat\nabla$.

A smooth curve $c\colon (a,b)\to M$ is called a \emph{$J$-planar curve} with
respect to a complex connection $\nabla$, if $\nabla_{\dot c}\dot c$ lies in
the span of $\dot c$ and $J\dot c$. The notion of $J$-planar curves gives rise
to the following geometric interpretation of a c-projective equivalence class
of complex connections.

\begin{prop}[\cite{Hrdina,MikesSinyukov,OtsTash,Tashiro}]
\label{prop:Jplanar-curves} Suppose $(M,J)$ is an almost complex manifold
and let $\nabla$ and $\hat\nabla$ be complex connections on $M$ with the same
torsion. Then $\nabla$ and $\hat\nabla$ are c-projectively equivalent if and
only if they have the same $J$-planar curves.
\end{prop}
\begin{proof}
Suppose $\nabla$ and $\hat\nabla$ are complex connections with the same
torsion. If $\nabla$ and $\hat\nabla$ are c-projectively equivalent, then they
clearly have the same $J$-planar curves. Conversely, assume that $\nabla$ and
$\hat\nabla$ share the same $J$-planar curves and consider the difference
tensor $A_{\alpha\beta}{}^\gamma Y^\beta=\hat\nabla_\alpha
Y^\gamma-\nabla_\alpha Y^\gamma$. As both connections are complex and have the
same torsion, the difference tensor satisfies
$A_{\alpha\beta}{}^\gamma=A_{(\alpha\beta)}{}^\gamma$ and
$A_{\alpha\beta}{}^\gamma J_\gamma{}^\delta=A_{\alpha\gamma}{}^\delta
J_{\beta}{}^\gamma$. The fact that $\hat\nabla$ and $\nabla$ have the same
$J$-planar curves and that any tangent vector can be realised as the derivative
of such a curve implies that at any point $x\in M$ and for any nonzero vector
$Y\in T_xM$ there exist uniquely defined real numbers $\gamma(Y)$ and $\mu(Y)$
such that
\begin{equation}\label{difference_tensor}
A(Y, Y)=\gamma(Y)Y+\mu(Y)JY.
\end{equation}
Note that $\gamma$ and $\mu$ give rise to well-defined smooth functions
on~$TM\setminus 0$. Extending $\gamma$ and $\mu$ to functions on all of $TM$ by
setting $\gamma(0)=\mu(0)=0$, formula \eqref{difference_tensor} becomes valid
for any tangent vector, and by construction $\gamma$ and $\mu$ are then clearly
homogeneous of degree one. {From} $A(Y,Y)=-A(JY,JY)$ we deduce that
$\mu(X)=-\gamma(JX)$ whence
\[
A(Y,Y)=\gamma(Y)Y-\gamma(JY)JY.
\]
By polarisation we have for any tangent vectors $X$ and $Y$
\begin{align}\label{polarisation}
A(X,Y)&=\tfrac{1}{2}\bigl(A(X+Y, X+Y)-A(X,X)-A(Y,Y)\bigr)\\
&=\tfrac{1}{2}
\bigl((\gamma(X+Y)-\gamma(X))X+(\gamma(X+Y)-\gamma(Y))Y\bigr)\nonumber\\
&\quad-\tfrac{1}{2}
\bigl((\gamma(JX+JY)-\gamma(JX))JX-(\gamma(JX+JY)-\gamma(JY))JY\bigr).\nonumber
\end{align}
Suppose that $X$ and $Y$ are linearly independent and expand the identity
$A(X,tY)=tA(X,Y)$ for all $t\in\R$ using \eqref{polarisation}. Then a
comparison of coefficients shows that
$$\gamma(X+tY)-t\gamma(Y)=\gamma(X+Y)-\gamma(Y).$$
Taking the limit $t\rightarrow 0$, shows that $\gamma(X+Y)=\gamma(X)+\gamma(Y)$.
Hence, $\gamma$ defines a (smooth) $1$-form and 
\begin{align*}
A(X,Y)&=\tfrac{1}{2}\bigl(A(X+Y, X+Y)-A(X,X)-A(Y,Y)\bigr)\\
&=\tfrac{1}{2}
\bigl(\gamma(X)Y+\gamma(Y)X-\gamma(JX)(JY)-\gamma(JY)JX\bigr),
\end{align*}
for any tangent vector $X$ and $Y$ as desired. 
\end{proof}

\begin{defin}\label{definition_cprojective_structure} 
Suppose that $M$ is manifold of real dimension $2n\geq4$. 
\begin{enumerate}
\item An \emph{almost c-projective structure} on $M$ consists of an almost
complex structure $J$ on $M$ and a c-projective equivalence class $[\nabla]$ of
minimal complex connections.
\item The \emph{torsion} of an almost c-projective structure $(M, J, [\nabla])$
is the torsion $T$ of one, hence any, of the connections in $[\nabla]$, 
i.e.~$T=-\frac{1}{4}N^J$.
\item An almost c-projective structure $(M, J,[\nabla])$ is called a
\emph{c-projective structure}, if $J$ is integrable. (This is the case if and 
only if some and hence all connections in the c-projective class are 
torsion-free.)
\end{enumerate}
\end{defin}

\begin{rem}
If $M$ is a $2$-dimensional manifold, any almost complex structure $J$ is
integrable and any two torsion free complex connections are c-projectively
equivalent. Therefore, in this case one needs to modify the definition of a
c-projective structure in order to have something nontrivial
(cf.~\cite{DCMoebius1,DCMoebius2}). We shall not pursue this here.
\end{rem}
\begin{rem}
Recall that the geodesics of an affine connection can be also realised as the
geodesics of a torsion-free connection; hence the definition of a projective
structure as an equivalence class of torsion-free connections does not
constrain the considered families of geodesics. The analogous statement for
$J$-planar curves does not hold: the $J$-planar curves of a complex connection
cannot in general be realised as the $J$-planar curves of a minimal connection.
We discuss the motivation for the restriction to minimal connections in the
definition of almost c-projective manifolds in
Remark~\ref{rem:non-normal-case}.
\end{rem}

\begin{defin}\label{CprojTransDefinition} Let $(M, J_M, [\nabla^M])$
and $(N, J_N, [\nabla^N])$ be almost c-projective manifolds of dimension
$2n\geq 4$.  A diffeomorphism $\Phi\colon M\to N$ is called \emph{c-projective
  transformation or automorphism}, if $\Phi$ is complex (i.e.\, $T\Phi\circ
J_M=J_N\circ T\Phi$) and for a (hence any) connection $\nabla^N\in[\nabla^N]$
the connection $\Phi^*\nabla^N$ is a connection in $[\nabla^M]$.
\end{defin}

{From} Proposition \ref{prop:Jplanar-curves} one deduces straightforwardly
that also the following characterisation of c-projective transformations
holds:

\begin{prop}\label{CprojTransCharacterisation} 
Let $(M, J_M, [\nabla^M])$ and $(N, J_N, [\nabla^N])$ be almost c-projective
manifolds of dimension $2n\geq 4$.  Then a complex diffeomorphism $\Phi\colon
M\to N$ is a c-projective transformation if and only if $\Phi$ maps
$J_M$-planar curves to $J_N$-planar curves.
\end{prop}

Suppose that $(M, J, [\nabla])$ is an almost c-projective manifold. Let
$\hat\nabla$ and $\nabla$ be connections of the c-projective class $[\nabla]$
that differ by $\Upsilon_\alpha$ as in (\ref{cprojchange}). Then $\hat\nabla$
and $\nabla$ give rise to linear connections on $\C TM=T^{1,0}M\oplus T^{0,1}M$
that preserve the decomposition into types. Hence, they induce connections on
the complex vector bundles $T^{1,0}M$ and $T^{0,1}M$. To deduce the difference
between the connections $\hat\nabla$ and $\nabla$ on $T^{1,0}M$ (respectively
$T^{0,1}M$), we just need to apply the splittings $\Pi_\alpha^a$ and
$\Pi_a^\alpha$ (respectively their conjugates) from the previous section to
(\ref{cprojchange}). Using the identities~\eqref{usefulidentities}, we obtain
\begin{align*}
\Pi_a^\alpha\Pi_b^\beta\upsilon_{\alpha\beta}{}^\gamma\Pi_\gamma^c&=
\tfrac12\Pi_a^\alpha\Pi_b^\beta(\Upsilon_\alpha \delta_\beta{}^\gamma+
\delta_\alpha{}^\gamma\Upsilon_\beta
-J_\alpha{}^\delta\Upsilon_\delta J_\beta{}^\gamma 
-J_\alpha{}^\gamma\Upsilon_\delta J_\beta{}^\delta)\Pi_\gamma^c\\
&=\tfrac12\Pi_a^\alpha\Pi_b^\beta\bigl(
(\Upsilon_\alpha-iJ_\alpha{}^\delta\Upsilon_\delta)\Pi_\beta^c
+(\Upsilon_\beta-iJ_\beta{}^\delta\Upsilon_\delta)\Pi_\alpha^c\bigr)\\
&=\tfrac12
\Pi_a^\alpha(\Upsilon_\alpha-iJ_\alpha{}^\delta\Upsilon_\delta)\delta_b{}^c
+\tfrac12
\Pi_b^\beta(\Upsilon_\beta-iJ_\beta{}^\delta\Upsilon_\delta)\delta_a{}^c\\
&=\Upsilon_a\delta_b{}^c+\Upsilon_b\delta_a{}^c,\qquad\mbox{where}\quad
\Upsilon_a\equiv\Pi_a^\alpha\Upsilon_\alpha\,.
\end{align*}
Similarly, we find that $\overline{\Pi}{}_{\bar{a}}^\alpha\Pi_b^\beta
\upsilon_{\alpha\beta}{}^\gamma\Pi_\gamma^c=0$. These identities are the key
to the following:

\begin{prop}\label{change10}
Suppose $(M,J,[\nabla])$ is an almost c-projective manifold of dimension
$2n\geq4$. Assume two connections $\hat\nabla$ and $\nabla$ in $[\nabla]$
differ by $\Upsilon_\alpha$ as in {\rm(\ref{cprojchange})}, and set
$\Upsilon_a:=\Pi_a^\alpha\Upsilon_\alpha$ and
$\Upsilon_{\bar{a}}:=\Pi_{\bar{a}}^\alpha\Upsilon_\alpha$. Then we have the
following transformation rules for the induced connections on $T^{1,0}M$ and
$T^{0,1}M$.
\begin{enumerate}
\item $\hat\nabla_aX^c=\nabla_aX^c+\Upsilon_aX^c+\delta_a{}^c\Upsilon_bX^b$
and $\hat\nabla_{\bar a}X^c=\nabla_{\bar a}X^c$,
\item $\hat\nabla_{\bar a}X^{\bar c}=\nabla_{\bar a}X^{\bar c}+
\Upsilon_{\bar a}X^{\bar c}+
\delta_{\bar a}{}^{\bar c}\Upsilon_{\bar b}X^{\bar b}$
and $\hat\nabla_{a}X^{\bar c}=\nabla_{a}X^{\bar c}$.
\end{enumerate}
\end{prop}
\begin{proof} We compute
\begin{align*}\hat\nabla_aX^c&=
\Pi_a^\alpha\hat\nabla_\alpha(\Pi_\gamma^c\Pi_b^\gamma X^b)=
\Pi_a^\alpha\Pi_\gamma^c\hat\nabla_\alpha(\Pi_b^\gamma X^b)\\
&=\Pi_a^\alpha\Pi_\gamma^c\nabla_\alpha(\Pi_b^\gamma X^b)
+\Pi_a^\alpha\Pi_\gamma^c\upsilon_{\alpha\beta}{}^\gamma\Pi_b^\beta X^b\\
&=\Pi_a^\alpha\nabla_\alpha(\Pi_\gamma^c\Pi_b^\gamma X^b)
+(\Pi_a^\alpha\Pi_b^\beta\upsilon_{\alpha\beta}{}^\gamma\Pi_\gamma^c)X^b\\
&=\nabla_aX^c+(\Upsilon_a\delta_b{}^c+\Upsilon_b\delta_a{}^c)X^b=
\nabla_aX^c+\Upsilon_aX^c+\delta_a{}^c\Upsilon_bX^b,
\end{align*}
as required. The remaining calculations are similar.
\end{proof}
\begin{rem}\label{rem:proj-vs-cproj} The differential
operator $\nabla_{\bar{a}}\colon T^{1,0}M\to \Wedge^{0,1}M\otimes T^{1,0} M$
is c-projectively invariant, as is its conjugate $\nabla_{a}\colon T^{0,1}M\to
\Wedge^{1,0}M\otimes T^{0,1} M$.  (Here and throughout, the domain and
codomain of a differential operator are declared as bundles, although the
operator is a map between corresponding spaces of sections.) This is
unsurprising: it is the usual $\bar{\partial}$-operator on an almost complex
manifold whose kernel (in the integrable case) comprises the holomorphic
vector fields.

In contrast, the transformation rules for $\nabla_{a}\colon T^{1,0}M\to
\Wedge^{1,0}M\otimes T^{1,0}M$ and its conjugate are analogues of projective
equivalence~\eqref{projchange} in the $(1,0)$ and $(0,1)$ directions
respectively. When $(M,J)$ is real-analytic and the c-projective class
contains a real-analytic connection $\nabla$, this can be made precise by
extending $J$ and $[\nabla]$ to a complexification $M^\C$ of $M$, so that
$T^{1,0}M$ and $T^{0,1}M$ extend to distributions on $M^\C$.  If $J$ is
integrable, these distributions integrate to two foliations of $M^\C$, and
$[\nabla]$ induces projective structures on the leaves of these foliations.
\end{rem}

Taking the trace in equation (\ref{cprojchange}) and in the formulae in
Proposition~\ref{change10}, we deduce:
 
\begin{cor}\label{changedensity}
On an almost c-projective manifold $(M,J,[\nabla])$, the transformation rules
for the induced linear connections on $\Wedge^{2n}TM$, $\Wedge^{n}T^{1,0}M$,
and $\Wedge^{n} T^{0,1}M$ are\textup:
\begin{enumerate}
\item $\hat\nabla_\alpha\Sigma=
\nabla_\alpha\Sigma+(n+1)\Upsilon_\alpha\Sigma$,\enskip
for $\Sigma\in\Gamma(\Wedge^{2n}TM)$
\item $\hat\nabla_a\sigma=\nabla_a\sigma+(n+1)\Upsilon_a\sigma$ 
and $\hat\nabla_{\bar a}\sigma=\nabla_{\bar a}\sigma$,\enskip
for $\sigma\in\Gamma(\Wedge^{n}T^{1,0}M)$
\item $\hat\nabla_{\bar a}\bar\sigma=
\nabla_{\bar a}\bar\sigma+(n+1)\Upsilon_{\bar a}\bar\sigma$ 
and $\hat\nabla_{ a}\bar\sigma=\nabla_{a}\bar\sigma$,\enskip
for $\bar\sigma\in\Gamma(\Wedge^{n} T^{0,1}M)$.
\end{enumerate}
\end{cor}

For the convenience of the reader, let us also record the transformation rules
for the induced connections on~$T^*M$, respectively $\Wedge^{1,0}$ and
$\Wedge^{0,1}$. If two complex connections $\hat\nabla$ and $\nabla$ are
related via $\Upsilon_\alpha$ as in~(\ref{cprojchange}), then the induced
connections on $T^*M$ are related by
\begin{equation}
\hat\nabla_\alpha \phi_\gamma = \nabla_\alpha \phi_\gamma-
\tfrac{1}{2}(\Upsilon_\alpha\phi_\gamma+
\Upsilon_\gamma \phi_\alpha-
J_\alpha{}^\beta\Upsilon_\beta J_\gamma{}^\delta \phi_\delta
-J_\alpha{}^\beta\phi_\beta J_\gamma{}^\delta \Upsilon_\delta).
\end{equation}
Therefore, we obtain:
\begin{prop}\label{changeform}
Suppose $(M,J,[\nabla])$ is an almost c-projective manifold of dimension
$2n\geq4$. Assume two connections $\hat\nabla$ and $\nabla$ in $[\nabla]$
differ by $\Upsilon_\alpha$ as in {\rm(\ref{cprojchange})} and set
$\Upsilon_a:=\Pi_a^\alpha\Upsilon_\alpha$ and
$\Upsilon_{\bar{a}}:=\Pi_{\bar{a}}^\alpha\Upsilon_\alpha$. Then we have the
following transformation rules for the induced connections on $\Wedge^{1,0}$
and $\Wedge^{0,1}$\textup:
\begin{enumerate}
\item $\hat\nabla_a\phi_c=\nabla_a\phi_c-\Upsilon_a\phi_c-\phi_a\Upsilon_c$
and $\hat\nabla_{\bar a}\phi_c=\nabla_{\bar a}\phi_c$\,,
\item $\hat\nabla_{\bar a}\phi_{\bar c}=\nabla_{\bar a}\phi_{\bar c}
-\Upsilon_{\bar a}\phi_{\bar c}-\phi_{\bar a}\Upsilon_{\bar c}$
and $\hat\nabla_{a}\phi_{\bar c}=\nabla_{a}\phi_{\bar c}$\,.
\end{enumerate}
\end{prop}
Note that the real line bundle $\Wedge^{2n}TM$ is oriented and hence admits
oriented roots.  We denote $(\Wedge^{2n}TM)^{\frac{1}{n+1}}$ by $\cE_\R(1,1)$
and for any $k\in\Z$ we set $\cE_\R(k,k):= \cE(1,1)^{\otimes k}$, where
$\cE_\R(k,k)^*=\cE_\R(-k,-k)$.  It follows from Corollary~\ref{changedensity}
that for a section $\Sigma$ of $\cE_\R(k,k)$ we have
\begin{equation}\label{changes_on_realdensities}
\hat\nabla_\alpha\Sigma=\nabla_\alpha\Sigma+k\Upsilon_\alpha\Sigma.
\end{equation}
In particular, we immediately deduce the following result.
\begin{prop}\label{scalebundle} Suppose $(M, J,[\nabla])$ is an almost
c-projective manifold of dimension $2n\geq 4$. The map sending an affine
connection to its induced connection on $\cE_\R(1,1)$ induces a bijection from
connections in $[\nabla]$ to linear connections on $\cE_\R(1,1)$.
\end{prop} 
Since $\Wedge^{2n}TM$ and $\cE_\R(1,1)$ are oriented, they can be trivialised
by choosing a positive section. Such a positive section $\scale$ of
$\cE_\R(1,1)$ gives rise to a linear connection on $\cE_\R(1,1)$ by decreeing
that $\scale$ is parallel and therefore, by Proposition \ref{scalebundle}, to a
connection in the c-projective class. We call a connection $\nabla\in[\nabla]$
that arises in this way a \emph{special connection}.  Suppose $\hat\scale$ and
$\scale$ are two nowhere vanishing sections of $\cE_\R(1,1)$ and denote by
$\hat\nabla$ and $\nabla$ the corresponding connections. Then
$\hat\scale=e^{-f}\scale$ for some smooth function $f$ on $M$ and any
$\sigma\in\Gamma(\cE_\R(1,1))$ can be written as
$\sigma=h\scale=he^f\hat\scale$ for a smooth function $h$ on $M$.  Since
$\nabla\sigma=dh\otimes \scale$, we have
\begin{equation*}
\hat\nabla\sigma=d(he^f)\otimes \hat\scale= dh\otimes\scale+df\otimes\sigma
=\nabla\sigma+(\nabla f)\sigma.
\end{equation*}
Therefore, $\hat\nabla$ and $\nabla$ differ by an exact $1$-form, namely
$\Upsilon_\alpha\equiv\nabla_\alpha f$.

In some of the following sections, like for instance in
Section~\ref{standard_tractors}, we shall assume also that the complex line
bundle $\Wedge^n T^{1,0}M$ admits a $(n+1)^{\mathrm{st}}$ root and that we have
chosen one, which we will denote by $\cE(1,0)$ (following a standard notation
on~$\CP^n$). In that case we shall denote its conjugate bundle
$\overline{\cE(1,0)}$ by $\cE(0,1)$ and the dual bundle $\cE(1,0)^*$ by
$\cE(-1,0)$. In general, we shall also write $\cE(k,\ell):=\cE(1,0)^{\otimes k}
\otimes \cE(0,1)^{\otimes \ell}$ for $(k,\ell)\in\Z\times\Z$ and refer to its
sections as \emph{c-projective densities of weight}~$(k,\ell)$. By
Corollary~\ref{changedensity} we see that, for a c-projective density $\sigma$
of weight~$(k,\ell)$, we have
\begin{equation}\label{changes_on_densities}
\hat\nabla_a\sigma=\nabla_a\sigma+k\Upsilon_a\sigma\qquad
\hat\nabla_{\bar{a}}\sigma=
\nabla_{\bar{a}}\sigma+\ell\Upsilon_{\bar{a}}\sigma.
\end{equation}
Our notion of c-projective density means, in particular, that we may identify
$\Wedge^{n,0}$ with $\cE(-n-1,0)$ and it is useful to have a notation for this
change of viewpoint.  Precisely, we may regard our identification
$\cE(-n-1,0)\mapsisoto\Wedge^{n,0}$ as a tautological section $\epst_{ab\cdots
  c}$ of $\Wedge^{n,0}(n+1,0)$, such that a c-projective density $\rho$ of
weight $(-n-1,0)$ corresponds to $\rho\epst_{ab\cdots c}$, a form of type
$(n,0)$. Note that $\cE(k,k)\cong \cE_\R(k,k)\otimes\C$.

\subsection{Parabolic geometries}\label{intropargeom}

For the convenience of the reader we recall here some basics of parabolic
geometries; for a comprehensive introduction see~\cite{csbook}.

A \emph{parabolic geometry} on a manifold $M$ is a Cartan geometry of type
$(G,P)$, where $G$ is a semisimple Lie group and $P\subset G$ a so-called
\emph{parabolic subgroup}. Hence, it is given by the following data:
\begin{itemize}
\item a principal $P$-bundle $p\colon\G\to M$ 
\item a Cartan connection $\omega\in\Omega^1(\G,\g)$---that is, a
$P$-equivariant $1$-form on $\G$ with values in $\g$ defining a trivialisation
$T\G\cong\G\times\g$ and reproducing the generators of the fundamental vector
fields,
\end{itemize}
where $\g$ denotes the Lie algebra of $G$. Note that the projection
$G\to G/P$, equipped with the (left) Maurer--Cartan form
$\omega_{G}\in\Omega^1(G,\g)$ of $G$, defines a parabolic geometry on $G/P$,
which is called the \emph{homogeneous} or \emph{flat model} for parabolic
geometries of type $(G,P)$.

The \emph{curvature} of a parabolic geometry $(\G\stackrel{p}{\to}M,\omega)$ is
a $2$-form $K$ on $\G$ with values in~$\g$, defined by
\[
K(\chi,\xi)=d\omega(\chi,\xi)+[\omega(\chi),\omega(\xi)]\, \text{ for vector
  fields $\chi$ and $\xi$ on }\G,
\] 
where $d$ denotes the exterior derivative and $[\enskip,\enskip]$ the Lie
bracket of $\g$.

The curvature of the homogeneous model $(G\to G/P,\omega_{G})$
vanishes identically. Furthermore, the curvature $K$ of a parabolic geometry
of type $(G,P)$ vanishes identically if and only if it is locally isomorphic
to $(G\to G/P,\omega_{G})$. Thus, the curvature $K$ measures the
extent to which the geometry differs from its homogeneous model.

Given a parabolic geometry $(\G\stackrel{p}{\to}M,\omega)$ of type $(G,P)$, any
representation $\E$ of $P$ gives rise to an associated vector bundle
$E:=\G\times_P\E$ over~$M$. These are the natural vector bundles on a parabolic
geometry. Notice that the Cartan connection $\omega$ induces an isomorphism
\begin{align*}
\G\times_P\g/\p&\cong TM\\
{[u,X+\p]}&\mapsto T_up\bigl(\omega^{-1}(X)\bigr),
\end{align*}
where $\p$ denotes the Lie algebra of $P$ and the action of $P$ on $\g/\p$ is
induced by the adjoint action of~$G$. Similarly, $\omega$ allows us to identify
all tensor bundles on $M$ with associated vector bundles. The vector bundles
corresponding to $P$-modules obtained by restricting a representation of $G$ to
$P$ are called \emph{tractor bundles}. These bundles play an important role in
the theory of parabolic geometries, since the Cartan connection induces linear
connections, called \emph{tractor connections}, on these bundles. An important
example of a tractor bundle is the \emph{adjoint tractor} bundle $\A
M=\G\times_P\g$, which has a canonical projection to $TM$ corresponding to the
$P$-equivariant projection $\g\to\g/\p$.

\begin{rem}
The abstract theory of tractor bundles and connections even
provides an alternative description of parabolic geometries (see~\cite{CG}).
\end{rem}

By normalising the curvature of a parabolic geometry, the prolongation
procedures of~\cite{CSch,Morimoto,Tanaka} leads to an equivalence of
categories between so-called \emph{regular normal} parabolic geometries and
certain underlying structures, which may be described in more conventional
geometric terms. Among the most prominent of these are conformal structures,
projective structures, and CR-structures of hypersurface type. In the next
section we shall see that almost c-projective manifolds form another class of
examples.

{From} the defining properties of a Cartan connection it follows immediately
that the curvature $K$ of a parabolic geometry of type $(G,P)$ is
$P$-equivariant and horizontal. Hence, $K$ can be identified with a section of
the vector bundle $\Wedge^2T^*M\otimes\A M$ and therefore corresponds via
$\omega$ to a section $\kappa$ of the vector bundle
$$\G\times_P\Wedge^2(\g/\p)^*\otimes\g
\cong\G\times_P\Wedge^2\p_+\otimes\g,$$
where $\p_+$ is the nilpotent radical of $\p$ and the latter isomorphism is
induced by the Killing form of~$\g$. Now consider the complex for computing the
Lie algebra homology $H_*(\p_+,\g)$ of $\p_+$ with values in~$\g$:
\begin{equation*}
0\leftarrow\g\stackrel{\,\,\,\partial^*}{\leftarrow}\p_+\otimes\g
\stackrel{\,\,\,\partial^*}{\leftarrow}\Wedge^2\p_+\otimes\g\leftarrow\ldots
\end{equation*}
Since the linear maps $\partial^*$ are $P$-equivariant, they induce vector
bundle maps between the corresponding associated vector bundles. Moreover, the
homology spaces $H_i(\p_+,\g)$ are naturally $P$-modules and therefore give
rise to natural vector bundles. A parabolic geometry is called \emph{normal},
if $\partial^*\kappa=0$. In this case, we can project $\kappa$ to a section
$\kappa_h$ of $\G\times_PH_2(\p_+,\g)$, called the \emph{harmonic curvature}.
The spaces $H_i(\p_+,\g)$ are completely reducible $P$-modules and hence arise
as completely reducible representations of the reductive Levi factor $G_0$ of
$P$ via the projection $P\to P/\exp(\p_+)=G_0$. In particular, the
harmonic curvature is a section of a completely reducible vector bundle, which
makes it a much simpler object than the full curvature.  Moreover, using the
Bianchi identities of~$\kappa$, it can be shown that the harmonic curvature is
still a complete obstruction to local flatness:

\begin{prop}[see e.g.~\cite{csbook}] \label{harmcurv}
Suppose that $(\G\to M,\omega)$ is a regular normal
parabolic geometry. Then $\kappa\equiv 0$ if and only if $\kappa_h\equiv 0$.
\end{prop} 

\begin{rem}
The machinery of BGG sequences shows that the curvature of a regular normal
parabolic geometry can be reconstructed from the harmonic curvature by applying
a BGG splitting operator (see~\cite{CD}).
\end{rem}

\subsection{Almost c-projective manifolds as parabolic geometries}
\label{cproCartan}

It is convenient for our purposes to realise the Lie algebra
$\g:=\sgl(n+1,\C)$ of complex trace-free linear endomorphisms of
$\C^{n+1}$ as block matrices of the form
\begin{equation}\label{block_decom}
\g=\left\{
\begin{pmatrix}-\mathrm{tr}\,A & Z \\ X & A\end{pmatrix}: 
A\in\g\mathfrak l(n,\C), X\in\C^n, Z\in(\C^n)^*\right\},
\end{equation}
where $\mathrm{tr}\colon {\mathfrak{gl}}(n,\C)\to \C$ denotes the
trace. The block form equips $\g$ with the structure of a graded Lie algebra:
\begin{equation*}
\g=\g_{-1}\oplus\g_0\oplus\g_1,
\end{equation*} 
where $\g_0$ is the block diagonal subalgebra isomorphic to $\gl(n,\C)$ and
$\g_{-1}\cong\C^n$, respectively $\g_{1}\cong(\C^n)^*$, as $\g_0$-modules. Note
that the subspace $\p:=\g_0\oplus\g_1$ is a subalgebra of $\g$ (with
$\p\cong \g_0\ltimes\g_1$ as Lie algebra). Furthermore, $\p$ is a parabolic
subalgebra with Abelian nilpotent radical $\p_+:=\g_1$ and Levi factor
isomorphic to~$\g_0$. For later purposes let us remark here that we may
conveniently decompose an element $A\in\g_0$ into its trace-free part and into
its trace part as follows
\begin{equation}\label{decompg0}
\begin{pmatrix}
0& 0 \\ 0 & A -\frac{\mathrm{tr}\,A}{n} \Id_n\end{pmatrix}
\;+\;
\frac{n+1}{n}\mathrm{tr}\,A\!
\begin{pmatrix}
-\frac{n}{n+1} & 0 \\ 0 & \frac{1}{n+1}\Id_n\end{pmatrix}.
\end{equation}

Now set $G:={\mathrm{PSL}}(n+1,\C)$ and let $P$ be the stabiliser in $G$ of the
complex line generated by the first standard basis vector of $\C^{n+1}$. Let
$G_0$ be the subgroup of $P$ that consists of all elements $g\in P$ whose
adjoint action $\mathrm{Ad}(g)\colon \g\to\g$ preserve the grading.
Hence, it consists of equivalence classes of matrices of the form
\begin{equation*}
\begin{pmatrix} (\det_\C C)^{-1} & 0 \\ 0 & C\end{pmatrix}
\quad\text{ where } C\in {\mathrm{GL}}(n,\C),
\end{equation*}
and the adjoint action of $G_0$ on $\g$ induces an isomorphism
$$G_0\cong{\mathrm{GL}}(\g_{-1}, \C)\cong{\mathrm{GL}}(n,\C).$$
Obviously, the subgroups $G_0$ and $P$ of $G$ have corresponding Lie algebras
$\g_0$ and~$\p$, respectively.

{From} now on we shall view $G_0\subset P\subset G$ as real Lie groups in
accordance with the identification of ${\mathrm{GL}}(n+1,\C)$ with the real
subgroup of ${\mathrm{GL}}(2n+2,\R)$ that is given by
\[
{\mathrm{GL}}(2n+2,\J_{2(n+1)})=\left\{A\in
\mathrm{GL}(2(n+1),\R): A\J_{2(n+1)}=\J_{2(n+1)}A\right\},
\]
where $\J_{2(n+1)}$ is the following complex structure on $\R^{2n+2}$:
\begin{equation*} \J_{2(n+1)}=
\begin{pmatrix} \J_2 &  &\\ &\ddots& \\
&&\J_2\end{pmatrix}
\quad\text{ with}\enskip
\J_2=\begin{pmatrix}0&-1\\1&0\end{pmatrix}.
\end{equation*}

Suppose now that $(M,J,[\nabla])$ is an almost c-projective manifold of real
dimension $2n\geq4$. Then $J$ reduces the frame bundle ${\mathcal F}M$ of $M$ to
a principal bundle $p_0\colon\G_0\to M$ with structure group $G_0$
corresponding to the group homomorphism
$$G_0\cong {\mathrm{GL}}(n,\C)\cong {\mathrm{GL}}(2n,\J_{2n})
\hookrightarrow {\mathrm{GL}}(2n,\R).$$
The general prolongation procedures of~\cite{CSch,Morimoto,Tanaka} further show
that $\G_0\to M$ can be canonically extended to a principal $P$-bundle
$p\colon\G\to M$, equipped with a normal Cartan connection
$\omega\in\Omega^1(\G,\g)$ of type $(G,P)$. Moreover,
$(\G\stackrel{p}{\to}M,\omega)$ is uniquely defined up to isomorphism and these
constructions imply:

\begin{thm}[see also~\cite{Hrdina,Y}]\label{EquivCat}
There is an equivalence of categories between almost c-projective manifolds of
real dimension $2n\geq 4$ and normal parabolic geometries of type $(G,P)$,
where $G$ and $P$ are viewed as real Lie groups. The homogeneous model
$(G\to G/P,\omega_{G})$ corresponds to the c-projective manifold
\[
(\CP^n, J_{\mathrm{can}}, [\nabla^{g_{FS}}]),
\] 
where $J_{\mathrm{can}}$ denotes the canonical complex structure on
$\CP^n$ and $\nabla^{g_{FS}}$ the Levi-Civita connection of the
Fubini--Study metric $g_{FS}$.
\end{thm}

Let us explain briefly how the Cartan bundle $\G$ and the normal Cartan
connection $\omega$ of an almost c-projective manifold $(M,J,[\nabla])$ of
dimension $2n\geq4$ are constructed.  The reduction $\G_0\to{\mathcal F}M$ is
determined by the pullback of the soldering form on ${\mathcal F}M$ and
hence can be encoded by a strictly horizontal $G_0$-equivariant $1$-form
$\theta\in\Omega^1(\G_0,\g_{-1})$. Recall also that any connection
$\nabla\in[\nabla]$ can be equivalently viewed as a principal connection
$\gamma^\nabla\in\Omega^1(\G_0,\g_{0})$ on $\G_0$. Then $\G$ is defined to be
the disjoint union $\sqcup_{u\in\G_0} \G_{u}$, where
\begin{equation*}
\G_{u}:=\{\theta(u)+\gamma^\nabla(u): 
\nabla\in[\nabla]\}\quad\text{ for any }u\in\G_0.
\end{equation*}
The projection $p:= p_0\circ q\colon\G\to M$, where
$\G\stackrel{q}{\to}\G_0\stackrel{p_0}{\to}M$, naturally acquires the structure
of a $P$-principal bundle. Any element $p\in P$ can be uniquely written as
$p=g_0\exp(Z)$, where $g_0\in G_0$ and $Z\in\g_1$. The right action of an
element $g_0\exp(Z)\in P$ on an element $\theta(u)+\gamma^\nabla(u)\in\G_u$ is
given by
\begin{equation}\label{rightaction}
(\theta(u)+\gamma^\nabla(u))\cdot g_0\exp(Z):=
\theta(u\cdot g_0)(\cdot)+\gamma^\nabla(u\cdot g_0)(\cdot)
+[Z,\theta(u\cdot g_0)(\cdot)],\end{equation}
where $[\enskip,\enskip]$ denotes the Lie bracket
$\g_1\times\g_{-1}\to\g_0$.

\begin{rem}
The soldering form $\theta\in\Omega^1(\G_0,\g_{-1})$ gives rise to isomorphisms
$TM\cong\G_0\times_{G_0}\g_{-1}$ and $T^*M\cong\G_0\times_{G_0}\g_{1}$. For
elements $X\in \g_{-1}$ and $Z\in\g_1$, the Lie bracket
$[Z,X]\in\g_0\cong\gl(\g_{-1},\J_{2n})$ evaluated on an element
$Y\in\g_{-1}$ is given by
\begin{equation}\label{bracket}
[[Z,X],Y]=-(ZXY+ZYX-Z\J_{2n}X\J_{2n}Y-
Z\J_{2n}Y\J_{2n}X).
\end{equation} 
This shows that changing a connection form $\theta+\gamma^\nabla$ by a
$G_0$-equivariant function $Z\colon\G_0\to \g_1$ according to
(\ref{rightaction}) corresponds precisely to changing it c-projectively
(cf.~formula \eqref{cprojchange}).
\end{rem}

The definition of $\G$ easily implies that the following holds.
\begin{cor}
The projection $q\colon\G\to\G_0$ is a trivial principal bundle with
structure group $P_+:=\exp(\p_+)$ and its global $G_0$-equivariant sections,
called Weyl structures, are in bijection with principal connections in the
c-projective class. Moreover, any Weyl structure
$\sigma\colon\G_0\to\G$ induces an vector bundle isomorphism
\begin{align*}
\G_0\times_{G_0}\E&\cong\G\times_{P}\E\\
{[u,X]}&\mapsto[\sigma(u),X],
\end{align*}
for any $P$-module~$\E$.
\end{cor}

Note that there is a tautological $1$-form
$\nu\in\Omega^1(\G,\g_{-1}\oplus\g_0)$ on $\G$ given by
\begin{equation}\label{taut}
\nu(\theta(u)+\gamma^\nabla(u))(\xi):=
(\theta(u)+\gamma^\nabla(u))((Tq)\xi).
\end{equation}
Extending this form to a normal Cartan connection $\omega\in\Omega^1(\G,\g)$
establishes the equivalence of categories in Theorem~\ref{EquivCat}.

\begin{rem}
In Section~\ref{almost_c-projective} we observed that there are always
so-called special connections in the c-projective class. A Weyl structure
corresponding to a special connection is precisely what in the literature on
parabolic geometries is called an \emph{exact Weyl structure} (see 
\cite{csbook,CSl}). The name is due to the fact that they form an affine 
space over the space of exact $1$-forms on $M$.
\end{rem}

Note also that the almost complex structure $J$ on $M$ induces an almost
complex structure $J^{\G_0}$ on the complex frame bundle $\G_0$ of~$M$. If $J$
is integrable, so is $J^{\G_0}$ and $\G_0$ is a holomorphic vector bundle over
$M$. Moreover, the complex structure on $\g$ induces, by means of the
isomorphism $\omega\colon T\G\cong\G\times\g$, an almost complex structure
$J^\G$ on $\G$, satisfying $Tp\circ J^\G=J\circ Tp$ and $Tq\circ
J^\G=J^{\G_0}\circ Tq$. Note that the definition of the almost complex
structure on $J^{\G_0}$ and $J^{\G}$ implies that $\theta$ and $\omega$ are of
type~$(1,0)$.

Let us also remark that an immediate consequence of Theorem~\ref{EquivCat} and
the Liouville Theorem for Cartan geometries (see e.g.~\cite[Proposition
  1.5.3]{csbook}) is the following classical result.

\begin{prop}\label{CprojTrans_CPn}  For $n\geq 2$ the c-projective
transformations of $(\CP^n, J_{\mathrm{can}}, [\nabla^{g_{FS}}])$
\textup(which by Proposition~\textup{\ref{CprojTransCharacterisation}} are the
complex diffeomorphisms of $\CP^n$ that map complex lines to complex
lines\textup) are precisely given by the left multiplications of elements in
$\emph{PSL}(n+1,\C)$.  Moreover, any local c-projective transformation of
$(\CP^n, J_{\mathrm{can}}, [\nabla^{g_{FS}}])$ uniquely extends to a
global one.
\end{prop}

We finish this section by introducing some notation. The $P$-module $\g$ admits
an invariant filtration $\g\supset \p\supset \g_1$ and hence the adjoint
tractor bundle $\A M:=\G\times_P\g$ is naturally filtered
\begin{equation*}
\A M=\A^{-1}M\supset\A^0M\supset \A^1M,
\end{equation*}
with $\A^1M\cong T^*M$ and  $\A M/\A^0 M \cong TM$.
Hence, the associated graded vector bundle to $\A M$ is given by
\begin{equation}\label{associated_graded_adjoint_tractor}
\mathrm{gr}(\A M)=\mathrm{gr}_{-1}(\A M)\oplus\mathrm{gr}_0(\A M)\oplus
\mathrm{gr}_1(\A M)=TM\oplus\gl(TM,J)\oplus T^*M,
\end{equation}
which can be identified with $\G_0\times_{G_0}\g$.

\subsection{The curvature of the canonical Cartan connection}
\label{almost_cproj_curvature}

Suppose $\sigma\colon \G_0\to\G$ is a Weyl structure and let
$\gamma^\nabla$ be the corresponding principal connection in the c-projective
class. Since the normal Cartan connection $\omega$ is $P$-equivariant and
$\sigma$ is $G_0$-equivariant, the pullback $\sigma^*\omega\in\Omega^1(\G_0,\g)$
is $G_0$-equivariant and hence decomposes according to the grading on $\g$ into
three components. Since $\omega$ extends the tautological form $\nu$ on $\G$, 
defined by~(\ref{taut}), we deduce that
\begin{equation}\label{cartanconn}
\sigma^*\omega=\theta+\gamma^\nabla-{\mathsf p}^\nabla,
\end{equation}
where ${\mathsf p}^\nabla\in\Omega^1(\G_0,\g_1)$ is horizontal and
$G_0$-equivariant and hence can be viewed as a section $\Rho^\nabla$ of
$T^*M\otimes T^*M$, called the \emph{Rho tensor} of~$\nabla$. Via~$\sigma$, the
curvature $\kappa\in\Omega^2(M,\A M)$ of $\omega$ can be identified with a
section $\kappa^\sigma$ of
\begin{multline*}
\Wedge^2 T^*M\otimes\mathrm{gr}(\A M)\\
=(\Wedge^2T^*M\otimes TM)\oplus(\Wedge^2T^*M\otimes\gl(TM,J))
\oplus(\Wedge^2T^*M\otimes T^*M),
\end{multline*} 
which decomposes according to this splitting into three components
\begin{equation*}
\kappa^\sigma=T+W^\nabla-C^\nabla.
\end{equation*}
One computes straightforwardly that $T\in\Gamma(\Wedge^2T^*M\otimes TM)$ is the
torsion of the almost c-projective structure and that $C^\nabla= d^\nabla
\Rho^\nabla\in\Gamma(\Wedge^2T^*M\otimes T^*M)$, where $d^\nabla$ denotes the
covariant exterior derivative on differential forms with values in $T^*M$
induced by $\nabla$. The tensor $C^\nabla$ is called the \emph{Cotton--York
tensor} of~$\nabla$. To describe the component
$W^\nabla\in\Gamma(\Wedge^2T^*M\otimes\gl(TM,J))$, called the (c-projective)
\emph{Weyl curvature} of $\nabla$, let us denote by
$R^\nabla\in\Omega^2(M,\gl(TM,J))$ the curvature of $\nabla$. Then one has
\begin{equation*}
W^\nabla= R^\nabla-\partial \Rho^\nabla,
\end{equation*}
where 
\begin{equation}\label{partial}
(\partial\Rho^\nabla)_{\alpha\beta}{}^\gamma{}_\epsilon:=
\delta_{[\alpha}{}^\gamma\Rho^\nabla_{\beta]\epsilon}
-J_{[\alpha}{}^\gamma\Rho^\nabla_{\beta]\zeta}J_\epsilon{}^\zeta
-\Rho^\nabla_{[\alpha\beta]}\delta_{\epsilon}{}^\gamma
- J_{[\alpha}{}^\zeta\Rho^\nabla_{\beta]\zeta}J_\epsilon{}^{\gamma}.
\end{equation}

\begin{rem}\label{partial_remark}
The map $\partial\colon T^*M\otimes T^*M\to\Wedge^2T^*M\otimes
\gl(TM,J)$ given by (\ref{partial}) is related to Lie algebra cohomology. It is
easy to see that the Lie algebra differentials in the complex computing the Lie
algebra cohomology of the Abelian real Lie algebra $\g_{-1}$ with values in the
representation $\g$ are $G_0$-equivariant and that $\partial$ is induced by the
restriction to $\g_{-1}^*\otimes\g_1\cong \g_{1}\otimes\g_1$ of half of the
second differential in this complex.
\end{rem}

The normal Cartan connection $\omega$ is characterised as the unique extension
of $\nu$ to a Cartan connection such that $\partial^*\kappa^\sigma=0$ for all
Weyl structures $\sigma\colon\G_0\to\G$. Analysing $\ker\partial^*$
shows that $T_{\alpha\beta}{}^\gamma$ is in there, since forms of type $(0,2)$
are, and $C^\nabla$ is, since $\Wedge^2T^*M\otimes T^*M\subset
\ker\partial^*$. Hence, $\Rho^\nabla$ is uniquely determined by requiring
that $W^\nabla$ be in the kernel of $\partial^*$.

\begin{rem}\label{rem:non-normal-case} Recall that in
Definition~\ref{definition_cprojective_structure} we restricted our definition
of almost c-projective structures to c-projective equivalence classes of
minimal connections. Since the kernel of
\[
\partial^*\colon
\Wedge^2 T^*M\otimes TM\to T^*M\otimes \mathfrak{gl}(TM,J)
\]
consists precisely of all the $2$-forms with values in $TM$ of type $(0,2)$,
the discussion of the construction of the Cartan connection above shows that
the minimality condition is forced by the normalisation condition of the Cartan
connection. The requirement for the almost c-projective structure to be minimal
is however not necessary in order to construct a canonical Cartan connection.
In fact, starting with any complex connection, one can show that there is a
complex connection with the same $J$-planar curves whose torsion has only two
components, namely the $(0,2)$-component $-\frac{1}{4}N_J$ and a component in
the subspace of $(1,1)$-tensors in $\Wedge^2T^*M\otimes TM$ that are trace and
$J$-trace free. Imposing this normalisation condition on an almost c-projective
structure allows then analogously as above to associate a canonical Cartan
connection (see \cite{KMT}).
\end{rem}

\begin{prop} Suppose $(M,J,[\nabla])$ is an almost c-projective manifold of
dimension $2n\geq 4$. Let $\nabla\in[\nabla]$ be a connection in the
c-projective class. Then the Rho tensor corresponding to $\nabla$ is given by
\begin{equation}\label{RhoTensor}
\Rho^\nabla_{\alpha\beta}=\tfrac{1}{n+1}(\Ric^\nabla_{\alpha\beta}+
\tfrac{1}{n-1}(\Ric^\nabla_{(\alpha\beta)}-J_{(\alpha}{}^\gamma J_{\beta)}{}^\delta
\Ric^\nabla_{\gamma\delta})),
\end{equation}
where $\Ric^\nabla_{\alpha\beta}:= R^\nabla_{\gamma\alpha}{}^\gamma{}_\beta$ is
the Ricci tensor of~$\nabla$. Moreover, if $\hat\nabla\in[\nabla]$ is another
connection in the class, related to $\nabla$ according
to~{\rm(\ref{cprojchange})}, then
\begin{equation}\label{changeRho}
\Rho^{\hat\nabla}_{\alpha\beta} =
\Rho^\nabla_{\alpha\beta}-\nabla_\alpha\Upsilon_\beta
+\tfrac{1}{2}(\Upsilon_\alpha\Upsilon_\beta -J_\alpha{}^\gamma
J_\beta{}^\delta\Upsilon_\gamma\Upsilon_\delta).
\end{equation}
\end{prop}
\begin{proof}
The map 
$\partial^*\colon \Wedge^2 T^*M\otimes \gl(TM,J)\to T^*M\otimes T^*M$
is a multiple of a Ricci-type contraction. Hence, the normality of $\omega$
implies
\begin{equation}
R^\nabla_{\alpha\beta}{}^\alpha{}_\epsilon=
(\partial \Rho^\nabla)_{\alpha\beta}{}^\alpha{}_\epsilon=
(n+\tfrac12)\Rho^\nabla_{\beta\epsilon}-
\tfrac{1}{2}\Rho^\nabla_{\epsilon\beta}
+J_{(\beta}{}^\gamma J_{\epsilon)}{}^\zeta\Rho^\nabla_{\gamma\zeta}.
\end{equation} 
Therefore, $\Ric^\nabla_{[\beta\epsilon]}=(n+1)\Rho^\nabla_{[\beta\epsilon]}$
and $\Ric^\nabla_{(\beta\epsilon)}=n\Rho^\nabla_{(\beta\epsilon)}
+J_{(\beta}{}^\gamma J_{\epsilon)}{}^\zeta\Rho_{\gamma\zeta}$, which implies
that
\begin{equation*}
\Ric^\nabla_{(\beta\epsilon)}
-J_{(\beta}{}^\gamma J_{\epsilon)}{}^\zeta \Ric^\nabla_{\gamma\zeta}=
(n-1)(\Rho^\nabla_{(\beta\epsilon)}
-J_{(\beta}{}^\gamma J_{\epsilon)}{}^\zeta\Rho^\nabla_{\gamma\zeta}).
\end{equation*} 
Using these identities one verifies immediately that formula (\ref{RhoTensor})
holds. The formula (\ref{changeRho}) for the change of the Rho tensor can
easily be verified directly or follows from the general theory of Weyl
structures for parabolic geometries established in~\cite{CSl} taking into
account that the Rho tensor in ~\cite{CSl} is $-\frac{1}{2}$ times the Rho
tensor given by \eqref{RhoTensor} and our conventions for the definition of
$\Upsilon$ as in \eqref{cprojchange}.
\end{proof}

As an immediate consequence (writing out (\ref{changeRho}) in terms of its
components using the various projectors $\Pi_\alpha^a,\ldots$ and the
formulae (\ref{usefulidentities})) we have:

\begin{cor}\label{Rho_changes}
\mbox{ }\begin{itemize}
\item $\overline{\Rho^\nabla_{ab}}=\Rho^\nabla_{\bar a\bar b}$
and $\overline{\Rho^\nabla_{\bar a b}}=\Rho^\nabla_{a\bar b}$ 
\item $\Rho^{\hat\nabla}_{ab} = 
\Rho^\nabla_{ab}-\nabla_a\Upsilon_b+\Upsilon_a\Upsilon_b$
\item $\Rho^{\hat\nabla}_{\bar a b} = 
\Rho^\nabla_{\bar a b}-\nabla_{\bar a}\Upsilon_{b}$
\end{itemize}
\end{cor}

For any connection $\nabla\in[\nabla]$, its Weyl curvature $W^\nabla$ is, by
construction, a section of $\Wedge^2T^*M\otimes\gl(TM,J)$ that satisfies
$W^\nabla_{\alpha\beta}{}^\alpha{}_\gamma\equiv 0$. This implies that also
$J_{\zeta}{}^\alpha W^\nabla_{\alpha\beta}{}^\zeta{}_\epsilon=
W^\nabla_{\alpha\beta}{}^{\alpha}{}_\zeta J_{\epsilon}{}^\zeta\equiv 0$. In
the sequel we will often simply write $W$ instead of $W^\nabla$, and similarly
for other tensors such as the Rho tensor, the dependence of $\nabla$ being
understood. Viewing $W$ as a $2$-form with values in the complex bundle vector
bundle ${\mathfrak{gl}}(TM,J)\cong \mathfrak{gl}(T^{1,0}M,\C)$, it decomposes
according to $(p,q)$-types into three components:
$$W_{ab}{}^c{}_d\quad\quad W_{a\bar b}{}^c{}_{d}\quad\quad
W_{\bar a\bar b}{}^c{}_{d}.$$
The vanishing of the trace and $J$-trace above, then imply that
$$W_{ab}{}^a{}_d=W_{a\bar b}{}^a{}_{d}\equiv 0.$$

In these conclusions and in Corollary~\ref{Rho_changes} we begin to see the
utility of writing our expressions in using the barred and unbarred indices
introduced in Section~\ref{almostcomplexmanifolds}. In the following discussion
we pursue this systematically, firstly by describing exactly how the curvature
of a complex connection decomposes. We analyse these decompositions from the
perspective of c-projective geometry: some pieces are invariant whilst others
transform simply. For the convenience of the reader, we reiterate some of our
previous conclusions in the following theorem (but prove them more easily using
barred and unbarred indices, as just advocated).

\begin{prop}\label{rosetta} Suppose $(M,J,[\nabla])$ is an almost
c-projective manifold of dimension $2n\geq 4$. Let $T_{\bar{a}\bar{b}}{}^c$
denote its torsion \textup(already observed to be a constant multiple of the
Nijenhuis tensor of $(M,J)$\textup). Then the curvature $R$ of a connection
$\nabla$ in the c-projective class decomposes as follows\textup:
\begin{equation}\label{full_curvature_decomposition}
\begin{split}
R_{ab}{}^c{}_d&=
W_{ab}{}^c{}_d+2\delta_{[a}{}^c\Rho_{b]d}+\beta_{ab}\delta_d{}^c\\
R_{a\bar{b}}{}^c{}_d&=
W_{a\bar{b}}{}^c{}_d+\delta_a{}^c\Rho_{\bar{b}d}+\delta_d{}^c\Rho_{\bar{b}a}\\
W_{a\bar{b}}{}^c{}_d&=H_{a\bar{b}}{}^c{}_d
-\tfrac{1}{2(n+1)}\bigl(\delta_a{}^cT_{df}{}^{\bar{e}}T_{\bar{e}\bar{b}}{}^f
+\delta_d{}^cT_{af}{}^{\bar{e}}T_{\bar{e}\bar{b}}{}^f\bigr)
-\tfrac12 T_{ad}{}^{\bar{e}}T_{\bar{e}\bar{b}}{}^c\\
R_{\bar{a}\bar{b}}{}^c{}_d&=W_{\bar{a}\bar{b}}{}^c{}_d
\enskip=\enskip\nabla_dT_{\bar{a}\bar{b}}{}^c
\end{split}\end{equation}
where
\begin{gather*}
W_{ab}{}^c{}_d=W_{[ab]}{}^c{}_d\qquad W_{[ab}{}^c{}_{d]}=0
\qquad W_{ab}{}^a{}_d=0\qquad\beta_{ab}=-2\Rho_{[ab]}\\
H_{a\bar{b}}{}^c{}_d=H_{d\bar{b}}{}^c{}_a\hspace{60pt}
H_{a\bar{b}}{}^a{}_d=0.
\end{gather*}
Let $\hat\nabla$ be another connection in the c-projective class, related to
$\nabla$ by \eqref{cprojchange}, and denote its curvature components by $\hat
W$, $\hat H$, and $\hat\Rho$. Then we have\textup:
\begin{enumerate}
\item $\hat W_{ab}{}^c{}_d=W_{ab}{}^c{}_d$ and
$\hat W_{a\bar b}{}^c{}_{d}=W_{a\bar b}{}^c{}_{d}$ and
$\hat{H}_{a\bar{b}}{}^c{}_d=H_{a\bar{b}}{}^c{}_d$, 
\item $\hat W_{\bar a\bar b}{}^c{}_{d}=
W_{\bar a\bar b}{}^c{}_{d}+ T_{\bar a\bar b}{}^e\upsilon_{ed}{}^c$ and if
$J$ is integrable, then $W_{\bar{a}\bar{b}}{}^c{}_d\equiv 0$,
\item $W_{ab}{}^c{}_c\equiv 0$,
\item $W_{a\bar{b}}{}^c{}_c=T_{fa}{}^{\bar{e}}T_{\bar{e}\bar{b}}{}^f$,
\end{enumerate}
whilst we recall that 
$\hat{\Rho}_{ab}=\Rho_{ab}-\nabla_a\Upsilon_b+\Upsilon_a\Upsilon_b,\quad
\hat{\Rho}_{\bar{b}d}=\Rho_{\bar{b}d}-\nabla_{\bar{b}}\Upsilon_d$.

The tensor $\beta_{ab}=-2\Rho_{[ab]}$ satisfies 
\begin{equation}\label{what_beta_satifies}
\nabla_{[b}\beta_{ce]}=\Rho_{\bar{f}[b}T_{ce]}{}^{\bar{f}}
-\tfrac{1}{n+1}T_{[bc}{}^{\bar{f}}T_{e]a}{}^{\bar{d}}T_{\bar{d}\bar{f}}{}^a.
\end{equation}
Finally, the Cotton--York tensors~$C_{abc}$ and~$C_{a\bar bc}$ are defined as
\begin{equation}\label{what_is_cottonyork}
C_{abc}:=\nabla_a\Rho_{bc}-\nabla_b\Rho_{ac}
+T_{ab}{}^{\bar{d}}\Rho_{\bar{d}c} \quad\text{and}\quad
C_{a\bar b c}:=\nabla_a\Rho_{\bar b c}-\nabla_{\bar b}\Rho_{ac}.
\end{equation}
The first of these satisfies a Bianchi identity
\begin{multline}\label{contracted_bianchi}
\nabla_aW_{bc}{}^a{}_e-(n-2)C_{bce}\\
\enskip{}=2T_{a[b}{}^{\bar{f}}H_{c]\bar{f}}{}^a{}_e
+\frac{2}{n+1}T_{bc}{}^{\bar{f}}T_{ea}{}^{\bar{d}}T_{\bar{d}\bar{f}}{}^a
-\frac{n}{n+1}
T_{e[b}{}^{\bar{f}}T_{c]a}{}^{\bar{d}}T_{\bar{d}\bar{f}}{}^a
\end{multline}
and transforms as
\begin{equation}\label{cottonyork_transformation}
\hat{C}_{bce}=C_{bce}+\Upsilon_aW_{bc}{}^a{}_e
\end{equation}
under c-projective change~{\rm(\ref{cprojchange})}.  Another part of
the Bianchi identity reads
\begin{equation}\label{anotherBianchi}
C_{a\bar{b}c}-C_{c\bar{b}a} =\tfrac{1}{n+1}\bigl(
T_{\bar{b}\bar{f}}{}^d\nabla_dT_{ac}{}^{\bar{f}}
+R_{\bar{b}\bar{f}}{}^d{}_dT_{ac}{}^{\bar{f}}
-2R_{\bar{b}\bar{f}}{}^d{}_{[a}T_{c]d}{}^{\bar{f}}\bigr).
\end{equation}
\end{prop}
\begin{proof}
In this proof we also take the opportunity to develop various useful formulae
for torsion and curvature and for how these quantities transform under
c-projective change (\ref{cprojchange}). As in the statement of
Proposition~\ref{rosetta}, we express all these formulae in terms of the
abstract indices on almost complex manifolds developed in
Section~\ref{almostcomplexmanifolds}. Firstly, recall that since we are working
with minimal connections
(cf.~Definition~\ref{definition_cprojective_structure}), their torsions are
restricted to being of type $(0,2)$ and this means precisely that
\begin{equation}\label{torsion_in_action}\begin{aligned}
(\nabla_a\nabla_b-\nabla_b\nabla_a)f+T_{ab}{}^{\bar{c}}\nabla_{\bar{c}}f&=0
&\quad(\nabla_{\bar{a}}\nabla_b-\nabla_b\nabla_{\bar{a}})f&=0\\
(\nabla_{\bar{a}}\nabla_{\bar{b}}-\nabla_{\bar{b}}\nabla_{\bar{a}})f
+T_{\bar{a}\bar{b}}{}^c\nabla_cf&=0
&\quad(\nabla_a\nabla_{\bar{b}}-\nabla_{\bar{b}}\nabla_a)f&=0,
\end{aligned}\end{equation}
where $T_{ab}{}^{\bar{c}}\equiv\Pi_a^\alpha\Pi_b^\beta\overline{\Pi}{}_\gamma{}^c
T_{\alpha\beta}{}^\gamma$, equivalently its complex
conjugate~$T_{\bar{a}\bar{b}}{}^c=\overline{T_{ab}{}^{\bar{c}}}$, represents
the Nijenhuis tensor as in~(\ref{torsion_types_with_indices}). Notice that the
second line of (\ref{torsion_in_action}) is the complex conjugate of the first.
In this proof, we take advantage of this general feature by listing only
one of such conjugate pairs, its partner being implicitly valid. For example,
here are characterisations of sufficiently many components of the general
curvature tensor $R_{\alpha\beta}{}^\gamma{}_{\delta}$.
\begin{equation}\label{curvatures}\begin{split}
(\nabla_a\nabla_b-\nabla_b\nabla_a)X^c+T_{ab}{}^{\bar{d}}\nabla_{\bar{d}}X^c
&= R_{ab}{}^c{}_dX^d\\
(\nabla_a\nabla_{\bar{b}}-\nabla_{\bar{b}}\nabla_a)X^c
&= R_{a{\bar{b}}}{}^c{}_dX^d\\
\bigl[\mbox{or\qquad}(\nabla_{\bar{a}}\nabla_b-\nabla_b\nabla_{\bar{a}})X^c
&= R_{{\bar{a}}b}{}^c{}_dX^d,\quad\mbox{if preferred}\bigr]\\
(\nabla_a\nabla_b-\nabla_b\nabla_a)X^{\bar{c}}
+T_{ab}{}^{\bar{d}}\nabla_{\bar{d}}X^{\bar{c}}
&= R_{ab}{}^{\bar{c}}{}_{\bar{d}}X^{\bar{d}}.
\end{split}\end{equation}
For convenience, the dual formulae are sometimes preferred: for example,
\begin{equation}\label{dual_curvature}
(\nabla_a\nabla_b-\nabla_b\nabla_a)\phi_d
+T_{ab}{}^{\bar{c}}\nabla_{\bar{c}}\phi_d
=-R_{ab}{}^c{}_d\phi_c\,.
\end{equation}
The tensor $\upsilon_{\alpha\beta}{}^\gamma$ employed in a c-projective change
of connection~(\ref{cprojchange}) was already broken into irreducible pieces in
deriving Proposition~\ref{change10}, e.g.\
\begin{flalign}
\upsilon_{ab}{}^c&=
\Pi_a^\alpha\Pi_b^\beta\Pi_\gamma^c\upsilon_{\alpha\beta}{}^\gamma=
\Upsilon_a\delta_b{}^c+\Upsilon_b\delta_a{}^c &\Rightarrow\;&
\hat\nabla_aX^c=\nabla_aX^c+\Upsilon_aX^c+\Upsilon_bX^b\delta_a{}^c\\
\text{and }\upsilon_{\bar{a}b}{}^c&=
\overline{\Pi}{}_{\bar{a}}^\alpha\Pi_b^\beta\Pi_\gamma^c
\upsilon_{\alpha\beta}{}^\gamma=0 &\Rightarrow\;&
\hat\nabla_{\bar{a}}X^c=\nabla_{\bar{a}}X^c.
\end{flalign}
It is an elementary matter, perhaps more conveniently executed in the dual
formulation
\begin{equation}\label{dual_rules}
\hat\nabla_a\phi_b=\nabla_a\phi_b-\Upsilon_a\phi_b-\Upsilon_b\phi_a\qquad
\hat\nabla_{\bar{a}}\phi_b=\nabla_{\bar{a}}\phi_b\,,
\end{equation}
to compute the effect of these changes on curvature, namely
\begin{equation}\label{changes_on_curvature}\begin{split}
\hat{R}_{ab}{}^c{}_d&=R_{ab}{}^c{}_d
-2\delta_{[a}{}^d(\nabla_{b]}\Upsilon_c)
+2\delta_{[a}{}^d\Upsilon_{b]}\Upsilon_c
+2(\nabla_{[a}\Upsilon_{b]})\delta_d{}^c\\
\hat{R}_{a{\bar{b}}}{}^c{}_d&=R_{a{\bar{b}}}{}^c{}_d
-\delta_a{}^c\nabla_{\bar{b}}\Upsilon_d
-\delta_d{}^c\nabla_{\bar{b}}\Upsilon_a\\
\hat{R}_{ab}{}^{\bar{c}}{}_{\bar{d}}&=R_{ab}{}^{\bar{c}}{}_{\bar{d}}+
T_{ab}{}^{\bar{c}}\Upsilon_{\bar{d}}
+\Upsilon_{\bar{e}}T_{ab}{}^{\bar{e}}\delta_{\bar{d}}{}^{\bar{c}}
=R_{ab}{}^{\bar{c}}{}_{\bar{d}}+
T_{ab}{}^{\bar{e}}\upsilon_{\bar{e}\bar{d}}{}^{\bar{c}}.
\end{split}\end{equation}

We shall also need the Bianchi symmetries derived from (\ref{curvatures}) or,
more conveniently in the dual formulation, as follows. Evidently,
\begin{multline*}\nabla_a(\nabla_b\nabla_c-\nabla_c\nabla_b)f+
\nabla_b(\nabla_c\nabla_a-\nabla_a\nabla_c)f+
\nabla_c(\nabla_a\nabla_b-\nabla_b\nabla_a)f\\
\quad{}=(\nabla_a\nabla_b-\nabla_b\nabla_a)\nabla_cf+
(\nabla_b\nabla_c-\nabla_c\nabla_b)\nabla_af+
(\nabla_c\nabla_a-\nabla_a\nabla_c)\nabla_bf,
\end{multline*}
which we may expand using (\ref{torsion_in_action}) and (\ref{curvatures}) to
obtain
$$(\nabla_aT_{bc}{}^{\bar{d}}+\nabla_bT_{ca}{}^{\bar{d}}
+\nabla_cT_{ab}{}^{\bar{d}})\nabla_{\bar{d}}f
=(R_{ab}{}^d{}_c+R_{bc}{}^d{}_a+R_{bc}{}^d{}_a)\nabla_df$$
and hence that
\begin{equation}\label{bianchi1}
\nabla_{[a}T_{bc]}{}^{\bar{d}}=0\qquad R_{[ab}{}^c{}_{d]}=0.
\end{equation}
Similarly, by looking at different orderings for the indices of
$\nabla_a\nabla_{\bar{b}}\nabla_cf$, we find that
\begin{equation}\label{bianchi2}
R_{a{\bar{b}}}{}^c{}_{d}-R_{d{\bar{b}}}{}^c{}_{a}
+T_{ad}{}^{\bar{e}}T_{{\bar{e}}{\bar{b}}}{}^c=0
\qquad R_{ab}{}^{\bar{c}}{}_{\bar{d}}=\nabla_{\bar{d}}T_{ab}{}^{\bar{c}}.
\end{equation}
Already, the final statement of~\eqref{full_curvature_decomposition} is
evident and if $T_{ab}{}^{\bar{c}}=0$ then both
$R_{ab}{}^{\bar{c}}{}_{\bar{d}}$ and its complex conjugate
$R_{\bar{a}\bar{b}}{}^c{}_d$ vanish. Notice that $\partial\Rho$ does not
contribute to this piece of curvature. Specifically, from~(\ref{partial})
$$(\partial\Rho)_{\bar{a}\bar{b}}{}^c{}_d
=\overline{\Pi}{}_{\bar{a}}^\alpha\overline{\Pi}{}_{\bar{b}}^\beta\Pi_\gamma^c
\Pi_d^\epsilon(\partial\Rho)_{\alpha\beta}{}^\gamma{}_\epsilon
=-\Rho_{[\bar{a}\bar{b}]}\delta_d{}^c-\Rho_{[\bar{b}\bar{a}]}\delta_d{}^c=0.$$
It follows that $W_{\bar{a}\bar{b}}{}^c{}_d=R_{\bar{a}\bar{b}}{}^c{}_d$ in
general and that $W_{\bar{a}\bar{b}}{}^c{}_d=R_{\bar{a}\bar{b}}{}^c{}_d=0$ in
the integrable case. The rest of statement (2) also follows, either from the
last line of (\ref{changes_on_curvature}) or, more easily, from the
c-projective invariance of~$T_{ab}{}^{\bar{c}}$ (depending only on the
underlying almost complex structure), the second identity of (\ref{bianchi2}),
and the transformation rules~(\ref{dual_rules}).

Now let us consider the curvature $R_{ab}{}^c{}_d$. {From}~(\ref{partial}), we
compute that
\[
(\partial\Rho)_{ab}{}^c{}_d=\Pi_a^\alpha\Pi_b^\beta\Pi_\gamma^c
\Pi_d^\epsilon(\partial\Rho)_{\alpha\beta}{}^\gamma{}_\epsilon
=2\delta_{[a}{}^c\Rho_{b]d}-2\Rho_{[ab]}\delta_d{}^c
\]
and from (\ref{RhoTensor}) that
\[
\Rho_{ab}=\Pi_a^\alpha\Pi_b^\beta\Rho_{\alpha\beta}=
\tfrac{1}{n+1}\bigl(\Ric_{ab}+\tfrac{2}{n-1}\Ric_{(ab)}\bigr)=
\tfrac{1}{n+1}\bigl(\Ric_{ab}+\tfrac{2}{n-1}\Ric_{(ab)}\bigr),
\]
equivalently that $\Ric_{ab}=(n-1)\Rho_{ab}+2\Rho_{[ab]}$. Bearing in mind the
Bianchi symmetry (\ref{bianchi1}) for $R_{ab}{}^c{}_d$, this means that we may
write
\begin{equation}\label{Rabcd_decomposition}
R_{ab}{}^c{}_d=W_{ab}{}^c{}_d+2\delta_{[a}{}^c\Rho_{b]d}
+\beta_{ab}\delta_d{}^c,\end{equation}
where 
$$W_{ab}{}^c{}_d=W_{[ab]}{}^c{}_d\qquad W_{[ab}{}^c{}_{d]}=0
\qquad W_{ab}{}^a{}_d=0\qquad
\beta_{ab}=-2\Rho_{[ab]}.$$
Comparing this decomposition with the first line of
(\ref{changes_on_curvature}) implies that $W_{ab}{}^c{}_d$ is invariant and
confirms that $\Rho_{ab}$ transforms according to Corollary~\ref{Rho_changes}.
In summary,
$$\hat{W}_{ab}{}^c{}_d=W_{ab}{}^c{}_d\qquad
\hat{\Rho}_{ab}=\Rho_{ab}-\nabla_a\Upsilon_c+\Upsilon_b\Upsilon_c\qquad
\hat{\beta}_{ab}=\beta_{ab}+2\nabla_{[a}\Upsilon_{b]}.$$
We have shown (3) and the first statement of (1).

The remaining statements concern the curvature $R_{a\bar{b}}{}^c{}_d$. {From}
(\ref{partial}), we compute that
\begin{equation*}
(\partial\Rho)_{a\bar{b}}{}^c{}_d
=\Pi_a^\alpha\overline{\Pi}{}_{\bar{b}}^\beta\Pi_\gamma^c\Pi_d^\epsilon
(\partial\Rho)_{\alpha\beta}{}^\gamma{}_\epsilon
=\delta_a{}^c\Rho_{\bar{b}d}+\delta_d{}^c\Rho_{\bar{b}a}
\end{equation*}
and from (\ref{RhoTensor}) that
\begin{equation*}
\Rho_{\bar{b}d}=\overline{\Pi}{}_{\bar{b}}^\beta\Pi_d^\epsilon
\Rho_{\beta\epsilon}=\tfrac{1}{n+1}\Ric_{\bar{b}d}=
\tfrac1{n+1}R_{a\bar{b}}{}^a{}_d.
\end{equation*}
{From} (\ref{bianchi2}) it now follows that
\begin{equation}\label{another_curvature_decomposition}
R_{a\bar{b}}{}^c{}_d+\tfrac12T_{ad}{}^{\bar{e}}T_{\bar{e}\bar{b}}{}^c
=H_{a\bar{b}}{}^c{}_d
-\tfrac{1}{2(n+1)}\bigl(\delta_a{}^cT_{df}{}^{\bar{e}}T_{\bar{e}\bar{b}}{}^f
+\delta_d{}^cT_{af}{}^{\bar{e}}T_{\bar{e}\bar{b}}{}^f\bigr)
+(\partial\Rho)_{a\bar{b}}{}^c{}_d,
\end{equation}
where
$$H_{a\bar{b}}{}^c{}_d=H_{d\bar{b}}{}^c{}_a\qquad H_{a\bar{b}}{}^a{}_d=0.$$
Recall that by definition
$$W_{a\bar{b}}{}^c{}_d=
R_{a\bar{b}}{}^c{}_d-(\partial\Rho)_{a\bar{b}}{}^c{}_d.$$
Therefore
\[
W_{a\bar{b}}{}^c{}_d=H_{a\bar{b}}{}^c{}_d
-\tfrac{1}{2(n+1)}\bigl(\delta_a{}^cT_{df}{}^{\bar{e}}T_{\bar{e}\bar{b}}{}^f
+\delta_d{}^cT_{af}{}^{\bar{e}}T_{\bar{e}\bar{b}}{}^f\bigr)
-\tfrac12T_{ad}{}^{\bar{e}}T_{\bar{e}\bar{b}}{}^c.
\]
Comparison with the formula for $\hat{R}_{a\bar{b}}{}^c{}_d$ in
(\ref{changes_on_curvature}) immediately shows that $W_{a\bar{b}}{}^c{}_d$ and
$H_{a\bar{b}}{}^c{}_d$ are c-projectively invariant and also that
\[
W_{a\bar{b}}{}^c{}_c=-T_{af}{}^{\bar{e}}T_{\bar{e}\bar{b}}{}^f,
\]
as required to complete (1) and~(4). Next we demonstrate the behaviour of the
Cotton--York tensor. For this, we need a Bianchi identity with torsion, which
may be established as follows. Evidently,
\begin{multline*}\nabla_a(\nabla_b\nabla_c-\nabla_c\nabla_b)\phi_e+
\nabla_b(\nabla_c\nabla_a-\nabla_a\nabla_c)\phi_e+
\nabla_c(\nabla_a\nabla_b-\nabla_b\nabla_a)\phi_e\\
\;{}=(\nabla_a\nabla_b-\nabla_b\nabla_a)\nabla_c\phi_e+
(\nabla_b\nabla_c-\nabla_c\nabla_b)\nabla_a\phi_e+
(\nabla_c\nabla_a-\nabla_a\nabla_c)\nabla_b\phi_e,
\end{multline*}
the left hand side of which may be expanded by (\ref{dual_curvature}) as
$$\nabla_a(-R_{bc}{}^d{}_e\phi_d-T_{bc}{}^{\bar{d}}\nabla_{\bar{d}}\phi_e)
+\cdots+\cdots,$$
where $\cdots$ represent similar terms where the indices $abc$ are cycled
around. On the other hand, the right hand side may be expanded as
$$\cdots
-R_{bc}{}^d{}_a\nabla_d\phi_e
-R_{bc}{}^d{}_e\nabla_a\phi_d
-T_{bc}{}^{\bar{d}}\nabla_{\bar{d}}\nabla_a\phi_e-\cdots.$$
Comparison yields
$$(\nabla_{[a}R_{bc]}{}^d{}_e)\phi_d+
(\nabla_{[a}T_{bc]}{}^{\bar{d}})\nabla_{\bar{d}}\phi_e
-T_{[bc}{}^{\bar{d}}R_{a]\bar{d}}{}^f{}_e\phi_f=
R_{[bc}{}^d{}_{a]}\nabla_d\phi_e$$
and, from the Bianchi symmetries~(\ref{bianchi1}), we conclude that
$$\nabla_{[a}R_{bc]}{}^d{}_e
=T_{[ab}{}^{\bar{f}}R_{c]\bar{f}}{}^d{}_e.$$
Using (\ref{full_curvature_decomposition}) and tracing over $a$ and $d$ yields
\begin{multline*}
\nabla_aW_{bc}{}^a{}_e-2(n-2)\nabla_{[b}\Rho_{c]e}
+3\nabla_{[b}\beta_{ce]}\\
\enskip{}=2T_{a[b}{}^{\bar{f}}H_{c]\bar{f}}{}^a{}_e
+\frac{1}{n+1}T_{bc}{}^{\bar{f}}T_{ea}{}^{\bar{d}}T_{\bar{d}\bar{f}}{}^a
-\frac{n+2}{n+1}
T_{e[b}{}^{\bar{f}}T_{c]a}{}^{\bar{d}}T_{\bar{d}\bar{f}}{}^a
+(n-2)T_{bc}{}^{\bar{f}}\Rho_{\bar{f}e}+3\Rho_{\bar{f}[b}T_{ce]}{}^{\bar{f}}.
\end{multline*}
Skewing this identity over $bce$ gives (\ref{what_beta_satifies}) and
substituting back gives
\begin{multline*}
\nabla_aW_{bc}{}^a{}_e-2(n-2)\nabla_{[b}\Rho_{c]e}\\
=2T_{a[b}{}^{\bar{f}}H_{c]\bar{f}}{}^a{}_e
+\frac{2}{n+1}T_{bc}{}^{\bar{f}}T_{ea}{}^{\bar{d}}T_{\bar{d}\bar{f}}{}^a
-\frac{n}{n+1}T_{e[b}{}^{\bar{f}}T_{c]a}{}^{\bar{d}}T_{\bar{d}\bar{f}}{}^a
+(n-2)T_{bc}{}^{\bar{f}}\Rho_{\bar{f}e}.
\end{multline*}
The contracted Bianchi identity (\ref{contracted_bianchi}) follows from the
definition (\ref{what_is_cottonyork}) of the Cotton--York tensor.  Notice that
the right hand side of (\ref{contracted_bianchi}) is c-projectively
invariant. Also, by computing that
\begin{align*}
\hat\nabla_a\hat{W}_{bc}{}^d{}_e&=\hat\nabla_aW_{bc}{}^d{}_e\\
&=\nabla_aW_{bc}{}^d{}_e
-2\Upsilon_aW_{bc}{}^d{}_e
-\Upsilon_bW_{ac}{}^d{}_e
-\Upsilon_cW_{ba}{}^d{}_e
+\delta_a{}^d\Upsilon_fW_{bc}{}^f{}_e
-\Upsilon_eW_{bc}{}^d{}_a
\end{align*}
and tracing over $a$ and $d$, we see that
$$\hat\nabla_a\hat{W}_{bc}{}^a{}_e=\nabla_aW_{bc}{}^a{}_e
+(n-2)\Upsilon_aW_{bc}{}^a{}_e$$
and for $n>2$ conclude that (\ref{cottonyork_transformation}) is valid. The
case $n=2$ is somewhat degenerate. Although (\ref{cottonyork_transformation})
is still valid, as we shall see below in Proposition~\ref{Weyl_gone}, the Weyl
curvature $W_{bc}{}^a{}_d$ vanishes by symmetry considerations and
(\ref{cottonyork_transformation}) reads $\hat{C}_{bce}=C_{bce}$, the
straightforward verification of which is left to the reader. Similarly, by
considering different orderings for the indices of
$\nabla_a\nabla_{\bar{b}}\nabla_c\phi_e$, we are rapidly led to
$$\nabla_aR_{c\bar{b}}{}^d{}_e-\nabla_cR_{a\bar{b}}{}^d{}_e
+\nabla_{\bar{b}}R_{ac}{}^d{}_e=T_{ac}{}^{\bar{f}}R_{\bar{b}\bar{f}}{}^d{}_e$$
as another piece of the Bianchi identity, which may then be further split into 
irreducible parts. In particular, tracing over $d$ and~$e$ (equivalently, 
tracing over $a$ and $d$ and then skewing over $c$ and $e$) 
gives~(\ref{anotherBianchi}).
\end{proof}

\begin{prop}\label{Weyl_gone} Suppose that $W_{ab}{}^c{}_d\in
\Wedge^{1,0}\otimes\Wedge^{1,0}\otimes T^{1,0}M\otimes\Wedge^{1,0}$ has
the following symmetries\textup:
\begin{equation*}
W_{ab}{}^c{}_d=W_{[ab]}{}^c{}_d\qquad W_{[ab}{}^c{}_{d]}=0\qquad W_{ab}{}^a{}_d=0.
\end{equation*}
If $2n=4$, then $W_{ab}{}^c{}_d\equiv 0$.
\end{prop}
\begin{proof} Fix a nonzero skew tensor~$V_{ab}$. As $W_{ab}{}^c{}_d$ is skew
in $a$ and $b$, it follows that $W_{ab}{}^c{}_d=V_{ab}S^c{}_d$ for some unique
tensor~$S^c{}_d$. Now $W_{ab}{}^a{}_d=V_{ab}S^a{}_d$ but $V_{ab}$ is also
nondegenerate so $W_{ab}{}^a{}_d=0$ implies $S^c{}_d=0.$
\end{proof}

\begin{rem}
When $n=2$, the identity (\ref{contracted_bianchi}) is vacuous.
Proposition~\ref{Weyl_gone} implies that the left hand side vanishes. For the
right hand side, the vanishing of $T_{a[b}{}^{\bar{f}}H_{c]\bar{f}}{}^a{}_e$
follows by tracing the identity $T_{[ab}{}^{\bar{f}}H_{c]\bar{f}}{}^d{}_e=0$
over $a$ and $d$, bearing in mind that $H_{a\bar{f}}{}^d{}_e$ is trace-free in
$a$ and~$d$. The remaining terms also evaporate because, when $n=2$, the tensor
$T_{bc}{}^{\bar{f}}T_{ea}{}^{\bar{d}}$ is symmetric in ${\bar{f}\bar{d}}$
whereas $T_{\bar{d}\bar{f}}{}^a$ is skew.
\end{rem}

The torsion $T_{ab}{}^{\bar{c}}$ (equivalently, its complex
conjugate~$T_{\bar{a}\bar{b}}{}^c$) is c-projectively invariant. The same is
true, not only of the Weyl curvature $W_{a\bar{b}}{}^c{}_d$, but also of its
trace-free symmetric part $H_{a\bar{b}}{}^c{}_d$ (which will be identified as
part of the \emph{harmonic curvature} in Section~\ref{sectionharmcurv}). The
Weyl curvature $W_{ab}{}^c{}_d$ is c-projectively invariant and forms the final
piece of harmonic curvature except when $2n=4$, in which case $W_{ab}{}^c{}_d$
necessarily vanishes, its role being taken by~$C_{abc}$, the c-projectively
invariant part of the Cotton--York tensor. In Section~\ref{sectionharmcurv}, we
place this discussion in the context of general parabolic geometry but, before
that, we collect in the following section some useful formulae for the various
curvature operators on c-projective densities.

\subsection{Curvature operators on c-projective densities}
\label{curvature_operators_on_densities}

Suppose $X^{cd\cdots e}=X^{[cd\cdots e]}$ is a section of
$\cE(n+1,0)=\Wedge^nT^{1,0}M$ and $Y^{\bar c\bar d\cdots\bar e}$ a section of
$\cE(0,n+1)=\Wedge^nT^{0,1}M$. Then it follows from (\ref{curvatures}) that
\begin{align*}
(\nabla_{\bar{a}}\nabla_b-\nabla_b\nabla_{\bar{a}})X^{cd\cdots e}
&=R_{{\bar{a}}b}{}^c{}_fX^{fd\cdots e}+R_{{\bar{a}}b}{}^d{}_fX^{cf\cdots e}
+\cdots +R_{{\bar{a}}b}{}^e{}_fX^{cd\cdots f}\\
&=R_{{\bar{a}}b}{}^f{}_fX^{cd\cdots e}.\\
(\nabla_{\bar{a}}\nabla_b-\nabla_b\nabla_{\bar{a}})Y^{\bar c\bar d\cdots \bar e}
&=R_{{\bar{a}}b}{}^{\bar c}{}_{\bar f}Y^{\bar f\bar d\cdots \bar e}
+R_{{\bar{a}}b}{}^{\bar d}{}_{\bar f}Y^{\bar c\bar f\cdots\bar e}
+ \cdots +R_{{\bar{a}}b}{}^{\bar e}{}_{\bar f}Y^{\bar c\bar d\cdots \bar f}\\
&=R_{{\bar{a}}b}{}^{\bar f}{}_{\bar f}Y^{cd\cdots e}.
\end{align*}
However, from Proposition~\ref{rosetta} part~(4), we find that
\begin{align*}
R_{{\bar{a}}b}{}^f{}_f&=-R_{b{\bar{a}}}{}^f{}_f=-W_{b{\bar{a}}}{}^f{}_f
-(\partial\Rho)_{b\bar{a}}{}^f{}_f
=-T_{fb}{}^{\bar{e}}T_{\bar{e}\bar{a}}{}^f-(n+1)\Rho_{\bar{a}b}\\
R_{{\bar{a}}b}{}^{\bar f}{}_{\bar f}
&=W_{{\bar{a} b}}{}^{\bar f}{}_{\bar f}+(\partial\Rho)_{\bar{a} b}{}^{\bar f}{}_{\bar f}
=T_{fb}{}^{\bar{e}}T_{\bar{e}\bar{a}}{}^f+(n+1)\Rho_{b\bar a}.
\end{align*}
We conclude immediately that for a section $\sigma$ of $\cE(k,\ell)$ we have
\begin{equation}\label{curvature_on_densities}
(\nabla_{\bar{a}}\nabla_b-\nabla_b\nabla_{\bar{a}})\sigma=
\tfrac{\ell-k}{n+1}\,T_{fb}{}^{\bar{e}}T_{\bar{e}\bar{a}}{}^{\smash f}\sigma
+\ell \Rho_{b\bar a}\sigma-k\Rho_{\bar{a}b}\sigma.
\end{equation}
Similarly, from (\ref{curvatures}) it also follows that 
\begin{align*}
\nabla_a\nabla_b-\nabla_b\nabla_a)X^{cd\cdots e}
+T_{ab}{}^{\bar{f}}\nabla_{\bar{f}}X^{cd\cdots e}&=
R_{ab}{}^f{}_fX^{cd\cdots e}\\
(\nabla_a\nabla_b-\nabla_b\nabla_a)Y^{\bar c\bar d\cdots \bar e}
+T_{ab}{}^{\bar{f}}\nabla_{\bar{f}}Y^{\bar c\bar d\cdots \bar e}&=
R_{ab}{}^{\bar f}{}_{\bar f}Y^{\bar c\bar d\cdots \bar e}.
\end{align*}
{From} Proposition~\ref{rosetta} we conclude that
\begin{align*}
R_{ab}{}^f{}_f&=2\Rho_{[ba]}+n\beta_{ab}=(n+1)\beta_{ab}\\
R_{ab}{}^{\bar f}{}_{\bar f}&=\nabla_{\bar f}T_{ab}{}^{\bar f}.
\end{align*}
Therefore, if $\sigma$ is c-projective density of weight $(k,\ell)$, then 
\begin{equation}\label{another_curvature_on_densities}
(\nabla_a\nabla_b-\nabla_b\nabla_a)\sigma+T_{ab}{}^{\smash{\bar f}}\nabla_{\bar f}\sigma
=k\beta_{ab}\sigma+\tfrac{\ell}{n+1}(\nabla_{\bar f}T_{ab}{}^{\smash{\bar f}})\sigma
\end{equation}
and, accordingly,
\begin{equation}\label{final_curvature_on_densities}
(\nabla_{\bar{a}}\nabla_{\bar{b}}-\nabla_{\bar{b}}\nabla_{\bar{a}})\sigma
+T_{\bar{a}\bar{b}}{}^f\nabla_f\sigma
=\ell \beta_{\bar a\bar b}+\tfrac{k}{n+1}(\nabla_fT_{\bar{a}\bar{b}}{}^f)\sigma.
\end{equation}    

Recall that for any connection $\nabla\in[\nabla]$ its Rho tensor, by
definition, satisfies $\Rho_{\bar a b}=\frac{1}{n+1}\Ric_{\bar a b}$ and
$\Rho_{[ab]}=\frac{1}{n+1}\Ric_{[ab]}$. Hence, the identities
(\ref{curvature_on_densities}) and (\ref{another_curvature_on_densities}) imply
that the Ricci tensor of a special connection $\nabla\in[\nabla]$ satisfies
\begin{itemize}
\item $\Ric_{\bar a b}=\Ric_{b\bar a}$ 
\item $\Ric_{[ab]}=\frac{1}{2}\nabla_{\bar c}T_{ab}{}^{\bar c}$.
\end{itemize}
If $\nabla_{\bar c}T_{ab}{}^{\bar c}$ vanishes, the special connection has
symmetric Ricci tensor. In particular, if $J$ is integrable all special
connections have symmetric Ricci tensor.

\subsection{The curvature of complex projective space}\label{CPn_curvature}
In Section~\ref{almost_cproj_curvature}, and especially in
Proposition~\ref{rosetta}, the curvature of a complex connection on a general
almost complex manifold was decomposed into various irreducible pieces
(irreducibility to be further discussed in Section~\ref{BGG}). Here, we pause
to examine this decomposition on complex projective space~$\CP^n$
with its standard Fubini--Study metric.
\begin{lem}
The Riemannian curvature tensor for the Fubini--Study metric $g_{\alpha\beta}$
on $\CP^n$ is given by
\begin{equation}\label{FS_curvature}
R_{\alpha\beta\gamma\delta}=
g_{\alpha\gamma}g_{\beta\delta}-g_{\beta\gamma}g_{\alpha\delta}
+\Omega_{\alpha\gamma}\Omega_{\beta\delta}
-\Omega_{\beta\gamma}\Omega_{\alpha\delta}
+2\Omega_{\alpha\beta}\Omega_{\gamma\delta}\end{equation}
where $J_\alpha{}^\beta$ is the complex structure and
$\Omega_{\alpha\gamma}\equiv J_\alpha{}^\beta g_{\beta\gamma}$ \textup(the K\"ahler
form\textup).
\end{lem}
\begin{proof} A direct calculation from the definition of the Fubini--Study
metric (e.g.~\cite{CdS}) or by invariant theory noting that (up to scale) the 
right hand side of (\ref{FS_curvature}) is the only covariant expression in 
$g_{\alpha\beta}$ and $\Omega_{\alpha\beta}$ such that
\[
R_{\alpha\beta\gamma\delta}=R_{[\alpha\beta][\gamma\delta]}\qquad R_{[\alpha\beta\gamma]\delta}=0
\qquad R_{\alpha\beta\gamma[\delta}J_{\epsilon]}{}^\gamma=0
\]
where the last condition is a consequence of the K\"ahler
condition~$d\Omega=0$ (or, more precisely, a consequence of
$\nabla_\alpha\Omega_{\beta\gamma}=0$ as one can check, by direct computation
in case the almost complex structure $J_\alpha{}^\beta$ is orthogonal
(i.e.~$J_\alpha{}^\beta g_{\beta\gamma}$ is skew), that
$$2\nabla_\alpha\Omega_{\beta\gamma}=3\nabla_{[\alpha}\Omega_{\beta\gamma]}-
3J_\alpha{}^\delta J_\beta{}^\epsilon\nabla_{[\alpha}\Omega_{\delta\epsilon]}-
\Omega_{\alpha\delta}N_{\alpha\beta}{}^\delta,$$
where recall that $N_{\alpha\beta}{}^\gamma$ is the Nijenhuis 
tensor~(\ref{nijenhuis}),
which vanishes when the complex structure is integrable, as it is
on~$\CP^n$).
\end{proof}
To apply the decompositions of Proposition~\ref{rosetta} to (\ref{FS_curvature})
we should raise an index
$$R_{\alpha\beta}{}^\gamma{}_\epsilon=
\delta_\alpha{}^\gamma g_{\beta\epsilon}
-\delta_\beta{}^\gamma g_{\alpha\epsilon}
+J_\alpha{}^\gamma\Omega_{\beta\epsilon}
-J_\beta{}^\gamma\Omega_{\alpha\epsilon}
-2\Omega_{\alpha\beta}J_\epsilon{}^\gamma$$
and then apply the various projectors such as 
$\Pi_a^\alpha\Pi_b^\beta\Pi_\gamma^c\Pi_d^\epsilon$. However, firstly note 
that applying $\Pi_a^\alpha\Pi_c^\gamma$ to 
$J_\alpha{}^\beta g_{\beta\gamma}+J_\gamma{}^\beta g_{\beta\alpha}=0$ implies
that $g_{ac}=0$ (consequently $\Omega_{ac}=0$) whilst applying 
$\Pi_a^\alpha\overline\Pi{}_{\bar{c}}^\gamma$ to 
$\Omega_{\alpha\gamma}=J_\alpha{}^\beta g_{\beta\gamma}$ shows that 
$\Omega_{a\bar{c}}=ig_{a\bar{c}}$. We conclude that
$R_{ab}{}^c{}_d=0$ and
$$R_{a\bar{b}}{}^c{}_d
\equiv\Pi_a^\alpha\overline\Pi{}_{\bar{b}}^\beta\Pi_\gamma^c\Pi_d^\epsilon 
R_{\alpha\beta}{}^\gamma{}_\epsilon=
\delta_a{}^cg_{d\bar{b}}-i\delta_a{}^c\Omega_{d\bar{b}}
-2i\Omega_{a\bar{b}}\delta_d{}^c=2\delta_a{}^cg_{d\bar{b}}
+2\delta_d{}^cg_{a\bar{b}}.$$
Thus, with reference to Proposition~\ref{rosetta}, we see that all irreducible
pieces of curvature vanish save for $\Rho_{\bar{b}d}=2g_{d\bar{b}}$. In
particular, all invariant pieces
$$T_{\bar{a}\bar{b}}{}^c\qquad H_{a\bar{b}}{}^c{}_d\qquad W_{ab}{}^c{}_d$$
of \emph{harmonic curvature} (as identified the following section) vanish.
This is, of course, consistent with~$\CP^n$, equipped with its
standard complex structure and Fubini--Study connection, being the flat model of
c-projective geometry, as discussed in Section~\ref{cproCartan} and especially
Theorem~\ref{EquivCat}.

Finally, observe that if we regard $\CP^n$ as 
\[
{\mathrm{SL}}(n+1,\C)\bigg/\left\{\begin{pmatrix}
\lambda&*&\cdots&*\\
0&*&\cdots&*\\
\vdots&\vdots&\ddots&\vdots\\
0&*&\cdots&*
\end{pmatrix}\right\},
\] 
rather than as a homogeneous ${\mathrm{PSL}}(n+1,\C)$-space as in
Section~\ref{cproCartan}, then the character
$\lambda\mapsto\lambda^{-k}\overline\lambda{}^{-\ell}$ induces a homogeneous
line bundle $\cE(k,\ell)$ on $\CP^n$ as we were supposing earlier and
as we shall soon suppose in Section~\ref{standard_tractors}. This observation
also explains our copacetic choice of notation: on $\CP^n$ it is
standard to write $\cO(k)$ for the holomorphic bundle that is $\cE(k,0)$ just
as a complex bundle (and then $\overline{\cE(k,0)}=\cE(0,k)$).

\subsection{The harmonic curvature}\label{sectionharmcurv}

A normal Cartan connection gives rise to a simpler local invariant than the
Cartan curvature $\kappa$, called the \emph{harmonic curvature} $\kappa_h$,
which still provides a full obstruction to local flatness, as discussed in
Section~\ref{intropargeom} (cf.~especially Proposition \ref{harmcurv}). The
harmonic curvature $\kappa_h$ of an almost c-projective manifold is the
projection of $\kappa\in\ker\partial^*$ to its homology class in
\[
\G\times_P H_2(\g_1,\g)\cong \G_0\times_{G_0} H_2(\g_1,\g).
\]
By Kostant's version of the Bott--Borel--Weil Theorem \cite{Kostant} the
$G_0$-module $H_2(\g_{1},\g)$ can be naturally identified with a
$G_0$-submodule in
$\Wedge^{2}\g_{1}\otimes_\R\g\cong\Wedge^{2}\g_{-1}^*\otimes_\R\g$ that
decomposes into three irreducible components as follows:
\begin{itemize}
\item for $n=2$
$$(\Wedge^{0,2}\g_{-1}^*\otimes_\C \g_{-1})
\oplus(\Wedge^{1,1}\g_{-1}^*\circledcirc_\C\sgl(\g_{-1},\C))
\oplus(\Wedge^{2,0}\g_{-1}^*\otimes_\C\g_1)$$
\item for $n>2$ 
$$(\Wedge^{0,2}\g_{-1}^*\otimes_\C \g_{-1})
\oplus(\Wedge^{1,1}\g_{-1}^*\circledcirc_\C\sgl(\g_{-1},\C))
\oplus(\Wedge^{2,0}\g_{-1}^*\circledcirc_\C\sgl(\g_{-1},\C)),$$
\end{itemize}
where these are complex vector spaces but regarded as real, and where
$\circledcirc$ denotes the Cartan product. Correspondingly, we decompose
the harmonic curvature as
$$\kappa_h=\tor+\wc+\cy$$
in case $n=2$ and 
$$\kappa_h=\tor+\wc_1+\wc_2$$ 
in case~$n>3$. 

Note that $\partial^*$ preserves homogeneities,
i.e.~$\partial^*(\Wedge^{i}\g_1\otimes \g_j)\subset
\Wedge^{i-1}\g_1\otimes\g_{j+1}$. In particular, the induced vector bundle map
$\partial^*$ maps $\Wedge^3T^*M\otimes\A M$ to $\Wedge^2T^*M\otimes\A^0M$.
Hence, we conclude that $\tor$ must equal the torsion
$T_{\alpha\beta}{}^\gamma$. If $n=2$, then $\wc$ is the component
$H_{a\bar{b}}{}^c{}_d$ in $(\Wedge^{1,1}\circledcirc_\C\sgl(T^{1,0}M))$ of the
Weyl curvature of any connection in the c-projective class, and $\cy$ is the
$(2,0)$-part of the Cotton--York tensor. If $n>2$, then $\wc_1$,
respectively~$\wc_2$, is the totally trace-free $(1,1)$-part, respectively
$(2,0)$-part, of the Weyl curvature of any connection in the class.

We now give a geometric interpretation of the three harmonic curvature
components.
\begin{thm}\label{Intharmcurv}
Suppose $(M,J,[\nabla])$ is an almost c-projective manifold of dimension
$2n\geq 4$ and denote by $\kappa_h$ the harmonic curvature of its normal Cartan
connection. Then the following statements hold.
\begin{enumerate}
\item $\kappa_h\equiv 0$ if and only if the almost c-projective manifold
$(M,J,[\nabla])$ is locally isomorphic to $(\C P^n, J_{\mathrm{can}},
[\nabla^{g_{FS}}])$. 
\item $\tor$ is the torsion of $(M,J,[\nabla])$. In particular, $\tor\equiv 0$
if and only if $J$ is integrable, i.e.~$(M,J,[\nabla])$ is a c-projective
manifold. Moreover, in this case, $J^\G$ is integrable and the Cartan bundle
$p\colon\G\to M$ is a holomorphic principal $P$-bundle.
\item Suppose $\tor\equiv0$. Then $\wc_1\equiv 0$ \textup(resp.~$\wc\equiv
0$\textup) if and only if $\omega$ is a holomorphic Cartan connection on the
holomorphic principal bundle $p\colon\G\to M$. This is the
case if and only if $[\nabla]$ locally admits a holomorphic connection,
i.e.~for any connection $\nabla\in[\nabla]$ and any point $x\in M$ there is an
open neighbourhood $U\ni x$ such that $\nabla|_U$ is c-projectively equivalent
to a holomorphic connection on $U$.
\end{enumerate}
\end{thm}

\begin{proof}
We have already observed (1) and the first two assertions of (2). To prove the
last statement of (2) and (3), assume that $\tor\equiv0$, which says, in
particular, that the Cartan geometry is torsion-free. Since $P$ acts on the
complex vector space $\Wedge^2\g_1\otimes_\R\g$ by complex linear maps, $P$
preserves the decomposition of this vector space into the three
$(p,q)$-types. Therefore~\cite[Corollary 3.2]{CorrCap} applies and hence
$\kappa_h$ has components of type $(p,q)$ if and only if $\kappa$ has
components of type $(p,q)$. Therefore, $\tor\equiv0$ implies that $\kappa$ has
no $(0,2)$-part, which by the proof of~\cite[Theorem 3.4]{CorrCap}
(cf.~\cite[Proposition 3.1.17]{csbook}) implies that $J^\G$ is integrable and
$p\colon\G\to M$ a holomorphic principal bundle.  This finishes the
proof of (2). We know that the component $\wc_1$ (respectively $\wc$) vanishes
identically if and only if $\kappa$ is of type $(2,0)$, which by~\cite[Theorem
  3.4]{CorrCap} is the case if and only if $\omega\in\Omega^{1,0}(\G,\g)$ is
holomorphic, i.e.~$d\omega$ is of type $(2,0)$. Hence, it just remains to prove
the last assertion of (3). Assume firstly that $\wc_1$ (respectively $\wc$)
vanishes identically and hence that $(p\colon\G\to M,\omega)$ is a
holomorphic Cartan geometry. Then we can find around each point of $M$ an open
neighbourhood $U\subset M$ such that $\G$ and $\G_0$ trivialise as holomorphic
principal bundles over $U$.  Having chosen such trivialisations, the
holomorphic inclusion $G_0\hookrightarrow P$ induces a holomorphic
$G_0$-equivariant section $\sigma\colon p_0^{-1}(U)\to
p^{-1}(U)$. Since $d\omega$ is of type $(2,0)$ and $\sigma$ is holomorphic
$$\sigma^*d\omega=d\sigma^*\omega=d\theta+d\gamma^\nabla-d\mathsf{p}^\nabla$$ 
is also of type $(2,0)$. In particular, $d\gamma^\nabla$ is of type $(2,0)$ and
it follows that $\gamma^\nabla\in\Omega^{1,0}(p_0^{-1}(U),\g_0)$ is a
holomorphic principal connection in the c-projective class. Conversely, assume
that $U\subset M$ is an open set and that
$\gamma^\nabla\in\Omega^{1,0}(p_0^{-1}(U),\g_0)$ is a holomorphic principal
connection that belongs to the c-projective class. Since the Lie bracket on
$\g$ is complex linear, the holomorphicity of $\gamma^\nabla$ implies that its
curvature $d\gamma^\nabla+[\gamma^\nabla,\gamma^\nabla]$ is of type $(2,0)$. By
definition of the Weyl curvature this implies that also its Weyl curvature is
of type $(2,0)$ and hence so is $\kappa_h|_U$. By assumption there exists
locally around any point a holomorphic connection and hence $\kappa_h$ is of
type $(2,0)$ on all of $M$.
\end{proof}

\section{Tractor bundles and BGG sequences}\label{sec:tractorBGG}

The normal Cartan connection of an almost c-projective manifold induces a
canonical linear connection on all associated vector bundles corresponding to
representations of $\mathrm{PSL}(n+1,\C)$ (cf.~Section~\ref{intropargeom}).
These, in the theory of parabolic geometries, so-called \emph{tractor
  connections}, provide an efficient calculus, especially well suited for
explicit constructions of local invariants and invariant differential
operators. We develop in this section the basics of the theory of tractor
connections for almost c-projective manifolds, and explain their relation to
geometrically significant overdetermined systems of PDE and sequences of
invariant differential operators.

\subsection{Standard complex tractors}\label{standard_tractors}

Suppose that $(M,J,[\nabla])$ is an almost c-projective manifold of dimension
$2n\geq 4$. Further, assume that the complex line bundle $\Wedge^n T^{1,0}M$
admits an $(n+1)^{\mathrm{st}}$ root and choose one, denoted $\cE(1,0)$, with
conjugate $\cE(0,1)$. More generally, we write $\cE(k,\ell) =\cE(1,0)^{\otimes
  k}\otimes\cE(0,1)^{\ell}$ for any $(k,\ell)\in\Z\times \Z$
(cf.~Section~\ref{almost_c-projective}). Note that such a choice of a root
$\cE(1,0)$ is at least locally always possible and the assumption that such
roots exists globally is a relatively minor constraint. The choice of
$\cE(1,0)$ canonically extends the Cartan bundle of $(M,J,[\nabla])$ to a
$\tilde{P}$-principal bundle $\tilde{p}\colon\tilde\G\to M$, where
$\tilde{P}$ is the stabiliser in ${\mathrm{SL}}(n+1,\C)$ of the complex line
generated by the first basis vector in $\C^{n+1}$, and the normal Cartan
connection of $(M,J,[\nabla])$ naturally extends to a normal Cartan connection
on $\tilde\G$ of type $({\mathrm{SL}}(n+1,\C), \tilde{P})$, which we
also denote by~$\omega$. The groups $\mathrm{SL}(n+1,\C)$ and $\tilde{P}$ are
here viewed as real Lie groups as in Section~\ref{cproCartan}, and we obtains
in this way, analogously to Theorem \ref{EquivCat}, an equivalence of
categories between almost c-projective manifolds equipped with an $(n+1)$st
root $\cE(1,0)$ of $\Wedge^n T^{1,0}M$ and normal Cartan geometries of type
$(\mathrm{SL}(n+1,\C), \tilde{P})$.  The homogeneous model of such structures
is again $\CP^n$, but now viewed as a homogeneous space
$\mathrm{SL}(n+1,\C)/ \tilde{P}$ with $\cE(1,0)$ being dual to the
tautological line bundle $\cO(-1)$, cf.~Section~\ref{CPn_curvature}.

The extended normal Cartan geometry of type $(\mathrm{SL}(n+1,\C), \tilde{P})$ 
allows us to form the \emph{standard complex tractor bundle}
$$\T:=\tilde{\G}\times_{\tilde{P}}\V$$
of $(M,J,[\nabla],\cE(1,0))$, where $\V=\R^{2n+2}$ is the
defining representation of the real Lie group
$\mathrm{SL}(2(n+1),\J_{2(n+1)})\cong{\mathrm{SL}}(n+1,\C)$. Note
that the complex structure $\J_{2(n+1)}$ on $\V$ induces a complex
structure $J^{\T}$ on $\T$.  Analogously to the discussion of the tangent
bundle of an almost complex manifold in Section~\ref{almostcomplexmanifolds},
$(\T, J^{\T})$ can be identified 
with the $(1,0)$-part of its complexification $\T_\C$, on which $J^{\T}$ acts
by multiplication by $i$.  We will implement this identification in the sequel
without further comment, and similarly  for all the other tractor bundles with
complex structures in the following Sections~\ref{standard_tractors},
\ref{BGG} and~\ref{infinitesimal_automorphisms}. 
Since $\tilde P$ stabilises the complex line
generated by the first basis vector in $\C^{n+1}$, this line defines a complex
$1$-dimensional submodule of $\C^{n+1}$. Correspondingly, the standard complex
tractor bundle (identified with the $(1,0)$-part of its
complexification~$\T_\C$) is filtered as
\begin{equation}\label{filtrationT}
\T=\T^0\supset \T^1
\end{equation}
where $\T^1\cong\cE(-1,0)$ and $\T^0/\T^1\cong T^{1,0}M(-1,0)$. Since $\T$ is
induced by a representation of ${\mathrm{SL}}(n+1,\C)$, the Cartan connection
induces a linear connection $\nabla^\T$ on~$\T$, called the \emph{tractor
connection} (see Section~\ref{intropargeom}). Any choice of a linear connection
$\nabla\in[\nabla]$, splits the filtration of the tractor bundle $\T$ and the
splitting changes by
\begin{equation}\label{splittingT}
\widehat{\left( \begin{matrix}X^b\\ \rho \end{matrix}\right )}
=\left( \begin{matrix}X^b\\ \rho-\Upsilon_cX^c \end{matrix}\right ),
\text{ where }
\left\{\begin{array}{l} X^b\in T^{1,0}M(-1,0)\\ 
\rho\in \cE(-1,0),
\end{array}\right.
\end{equation}
if one changes the connection according to (\ref{cprojchange}). In terms of a
connection $\nabla\in[\nabla]$, the tractor connection is given by
\begin{equation}\label{tractorconn}
\nabla^\T_\alpha \begin{pmatrix}X^b\\ \rho \end{pmatrix}
=\begin{pmatrix}\nabla_\alpha X^b+\rho\Pi_\alpha^b\\
\nabla_\alpha\rho-\Rho_{\alpha b}X^b
\end{pmatrix}.
\end{equation}
Applying $\Pi_a^\alpha$ and $\overline{\Pi}{}_{\bar{a}}^\alpha$
to~(\ref{tractorconn}), we can write the tractor connection as
\begin{equation}\label{tractorconna}
\nabla^\T_a\left( \begin{matrix}X^b\ \\ \rho \end{matrix}\right )
=\left( \begin{matrix}\nabla_aX^b+\rho\delta_a{}^b\\
\nabla_a\rho-\Rho_{ab}X^b \end{matrix}\right )
\end{equation}
and
\begin{equation}\label{tractorconnbara}
\nabla^\T_{\bar a}\left( \begin{matrix}X^b\\ \rho \end{matrix}\right )
=\left( \begin{matrix} \nabla_{\bar{a}}X^b\\
\nabla_{\bar{a}}\rho-\Rho_{\bar{a}b}X^b \end{matrix}\right ).
\end{equation}
By (\ref{filtrationT}) the dual (or ``co-standard'') complex tractor bundle
$\T^*$ admits a natural subbundle isomorphic to $\Wedge^{1,0}(1,0)$
and the quotient $\T^*/\Wedge^{1,0}(1,0)$ is isomorphic to
$\cE(1,0)$. One immediately deduces from (\ref{tractorconna}) and
(\ref{tractorconnbara}) that in terms of a connection $\nabla\in[\nabla]$,
the tractor connection on $\T^*$ is given by
\begin{align*}
\nabla^{\T^*}_a \begin{pmatrix}\sigma\\ \mu_b \end{pmatrix}
&=\begin{pmatrix}\nabla_a\sigma-\mu_a\\ \nabla_a \mu_b+\Rho_{ab}\sigma
\end{pmatrix}\\
\nabla^{\T^*}_{\bar a} \begin{pmatrix}\sigma\\ \mu_b \end{pmatrix}
&=\begin{pmatrix}\nabla_{\bar a}\sigma\\
\nabla_{\bar a}\mu_b+\Rho_{\bar ab}\sigma \end{pmatrix}.
\end{align*}

For a choice of connection $\nabla\in[\nabla]$ consider now the following
overdetermined system of PDE on sections $\sigma$ of $\cE(1,0)$:
\begin{equation}\label{riccicur}
{\mathrm{(i)}}\: \nabla_{\bar a}\sigma=0\quad\quad\quad
{\mathrm{(ii)}}\: \nabla_{(a}\nabla_{b)}\sigma+\Rho_{(ab)}\sigma=0.
\end{equation}
Suppose $\hat\nabla\in[\nabla]$ is another connection in the c-projective
class. Then the formulae (\ref{changes_on_densities}) imply that
$\hat\nabla_{\bar a}\sigma=\nabla_{\bar a}\sigma$.  Moreover, we deduce from
Proposition \ref{changeform} and the formulae (\ref{changes_on_densities})
that
\begin{equation*}
\hat\nabla_a\hat\nabla_b\sigma=\nabla_a\nabla_b\sigma
+(\nabla_a\Upsilon_b)\sigma-\Upsilon_a\Upsilon_b\sigma,
\end{equation*}
which together with Corollary \ref{Rho_changes} implies that 
\begin{equation*}
\hat\nabla_a\hat\nabla_b\sigma+\hat \Rho_{ab}\sigma
=\nabla_a\nabla_b\sigma+\Rho_{ab}\sigma.
\end{equation*}
This shows that the overdetermined system (\ref{riccicur}) is c-projectively
invariant. Note also that by equation (\ref{another_curvature_on_densities})
we have
\begin{equation}\label{skewpart}
\nabla_{[a}\nabla_{b]}\sigma+P_{[ab]}\sigma
=-\tfrac{1}{2}T_{ab}{}^{\bar c}\nabla_{\bar c}\sigma.
\end{equation}
Therefore $\nabla_a\nabla_b\sigma+\Rho_{ab}\sigma$ is symmetric provided that
$J$ is integrable or that $\sigma$ satisfies equation (i) of
(\ref{riccicur}). The following proposition shows that, if $J$ is integrable,
the tractor connection on $\T^*$ can be viewed as the prolongation of
(\ref{riccicur}):

\begin{prop}\label{PDEcotractor} Suppose $(M,J,[\nabla],\cE(1,0))$ is
a c-projective manifold of dimension $2n\geq 4$. The projection $\pi\colon
\T^*\to\T^*/\Wedge^{1,0}(1,0)\cong\cE(1,0)$ induces a bijection between
sections of $\T^*$ parallel for $\nabla^{\T^*}$ and sections $\sigma$ of
$\cE(1,0)$ that satisfy~\eqref{riccicur} for some \textup(and hence
any\textup) connection $\nabla\in[\nabla]$. Moreover, suppose that
$\sigma\in\cE(1,0)$ is a nowhere vanishing section, then for any connection
$\nabla\in[\nabla]$ the connection
\begin{equation}\label{ricciconn}
\hat\nabla_a\sigma'=\nabla_a\sigma'-(\nabla_a\sigma)\sigma^{-1}\sigma'
\end{equation}
is induced from a connection in the c-projective class, and $\sigma$ with
$\nabla_{\bar a}\sigma=0$ is a solution of \eqref{riccicur} if and only
if \eqref{ricciconn} is Ricci-flat.
\end{prop}

\begin{proof} Suppose $s$ is a parallel section of $\T^*$ and
set $\sigma:=\pi(s)\in\Gamma(\cE(1,0))$. It follows from (\ref{splittingT})
that changing from connection to another in $[\nabla]$ changes the splitting of
$\T^*$ by $\widehat{(\sigma, \mu_b)}=(\sigma, \mu_b+\Upsilon_b\sigma)$.  Hence,
for any connection $\nabla\in[\nabla]$ we can identify $s$ with a section of
the form $(\sigma, \mu_b)$ for some $\mu_b\in\Gamma(\Wedge^{1,0}(1,0))$ and by
definition of the tractor connection $\sigma\in\Gamma(\cE(1,0))$ satisfies
(\ref{riccicur}) for any connection $\nabla\in[\nabla]$. So $\pi$ induces a map
from parallel sections of $\T^*$ to solutions of (\ref{riccicur}).

Conversely, let us contemplate the differential operator $L\colon
\cE(1,0)\to \T^*$, which, for a choice of
connection in~$[\nabla]$, is given by $L\sigma=(\sigma,\nabla_b\sigma)$.
Suppose now $\sigma$ is a solution of~\eqref{riccicur}. Then obviously
$\nabla^{\T^*}_aL\sigma=0$, since \eqref{skewpart} vanishes. By
(\ref{curvature_on_densities}) we have
\begin{align}\label{curv10}
\nabla_{\bar a}\nabla_b\sigma+\Rho_{\bar ab}\sigma&
=(\nabla_{\bar a}\nabla_b-\nabla_b\nabla_{\bar a})\sigma+
\Rho_{\bar ab}\sigma\\
&=\tfrac{1}{n+1}T_{bf}{}^{\bar{e}}T_{\bar{e}\bar{a}}{}^f\sigma
-\Rho_{\bar{a}b}\sigma+\Rho_{\bar ab}\sigma\nonumber\\
&=\tfrac{1}{n+1}
T_{bf}{}^{\bar{e}}T_{\bar{e}\bar{a}}{}^f\sigma\,,\nonumber
\end{align}
which vanishes, since $J$ is integrable. Hence, we also have
$\nabla^{\T^*}_{\bar a}L\sigma=0$. Therefore $L$ maps solutions
$\sigma\in\Gamma(\cE(1,0))$ of (\ref{riccicur}) to parallel sections of $\T^*$
and defines an inverse to the restriction of $\pi$ to parallel section. For the
second statement, assume now that $\sigma$ is a section of $\cE(1,0)$ that is
nowhere vanishing and let $\nabla\in[\nabla]$ be a connection in the
c-projective class. If we change $\nabla$ according to (\ref{cprojchange}) by
$\Upsilon_a=-(\nabla_a\sigma)\sigma^{-1}$ to a connection
$\hat\nabla\in[\nabla]$, then we deduce from Corollary~\ref{changedensity} that
the induced connection on $\cE(1,0)$ is given by~(\ref{ricciconn}).  Moreover,
$\hat\nabla_{a}\sigma=0$. Therefore, using that (\ref{curv10}) vanishes, we
deduce that $\sigma$ with $\hat\nabla_{\bar a}\sigma=0$ satisfies
(\ref{riccicur}) if and only if $\hat\Rho_{ab}\sigma=0$ and $\hat\Rho_{\bar a
  b}\sigma=0$, and hence if and only if $\hat\nabla$ is Ricci-flat.
\end{proof}

More generally, we immediately conclude from equation (\ref{curv10}) that, in
the case of an almost c-projective manifold, i.e.~$J$ is not necessarily
integrable, the corresponding proposition reads as follows:

\begin{prop}\label{PDEcotractor2} Let $(M,J,[\nabla],\cE(1,0))$ be an almost
c-projective manifold of dimension $2n\geq 4$. Then sections $\sigma$
of $\cE(1,0)$ that satisfy {\rm(\ref{riccicur})} are in bijection with
sections of $\T^*$ that are parallel for the connection given by
\begin{equation}\nabla^{\T^*}_a\quad\text{ and }\quad
\nabla^{\T^*}_{\bar a}\begin{pmatrix}\sigma\\ \mu_b
\end{pmatrix}-\frac{1}{n+1}\begin{pmatrix}0\\
T_{bf}{}^{\bar{e}}T_{\bar{e}\bar{a}}{}^f\sigma
\end{pmatrix}.
\end{equation} 
Moreover, suppose $\sigma\in\cE(1,0)$ is a nowhere vanishing section
with $\nabla_{\bar a}\sigma=0$. Then $\sigma$ is a solution of
\eqref{riccicur} if and only if~$\hat\nabla$, defined as in
\eqref{ricciconn}, satisfies $\hat\Rho_{ab}=0$ and $\hat\Rho_{\bar{a}b}
=\frac1{n+1} T_{bf}{}^{\bar{e}}T_{\bar{e}\bar{a}}{}^f$.
\end{prop}

\begin{rem}
Recall that a parallel section of a linear connection of a vector bundle over
a connected manifold, is already determined by its value at one point. The
correspondences established in Propositions \ref{PDEcotractor} and
\ref{PDEcotractor2} between solutions of (\ref{riccicur}) and parallel sections
of a linear connection on $\T^*$ therefore implies, that on a
connected almost c-projective manifold
\[
U:=\{x\in M: \sigma(x)\neq 0\}\subset M
\]
is a dense open subset for any nontrivial solution $\sigma\in\Gamma(\cE(1,0))$
of (\ref{riccicur}). In particular, the second statement of
Proposition~\ref{PDEcotractor}, respectively of
Proposition~\ref{PDEcotractor2}, holds always on the dense open subset~$U$.
\end{rem}

The equations (\ref{riccicur}) define an invariant differential operator of
order two
\begin{equation*}
D^{\T^*}\colon \cE(1,0)\to \Wedge^{0,1}(1,0)\oplus S^2\Wedge^{1,0}(1,0),
\end{equation*}
whose kernel are the solutions of~\eqref{riccicur}.  The differential operator
$D^{\T^*}$ is the first operator in the BGG sequence corresponding to
the co-standard complex tractor bundle; see~\cite{CD,CSS}. The
proof of Proposition~\ref{PDEcotractor2} implies that in order for a (nonzero)
parallel section of the tractor connection on $\T^*$ to exist, it is
necessary that Nijenhuis tensor of $J$ satisfy
$N^J_{\,bf}{}^{\bar{e}}N^J_{\,\bar{e}\bar{a}}{}^f\equiv 0$ and, in this case, parallel
sections of the tractor connection are in bijection with sections in the
kernel of $D^{\T^*}$.

Similarly, one may consider the first BGG operator in the sequence
corresponding to the standard complex tractor bundle $\T$, which is a
first order invariant differential operator, defined, for a choice of
connection $\nabla\in[\nabla]$, by
\begin{align}\notag
D^\T\colon T^{1,0}M(-1,0)&\to
(\Wedge^{0,1}\otimes T^{1,0}M(-1,0))\oplus 
(\Wedge^{1,0}\otimes_\circ T^{1,0}M(-1,0)),\\
X^b&\;\;\mapsto\;
(\nabla_{\bar a}X^b, \nabla_aX^b-\tfrac{1}{n}\nabla_cX^c\delta_{a}{}^b),
\label{StandardBGG}
\end{align}
where the subscript $\circ$ denotes the trace-free part.

\begin{prop}\label{standard_prolong}
Suppose $(M,J,[\nabla],\cE(1,0))$ is an almost c-projective manifold
of dimension $2n\geq 4$. The projection $\pi\colon\T\to
\T/\cE(-1,0)\cong T^{1,0}M(-1,0)$ induces a bijection between sections
of $\T$ that are parallel for the connection given by
\begin{equation}\label{prolongconn2}
\nabla^{\T}_a\quad\quad\mathrm{ and }\quad
\nabla^{\T}_{\bar a}\left(
\begin{matrix}X^b\\ \rho \end{matrix}\right )
-\frac{1}{n(n+1)}\left(
\begin{matrix}0\\ T_{bf}{}^{\bar{e}}T_{\bar{e}\bar{a}}{}^fX^b
\end{matrix}\right )
\end{equation} 
and sections $X^b\in \Gamma(T^{1,0}M(-1,0))$ that are in the kernel of $D^\T$. In
particular, elements $X^b\in\ker D^\T$ with
$N^J_{\,bf}{}^{\bar{e}}N^J_{\,\bar{e}\bar{a}}{}^fX^b=0$ are in bijection with
parallel sections of the tractor connection $\nabla^\T$.
\end{prop}

\begin{proof}
Suppose firstly that $s\in\Gamma(\T)$ is parallel for the connection
(\ref{prolongconn2}) and set $X^b:=\pi(s)$. For a choice of connection
$\nabla\in[\nabla]$ we can identify $s$ with an element of the form
$(X^b,\rho)$, where $\rho\in\Gamma(\cE(-1,0))$. By assumption
$\nabla_{\bar{a}}X^b=0$ and $\nabla_aX^b=-\rho\delta_a{}^b$. Taking the trace
of the latter equation shows that $\rho=-\frac{1}{n}\nabla_cX^c$. Hence, $X^b$
is in the kernel of $D^\T$.

Conversely, suppose $X^b\in\ker D^\T$ and pick a connection
$\nabla\in[\nabla]$.  Then we deduce from Proposition~\ref{rosetta} and
equation~\eqref{another_curvature_on_densities} that
\begin{align}\label{eq}
(\nabla_a\nabla_b-\nabla_b\nabla_a)X^c
&=(\nabla_a\nabla_b-\nabla_b\nabla_a)X^c+T_{ab}{}^{\bar d}\nabla_{\bar d}X^c\\
&=R_{ab}{}^c{}_dX^d+2\Rho_{[ab]}X^c=
W_{ab}{}^c{}_dX^d+2\delta_{[a}{}^c\Rho_{b]d}X^d.\nonumber
\end{align}
By assumption $\nabla_aX^b=\frac{1}{n}\nabla_cX^c\delta_{a}{}^b$ and therefore
(\ref{eq}) implies
\begin{equation}\label{eq2}
\tfrac{1}{n}
(\nabla_a\nabla_dX^d\delta_{b}{}^c-\nabla_b\nabla_dX^d\delta_{a}{}^c)=
W_{ab}{}^c{}_dX^d+2\delta_{[a}{}^c\Rho_{b]d}X^d.
\end{equation}
Taking the trace in (\ref{eq2}) over $a$ and $c$ shows that
$-\frac{1}{n}\nabla_b\nabla_dX^d=\Rho_{bd}X^d$. Hence,
$(X^b,-\frac{1}{n}\nabla_cX^c)$ defines a section $s$ of $\T$ that satisfies
$\nabla_a^\T s=0$. Similarly, since $\nabla_{\bar{a}}X^b=0$,
Proposition~\ref{rosetta} and equation~\eqref{curvature_on_densities} imply
\begin{align}\label{eq3}
\nabla_{\bar a}\nabla_bX^c&=
(\nabla_{\bar a}\nabla_b-\nabla_{b}\nabla_{\bar a})X^c\\
\nonumber&= R_{\bar{a}b}{}^c{}_dX^d
-\tfrac{1}{n+1}T_{bf}{}^{\bar e}T_{\bar e\bar a}{}^fX^c
+\Rho_{\bar{a}b}X^c\\
\nonumber&=
W_{\bar ab}{}^c{}_dX^d-2\Rho_{\bar{a}(b}\delta_{d)}{}^cX^d+\Rho_{\bar{a}b}X^c
-\tfrac{1}{n+1}T_{bf}{}^{\bar e}T_{\bar e\bar a}{}^fX^c.
\end{align} 
Taking the trace in (\ref{eq3}) over $b$ and $c$ implies that
\[
-\tfrac{1}{n}\nabla_{\bar{a}}\nabla_cX^c-\Rho_{\bar{a}c}X^c
=\tfrac{1}{n(n+1)}T_{cf}{}^{\bar e}T_{\bar e\bar a}{}^fX^c.
\] 
Hence, $s$ is parallel for the connection (\ref{prolongconn2}) and it follows
immediately that the differential operator
$X^b\mapsto(X^b,-\frac{1}{n}\nabla_cX^c)$
defines an inverse to the restriction of $\pi$ to parallel sections of
(\ref{prolongconn2}).
\end{proof}

\subsection{Cone description of almost c-projective structures}
\label{cone_construction}

For (real) projective structures there is an alternative description of the
tractor connection as an affine connection on a cone manifold over the
projective manifold \cite{csbook,Fox}. This point of view, which (at least in
spirit) goes back to work of Tracy Thomas, is often convenient---it has for
instance been used in \cite{Arm} to classify holonomy reductions of projective
structures. An analogue holds for almost c-projective manifolds, which we will
now sketch, following the presentation in \cite{csbook} of the projective
case.  This canonical cone connection was used in~\cite{Arm} to realise
c-projective structures as holonomy reductions of projective structures. It
also underlies metric cone constructions~\cite{MatRos,mikes} which we discuss
later.

Let $(M,J, [\nabla], \cE (1,0))$ be an almost c-projective manifold and, as in
Section~\ref{standard_tractors}, let $\tilde P\subset \tilde
G=\textrm{SL}(n+1,\C)$ be the stabiliser of the complex line $\V^{1}$ through
the first basis vector $e_1$ of $\V=\C^{n+1}$. Denote by $\tilde Q\subset
\tilde P$ the stabiliser of $e_1$, which is the derived group of $\tilde P$,
hence a normal complex Lie subgroup. Now set
\begin{equation*} \cC:=
\tilde\G/\tilde Q=\tilde\G\times_{\tilde P}\tilde P/\tilde Q.
\end{equation*}
The natural projection $p_\cC\colon \tilde\G\to \cC$ defines a
(real) principal bundle with structure group $\tilde Q$.  Since the canonical
Cartan connection $\omega$ of $(M,J, [\nabla], \cE (1,0))$ can also be viewed
as a Cartan connection of type $(\tilde G,\tilde Q)$ for $p_\cC$, it induces
an isomorphism
\begin{equation*}
T\cC\cong\tilde\G\times_{\tilde Q}\g/\tilde{\q}.
\end{equation*}
Note that $\cC$ inherits an almost complex structure $J^\cC$ from the almost
complex structure on $\tilde\G$ characterised by $Tp_\cC\circ J^\G=J^\cC\circ
Tp_\cC$.  Furthermore, the projection $\pi_\cC\colon \cC\to M$ defines a
principal bundle with structure group $\tilde P/\tilde Q\cong
\C^\times$. Since $\tilde P/\tilde Q$ can be identified with the nonzero
elements in the complex $\tilde P$-submodule $\V^{1}\subset \V$, we see that
$\cC$ can be identified with the space of nonzero elements in $\cE(-1,0)$ or
with the complex frame bundle of $\cE(-1,0)$. Note that, by construction
(recall that $Tp\circ J^\G=J\circ Tp$), we have $T\pi_\cC\circ J^\cC=J\circ
T\pi_\cC$. By the compatibility of the almost complex structures $J^\G$,
$J^\cC$ and $J$ with the various projections, it follows immediately that
vanishing of the Nijenhuis tensor $N^{J^\G}$ of $J^\G$ implies vanishing of
the Nijenhuis tensors $N^{J^\cC}$, which in turn implies vanishing of
$N^J$. Conversely, Theorem \ref{Intharmcurv} shows that $N^J\equiv 0$ implies
$N^{J^\G}\equiv 0$. This shows in particular that
\begin{equation}\label{integrability_cone} 
N^J\equiv 0\iff N^{J^\cC}\equiv 0,
\end{equation} 
in which case $\pi_\cC\colon\cC\to M$ is a holomorphic principal bundle with
structure group $\C^\times$.
\begin{lem}\label{tangent_bundle_cone} There are canonical
isomorphisms $T\cC\cong\tilde\G\times_{\tilde Q}\V\cong\pi_\cC^* \T$.
\end{lem}
\begin{proof} From the block decomposition~\eqref{block_decom} of $\g$ it
follows that $\g/\tilde{\q}=(\V^1)^*\otimes \V$ and hence $\g/\tilde{\q}\cong
\V$ as $\tilde Q$-modules, i.e.
\begin{equation*}
T\cC\cong\tilde\G\times_{\tilde Q}\g/\tilde{\q}
\cong\tilde\G\times_{\tilde Q}\V \cong \pi_\cC^*(
\tilde\G\times_{\tilde P}\V)=\pi_C^*\T. \qedhere
\end{equation*}
\end{proof}
Hence the standard tractor connection induces an affine connection
$\nabla^\cC$ on $\cC$ which preserves $J^\cC$ and the complex volume form
$\textrm{vol}^\cC\in\Wedge^{n,0}T^*\cC$ that is induced by the standard
complex volume on $\V=\C^{n+1}$. Alternatively, note that $\omega$ can be
extended to a principal connection on the principal $G$-bundle
$\tilde\G\times_{\tilde Q} G$, and since $\V=\C^{n+1}$ is a $G$-module, we
obtain an induced connection on $T\cC\cong\tilde\G\times_{\tilde Q}
G\times_G\V$.

If we identify a vector field $Y\in\mathfrak X(\cC)$ with a $\tilde
Q$-equivariant function $f\colon\tilde\G\to \V$ via
Lemma~\ref{tangent_bundle_cone}, then by~\cite[Theorem 1.5.8]{csbook}, the
equivariant function corresponding to $\nabla^\cC_{X}Y$ for a vector field
$X\in\mathfrak X(\cC)$ is given by
\begin{equation}\label{tractor_formula}
\tilde X\cdot f+\omega(\tilde X)f
\end{equation}
where $\tilde X\in\mathfrak X(\tilde\G)$ is an arbitrary lift of
$X$. Moreover~\cite[Corollary 1.5.7]{csbook} shows that the curvature
$R^\cC\in\Wedge^2 T^* \cC\otimes\A\cC$ of $\nabla^\cC$ is given by
\begin{equation}\label{cone_curvature}
R^\cC(X,Y)(Z)=\kappa(X, Y)\bullet Z,
\end{equation}
where $\A\cC=\tilde\G\times_{\tilde Q}\g \cong\mathfrak{sl}(T\cC, J^\cC)$ and
$\bullet\colon \A\cC\times T\cC\to T\cC$ denotes the vector bundle map induced
by the action of $\g$ on $\V$.  Let us write $T^\cC\in\Wedge^2 T^*\cC\otimes
T\cC$ for the torsion of $\nabla^\cC$. It follows straightforwardly from
\eqref{tractor_formula} that $T^\cC$ is the projection of
$\kappa\in\Gamma(\Wedge^2T^*\cC\otimes \A\cC)$ to $\Wedge^2 T^*\cC\otimes
T\cC$, i.e.~it is the torsion of $\omega$ viewed as a Cartan connection of
type $(\tilde Q, \tilde G)$. In particular, like $\kappa$, the $2$-forms
$T^\cC$ and $R^\cC$ vanish upon insertion of sections of the vertical bundle
of $\pi_\cC$, which is canonically trivialised by the fundamental vector
fields $E$ and $J^\cC E$ generated by $1$ and $i$ respectively.

\begin{prop}\label{cone_connection} 
Suppose $(M,J, [\nabla], \cE (1,0))$ is an almost c-projective manifold. Then
there is a unique affine connection $\nabla^\cC$ on the total space of the
principal bundle $\pi_\cC\colon\cC\to M$ with the following properties\textup:
\begin{enumerate}
\item $\nabla^\cC J^\cC=0$ and $\nabla^\cC \emph{vol}^\cC=0$\textup;
\item $\nabla_{X}^\cC E=X$ for all $X\in\mathfrak X(\cC)$\textup;
\item $\mathcal L_{E}\nabla^\cC=0$ and $\mathcal L_{J^\cC E}\nabla^\cC=0$\textup;
\item $i_{E}T^\cC=0$ and $i_{J^{\cC }E}T^\cC=0$\textup;
\item $T\pi\circ T^\cC$ is of type $(0,2)$ and the $(2,0)$-part of $T^\cC$
vanishes\textup;
\item $\nabla^\cC$ is Ricci-flat\textup;
\item for any \textup(local\textup) section $s$ of $\pi_\cC$ the connection
$s^*\nabla^\cC$ lies in $[\nabla]$.
\end{enumerate}
Moreover, if $J$ is integrable, $\pi_\cC\colon\cC\to M$ is a holomorphic
principal bundle and $T^\cC\equiv 0$.
\end{prop}
\begin{proof}
We already observed that (1) and (4) hold and (2) is an immediate consequence
of \eqref{tractor_formula}. Since we have in addition $i_{E}R^\cC=0$ and
$i_{J^{\cC }E}R^\cC=0$ by \eqref{cone_curvature}, statement (3) follows from
\eqref{LieNabla}.  The statements (5) and (6) are consequences of
$\partial^*\kappa=0$. More explicitly, note that (5) can be simply read off
Proposition \ref{rosetta}, which also shows that if $J$ is integrable,
$T^\cC\equiv 0$. In this case $N^{J^\cC}$, which is up to a constant multiple
the $(0,2)$-part of $T^\cC$, vanishes and $\pi_\cC$ is a holomorphic principal
bundle as calimed.  Statement (6) follows because $T\pi\circ T^\cC$,
viewed as a section of $\Wedge^{0,2} T^*M\otimes TM$, has vanishing trace and
$\partial^*\colon\Wedge^2 T^*M\otimes \mathfrak{gl}(TM,J)\to T^*M\otimes T^*M$
is a multiple of a Ricci-type contraction.  The proof of statement (7) and the
uniqueness of $\nabla^\cC$ we leave to the reader, but note that (1)--(6)
imply that $\nabla^\cC$ descends to the normal Cartan connection on
$T\cC/\C^\times\cong \T$.
\end{proof}

\subsection{BGG sequences}\label{BGG}
For a general parabolic geometry, it was shown in~\cite{CD,CSS} that there are
natural sequences of invariant linear differential operators generalising the
corresponding complexes on the flat model. These are the
Bernstein--Gelfand--Gelfand (BGG) sequences, named after the
constructors~\cite{bgg} of complexes of Verma modules, roughly dual to the
current circumstances.

Here is not the place to say much about the general theory. Instead, we would
like to like to present something of the theory as it applies in the
c-projective case. The point is that the invariant operators that we have
already encountered and are about to encounter, all can be seen as curved
analogues of operators from the BGG complex on~$\CP^n$ (as the flat
model of c-projective geometry).

In fact, the main hurdle in presenting the BGG complex and sequences is in
having a suitable notation for the vector bundles involved. Furthermore, this
notation is already of independent utility since, as foretold in
Remark~\ref{almost_complex_natural_bundles}, it neatly captures the natural
irreducible bundles on an almost complex manifold.

Recall from Section~\ref{almostcomplexmanifolds} that the complexified tangent 
bundle on an almost complex manifold decomposes
\[
\C TM=T^{1,0}M\oplus T^{0,1}M
\]
as does its dual. An alternative viewpoint on these decompositions is that the
tangent bundle $TM$ on any $2n$-dimensional manifold is tautologically induced
from its frame-bundle by the defining representation of ${\mathrm{GL}}(2n,\R)$,
that an almost complex structure is a reduction of structure group for the
frame bundle to ${\mathrm{GL}}(n,\C)\subset{\mathrm{GL}}(2n,\R)$, that the
defining representation of ${\mathrm{GL}}(2n,\R)$ on $\R^{2n}$ complexifies as
${\mathrm{GL}}(2n,\R)$ acting on~$\C^{2n}$ (as real matrices acting on complex
vectors), and finally that this complex representation when restricted to
${\mathrm{GL}}(n,\C)$ decomposes into two irreducibles inducing the bundles
$T^{1,0}M$ and $T^{0,1}M$, respectively. Of course, the dual decomposition
comes from the dual representation, namely ${\mathrm{GL}}(n,\C)$ acting
on~($\C^{2n})^*$. Our notation arises by systematically using the
representation theory of ${\mathrm{GL}}(n,\C)$ as a real Lie group but adapted
to its embedding
$${\mathrm{GL}}(n,\C)\cong G_0\subset P\subset 
G={\mathrm{PSL}}(n+1,\C)$$
as described in Section~\ref{cproCartan}.

For relatively simple bundles, there is no need for any more advanced notation.
In several complex variables, for example, it is essential to break up the
complex-valued differential forms into \emph{types} but that's about it. 
Recall already with $2$-forms
$$\Wedge^2=\Wedge^{0,2}\oplus\Wedge^{1,1}\oplus\Wedge^{2,0}$$
that this complex decomposition is finer that the real decomposition
\begin{equation}\label{real_two_forms}
\Wedge^2T^*M=\left[\Wedge^{0,2}\oplus\Wedge^{2,0}\right]_\R\oplus
\Wedge_\R^{1,1}\end{equation} already discussed in
Section~\ref{almostcomplexmanifolds} following~(\ref{types_of_two_form}). Of
course, as soon as one speaks of holomorphic functions on a complex manifold
one is obliged to work with complex-valued differential forms. However, even if
one is concerned only with real-valued forms and tensors, it is convenient
firstly to decompose the complex versions and then impose reality as, for
example, in~(\ref{real_two_forms}). In fact, this is already a feature of
representation theory in general.

For more complicated bundles, we shall use Dynkin diagrams from~\cite{csbook}
decorated in the style of~\cite{beastwood}. The formal definitions will not be
given here but the upshot is that the general complex irreducible bundle on an 
almost complex manifold will be denoted as 
\begin{equation}\label{the_general_bundle}
\raisebox{-14pt}{\weight pq \pweight abcd efgh\quad
\raisebox{14pt}{(in the $10$-dimensional case ($2n$ nodes in 
general))}}
\end{equation}
where $a,b,c,d,e,f,g,h$ are nonnegative integers whilst, in the first
instance, $p,q$ are real numbers restricted by the requirement that
\begin{equation}\label{integrality}
p+2a+3b+4c+5d = q+2e+3f+4g+5h \bmod{6}\end{equation} 
(again, in the $10$-dimensional case). For example,
\begin{equation*}
\raisebox{14pt}{$T^{1,0}M={}$}
\weight 10 \pweight 0001 0000
\quad\enskip\raisebox{14pt}{$\Wedge^{1,0}={}$}
\weight{-2}0 \pweight 1000 0000
\quad\enskip\raisebox{14pt}{$\Wedge^{0,2}={}$}
\weight 0{-3} \pweight 0000 0100
\end{equation*}
but the point is that this notation covers all bases and, in particular, the
various awkward bundles that have already arisen and will now arise in
this article. In general, the integrality condition (\ref{integrality}) is 
needed, as typified by 
\begin{equation*}
\raisebox{14pt}{$\det\Wedge^{1,0}=\Wedge^{5,0}={}$}
\weight{-6}0 \pweight 0000 0000
\qquad
\raisebox{14pt}{$\cE(p,p)={}$}
\weight pp \pweight 0000 0000
\end{equation*}
but, as already discussed at the end of Section~\ref{almost_c-projective}, on
an almost c-projective manifold we shall suppose that there is a bundle
$\cE(1,0)$ and an identification $\cE(n+1,0):=\cE(1,0)^{n+1}=\Wedge^nT^{1,0}M$,
in which case we shall add
\begin{equation*}
\raisebox{14pt}{$\cE(1,0)={}$}
\weight 10 \pweight 0000 0000
\end{equation*}
to our notation and relax (\ref{integrality}) to requiring merely that 
$p-q$ be integral. In fact, all of $p,q,a,b,c,d,e,f,g,h$ will be integral for 
the rest of this article. 

Our Dynkin diagram notation is well suited to the barred and unbarred
indices that we have already been using. Specifically, a section of
\begin{equation*}
\weight p0 \pweight abcd 0000
\end{equation*}
may be realised as tensors with $a+2b+3c+4d$ unbarred covariant indices, 
having symmetries specified by the Young diagram
$$
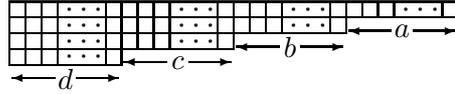
\begin{picture}(168,35)
\put(0,8){\line(1,0){42}}
\put(0,14){\line(1,0){84}}
\put(0,20){\line(1,0){126}}
\put(0,26){\line(1,0){168}}
\put(0,32){\line(1,0){168}}
\put(0,8){\line(0,1){24}}
\put(6,8){\line(0,1){24}}
\put(12,8){\line(0,1){24}}
\put(18,8){\line(0,1){24}}
\put(36,8){\line(0,1){24}}
\put(42,8){\line(0,1){24}}
\put(48,14){\line(0,1){18}}
\put(54,14){\line(0,1){18}}
\put(60,14){\line(0,1){18}}
\put(78,14){\line(0,1){18}}
\put(84,14){\line(0,1){18}}
\put(90,20){\line(0,1){12}}
\put(96,20){\line(0,1){12}}
\put(102,20){\line(0,1){12}}
\put(120,20){\line(0,1){12}}
\put(126,20){\line(0,1){12}}
\put(132,26){\line(0,1){6}}
\put(138,26){\line(0,1){6}}
\put(144,26){\line(0,1){6}}
\put(162,26){\line(0,1){6}}
\put(168,26){\line(0,1){6}}
\put(20.8,8.2){$\cdots$}
\put(20.8,14.2){$\cdots$}
\put(20.8,20.2){$\cdots$}
\put(20.8,26.2){$\cdots$}
\put(62.8,14.2){$\cdots$}
\put(62.8,20.2){$\cdots$}
\put(62.8,26.2){$\cdots$}
\put(104.8,20.2){$\cdots$}
\put(104.8,26.2){$\cdots$}
\put(146.8,26.2){$\cdots$}
\put(21,3){\makebox(0,0){\small $d$}}
\put(17,3){\vector(-1,0){16}}
\put(25,3){\vector(1,0){16}}
\put(63,9){\makebox(0,0){\small $c$}}
\put(59,9){\vector(-1,0){16}}
\put(67,9){\vector(1,0){16}}
\put(105,15){\makebox(0,0){\small $b$}}
\put(101,15){\vector(-1,0){16}}
\put(109,15){\vector(1,0){16}}
\put(147,21){\makebox(0,0){\small $a$}}
\put(143,21){\vector(-1,0){16}}
\put(151,21){\vector(1,0){16}}
\end{picture}$$
and of c-projective weight $(p+2a+3b+4c+5d,0)$. Indeed, for those reluctant 
to trace through the conventions in~\cite{beastwood}, this suffices as a 
definition and then
\begin{equation*}
\weight 0q \pweight 0000 efgh
\quad
\raisebox{14pt}{is the complex conjugate of}\quad
\weight q0 \pweight efgh 0000
\end{equation*}
corresponding to tensors with barred indices in the obvious fashion and
\begin{equation*}
\weight pq \pweight abcd efgh
\quad\raisebox{14pt}{$=$}\quad
\weight p0 \pweight abcd 0000
\enskip\raisebox{14pt}{$\otimes$}\enskip
\weight 0q \pweight 0000 efgh
\raisebox{14pt}{.}
\end{equation*}

Already, these bundles provide locations for the tensors we encountered
earlier. For example,
\begin{equation*}
\raisebox{14pt}{$T_{ab}{}^{\bar{c}}\in\Gamma\biggl($\:}
\weight{-3} 1 \pweight 0100 0001
\raisebox{14pt}{$\biggr)$}
\raisebox{14pt}{\quad\mbox{and}\quad$H_{a\bar{b}}{}^c{}_d\in\Gamma\biggl($\:}
\weight{-3}{-2} \pweight 2001 1000
\raisebox{14pt}{$\biggr)$}\raisebox{14pt}{.}
\end{equation*}
Although the Dynkin diagram notation may at first seem arcane, it comes into
its own when discussing invariant linear differential operators. The
complex-valued de~Rham complex
\begin{equation}\label{deRham}\Wedge^{0,0}\to
\begin{array}c\Wedge^{1,0}\\ \oplus\\ \Wedge^{0,1}\end{array}\to
\begin{array}c\Wedge^{2,0}\\ \oplus\\ \Wedge^{1,1}\\ \oplus\\
\Wedge^{0,2}\end{array}\to\cdots
\end{equation}
becomes
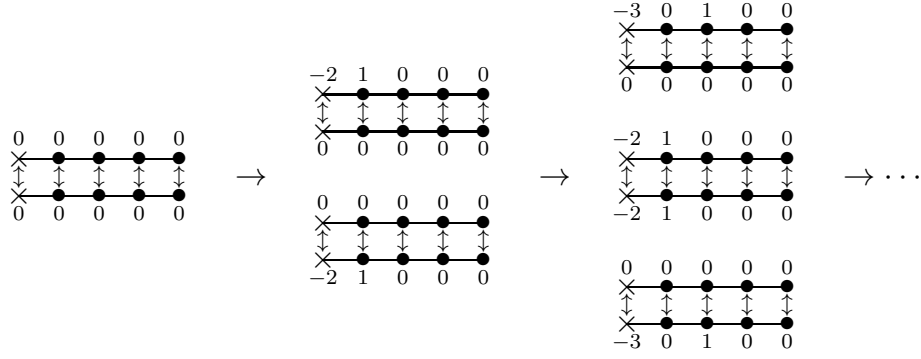
\begin{equation*}
\begin{array}{cccccc}&&&&
\weight{-3}0 \pweight 0100 0000\\[-16pt]
&&
\weight{-2}0 \pweight 1000 0000\\[-16pt] 
\weight 00 \pweight 0000 0000&
\raisebox{14pt}{\enskip$\to$\enskip}&
&\raisebox{14pt}{\enskip$\to$\enskip}&
\weight{-2}{-2} \pweight 1000 1000
&\raisebox{14pt}{\enskip$\to\cdots$}\\[-16pt]
&&
\weight 0{-2} \pweight 0000 1000\\[-16pt]
&&&&
\weight 0{-3} \pweight 0000 0100
\end{array}
\end{equation*}
and in either of them one sees the torsion
$T_{ab}{}^{\bar{c}}\colon\Wedge^{0,1}\to\Wedge^{2,0}$ and its complex
conjugate $T_{\bar{a}\bar{b}}{}^c\colon\Wedge^{1,0}\to\Wedge^{0,2}$ as the
restriction of the exterior derivative $d\colon \Wedge^1 \to \Wedge^2$ to the
relevant bundles (note that
\begin{equation*}
\raisebox{14pt}{${\mathrm{Hom}}(\Wedge^{0,1},\Wedge^{2,0})
=T^{0,1}\otimes\Wedge^{2,0}={}$}
\weight{-3} 1 \pweight 0100 0001
\raisebox{14pt}{,}
\end{equation*}
as expected). In the torsion-free case, the de~Rham complex takes the form
\begin{equation}\label{beginHasse}\setlength{\arraycolsep}{2pt}
\begin{array}{ccccc}
&&&&\Wedge^{2,0}\\[-8pt]
&&\Wedge^{1,0}&\begin{array}c\nearrow\\ \searrow\end{array}\\[-8pt]
\Wedge^{0,0}&\begin{array}c\nearrow\\ \searrow\end{array}
&&&\Wedge^{1,1}\\[-8pt]
&&\Wedge^{0,1}&\begin{array}c\nearrow\\ \searrow\end{array}\\[-8pt]
&&&&\Wedge^{0,2}
\end{array}\end{equation}
familiar from complex analysis and the remarkable fact about c-projectively
invariant linear differential operators is firstly that this pattern is
repeated on the flat model starting with any bundle (\ref{the_general_bundle})
with $p,q\in \Z_{\geq 0}$, for example
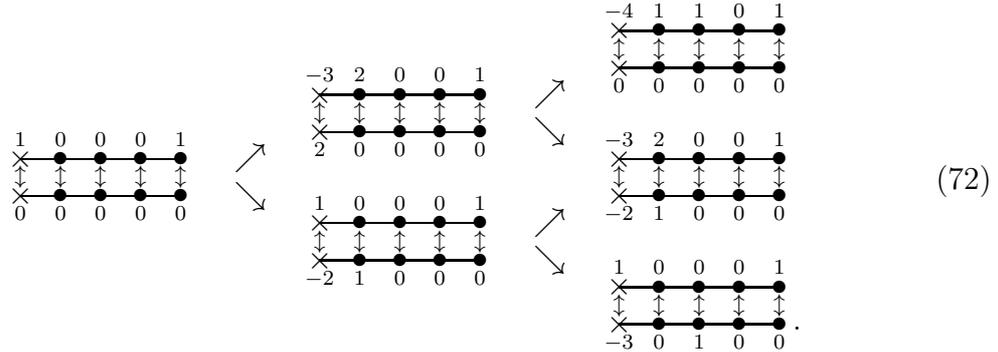
\begin{equation}\label{adjoint_resolution}
\begin{array}{lllll}&&&&
\weight{-4} 0 \pweight 1101 0000\\[-16pt]
&&
\weight{-3} 2 \pweight 2001 0000
&\raisebox{14pt}{$\begin{array}c\nearrow\\ \searrow\end{array}$}\\[-16pt] 
\weight 10 \pweight 0001 0000
&\raisebox{14pt}{$\begin{array}c\nearrow\\ \searrow\end{array}$}&&&
\weight{-3}{-2} \pweight 2001 1000 \\[-16pt]
&&
\weight 1{-2} \pweight 0001 1000
&\raisebox{14pt}{$\begin{array}c\nearrow\\ \searrow\end{array}$}\\[-16pt]
&&&&
\weight 1{-3} \pweight 0001 0100
\raisebox{8pt}{.}
\end{array}\end{equation}
The algorithm for determining the bundles in these patterns is detailed
in~\cite{beastwood} (it is the affine action $\lambda\mapsto
w(\lambda+\rho)-\rho$ of the Weyl group for $G$ along the Hasse diagram
corresponding to the parabolic subgroup~$P$). On $G/P$ in general,
these are complexes of differential operators referred to as
\emph{Bernstein--Gelfand--Gelfand \textup(BGG\textup)} complexes. In our case,
i.e.~on~$\CP^n$, they provide resolutions of the finite-dimensional
representations
\begin{equation*}
\raisebox{-14pt}{
\weight pq \pweight[\bullet] abcd efgh
\quad \raisebox{14pt}{(in case $n=5$ ($2n$ nodes in general))}}
\end{equation*}
of the group $G={\mathrm{PSL}}(n+1,\C)$ as a real Lie group. More
precisely, any finite-dimensional representation $\E$ of $G$ gives
rise to a constant sheaf $G/P\times\E$, which may in turn be
identified with the corresponding homogeneous bundle induced on $G/P$ by means of
\begin{equation}\label{untwist}
G/P\times_P\E\ni[g,e]\mapsto([g],ge)\in G/P\times\E.
\end{equation}
Since the first bundle in the BGG complex is a quotient of this bundle, we
obtain a mapping of $\E$ to the sections of this first bundle and to
say that the complex is a resolution of $\E$ is to say that these
sections are locally precisely the kernel of the first BGG operator (just as
the locally constant functions are precisely the kernel of the first exterior
derivative $d\colon\Wedge^0\to\Wedge^1$). In our
example~(\ref{adjoint_resolution}), this means that 
\begin{equation*}
\begin{array}{lll}
&&
\weight{-3} 0 \pweight 2001 0000\\[-16pt] 
\raisebox{14pt}{$0\to$\enskip}
\weight 10 \pweight[\bullet] 0001 0000
\raisebox{14pt}{\enskip$\longrightarrow$\enskip}
\weight 10 \pweight 0001 0000&
\raisebox{14pt}{$\begin{array}c\nearrow\\ \searrow\end{array}$}\\[-16pt]
&&
\weight 1{-2} \pweight 0001 1000
\end{array}
\end{equation*}
is exact, the $G$-module $\E$ in this case being the adjoint
representation of ${\mathrm{PSL}}(n+1,\C)$ as a complex Lie algebra. 
More generally, the BGG resolutions on $\CP^n$ as a homogeneous 
space for ${\mathrm{PSL}}(n+1,\C)$ begin
\begin{equation}\label{general_firstBGG}
\begin{array}{lll}
&&\begin{picture}(70,36)(0,-7)
\put(5,3){\makebox(0,0){$\times$}}
\put(30,3){\makebox(0,0){$\bullet$}}
\put(50,3){\makebox(0,0){$\bullet$}}
\put(65,3){\makebox(0,0){$\bullet$}}
\put(80,3){\makebox(0,0){$\bullet$}}
\put(5,3){\line(1,0){75}}
\put(5,10.2){\makebox(0,0){\tiny$\updownarrow$}}
\put(30,10.2){\makebox(0,0){\tiny$\updownarrow$}}
\put(50,10.2){\makebox(0,0){\tiny$\updownarrow$}}
\put(65,10.2){\makebox(0,0){\tiny$\updownarrow$}}
\put(80,10.2){\makebox(0,0){\tiny$\updownarrow$}}
\put(5,17){\makebox(0,0){$\times$}}
\put(30,17){\makebox(0,0){$\bullet$}}
\put(50,17){\makebox(0,0){$\bullet$}}
\put(65,17){\makebox(0,0){$\bullet$}}
\put(80,17){\makebox(0,0){$\bullet$}}
\put(5,17){\line(1,0){75}}
\put(0,24){\makebox(0,0){$\scriptstyle\vphantom{pd}-p-2$}}
\put(30,24){\makebox(0,0){$\scriptstyle\vphantom{pd}p+a+1$}}
\put(50,24){\makebox(0,0){$\scriptstyle\vphantom{pd}b$}}
\put(65,24){\makebox(0,0){$\scriptstyle\vphantom{pd}c$}}
\put(80,24){\makebox(0,0){$\scriptstyle\vphantom{pd}d$}}
\put(5,-4){\makebox(0,0){$\scriptstyle\vphantom{pd}q$}}
\put(30,-4){\makebox(0,0){$\scriptstyle\vphantom{pd}e$}}
\put(50,-4){\makebox(0,0){$\scriptstyle\vphantom{pd}f$}}
\put(65,-4){\makebox(0,0){$\scriptstyle\vphantom{pd}g$}}
\put(80,-4){\makebox(0,0){$\scriptstyle\vphantom{pd}h$}}
\end{picture}\\[-16pt] 
\raisebox{14pt}{$0\to$\enskip}
\weight pq \pweight[\bullet] abcd efgh
\raisebox{14pt}{\enskip$\longrightarrow$\enskip}
\weight pq \pweight abcd efgh
&\raisebox{14pt}{$\begin{array}c\nearrow\\ \searrow\end{array}$}\\[-16pt]
&&\begin{picture}(85,36)(0,-7)
\put(5,3){\makebox(0,0){$\times$}}
\put(30,3){\makebox(0,0){$\bullet$}}
\put(50,3){\makebox(0,0){$\bullet$}}
\put(65,3){\makebox(0,0){$\bullet$}}
\put(80,3){\makebox(0,0){$\bullet$}}
\put(5,3){\line(1,0){75}}
\put(5,10.2){\makebox(0,0){\tiny$\updownarrow$}}
\put(30,10.2){\makebox(0,0){\tiny$\updownarrow$}}
\put(50,10.2){\makebox(0,0){\tiny$\updownarrow$}}
\put(65,10.2){\makebox(0,0){\tiny$\updownarrow$}}
\put(80,10.2){\makebox(0,0){\tiny$\updownarrow$}}
\put(5,17){\makebox(0,0){$\times$}}
\put(30,17){\makebox(0,0){$\bullet$}}
\put(50,17){\makebox(0,0){$\bullet$}}
\put(65,17){\makebox(0,0){$\bullet$}}
\put(80,17){\makebox(0,0){$\bullet$}}
\put(5,17){\line(1,0){75}}
\put(5,24){\makebox(0,0){$\scriptstyle\vphantom{pd}p$}}
\put(30,24){\makebox(0,0){$\scriptstyle\vphantom{pd}a$}}
\put(50,24){\makebox(0,0){$\scriptstyle\vphantom{pd}b$}}
\put(65,24){\makebox(0,0){$\scriptstyle\vphantom{pd}c$}}
\put(80,24){\makebox(0,0){$\scriptstyle\vphantom{pd}d$}}
\put(0,-4){\makebox(0,0){$\scriptstyle\vphantom{pd}-q-2$}}
\put(30,-4){\makebox(0,0){$\scriptstyle\vphantom{pd}q+e+1$}}
\put(50,-4){\makebox(0,0){$\scriptstyle\vphantom{pd}f$}}
\put(65,-4){\makebox(0,0){$\scriptstyle\vphantom{pd}g$}}
\put(80,-4){\makebox(0,0){$\scriptstyle\vphantom{pd}h$}}
\end{picture}
\end{array}\end{equation}
for nonnegative integers $p,a,b,c,d,q,e,f,g,h$ constrained
by~(\ref{integrality}). We may drop the constraint (\ref{integrality}) by
considering $\CP^n$ instead as a homogeneous space for
${\mathrm{SL}}(n+1,\C)$, as is perhaps more usual. Having done that,
the standard representation of ${\mathrm{SL}}(n+1,\C)$ on
$\C^{n+1}$ gives rise to the BGG resolution
\begin{equation}\label{standardBGGresolution}\begin{array}{lll}
&&
\weight{-2}0 \pweight 1001 0000\\[-16pt] 
\raisebox{14pt}{$0\to$\enskip}
\weight 00 \pweight[\bullet] 0001 0000
\raisebox{14pt}{\enskip$\longrightarrow$\enskip}
\weight 00 \pweight 0001 0000
&\raisebox{14pt}{$\begin{array}c
\nabla_a\\ \nearrow\\ \searrow\\ \nabla_{\bar{a}}
\end{array}$}\\[-16pt]
&&
\weight 0{-2} \pweight 0001 1000
\end{array}\end{equation}
where the operators $\nabla_a$ and $\nabla_{\bar{a}}$ are, more explicitly and 
as noted in~(\ref{StandardBGG}),
\begin{equation}\label{standard_firstBGGoperator}
X^b\mapsto(\nabla_aX^b)_\circ\quad\mbox{and}\quad
X^b\mapsto\nabla_{\bar{a}}X^b\end{equation}
where $X^b$ is a vector field of type $(1,0)$ and of c-projective weight
$(-1,0)$ and the subscript $\circ$ means to take the trace-free part. Notice
that these are exactly the operators implicitly encoded in the standard tractor
connection (\ref{tractorconna}) and~(\ref{tractorconnbara}). More precisely,
the filtration (\ref{filtrationT}) is equivalent to the short exact sequence 
of vector bundles
$$\begin{array}{ccccccc}\raisebox{14pt}{$0\to$}&
\weight{-1} 0 \pweight 0000 0000
&\raisebox{14pt}{$\longrightarrow$}&
\weight 00 \pweight[\bullet] 0001 0000
&\raisebox{14pt}{$\longrightarrow$}&
\weight 00 \pweight 0001 0000
&\raisebox{14pt}{$\to 0$}\\
&\|&&\|&&\|\\[4pt]
&\cE(-1,0)&&\T&&T^{1,0}M(-1,0)\end{array}$$
and on the flat model, namely $\CP^n$ as a homogeneous space 
for~${\mathrm{SL}}(n+1,\C)$, the tractor connection is the exactly 
the flat connection induced by~(\ref{untwist}). In the c-projectively flat
case, the remaining entries in~(\ref{tractorconna}) and 
(\ref{tractorconnbara}), namely
$$\nabla_a\rho-\Rho_{ab}X^b\quad\mbox{and}\quad
\nabla_{\bar{a}}\rho-\Rho_{\bar{a}b}X^b$$
may be regarded as quantities whose vanishing are differential consequences of
setting
\[
\nabla_aX^b+\rho\delta_a{}^b=0\quad\mbox{and}\quad\nabla_{\bar{a}}X^b=0.
\]
Hence, they add no further conditions to being in the kernel of the first BGG
operator (\ref{standard_firstBGGoperator}) and the exactness of
(\ref{standardBGGresolution}) follows. The same reasoning pertains in the
curved but torsion-free setting and leads to the standard tractor connection
being obtained by prolongation of the first BGG operator. This is detailed in
Proposition~\ref{standard_prolong}. For more complicated representations, the
tractor connection may not be obtained by prolongation in the curved setting,
even if torsion-free. This phenomenon will soon be seen in two key examples,
specifically in the connection~(\ref{Infautconn}) and
Proposition~\ref{Infautconn2} concerned with infinitesimal automorphisms and in
Proposition~\ref{compKaehler}, Theorem~\ref{MetriProlongation}, and
Corollary~\ref{MetriCorollary} dealing with the metrisability of c-projective 
structures. With reference to the general first BGG 
operators~(\ref{general_firstBGG}), the following cases occur prominently in 
this article.

\smallskip\noindent\framebox{
\weight 00 \pweight[\bullet] 0001 0000
\raisebox{14pt}{$\;\xrightarrow{\;\pi\;}$}
\weight 00 \pweight 0001 0000} 
This is the standard complex tractor bundle $\T$ and its canonical
projection to $T^{1,0}M(-1,0)$.

\noindent\framebox{
\weight 10 \pweight[\bullet] 0001 0000 
\raisebox{14pt}{$\;\xrightarrow{\;\pi\;}$}
\weight 10 \pweight 0001 0000} 
This is the adjoint tractor bundle $\A M$ to be considered in
Section~\ref{infinitesimal_automorphisms} and its canonical projection to
$T^{1,0}M$. A first BGG operator acting on $T^{1,0}M$ is given in
Remark~\ref{firstBGG_is_different}.

\noindent\framebox{
\weight 00 \pweight[\bullet] 0001 0001
\raisebox{14pt}{$\;\xrightarrow{\;\pi\;}$}
\weight 00 \pweight 0001 0001} 
This is the tractor bundle arising in the metrisability of c-projective
structures to be discussed in Section~\ref{sec:metrisability} and a first 
BGG operator is given in~(\ref{firstBGG_metrisability}).

\noindent\framebox{
\weight 11 \pweight[\bullet] 0000 0000
\raisebox{14pt}{$\;\xrightarrow{\;\pi\;}$}
\weight 11 \pweight 0000 0000} 
This is the dual of the previous case and arises in
Section~\ref{cpHessianSection}, which is concerned with the first BGG operator
$\D^\cW$ defined in (\ref{cprojHessian})
and acting on c-projective densities of weight~$(1,1)$. It is a second order 
and c-projectively invariant operator.

In fact, there is quite a bit of flexibility in what one might allow as BGG
operators, already for the first ones~(\ref{general_firstBGG}). For example,
the operator $\D^{\A}$ in
Remark~\ref{firstBGG_is_different} is rather different from the c-projectively
invariant operators occurring as the left hand sides of (1) and (2) in
Proposition~\ref{Killingoperator}. Even for the bundle $T^{1,0}M(-1,0)$ in
(\ref{standardBGGresolution}) corresponding to standard complex tractors, there
is the option of replacing the second operator in
(\ref{standard_firstBGGoperator}) by
$$X^c\mapsto\nabla_{\bar{a}}X^c+T_{\bar{a}\bar{b}}{}^cX^{\bar{a}}$$
in line with equation~(1) in Proposition~\ref{Killingoperator}. Only in the
torsion-free case do these operators agree (with each other and the usual
$\bar\partial$-operator).

On the flat model, however, there is no choice. The operators occurring in the
BGG complexes are unique up to scale. Moreover, there are no other
c-projectively invariant linear differential operators: every such operator is
determined by its symbol and the BGG operators comprise a classification. In
the curved setting it is necessary to add curvature correction terms and
there is almost always some choice. Regarding existence, it is shown in
\cite{CD} and \cite{CSS} that such curved analogues always exist. However,
even for the BGG sequence associated to the trivial representation, the
resulting operators are different if there is torsion. Specifically, the
construction in~\cite{CSS} follows the Hasse diagram beginning as
in~(\ref{beginHasse}). In particular, there is no place for the torsion as an
operator $\Wedge^{0,1}\to\Wedge^{2,0}$ whereas, in~\cite{CD}, the first BGG
sequence associated to the trivial representation for the case of $|1|$-graded
geometry such as c-projective geometry, is just the de~Rham
complex~(\ref{deRham}).

In summary, the BGG operators on $\CP^n$ provide models for what one
should expect in the curved setting. In the flat case, there is no choice. In
the curved case, there is a certain degree of flexibility, more so when there
is torsion. Finally, the general theory of parabolic geometry~\cite{csbook}
provides a location for \emph{harmonic curvature}, as already discussed in
Sections~\ref{intropargeom} and~\ref{sectionharmcurv} and Kostant's
Theorem~\cite{Kostant} on Lie algebra cohomology provides the location for this 
curvature, namely the three bundles appearing in the second step of the BGG 
sequence~(\ref{adjoint_resolution}) for the adjoint representation whilst the 
two bundles at the first step locate the infinitesimal deformations of an 
almost c-projective structure, in line with the general theory~\cite{InfautCap}.

\subsection{Adjoint tractors and infinitesimal automorphisms}
\label{infinitesimal_automorphisms}

For a vector field $X$ on a manifold $M$ we write $\cL_X$ for the \emph{Lie
derivative} along $X$ of tensor fields on $M$.  Recall that there is also a
notion of a \emph{Lie derivative of an affine connection} $\nabla$ along a
vector field $X\in\Gamma(TM)$. It is given by the tensor field
$$\cL_X\nabla\colon TM\to T^*M\otimes TM$$
characterised by
$$(\cL_X\nabla)(Y)\equiv \cL_X(\nabla Y)-\nabla \cL_XY$$
for any vector field $Y\in\Gamma(TM)$. In abstract index notation we adopt the
convention that $(\cL_X\nabla)_{\alpha\beta}{}^\gamma
Y^\beta=\cL_X(\nabla_\alpha Y^\gamma)$.

\begin{defin}
A \emph{c-projective vector field} or \emph{infinitesimal automorphism} of
an almost c-projective manifold $(M,J,[\nabla])$ of dimension $2n\geq 4$ is
a vector field $X$ on $M$ that satisfies
\begin{itemize}
\item $\cL_XJ\equiv 0$ (i.e.~$[X,JY]=J[X,Y]$ for all vector fields
$Y\in\Gamma(TM)$) 
\item $(\cL_X\nabla)_{\alpha\beta}{}^\gamma=
\upsilon_{\alpha\beta}{}^\gamma$, where 
$\upsilon_{\alpha\beta}{}^\gamma\in\Gamma(S^2T^*M\otimes TM)$ is a tensor of 
the form~(\ref{cprojchange}).
\end{itemize}
\end{defin}

\noindent Note that $X\in\Gamma(TM)$ is an infinitesimal automorphism of an
almost c-projective manifold precisely if its flow acts by local automorphisms
thereof. 

Let us rewrite the two conditions defining a c-projective vector field as a
system of differential equations on a vector field $X$ of $M$. Expressing the
Lie bracket in terms of a connection $\nabla\in[\nabla]$ and its torsion shows
that $\cL_XJ=0$ is equivalent to
\begin{equation}\label{LieJ}
T_{\alpha\beta}{}^\gamma X^\alpha=-\tfrac{1}{2}(\nabla_\beta X^\gamma
+J_{\epsilon}{}^\gamma J_{\beta}{}^\zeta\nabla_\zeta X^\epsilon).
\end{equation} 
for one (and hence any) connection $\nabla\in[\nabla]$. Moreover, one deduces
straightforwardly from the definition of the Lie derivative of a connection
that for any connection $\nabla\in[\nabla]$ we have
\begin{equation}\label{LieNabla}
(\cL_X\nabla)_{\beta\delta}{}^\gamma=
R_{\alpha\beta}{}^\gamma{}_\delta X^\alpha
+\nabla_\beta\nabla_\delta X^\gamma
+\nabla_\beta(T_{\alpha\delta}{}^\gamma X^\alpha).
\end{equation}
Via the isomorphism $TM\cong T^{1,0}M$ we may write the result as a
differential equation on~$X^a$: equation~(\ref{LieJ}) then becomes
\begin{equation}\label{LieJ2}
\nabla_{\bar b}X^c+T_{\bar a\bar b}{}^cX^{\bar a}=0.
\end{equation}
Since a tensor $\upsilon_{\beta \delta}{}^\gamma$ of the form
(\ref{cprojchange}) satisfies
$\upsilon_{\bar{b}d}{}^c=0=\upsilon_{\bar{b}\bar{d}}{}^c$, the
equation $(\cL_X\nabla)_{\beta}{}^\gamma{}_\delta
=\upsilon_{\beta\delta}{}^\gamma$ can be equivalently encoded by the three
equations
\begin{equation}
(\cL_X\nabla)_{bd}{}^c=\upsilon_{bd}{}^c\qquad
(\cL_X\nabla)_{\bar bd}{}^c\equiv 0\qquad
(\cL_X\nabla)_{\bar b\bar d}{}^c\equiv 0,
\end{equation}
or, alternatively, their complex conjugates.
\begin{lem}\label{Lemma1}
If $X^a\in\Gamma(T^{1,0}M)$ satisfies the invariant differential equation
{\rm(\ref{LieJ2})}, then
$$(\cL_X\nabla)_{\bar bd}{}^c\equiv 0\qquad
(\cL_X\nabla)_{\bar b\bar d}{}^c\equiv 0,$$
for any connection $\nabla\in[\nabla]$.
\end{lem}
\begin{proof}
Equation \eqref{LieNabla} and the formulae of Proposition \ref{rosetta} imply 
\begin{align*}
(\cL_X\nabla)_{\bar bd}{}^c&
=R_{a\bar b}{}^c{}_dX^a+R_{\bar a\bar b}{}^c{}_dX^{\bar a}
+\nabla_{\bar b}\nabla_d X^c\\
&=-R_{\bar ba}{}^c{}_dX^a+R_{\bar bd}{}^c{}_aX^a
+\nabla_d\nabla_{\bar b}X^c+R_{\bar a\bar b}{}^c{}_dX^{\bar a}\\
&=2 R_{\bar b[d}{}^c{}_{a]}X^a+\nabla_d\nabla_{\bar b}X^c
-(\nabla_dT_{\bar b\bar a}{}^c)X^{\bar a}.
\end{align*}
Hence, the Bianchi symmetry \eqref{bianchi2} shows that
\[
(\cL_X\nabla)_{\bar bd}{}^c=T_{da}{}^{\bar e}
T_{\bar e\bar b}{}^c X^a+\nabla_d\nabla_{\bar b}X^c
-(\nabla_dT_{\bar b\bar a}{}^c)X^{\bar a},
\]
which evidently vanishes if 
$\nabla_{\bar b}X^c=T_{\bar b\bar a}{}^cX^{\bar a}$, and consequently also  
$\nabla_dX^{\bar e}=T_{da}{}^{\bar{e}}X^a$. 
As $(\cL_X\nabla)_{\bar b\bar d}{}^c{}=
\nabla_{\bar b}\nabla_{\bar d} X^c
+\nabla_{\bar b}(T_{\bar a\bar d}{}^cX^{\bar a})$, 
the second assertion is obvious.
\end{proof}

According to Lemma~\ref{Lemma1}, it remains to rewrite
$(\cL_X\nabla)_{b}{}^c{}_d=\upsilon_{bd}{}^c$ as a differential equation
on~$X^a$. Note that we have
\begin{equation}\label{LieNabla2}
(\cL_X\nabla)_{b}{}^c{}_d=\nabla_b\nabla_d X^c
+R_{ab}{}^c{}_dX^a+R_{\bar ab}{}^c{}_dX^{\bar a}.
\end{equation}
The Bianchi symmetry~\eqref{bianchi1} $R_{[bd}{}^c{}_{a]}\equiv 0$ implies 
$R_{bd}{}^c{}_a=-2 R_{a[b}{}^c{}_{d]}$. Moreover, 
\begin{align*}
\nabla_b\nabla_d X^c=\nabla_{(b}\nabla_{d)} X^c+\nabla_{[b}\nabla_{d]} X^c
=\nabla_{(b}\nabla_{d)} X^c
+\tfrac{1}{2}(R_{bd}{}^c{}_aX^a-T_{bd}{}^{\bar e}\nabla_{\bar e}X^c). 
\end{align*}
Therefore, we may rewrite (\ref{LieNabla2}) as 
\begin{align}\label{LieNabla3}
(\cL_X\nabla)_{b}{}^c{}_d&=\nabla_{(b}\nabla_{d)} X^c
+R_{a(b}{}^c{}_{d)}X^a+R_{\bar a(b}{}^c{}_{d)}X^{\bar a}
\nonumber\\
&\qquad{}-\tfrac{1}{2}T_{bd}{}^{\bar e}\nabla_{\bar e}X^c
+\tfrac{1}{2}T_{bd}{}^{\bar e}T_{\bar e\bar a}{}^cX^{\bar a},
\end{align}
where we used the Bianchi symmetry \eqref{bianchi2} given by 
$R_{\bar a[b}{}^c{}_{d]}=
\frac{1}{2}T_{bd}{}^{\bar e}T_{\bar e\bar a}{}^cX^{\bar a}$. The 
torsion terms of (\ref{LieNabla3}) evidently cancel if $X^a$
satisfies~(\ref{LieJ2}). 

Suppose now that $2n\geq 6$. Then we deduce from Proposition~\ref{rosetta} that
\begin{equation}
R_{a(b}{}^c{}_{d)}X^{a}= W_{a(b}{}^c{}_{d)}X^{a}+\Rho_{(bd)}X^{c}
+\delta_{(b}{}^c\Rho_{d)a}X^{a}-\delta_{b}{}^c\Rho_{ad}X^{a}-\delta_{d}{}^c\Rho_{ab}X^{a},
\end{equation}
where the third term and the two last terms already define two tensors of the 
form~(\ref{cprojchange}). Moreover, we 
obtain by Proposition \ref{rosetta} that
\begin{align}
R_{\bar a(b}{}^c{}_{d)}X^{\bar a}=H_{\bar{a}b}{}^c{}_{d}X^{\bar{a}}
+\tfrac{1}{n+1} \delta_{(b}{}^cT_{d)f}{}^{\bar{e}}T_{\bar{e}\bar{a}}{}^fX^{\bar{a}}
-2\Rho_{\bar a(b}\delta_{d)}{}^cX^{\bar a},
\end{align}
where the last two terms are again already of the form~\eqref{cprojchange}.
Therefore, we conclude:

\begin{prop}\label{Killingoperator} Suppose $(M,J,[\nabla])$ is an almost
c-projective manifold of dimension $2n\geq 6$. A vector field 
$X^a\in\Gamma(T^{1,0}M)$ is c-projective if and only if it satisfies the 
following equations
\begin{enumerate}
\item $\nabla_{\bar b}X^c+T_{\bar a\bar b}{}^cX^{\bar a}=0$
\item $(\nabla_{(b}\nabla_{d)}X^c+\Rho_{(bd)}X^c
+W_{a(b}{}^c{}_{d)}X^a+H_{\bar a b}{}^c{}_dX^{\bar a})_\circ=0,$
\end{enumerate}
where the subscript $\circ$ denotes the trace-free part.
\end{prop}
Due to Proposition \ref{Weyl_gone} for $2n=4$ the equations take a simpler form:
\begin{prop}\label{Killingoperator2} Suppose $(M,J,[\nabla])$ is an almost
c-projective manifold of dimension $2n=4$. A vector field 
$X^a\in\Gamma(T^{1,0}M)$ is c-projective if and only if it satisfies the
following equations
\begin{enumerate}
\item $\nabla_{\bar b}X^c+T_{\bar a\bar b}{}^cX^{\bar a}=0$
\item $(\nabla_{(b}\nabla_{d)}X^c+\Rho_{(bd)}X^c
+H_{\bar a b}{}^c{}_dX^{\bar a})_\circ=0,$
\end{enumerate}
where the subscript $\circ$ denotes the trace-free part.
\end{prop}

The equations in Propositions~\ref{Killingoperator} and~\ref{Killingoperator2}
define an invariant differential operator
\begin{equation*}
D^{\mathrm{aut}}\colon T^{1,0}M\to
(\Wedge^{1,0}\otimes T^{0,1}M)\oplus (S^2\Wedge^{1,0}\otimes T^{1,0}M)_\circ,
\end{equation*}
whose kernel comprises the infinitesimal automorphisms of $(M,J,[\nabla])$.

Let us recall some facts about infinitesimal automorphisms of Cartan
geometries.
\begin{defin} Suppose $(p\colon\G\to M,\omega)$ is a Cartan geometry.
A vector field $\tilde{X}\in\Gamma(T\G)$ is called an \emph{infinitesimal
  automorphism} of $(p\colon\G\to M,\omega)$, if $\tilde{X}$ is
right-invariant for the principal right action on $\G$ and
$\cL_{\tilde{X}}\omega=0$.
\end{defin}
A Cartan connection $\omega$ on $p\colon\G\to M$ induces a bijection
between right-invariant vector fields $\tilde{X}\in\Gamma(T\G)$ and
equivariant functions $\omega(\tilde{X})\colon\G\to\g$. Hence,
right-invariant vector fields on $\G$ are in bijection with sections of the
adjoint tractor bundle~$\A M$. A section $s$ of $\A M$ corresponds to an
infinitesimal automorphism of the Cartan geometry if and only if $s$ is
parallel for the linear connection
\begin{equation}\label{Infautconn}
\nabla^{\A}s+\kappa(\Pi (s),\,\cdot),
\end{equation}
where $\nabla^{\A}$ is the adjoint tractor connection, $\Pi\colon \A
M\to TM$ the natural projection, and $\kappa\in\Omega^2(M,\A M)$ the
curvature of the Cartan geometry; see~\cite{InfautCap,csbook}.

The equivalence of categories established in Theorem~\ref{EquivCat} implies
that any infinitesimal automorphism $X\in\Gamma(TM)$ of an almost c-projective
manifold can be lifted uniquely to an infinitesimal automorphism of its normal
Cartan geometry and conversely, any infinitesimal automorphism of the Cartan
geometry projects to an infinitesimal automorphism of the underlying almost
c-projective manifold.  This implies, in particular, that $\Pi$ induces a
bijection between sections of the adjoint tractor bundle of the almost
c-projective manifold that are parallel for the connection (\ref{Infautconn})
and infinitesimal automorphisms of the almost c-projective manifold.

For the convenience of the reader let us explicitly compute the modified
adjoint tractor connection~(\ref{Infautconn}). For these purposes let us
identify the adjoint tractor bundle with the $(1,0)$-part of its
complexification. As such it is filtered as
\begin{equation*}
\A M=\A^{-1}M\supset \A^0M\supset \A^1M,
\end{equation*}
where $\A^{-1}M/\A^{0}M\cong T^{1,0}M$, $\A^{0}M/\A^{1}M\cong\gl(T^{1,0}M,\C)$
and $\A^{1}M\cong \Wedge^{1,0}$. Hence, for any choice of
connection~$\nabla\in[\nabla]$, we can identify an element of $\A M$ with a
triple
\begin{equation*}
\begin{pmatrix}X^b\\ \adend_b{}^c \\ \mu_b \end{pmatrix},\text{ where }
\left\{\begin{array}{l} X^b\in T^{1,0}M\\
\adend_b{}^c\in \gl(T^{1,0}M,\C)\\
\mu_b\in \Wedge^{1,0}.
\end{array}\right.
\end{equation*}
Note that $\adend_b{}^c$ may be decomposed further into its trace-free and
trace parts according to the decomposition~\eqref{decompg0}
$\gl(T^{1,0}M,\C)\cong \sgl(T^{1,0}M,\C)\oplus\cE(0,0)$.  However, we
shall not make use of this decomposition. {From} the formulae
(\ref{tractorconna}) and (\ref{tractorconnbara}) defining the tractor
connection on the standard complex tractor bundle $\T$ one easily deduces that
tractor connection on $\A M=\sgl (\T)$ is given by
\begin{align}\label{Atractor1}
\nabla^{\A}_a \begin{pmatrix}X^b\\ \adend_b{}^c\\ \mu_b \end{pmatrix}
&=\begin{pmatrix}
\nabla_aX^b-\adend_a{}^{b}\\
\nabla_a\adend_{b}{}^c+\delta_{a}{}^c\mu_b+\Rho_{ab}X^c+
(\mu_a+\Rho_{ad}X^d)\delta_{b}{}^c\\
\nabla_a\mu_b-\Rho_{ac}\adend_{b}{}^c
\end{pmatrix}
\\
\label{Atractor2}
\nabla^{\A}_{\bar{a}}
\begin{pmatrix}X^b\\ \adend_b{}^c\\ \mu_b
\end{pmatrix}
&=\begin{pmatrix}\nabla_{\bar{a}}X^b\\
\nabla_{\bar{a}}\adend_{b}{}^c+\Rho_{\bar{a}b}X^c
+\Rho_{\bar{a}d}X^d\delta_{b}{}^c\\
\nabla_{\bar{a}}\mu_b-\Rho_{\bar{a}c}\adend_{b}{}^c
\end{pmatrix}.
\end{align}
{From} (\ref{Infautconn}) we deduce that:
\begin{prop} \label{Infautconn2}
Suppose $(M,J,[\nabla])$ is an almost c-projective manifold of dimension
$2n\geq6$. Then the projection $\Pi\colon \A M\to\A M/\A^{0}M\cong T^{1,0}M$
induces a bijection between sections of $\A M$ that are parallel for
\begin{gather*}
\nabla^{\A}_a\begin{pmatrix}X^b\\ \adend_b{}^c\\ \mu_b
\end{pmatrix}
+\begin{pmatrix}0\\ 
W_{da}{}^c{}_b X^d+W_{\bar{d}a}{}^c{}_b X^{\bar{d}}
+T_{af}{}^{\bar e}T_{\bar e\bar d}{}^fX^{\bar d}\delta _b{}^c\\
C_{cab}X^c+C_{\bar c ab}X^{\bar c}
\end{pmatrix}
\\
\nabla^{\A}_{\bar{a}}\begin{pmatrix}X^b\\ \adend_b{}^c\\ \mu_b\end{pmatrix}
+\begin{pmatrix}T_{\bar c\bar a}{}^bX^{\bar c}\\ 
W_{d\bar a}{}^c{}_bX^d+ W_{\bar d\bar a}{}^c{}_bX^{\bar d}
+(\nabla_eT_{\bar d\bar a}{}^eX^{\bar d}-T_{df}{}^{\bar e}T_{\bar e\bar a}{}^fX^d)
\delta_{b}{}^c\\
C_{c\bar ab}X^c+ C_{\bar c\bar a b}X^{\bar c}
\end{pmatrix}
\end{gather*}
and infinitesimal automorphisms of the almost c-projective manifold.
\end{prop}
Proposition~\ref{Infautconn2} can, of course, also be obtained directly by
prolonging cleverly the equations of Proposition \ref{Killingoperator}. Note
that the form of the equations in Proposition \ref{Killingoperator}
immediately shows that $\Pi$ maps parallel sections for the connection in
Proposition \ref{Infautconn2} to c-projective vector fields. To see the
converse, one may verify that that for a c-projective vector field $X^b$ and
for any choice of $\nabla\in[\nabla]$ the section
$$\begin{pmatrix}X^b\\ \adend_b{}^c\\ \mu_b
\end{pmatrix}=\begin{pmatrix}X^b\\ \nabla_bX^c\\
-\frac{1}{n+1}(\nabla_a\nabla_bX^a+2\Rho_{(ab)}X^b)
\end{pmatrix}$$
is parallel for the connection given in Proposition~\ref{Infautconn2} and
observe that this differential operator indeed defines the inverse to the
claimed bijection.

\begin{rem}
If the dimension of the almost c-projective manifold is $2n=4$, then
Proposition~\ref{Infautconn2} still holds taking into account that
$W_{ab}{}^c{}_d\equiv 0$.
\end{rem}

\begin{rem}\label{firstBGG_is_different}
Note that the differential operator
\begin{align*}
D^{\A}\colon T^{1,0}M&\to(\Wedge^{1,0}\otimes T^{0,1}M)\oplus
(S^2\Wedge^{1,0}\otimes T^{1,0}M)_\circ\\
X^c&\mapsto
(\nabla_{\bar b} X^c, (\nabla_{(b}\nabla_{d)}X^c+X^c\Rho_{(bd)})_\circ)
\end{align*}
is also invariant. It is the first operator in the BGG sequence of the adjoint
tractor bundle. As for projective structures, this operator differs from
$\D^{\mathrm{aut}}$, the operator that controls the infinitesimal
automorphisms of the almost c-projective manifold. For a discussion of this
phenomena in the context of general parabolic geometries see~\cite{InfautCap}.
\end{rem}

\section{Metrisability of almost c-projective structures}
\label{sec:metrisability}

On any \bps/K\"ahler manifold $(M,J,g)$ one may consider the c-projective
structure that is induced by the Levi-Civita connection of~$g$. The
c-projective manifolds that arise in this way from a \bps/K\"ahler metric are
the most extensively studied c-projective manifolds;
see~\cite{DM,IT,Mikes,OtsTash} and, more recently,~\cite{FKMR,MR,Mettler}. A
natural but difficult problem in this context is to characterise the
c-projective structures that arise from \bps/K\"ahler metrics or, more
generally, the almost c-projective structures that arise from
$(2,1)$-symplectic (also called quasi-K\"ahler) metrics.  In the following
sections we shall show that, suitably interpreted, this problem is controlled
by an invariant linear overdetermined system of PDE and we shall explicitly
prolong this system. Under the assumptions that $J$ is integrable and the
c-projective manifold $(M,J,[\nabla^g])$ arose via the Levi-Civita connection
$\nabla^g$ of a K\"ahler metric~$g$, a prolongation of the system of PDE
governing the K\"ahler metrics that are c-projectively equivalent to $g$ was
first given in~\cite{DM,Mikes} and rediscovered in the setting of Hamiltonian
$2$-forms on K\"ahler manifolds in~\cite{ACG}.

\subsection{Almost Hermitian manifolds}\label{Kaehlersection}
We begin by recalling some basic facts.

\begin{defin} Suppose $(M, J)$ is an almost complex manifold of dimension $2n
\geq 4$.  A \emph{Hermitian metric} on $(M, J)$ is a \bps/Riemannian metric
$g_{\alpha\beta}\in\Gamma(S^2T^*M)$ that is $J$-invariant:
\begin{equation*}
J_{\alpha}{}^\gamma J_{\beta}{}^\delta g_{\gamma\delta}=g_{\alpha\beta}.
\end{equation*}
We call such a triple $(M, J, g)$ an \emph{almost Hermitian manifold},
or, if $J$ is integrable, a \emph{Hermitian manifold}. Note that we drop the
awkward \bps/ prefix.
\end{defin}
\noindent To an almost Hermitian manifold $(M, J, g)$ one can associate a
nondegenerate $J$-invariant $2$-form $\Kf\in\Gamma(\Wedge^2T^*M)$ given by
\begin{equation}\label{Kaehlerform}
\Kf_{\alpha\beta}:= J_{\alpha}{}^\gamma g_{\gamma\beta}. 
\end{equation}
It is called the \emph{fundamental $2$-form} or \emph{K\"ahler form} of $(M,
J, g)$. If $\Kf$ is closed ($d\Kf=0$), we say $(M,J,g)$ is \emph{almost
  K\"ahler} or \emph{almost pseudo-K\"ahler} accordingly as $g$ is Riemannian
or pseudo-Riemannian; the ``almost'' prefix is dropped if $J$ is integrable.

We write $g^{\alpha\beta}$ for the inverse of the metric $g_{\alpha\beta}$:
\begin{equation*}
g_{\alpha\gamma}g^{\gamma\beta}=\delta_{\alpha}{}^{\beta}.
\end{equation*}
We raise and lower indices of tensors on an almost Hermitian manifold $(M, J,
g)$ with the metric and its inverse. The \emph{Poisson tensor} on $M$ is
$\Kf^{\alpha\beta}= J_{\gamma}{}^\beta g^{\alpha\gamma}$, with
\begin{equation}\label{Poisson_structure}
\Kf_{\alpha\beta}\Kf^{\beta\gamma}=-\delta_\alpha{}^\gamma.
\end{equation}
Viewing $\Wedge^{1,0}\otimes\Wedge^{0,1}$ as a complex vector bundle equipped
with the real structure given by swapping its factors, a Hermitian metric can,
by definition, also be seen as a real nondegenerate section $g_{a\bar b}$ of
$\Wedge^{1,0}\otimes\Wedge^{0,1}$. We denote by $g^{\bar a b}\in
\Gamma(T^{0,1}M\otimes T^{1,0}M)$ its inverse, characterised by
\[
g_{a\bar b}g^{\bar b c}=\delta_{a}{}^c\quad\text{and}\quad
g_{a\bar b}g^{\bar c a}=\delta_{\bar  b}{}^{\bar c}.
\]

Let us denote by $\nabla^g$ the Levi-Civita connection of a Hermitian metric
$g$.  Differentiating the identity $J_{\alpha}{}^{\gamma}J_{\gamma}{}^\beta
=-\delta_{\alpha}{}^\beta$ shows that
\begin{equation}\label{trivial_identity}
(\nabla_{\alpha}^gJ_{\beta}{}^\epsilon)J_{\epsilon}{}^{\gamma}
+J_{\beta}{}^\epsilon\nabla_{\alpha}^g J_{\epsilon}{}^\gamma=0.
\end{equation}
Since $\nabla_\alpha^g\Kf_{\beta\gamma}=g_{\gamma\epsilon} \nabla^g_\alpha
J_{\beta}{}^\epsilon$, it follows immediately from (\ref{trivial_identity})
that
\begin{equation}\label{(11)_part_vanishes}
\nabla_{\alpha}^g\Kf_{\beta\gamma}+
J_{\beta}{}^\epsilon J_{\gamma}{}^\zeta \nabla_\alpha^g\Kf_{\epsilon\zeta}=0.
\end{equation}
Viewing $\nabla_\alpha^g\Kf_{\beta\gamma}$ as $2$-form with values in $T^*M$,
equation (\ref{(11)_part_vanishes}) says that the part of type $(1,1)$
vanishes identically.  On the other hand, the vector bundle map
\begin{align*}
\Wedge^2T^*M\otimes T^*M&\to \Wedge^3T^*M\\
\Psi_{\alpha\beta\gamma}&\mapsto \Psi_{[\alpha\beta\gamma]}
\end{align*}
induces an isomorphism between $2$-forms with values in $T^*M$ of type $(0,2)$
and $3$-forms on $M$ of type $(2,1)+(1,2)$, i.e.~real sections of
$\Wedge^{2,1}\oplus\Wedge^{1,2}$. Since
\begin{equation}\label{ext_deriv_of_two_form}
\nabla_{[\alpha}^g\Kf_{\beta\gamma]}=\tfrac13(d\Kf)_{\alpha\beta\gamma},
\end{equation}
the identity
\begin{equation}\label{standard_identity}
2\nabla_\alpha^g\Kf_{\beta\gamma} = (d\Kf)_{\alpha\beta\gamma}
-J_{\beta}{}^\epsilon J_{\gamma}{}^\zeta (d\Kf)_{\alpha\epsilon\zeta}
-N^J_{\,\beta\gamma}{}^\epsilon\Kf_{\epsilon\alpha}
\end{equation}
shows that type $(0,2)$ component of $\nabla_\alpha^g\Kf_{\beta\gamma}$ is
identified with the $(2,1)+(1,2)$ component of $d\Kf$~\cite{Gauduchon}. The
type $(3,0)+(0,3)$ component of $d\Kf$ is determined by the Nijenhuis tensor
$N^J$, hence so is the type $(2,0)$ part of $\nabla_\alpha^g\Kf_{\beta\gamma}$
(which has type $(0,2)$ when viewed as a $2$-form with values in $TM$ using
$g$).

If $M$ has dimension $2n\geq 6$, $\nabla^g_\alpha\Kf_{\beta\gamma}$ can be
decomposed into $4$ components, which correspond to $4$ real irreducible
${\mathrm{U}}(p,q)$-submodules in $\Wedge^2\C^{n}\otimes\C^{n}$, where
${\mathrm{U}}(p,q)$ denotes the \bps/unitary group of signature $(p,q)$ with
$p+q=n$, the signature of $g_{\alpha\beta}$. If $2n=4$, then
$\nabla^g_\alpha\Kf_{\beta\gamma}$ has only two components. The different
possibilities of a subset of these invariants vanishing leads to the
Gray--Hervella classification of almost Hermitian manifolds into~16,
respectively~4, classes in dimension $2n\geq 6$, respectively $2n=4$,
see~\cite{GrayHervella}. In the following we shall be interested in the class
of almost Hermitian manifolds which in the literature (at least in the case of
metrics of definite signature) are referred to as \emph{quasi-K\"ahler} or
$(2,1)$-\emph{symplectic}, see~\cite{Gauduchon,GrayHervella}. We
extend this terminology to indefinite signature, as we have done for Hermitian
metrics in general.

\begin{defin} Suppose $(M,J,g)$ is an almost Hermitian manifold of dimension
\mbox{$2n\geq 4$}. Then $(M,J,g)$ is called a \emph{quasi-K\"ahler} or
$(2,1)$-\emph{symplectic} manifold, if
\begin{equation}\label{21symplectic}
\nabla_\alpha^g\Kf_{\beta\gamma}
+J_{\alpha}{}^\epsilon J_{\beta}{}^\zeta\nabla_{\epsilon}^g\Kf_{\zeta \gamma}=0,
\end{equation}
which is the case if and only if
\begin{equation}\label{21symplectic2}
\nabla_\alpha^gJ_{\beta}{}^\gamma
=-J_{\alpha}{}^\epsilon J_{\beta}{}^\zeta\nabla_{\epsilon}^gJ_{\zeta}{}^\gamma.
\end{equation}
\end{defin}
Since $\nabla_{\alpha}\Kf_{\beta\gamma}$, as a $2$-form with values in $T^*M$,
has no component of type $(1,1)$, (\ref{21symplectic}) means, equivalently,
that $\nabla_{\alpha}\Kf_{\beta\gamma}$ has type $(2,0)$, i.e.~has no $(0,2)$
part; equivalently $d\Kf$ has no component of type $(2,1)+(1,2)$, i.e.~it has
type $(3,0)+(0,3)$, which explains the ``$(2,1)$-symplectic'' terminology. The
class of $(2,1)$-symplectic manifolds of dimension $2n\geq 6$ contains as a
subclass the \emph{almost \bps/K\"ahler} manifolds, which are symplectic, and
the subclass of \emph{nearly K\"ahler} manifolds, i.e.~those almost Hermitian
manifolds that satisfy
$\nabla^g_\alpha\Kf_{\beta\gamma}=-\nabla^g_\beta\Kf_{\alpha\gamma}$, which is
manifestly equivalent to
$3\nabla^g_\alpha\Kf_{\beta\gamma}=(d\Kf)_{\alpha\beta\gamma}$. Since in
dimension $2n=4$ any $3$-form has type $(2,1)+(1,2)$, the condition for an
almost Hermitian manifold of dimension $4$ to be $(2,1)$-symplectic is
equivalent to the condition to be almost \bps/K\"ahler. If $J$ is integrable,
i.e.~$(M,J,g)$ a Hermitian manifold, then $(M,J,g)$ is $(2,1)$-symplectic if
and only if $d\Kf=0$, i.e.~$(M,J,g)$ is \bps/K\"ahler.

\begin{defin} Let $(M, J, g)$ be an almost Hermitian manifold of dimension
$2n\geq4$. Then a \emph{Hermitian connection} on $M$ is an affine connection
  $\nabla$ with $\nabla J=0$ and $\nabla g=0$.
\end{defin}

Such Hermitian connections exist and are uniquely determined by their torsion.
A discussion of Hermitian connections and of the freedom in prescribing their
torsion can for instance be found in \cite{Gauduchon} (see also
\cite{lich:connections}).  The following proposition shows that
$(2,1)$-symplectic manifolds can be characterised as those almost Hermitian
manifolds which admit a minimal Hermitian connection; for a proof see
\cite{Gauduchon}.

\begin{prop}\label{minimal_Hermitian_connection}  Suppose $(M, J, g)$ is
an almost Hermitian manifold of dimension $2n\geq 4$.  Then $(M, J, g)$ admits
a \textup(unique\textup) Hermitian connection whose torsion $T$ is of type
$(0,2)$ as $2$-form with values in $TM$, equivalently $T=-\frac{1}{4} N^J$, if
and only if it is $(2,1)$-symplectic.
\end{prop}

For a a $(2,1)$-symplectic manifold $(M,J,g)$ we refer to the unique Hermitian
connection $\nabla$ of Proposition \ref{minimal_Hermitian_connection} as the
\emph{canonical connection} of $(M,J,g)$. In terms of the Levi-Civita
connection $\nabla^g$ of $g$ it is given by
\begin{equation}\label{canonical_connection}
\nabla_{\alpha}X^\beta=\nabla_{\alpha}^g X^\beta
+\tfrac{1}{2}(\nabla_{\alpha}^gJ_{\gamma}{}^\beta)J_{\epsilon}{}^\gamma X^\epsilon. 
\end{equation}
For the convenience of the reader let us check that this connection has the
desired properties.  For an arbitrary almost Hermitian manifold $(M, J, g)$
the formula (\ref{canonical_connection}) is obviously a complex connection,
since
\begin{align*}
&\nabla_\alpha (J_{\gamma}{}^\beta X^\gamma)
=\tfrac{1}{2}(\nabla_\alpha^g (J_{\gamma}{}^\beta X^\gamma)
+J_{\gamma}{}^\beta\nabla_\alpha^g X^\gamma)\\
&J_{\gamma}{}^\beta\nabla_\alpha X^\gamma
=\tfrac{1}{2}(\nabla_\alpha^g (J_{\gamma}{}^\beta X^\gamma)
+J_{\gamma}{}^\beta\nabla_\alpha^g X^\gamma),
\end{align*}
which implies $(\nabla_{\alpha}J_{\beta}{}^\gamma)X^\beta=\nabla_\alpha
(J_{\gamma}{}^\beta X^\gamma)-J_{\gamma}{}^\beta\nabla_\alpha X^\gamma=0$ for
all vector fields $X^\alpha$.  Since $\nabla^g$ is a metric connection, the
connection given by \eqref{canonical_connection} is a metric connection if and
only if $(\nabla_{\alpha}^g J_{\zeta}{}^\epsilon) J_{\beta}{}^\zeta
g_{\gamma\epsilon}=(\nabla_{\alpha}^g\Kf_{\zeta\gamma})J_{\beta}{}^\zeta$ is
skew in $\gamma$ and $\beta$, which follows immediately from
\eqref{standard_identity}. Hence, on any almost Hermitian manifold formula
\eqref{canonical_connection} defines a Hermitian connection. Moreover, since
$\nabla^g$ is torsion free, the torsion $T$ of \eqref{canonical_connection}
satisfies
 \begin{equation}\label{Torsion_of_can_conn}
 T_{\alpha\beta}{}^\gamma=\tfrac{1}{2}((\nabla_{\alpha}^gJ_{\epsilon}{}^\gamma)
J_{\beta}{}^\epsilon-(\nabla_{\beta}^gJ_{\epsilon}{}^\gamma) J_{\alpha}{}^\epsilon).
\end{equation}
Recall that the Nijenhuis tensor can be expressed in terms of $\nabla^g$
(actually in terms of any torsion free connection) as
\begin{equation}
N^J_{\,\alpha\beta}{}^{\gamma}=
-(\nabla_{\alpha}^gJ_{\epsilon}{}^\gamma) J_{\beta}{}^\epsilon
+(\nabla_{\beta}^gJ_{\epsilon}{}^\gamma) J_{\alpha}{}^\epsilon-
J_{\alpha}{}^\epsilon\nabla_{\epsilon}J_{\beta}{}^\gamma
+ J_{\beta}{}^\epsilon\nabla_{\epsilon} J_{\alpha}{}^\gamma,
\end{equation}
which by (\ref{21symplectic2}) reduces in the case of a $(2,1)$-symplectic
manifold to the equation
\begin{equation}\label{Torsion_of_can_conn_2}
N^J_{\,\alpha\beta}{}^{\gamma}=
-2((\nabla_{\alpha}^gJ_{\epsilon}{}^\gamma) J_{\beta}{}^\epsilon
-(\nabla_{\beta}^gJ_{\epsilon}{}^\gamma) J_{\alpha}{}^\epsilon).
\end{equation}
Comparing (\ref{Torsion_of_can_conn}) with (\ref{Torsion_of_can_conn_2}) shows
that on a $(2,1)$-symplectic manifold the torsion $T$ of
(\ref{canonical_connection}) satisfies $T=-\frac{1}{4}N^J$ as required.  Note
that, if the Levi-Civita connection $\nabla^g$ of $(M, J, g)$ is a complex
connection, then also $\nabla^g_{\alpha}\Kf_{\beta\gamma}=0$, which by
(\ref{ext_deriv_of_two_form}) implies that $\Kf_{\alpha\beta}$ is
closed. Moreover, the identity (\ref{standard_identity}) shows that $J$ is
necessarily integrable in this case. Conversely, the same identity shows that,
if $J$ is integrable and the fundamental $2$-form closed, then the Levi-Civita
connection is a complex connection,
cf.~Corollary~\ref{torsion_free_complex_connection}. Hence, the connection in
\eqref{canonical_connection} coincides with the Levi-Civita connection of on
an almost Hermitian manifold if and only if $(M,J,g)$ is \bps/K\"ahler.

\begin{rem}
We have already observed that formula \eqref{canonical_connection} defines a
Hermitian connection on any almost Hermitian manifold, which is usually
referred to as the \emph{first canonical connection} following
\cite{lich:connections}. In the case of a $(2,1)$-symplectic manifold the first
canonical connection coincides also with \emph{the second canonical connection}
of \cite{lich:connections}, which is also called \emph{Chern connection}; see
\cite{Gauduchon}.
\end{rem}

Let $(M,J,g)$ be a $(2,1)$-symplectic manifold and denote by $R$ the curvature
of its canonical connection $\nabla$.  Since $\nabla$ is Hermitian, we have
$R\in\Omega^2(M,{\mathfrak u}(TM))$, where ${\mathfrak u}(TM)\subset
T^*M\otimes TM$ denotes the subbundle of unitary bundle endomorphisms of
$(TM,J,g)$. Setting $R_{\alpha\beta\gamma\delta}\equiv
R_{ab}{}^\epsilon{}_\delta g_{\epsilon\gamma}$, the property
$R\in\Omega^2(M,{\mathfrak u}(TM))$ of the curvature of a $(2,1)$-symplectic
manifold can be expressed as
\begin{equation}\label{21symp_Curv}
R_{\alpha\beta\gamma\delta}=R_{[\alpha\beta][\gamma\delta]}\quad \text{ and }\quad
R_{\alpha\beta\gamma[\delta}J_{\epsilon]}{}^\gamma=0.
\end{equation}
Moreover, recall that for any linear connection the Bianchi symmetry
holds. Hence, $R$ satisfies
\begin{equation}\label{Bianchi_with_torsion}
R_{[\alpha\beta}{}^\gamma{}_{\delta]}
=\nabla_{[\alpha} T_{\beta\delta]}{}^\gamma+T_{\epsilon[\alpha}{}^\gamma T_{\beta\delta]}{}^\epsilon,
\end{equation}
where $T_{\alpha\beta}{}^\gamma=-\frac{1}{4} N^J_{\alpha\beta}{}^\gamma$ is
the torsion of $\nabla$. Note that \eqref{Bianchi_with_torsion} for a minimal
connection is of course precisely equivalent to the already established
identities \eqref{bianchi1} and \eqref{bianchi2}.  Since $\nabla$ is a complex
connection, $R$ decomposes as a $2$-form with values in the complex
endomorphism of $TM$ into three components according to type as explained in
Section~\ref{almost_cproj_curvature}. In barred and unbarred indices $R$ can
therefore be encoded by the three tensors
\[
R_{ab}{}^c{}_d\quad\quad R_{a\bar b}{}^c{}_{d}\quad\quad R_{\bar a\bar b}{}^c{}_{d},
\]
or equivalently by their complex conjugates, where
$R_{ab}{}^c{}_d\equiv\Pi_a^\alpha \Pi{}_{b}^\beta\Pi{}^{c}_\gamma\Pi_d^\delta
R_{\alpha\beta}{}^{\gamma}{}_{\delta}$ and so on.  Since $\nabla$ preserves in
addition a \bps/Riemannian metric, the additional symmetry
$R_{\alpha\beta\gamma\delta}=-R_{\alpha\beta\delta\gamma}$ implies
\begin{align*}
&R_{ab\bar c d}\equiv R_{ab}{}^e{}_d g_{e\bar c}
=-R_{ab}{}^{\bar e}{}_{\bar c}g_{\bar ed}\equiv -R_{abd\bar c}
\qquad R_{\bar a\bar b\bar c d}\equiv R_{\bar a\bar b}{}^e{}_d g_{e\bar c}
=-R_{\bar a\bar b}{}^{\bar e}{}_{\bar c}g_{\bar ed}\equiv -R_{\bar a\bar bd\bar c}\\
&R_{a\bar b\bar c d}\equiv R_{a\bar b}{}^e{}_d g_{e\bar c}
=-R_{a\bar b}{}^{\bar e}{}_{\bar c}g_{\bar ed}\equiv -R_{a\bar bd\bar c}
\qquad R_{\bar ab\bar c d}\equiv R_{\bar ab}{}^e{}_d g_{e\bar c}
=-R_{\bar ab}{}^{\bar e}{}_{\bar c}g_{\bar ed}\equiv -R_{\bar abd\bar c}.
\end{align*}
Note that the first two identities (which are conjugates of each other) show
that for the canonical connection of a $(2,1)$-symplectic manifold (in
contrast to a general minimal complex connection) the curvature components
$R_{ab}{}^c{}_d$ and $R_{ab}{}^{\bar c}{}_{\bar d}=\overline{R_{\bar a\bar
    b}{}^c{}_d}$ are not independent of each other, since they are related by
$g$. Hence, the curvature $R$ of the canonical connection of a
$(2,1)$-symplectic manifold can be encoded by the two tensors
\[
R_{ab\bar c d}=R_{[ab] \bar c d}\quad\text{ and }\quad R_{a\bar b\bar cd}
\]
(or their complex conjugates).  By~\eqref{Bianchi_with_torsion},
\eqref{21symp_Curv} and the fact that the torsion of $\nabla$ has type
$(0,2)$ one deduces straightforwardly that the curvature and the torsion of a
$(2,1)$-symplectic manifold satisfy the symmetries
\begin{align}
&R_{ab\bar c d}=-\nabla_{\bar c} T_{abd}\qquad
R_{ab\bar c d}+ R_{bd\bar c a}+ R_{da\bar c b}=0\label{21Curv1a}\\
&\nabla_{[a} T_{bc]d}=0\label{21Curv1b}\\
&R_{a\bar b\bar cd}-R_{d\bar b\bar c a}
=- T_{ad}{}^{\bar e}T_{\bar e\bar b\bar c}\qquad
R_{a\bar b\bar cd}-R_{a\bar c\bar b d}
=- T_{\bar b\bar c}{}^{e}T_{e ad}\label{21Curv2}\\
&R_{a\bar b\bar c d}-R_{\bar c\bar d a\bar b}
=T_{\bar a\bar b\bar c} T_{da}{}^{\bar e} 
+T_{fda}T_{\bar c\bar b}{}^f,\label{21Curv3}
\end{align}
where $T_{abc}=T_{ab}{}^{\bar d}g_{c\bar d}$ and $T_{\bar a\bar b\bar c}
=T_{\bar a\bar b}{}^{d}g_{d\bar c}$ (cf.~also \eqref{bianchi1} and
\eqref{bianchi2}).  Now let us consider the Ricci tensor $\Ric$ of the
canonical connection of a $(2,1)$-symplectic manifold. By definition we have
\[
\Ric_{ab}=R_{ca}{}^c{}_b\qquad
\Ric_{\bar a\bar b}=R_{\bar c\bar a}{}^{\bar c}{}_{\bar b}\qquad
\Ric_{a\bar b}=R_{\bar ca}{}^{\bar c}{}_{\bar b}\qquad
\Ric_{\bar ab}=R_{c\bar a}{}^{c}{}_{b}.
\]
{From} the identities \eqref{21Curv1a} we we conclude that 
\begin{equation}\label{21Curv_Ricci1}
\begin{aligned}
\Ric_{ab}&=-\nabla_{\bar c} T^{\bar c}{}_{ab}&&&&&
\Ric_{[ab]}&=\tfrac{1}{2} \nabla_{\bar c} T_{ab}{}^{\bar c}\\
\Ric_{\bar a\bar b}&=-\nabla_{c} T^{c}{}_{\bar a\bar b} &&&&&
\Ric_{[\bar a\bar b]}&=\tfrac{1}{2} \nabla_{c} T_{\bar a\bar b}{}^{c}.
\end{aligned}\end{equation}
Moreover, taking a Ricci type contraction in \eqref{21Curv3} shows immediately
that the $J$-invariant part of the Ricci tensor of the
canonical connection of a $(2,1)$-symplectic manifold is symmetric:
\begin{equation}\label{21Curv_Ricci2}
\Ric_{a\bar b}=\Ric_{\bar b a}.
\end{equation}
The canonical connection of a $(2,1)$-symplectic manifold is special in the
sense of Section \ref{curvature_operators_on_densities}, since it preserves
the volume form of $g$; hence \eqref{21Curv_Ricci1} and \eqref{21Curv_Ricci2}
confirm in particular what we deduced there for the Ricci curvature of special
connections.

We already observed that $(2,1)$-symplectic is equivalent to \bps/K\"ahler when
$J$ is integrable. Hence, in this case, the canonical connection simply
coincides with the Levi-Civita connection. Suppose $(M,J,g)$ is now a
\bps/K\"ahler manifold.  Then the identities \eqref{21Curv1a}--\eqref{21Curv3}
imply that $R$ is determined by any of the following tensors
\[
R_{a\bar{b}c\bar{d}}\qquad R_{\bar{a}bc\bar{d}}\qquad
R_{a\bar{b}\bar{c}d}\qquad R_{\bar{a}b\bar{c}d}
\]
which are now subject to the following symmetries
\begin{equation}\label{Kurvature}
R_{a\bar{b}c\bar{d}}=-R_{\bar{b}ac\bar{d}} \qquad
R_{a\bar{b}c\bar{d}}=-R_{a\bar{b}\bar{d}c} \qquad
R_{a\bar{b}c\bar{d}}=R_{c\bar{b}a\bar{d}}  \qquad
R_{a\bar{b}c\bar{d}}=R_{a\bar{d}c\bar{b}}
\end{equation}
as well as
\[
\overline{R}_{\bar{a}b\overline{c}d}\equiv \overline{R_{a\bar{b}c\bar{d}}}
=R_{b\bar{a}d\bar{c}}.
\]
\begin{rem} Let us remark that for the curvature $R$ of a \bps/K\"ahler
manifold, we have $R_{[\alpha\beta}{}^\gamma{}_{\delta]}=0$ which, together
with the symmetries \eqref{21symp_Curv}, implies
\begin{equation}\label{KaehlerCurv}
R_{\alpha\beta\gamma\delta}=R_{\gamma\delta\alpha\beta}\qquad\text{and}\qquad
J_\alpha{}^\epsilon J_{\beta}{}^\zeta R_{\epsilon\zeta}{}^\gamma{}_\delta=
R_{\alpha\beta}{}^\gamma{}_\delta.
\end{equation}
The symmetries of \eqref{21symp_Curv} and \eqref{KaehlerCurv} are precisely
the ones in \eqref{Kurvature} expressed in real indices. Note also that
\eqref{KaehlerCurv} shows immediately that the Ricci tensor
$\Ric_{\alpha\beta}=R_{\epsilon\alpha}{}^\epsilon{}_\beta$ 
is symmetric and $J$-invariant, which is consistent with \eqref{21Curv_Ricci2}.
\end{rem}
Moreover, note that it is immediate that the rank of the bundle of
\bps/K\"ahler  curvatures is $(n(n+1)/2)^2$ and that this bundle further
decomposes under ${\mathrm{U}}(n)$ as
\begin{equation*}
S^2\Wedge^{1,0}\otimes S^2\Wedge^{0,1}=
(S^2\Wedge^{1,0}\otimes_\circ S^2\Wedge^{0,1})\;\oplus\;
(\Wedge^{1,0}\otimes_\circ\Wedge^{0,1})\;\oplus\;\R,
\end{equation*}
where the subscript $\circ$ means trace-free part and $\R$ stands for
the trivial bundle. Under this decomposition, the \bps/K\"ahler curvature
splits as
\begin{equation}\label{BochnerDecomp}
R_{a\bar{b}c\bar{d}}=\Bt_{a\bar{b}c\bar{d}}
-2(\Xi_{a\bar{b}}g_{c\bar{d}}+\Xi_{c\bar{d}}g_{a\bar{b}}
+\Xi_{a\bar{d}}g_{c\bar{b}}+\Xi_{c\bar{b}}g_{a\bar{d}})
-2\Lambda(g_{a\bar{b}}g_{c\bar{d}}+g_{a\bar{d}}g_{c\bar{b}}),
\end{equation}
where
\begin{equation*}
\Bt_{a\bar{b}c\bar{d}}=\Bt_{c\bar{b}a\bar{d}}=\Bt_{a\bar{d}c\bar{b}}\qquad
g^{\bar{b}c}\Bt_{a\bar{b}c\bar{d}}=0\qquad g^{\bar{b}a}\Xi_{a\bar{b}}=0.
\end{equation*}
This is a K\"ahler analogue of the usual decomposition of Riemannian curvature
into the conformal Weyl tensor, the trace-free Ricci tensor, and the scalar
curvature.  The tensor $\Bt_{a\bar{b}c\bar{d}}$ is called the \emph{Bochner
  curvature} (or \emph{tensor}) and is the orthogonal projection of the
conformal Weyl curvature onto the intersection of the space of K\"ahler
curvatures with the space of conformal Weyl tensors~\cite{ACG}.  The analogue
of constant curvature in \bps/K\"ahler geometry is to insist that
$R_{a\bar{b}\bar{d}c}=\Lambda(g_{a\bar{b}}g_{c\bar{d}}
+g_{c\bar{b}}g_{a\bar{d}})$, where the (a priori) smooth function $\Lambda$ is
constant by the Bianchi identity.  This is called \emph{constant holomorphic
  sectional curvature} and (for $\Lambda>0$) locally characterises
$\CP^n$ and its Fubini--Study metric as in Section~\ref{CPn_curvature}
(where the normalisation is such that $\Lambda=1$).

\subsection{Other curvature decompositions}
It will be useful, both in this article and elsewhere, to decompose the
\bps/K\"ahler curvature tensor from various different viewpoints, some of which
ignore the complex structure. Without a complex structure, barred and unbarred
indices are unavailable so firstly we should rewrite the irreducible
decomposition (\ref{BochnerDecomp}) using only real indices. We recall that
\begin{equation}\label{KK}
R_{\alpha\beta\gamma\delta}=R_{[\alpha\beta][\gamma\delta]}\qquad
R_{[\alpha\beta\gamma]\delta}=0\qquad
R_{\alpha\beta\gamma[\delta}J_{\epsilon]}{}^\gamma=0
\end{equation}
and the real version of (\ref{BochnerDecomp}) will apply to any tensor
satisfying these identities. Recalling that
$\Omega_{\alpha\beta}=J_{\alpha\beta}=J_\alpha{}^\gamma g_{\beta\gamma}$, we 
obtain
\begin{equation}\label{Bochner_in_real_money}
\begin{split}
R_{\alpha\beta\gamma\delta}
&=\Bt_{\alpha\beta\gamma\delta}\\
&\qquad+g_{\alpha\gamma}\Xi_{\beta\delta}-g_{\beta\gamma}\Xi_{\alpha\delta}
-g_{\alpha\delta}\Xi_{\beta\gamma}+g_{\beta\delta}\Xi_{\alpha\gamma}\\
&\qquad+\Omega_{\alpha\gamma}\Sigma_{\beta\delta}
-\Omega_{\beta\gamma}\Sigma_{\alpha\delta}
-\Omega_{\alpha\delta}\Sigma_{\beta\gamma}
+\Omega_{\beta\delta}\Sigma_{\alpha\gamma}
+2\Omega_{\alpha\beta}\Sigma_{\gamma\delta}
+2\Omega_{\gamma\delta}\Sigma_{\alpha\beta}\\
&\qquad
+\Lambda(g_{\alpha\gamma}g_{\beta\delta}-g_{\beta\gamma}g_{\alpha\delta}
+\Omega_{\alpha\gamma}\Omega_{\beta\delta}
-\Omega_{\beta\gamma}\Omega_{\alpha\delta}
+2\Omega_{\alpha\beta}\Omega_{\gamma\delta}),
\end{split}\end{equation}
where
\begin{itemize}
\item $\Bt_{\alpha\beta\gamma\delta}$ is totally trace-free with respect to
$g^{\alpha\beta}$ and $\Omega^{\alpha\beta}$ 
\item $\Sigma_{\alpha\beta}\equiv J_\alpha{}^\gamma\Xi_{\beta\gamma}$
whilst $\Xi_{\alpha\beta}$ is
symmetric, trace-free, and of type $(1,1)$:
$$\Xi_{\alpha\beta}=\Xi_{(\alpha\beta)}\qquad\Xi_\alpha{}^\alpha=0\qquad 
\Sigma_{\alpha\beta}=\Sigma_{[\alpha\beta]}.$$
\end{itemize}
A simple way to see this is to check that all parts of this decomposition 
satisfy (\ref{KK}) as they should and then apply 
$\Pi_a^\alpha\overline{\Pi}{}_{\bar{b}}^\beta
\Pi_c^\gamma\overline{\Pi}{}_{\bar{d}}^\delta$, using the various identities
from Section~\ref{barredandunbarred} including~(\ref{usefulidentities}), to 
recover (\ref{BochnerDecomp}).
One can also read off from (\ref{Bochner_in_real_money}) the corresponding
decomposition of the Ricci tensor in \bps/K\"ahler geometry. Specifically,
\[
\Ric_{\beta\delta}=2(n+2)\Xi_{\beta\delta}+2(n+1)\Lambda g_{\beta\delta}\qquad
\Scal=4n(n+1)\Lambda
\]
and, conversely,
\[
\Lambda=\tfrac1{4n(n+1)}\Scal\qquad
\Xi_{\alpha\beta}=\tfrac1{2(n+2)}\left(\Ric_{\alpha\beta}-\tfrac1{2n}\Scal\,
g_{\alpha\beta}\right).
\]

Other natural realms in which one may view \bps/K\"ahler geometry are
\begin{itemize}
\item projective
\item conformal
\item c-projective
\item symplectic
\end{itemize}
and in each case decompose the curvature accordingly. The projective Weyl
curvature tensor \cite{Eprojectivenotes} on a Riemannian manifold of dimension 
$m$ is given by
$$R_{\alpha\beta\gamma\delta}
-\tfrac1{m-1}g_{\alpha\gamma}\Ric_{\beta\delta}
+\tfrac1{m-1}g_{\beta\gamma}\Ric_{\alpha\delta}.$$
If this vanishes, then, in conjunction with the interchange symmetry
$R_{\alpha\beta\gamma\delta}=R_{\gamma\delta\alpha\beta}$, we deduce that
$R_{\alpha\beta\gamma\delta}
=\lambda(g_{\alpha\gamma}g_{\beta\delta}-g_{\beta\gamma}g_{\alpha\delta})$
where, if $m\geq 3$, the (a priori) smooth function $\lambda$ is constant by
the Bianchi identity. This is Beltrami's Theorem that the only projectively
flat \bps/Riemannian geometries are constant curvature (when $m=2$ one instead
uses that the projective Cotton--York tensor vanishes). In any case, comparison
with (\ref{Bochner_in_real_money}) shows that for $n\geq 2$ the only
projectively flat \bps/K\"ahler manifolds are flat. The conformal Weyl
curvature is given by
$$R_{\alpha\beta\gamma\delta}
-g_{\alpha\gamma}Q_{\beta\delta}+g_{\beta\gamma}Q_{\alpha\delta}
-g_{\beta\delta}Q_{\alpha\gamma}+g_{\alpha\delta}Q_{\beta\gamma},$$
where $Q_{\alpha\beta}$ is the Riemannian Schouten tensor
$$Q_{\alpha\beta}=
\frac1{m-2}\left(\Ric_{\alpha\beta}-\frac1{2(m-1)}\Scal\,g_{\alpha\beta}\right).$$
Thus, if the conformal Weyl curvature vanishes on a \bps/K\"ahler manifold,
then
\[
2R_\alpha{}^\delta{}_{\gamma[\delta}J_{\epsilon]}{}^\gamma=
J_{\epsilon\alpha}Q_\beta{}^\beta+2(n-2)J_\epsilon{}^\gamma Q_{\alpha\gamma}.
\]
{From} (\ref{KK}), we see that for $n\geq 3$ the only conformally flat
\bps/K\"ahler manifolds are flat. For $n=2$ it follows only that the geometry
is scalar flat and, in fact, Tanno~\cite{tanno72} showed that $4$-dimensional
conformally flat K\"ahler manifolds are locally of the form
$\CP^1\times\Sigma$ where $\CP^1$ has the Fubini--Study
metric up to constant scale and the complex surface $\Sigma$ has a constant
negative scalar curvature of equal magnitude but opposite sign to that
on~$\CP^1$.

{From} the c-projective viewpoint, if we compare the decomposition
(\ref{Bochner_in_real_money}) with (\ref{partial}), then we conclude, firstly
that $W_{\alpha\beta}{}^\gamma{}_\delta=H_{\alpha\beta}{}^\gamma{}_\delta$ (see
the proof of Proposition~\ref{kaehlerharmcurv} for a barred/unbarred index
proof of this), and then that
\begin{equation}\label{H_for_Kaehler}
\begin{split}H_{\alpha\beta\gamma\delta}
&=\Bt_{\alpha\beta\gamma\delta}
-g_{\alpha\delta}\Xi_{\beta\gamma}+g_{\beta\delta}\Xi_{\alpha\gamma}
+\tfrac1{n+1}(g_{\beta\gamma}\Xi_{\alpha\delta}-g_{\alpha\gamma}\Xi_{\beta\delta})\\
&+2\Omega_{\alpha\beta}\Sigma_{\gamma\delta}
-\Omega_{\alpha\delta}\Sigma_{\beta\gamma}+\Omega_{\beta\delta}\Sigma_{\alpha\gamma}
-\tfrac1{n+1}
(2\Omega_{\gamma\delta}\Sigma_{\alpha\beta}
-\Omega_{\beta\gamma}\Sigma_{\alpha\delta}
+\Omega_{\alpha\gamma}\Sigma_{\beta\delta}).
\end{split}\end{equation}
Notice, in particular, that 
\begin{equation}\label{trace_of_H}
H_{\alpha\beta\gamma}{}^\beta =2\tfrac{n(n+1)}{n+1}\Xi_{\alpha\gamma}
\end{equation}
from which we can deduce the following c-projective counterpart to 
Beltrami's Theorem.

\begin{thm}\label{counterpart_to_Beltrami}
Suppose a \bps/K\"ahler metric is c-projectively flat. Then it has 
constant holomorphic sectional curvature.
\end{thm}
\begin{proof}
To be c-projectively flat, the harmonic curvature tensor 
$H_{\alpha\beta}{}^\gamma{}_\delta$ must vanish. Then from (\ref{trace_of_H}) 
we find that $\Xi_{\alpha\beta}=0$ and from (\ref{H_for_Kaehler}) that 
also $U_{\alpha\beta\gamma\delta}=0$. According to 
(\ref{Bochner_in_real_money}) we find that $R_{\alpha\beta\gamma\delta}$ is of 
the required form. 
\end{proof}

Finally, we may view \bps/K\"ahler geometry from the purely symplectic
viewpoint as follows. For any torsion-free connection
preserving~$\Omega_{\alpha\beta}$, the tensor
$R_{\alpha\beta}{}^\epsilon{}_{\delta}\Omega_{\epsilon\gamma}$ is symmetric in
$\gamma\delta$ and may be decomposed into irreducible pieces under
${\mathrm{Sp}}(2n,\R)$:
\begin{equation}\label{symplectic_curvature}
R_{\alpha\beta}{}^\epsilon{}_{\delta}\Omega_{\epsilon\gamma}=
V_{\alpha\beta\gamma\delta}+\Omega_{\alpha\gamma}\Phi_{\beta\delta}
-\Omega_{\beta\gamma}\Phi_{\alpha\delta}
+\Omega_{\alpha\delta}\Phi_{\beta\gamma}
-\Omega_{\beta\delta}\Phi_{\alpha\gamma}
+2\Omega_{\alpha\beta}\Phi_{\gamma\delta},\end{equation}
where 
$$V_{\alpha\beta\gamma\delta}=V_{[\alpha\beta](\gamma\delta)}\qquad
V_{[\alpha\beta\gamma]\delta}=0\qquad
\Omega^{\alpha\beta}V_{\alpha\beta\gamma\delta}=0\qquad
\Phi_{\alpha\beta}=\Phi_{(\alpha\beta)}.$$
\begin{prop} On a \bps/K\"ahler manifold, if the tensor
$V_{\alpha\beta\gamma\delta}$ vanishes, then the metric has constant
holomorphic sectional curvature.
\end{prop}
\begin{proof} {From} (\ref{symplectic_curvature}), we find that 
$$\Omega^{\alpha\beta}
R_{\alpha\beta}{}^\epsilon{}_\delta\Omega_{\epsilon\gamma}
=\Omega^{\alpha\beta}\left[
\Omega_{\alpha\gamma}\Phi_{\beta\delta}
-\Omega_{\beta\gamma}\Phi_{\alpha\delta}
+\Omega_{\alpha\delta}\Phi_{\beta\gamma}
-\Omega_{\beta\delta}\Phi_{\alpha\gamma}
+2\Omega_{\alpha\beta}\Phi_{\gamma\delta}\right]
=4(n+1)\Phi_{\gamma\delta}$$
whereas computing according to~(\ref{Bochner_in_real_money}) leads to
$\Omega^{\alpha\beta}
R_{\alpha\beta}{}^\epsilon{}_\delta\Omega_{\epsilon\gamma}
=4\Xi_{\gamma\delta}$. We conclude that 
$\Xi_{\alpha\beta}=(n+1)\Phi_{\alpha\beta}$ at which point we may compare 
(\ref{symplectic_curvature}) with (\ref{Bochner_in_real_money}) when 
$V_{\alpha\beta\gamma\delta}=0$ to conclude 
that $\Bt_{\alpha\beta\gamma\delta}=0$ and $\Xi_{\alpha\beta}=0$, as required.
\end{proof}

\subsection{Metrisability of almost c-projective manifolds}
\label{metrisability}

Suppose $(M,J,[\nabla])$ is an almost c-projective manifold. It is natural to
ask whether $[\nabla]$ contains the canonical connection of a
$(2,1)$-symplectic metric on $(M,J)$.

\begin{defin} On an almost c-projective manifold $(M,J,[\nabla])$
a $(2,1)$-symplectic metric $g\in \Gamma(S^2T^*M)$ on $(M,J)$ is
\emph{compatible} with the c-projective class $[\nabla]$ if and only if its
canonical connection is contained in $[\nabla]$. The almost c-projective
structure on $M$ is said to be \emph{metrisable} or \emph{$(2,1)$-symplectic}
or \emph{quasi-K\"ahler} (or \emph{K\"ahler} or \emph{pseudo-K\"ahler} when
$J$ is integrable) if it admits a compatible $(2,1)$-symplectic metric $g$
(respectively a K\"ahler or pseudo-K\"ahler metric $g$, if $J$ is integrable).
\end{defin}
The volume form $\vol(g)$ of $g$ is a positive section of $\Wedge^{2n}T^*M$,
which we view as a c-projective density of weight $(-(n+1),-(n+1))$ under
the identification of oriented real line bundles
$\Wedge^{2n}T^*M=\cE_\R(-(n+1),-(n+1))$ determined by 
\[
\epst_{ab\cdots c}\, \bar\epst_{\bar d\bar e\cdots \bar f}\,
\in\Gamma(\Wedge^{2n}T^*(n+1,n+1)),
\]
where $\epst_{ab\cdots c}\in\Gamma(\Wedge^{n,0}(n+1,0))$ is the tautological
form from Section~\ref{almost_c-projective}. We now write
$\vol(g)=\scale_g^{\,-(n+1)}$ uniquely to determine a positive section
$\scale_g$ of $\cE_\R(1,1)$.  The canonical connection $\nabla$ of $g$ is a
special connection in the c-projective class, and for all $\ell\in\Z$,
$\scale_g^{\,\ell}= \vol(g)^{-\ell/(n+1)}\in\Gamma(\cE_\R(\ell,\ell))$ is a
$\nabla$-parallel trivialisation of $\cE_\R(\ell,\ell)$.

In the integrable case, the metrisability of a c-projective structure gives
easily the following constraints on the harmonic curvature.

\begin{prop}\label{kaehlerharmcurv}
Let $(M,J,[\nabla])$ be a c-projective manifold of dimension $2n\geq 4$. If
$[\nabla]$ is induced by the Levi-Civita of a \bps/K\"ahler metric on
$(M,J)$, then the harmonic curvature only consists of the $(1,1)$-part
$$W_{a\bar b}{}^c{}_d=H_{a\bar b}{}^c{}_d$$
of the \textup(c-projective\textup) Weyl curvature.
\end{prop}
\begin{proof}
Suppose first that $2n\geq 6$. Then we have to show that $W_{ab}{}^c{}_d$
vanishes. Recall that, by construction, $W_{ab}{}^c{}_d$ is the
connection-independent part of the $(2,0)$-component of the curvature of any
connection in the c-projective class.  Hence, if $[\nabla]$ is induced from
the Levi-Civita connection of a \bps/K\"ahler metric on $(M,J)$, then
$W_{ab}{}^c{}_d$ vanishes identically, since the curvature of a \bps/K\"ahler
metric is $J$-invariant. If $2n=4$, then $W_{ab}{}^c{}_d$ is always
identically zero and the $(2,0)$-part $C_{abc}$ of the Cotton--York tensor is
independent of the choice of connection in the c-projective class.  Since the
Ricci tensor $\Ric_{\alpha\beta}$ of a \bps/K\"ahler metric $g$ is
$J$-invariant~\eqref{KaehlerCurv}, we have $\Rho_{a\bar
  b}=\frac{1}{n+1}\Ric_{a\bar b}$ and $\Ric_{ab}=\Rho_{ab}=0$. Hence, if
$2n=4$ and the c-projective structure is metrisable, then $C_{abc}=\nabla_a
\Rho_{bc}-\nabla_{b}\Rho_{ac}$ vanishes identically, which proves the claim.
\end{proof}

We now link compatible metrics to solutions of the first BGG operator
associated to a real analogue $\cV$ of the standard complex tractor bundle
$\T$. Any almost c-projective manifold $(M,J,[\nabla])$ admits a complex
vector bundle
\[
\cV_\C=\T\otimes\overline\T.
\]
Although the construction of $\T$ and $\overline\T$ requires the existence and
choice of an $(n+1)^{\mathrm{st}}$ root $\cE(1,0)$ of $\Wedge^nT^{1,0}M$, the
vector bundle $\T\otimes\overline\T$ is defined independently of such a
choice.  Moreover, swapping the two factors defines a real structure on
$\T\otimes\overline{\T}$ and hence $\cV_\C$ is the complexification of a real
vector bundle $\cV$ over $M$ corresponding to that real structure. The
filtration (\ref{filtrationT}) of $\T$ induces filtrations on $\cV$ and
$\cV_\C$ given by
\begin{equation*}
\cV_\C=\cV_\C^{-1}\supset\cV_\C^{0}\supset\cV_\C^1 \,,
\end{equation*}
where
\begin{align*}
\cV_\C^{-1}/\cV_\C^{0}&\cong
T^{0,1}M\otimes T^{1,0}M(-1,-1)\\
\cV_\C^{0}/\cV_\C^1&\cong
(T^{1,0}M\oplus T^{0,1}M)(-1,-1)\\
\cV_\C^1&\cong\cE(-1,-1).
\end{align*}
For any choice of connection $\nabla\in[\nabla]$ we can therefore identify an
element of $\cV_\C$ with a quadruple
\begin{equation*}
\begin{pmatrix} \ms^{\bar b c} \\
X^b\enskip|\enskip Y^{\bar{b}}\\ \rho \end{pmatrix},
\text{ where } 
\left\{\begin{array}{l}
\ms^{\bar b c}\in T^{0,1}M\otimes T^{1,0}M(-1,-1),\\
X^b\in T^{1,0}M(-1,-1),\quad Y^{\bar{b}}\in T^{0,1}M(-1,-1),\\ 
\rho\in \cE(-1,-1),
\end{array}\right.
\end{equation*}
and elements of $\cV$ can be identified with the real elements of $\cV_\C$:
\begin{equation}\label{realsections}
\overline{\ms^{\bar c b}}=\ms^{\bar b c},\qquad
\overline {X^b}=Y^{\bar b}\quad \text{ and }\quad \bar\rho=\rho.
\end{equation}
The formulae (\ref{tractorconna}) and (\ref{tractorconnbara}) for the
tractor connection on $\T$ immediately imply that the tractor
connection on $\cV_\C=\T\otimes\overline\T$ is
given by
\begin{align}\label{Metritractor1}
\nabla^{\cV_\C}_a\begin{pmatrix} \ms^{\bar b c}\\ X^b\enskip|\enskip Y^{\bar{b}}\\
\rho \end{pmatrix}
&=\begin{pmatrix} \nabla_a\ms^{\bar b c}+\delta_{a}{}^cY^{\bar b}\\ 
\nabla_aX^b+\rho\delta_{a}{}^b-\Rho_{a\bar c}\ms^{\bar c b}\enskip|\enskip
\nabla_aY^{\bar{b}}-\Rho_{ac}\ms^{\bar b c}\\ 
\nabla_a\rho-\Rho_{a\bar b}Y^{\bar b}-\Rho_{ab}X^{b}
\end{pmatrix}\\
\label{Metritractor2}
\nabla^{\cV_\C}_{\bar a}\begin{pmatrix} \ms^{\bar b c}\\
X^b\enskip|\enskip Y^{\bar{b}}\\ \rho \end{pmatrix}
&=\begin{pmatrix} 
\nabla_{\bar a}\ms^{\bar b c}+\delta_{\bar a}{}^{\bar b}X^{c}\\ 
\nabla_{\bar a}X^b-\Rho_{\bar a\bar c}\ms^{\bar c b}\enskip|\enskip
\nabla_{\bar a}Y^{\bar{b}}+\rho\delta_{\bar a}{}^{\bar b}
-\Rho_{\bar ac}\ms^{\bar b c}\\
\nabla_{\bar a}\rho-\Rho_{\bar a\bar b}Y^{\bar b}-\Rho_{\bar ab}X^{b}
\end{pmatrix}.
\end{align}
Note that the real structure on $\cV_\C$ is parallel for this
connection and that, consequently, the tractor connection on $\cV$
is the restriction of (\ref{Metritractor1}) and (\ref{Metritractor2}) to real
sections~(\ref{realsections}).

Now consider, for a section $\ms^{\bar b c}$ of $T^{0,1}M\otimes
T^{1,0}M(-1,-1)$, the system of equations
\begin{equation}\label{metriequ1c}
\nabla_a\ms^{\bar b c}+\delta_{a}{}^cY^{\bar b}=0,\qquad
\nabla_{\bar{a}}\ms^{\bar b c}+\delta_{\bar{a}}{}^{\bar{b}}X^c=0
\end{equation} 
for some sections $X^c$ of $T^{1,0}M(-1,-1)$ and $Y^{\bar b}$ of
$T^{0,1}M(-1,-1)$. It follows immediately from the invariance of
(\ref{StandardBGG}) that the system (\ref{metriequ1c}) is c-projectively
invariant. In fact, if $\ms^{\bar b c}\in\Gamma(T^{0,1}M\otimes
T^{1,0}M(-1,-1))$ satisfies (\ref{metriequ1c}) for some connection
$\nabla\in[\nabla]$, for some $X^c\in\Gamma(T^{1,0}M(-1,-1))$, and for some 
$Y^{\bar{b}}\in\Gamma(T^{1,0}M(-1,-1))$, then $\ms^{\bar b c}$ satisfies
(\ref{metriequ1c}) for $\hat\nabla\in[\nabla]$ with
\begin{equation}\label{metrichange}
\hat{X}^c=X^c-\Upsilon_{\bar b}\ms^{\bar b c}\quad\text{and}\quad
\hat{Y}^{\bar b}=Y^{\bar b}-\Upsilon_{c}\ms^{\bar b c}.
\end{equation}
Moreover, if (\ref{metriequ1c}) is satisfied, one must have
$Y^{\bar{b}}=-\frac{1}{n}\nabla_{a}\ms^{\bar b a}$ and
$X^c=-\frac{1}{n}\nabla_{\bar{a}}\ms^{\bar a c}$.  If $\ms^{\bar b c}$ is a
real section, then the first equation in (\ref{metriequ1c}) is satisfied if and
only if the second equation of (\ref{metriequ1c}) holds, in which case
$\overline{X^b}=Y^{\bar b}$. We can reformulate these observations as follows.
There is an invariant differential operator
\begin{equation}\label{firstBGG_metrisability}
D^\cV_\C\colon T^{0,1}M\otimes T^{1,0}M(-1,-1)\to
\begin{matrix}
(\Wedge^{1,0}\otimes T^{0,1}M\otimes T^{1,0}M(-1,-1))_\circ\\
\oplus\\
(\Wedge^{0,1}\otimes T^{0,1}M\otimes T^{1,0}M(-1,-1))_\circ
\end{matrix}
\end{equation}
given by $\ms^{\bar b c}\mapsto (\nabla_a\ms^{\bar bc}
-\frac1n\delta_{a}{}^c\nabla_d\ms^{\bar bd}, \nabla_{\bar{a}}\ms^{\bar b c}
-\frac{1}{n}\delta_{\bar a}{}^{\bar b} \nabla_{\bar{d}}\ms^{\bar d c})$.
Restricting $D^{\cV_\C}$ to real sections $\ms^{\bar bc}=\overline{\ms^{\bar cb}}$
gives an invariant differential operator $D^\cV$. It is the first operator in
the BGG sequence corresponding to the tractor bundle $\cV$ and $D^{\cV_\C}$ is
its complexification.

\begin{prop}\label{compKaehler} Let $(M,J,[\nabla])$ be an almost c-projective
manifold of dimension $2n\geq 4$. Then, when $n$ is even, the map sending a
Hermitian metric $g_{b\bar c}$ to the real section $\ms^{\bar a b}=g^{\bar ab}
\scale_g^{\,-1}$ of $T^{0,1}M\otimes T^{1,0}M(-1,-1)$ restricts to a bijection
between compatible $(2,1)$-symplectic Hermitian metrics on $(M,J,[\nabla])$
and nondegenerate sections in the kernel of $D^\cV$. The inverse map sends
$\ms^{\bar a b}$ to the Hermitian metric $g_{b\bar c}$ with $g^{\bar a
  b}=(\det\ms)\ms^{\bar a b}$, where
\begin{equation}\label{deteta}
\det\ms:= \tfrac{1}{n!}\,\bar\epst_{\bar a\bar c\cdots \bar e}\,\epst_{bd\cdots f}\,
\ms^{\bar a b}\ms^{\bar c d}\cdots \ms^{\bar e f}\in \Gamma(\cE_\R(1,1))
\end{equation}
and $\epst_{ab\cdots c}$ denotes the tautological section of 
$\Wedge^{n,0}(n+1,0)$. When $n$ is odd, the mapping 
$\ms^{\bar a b}\mapsto g^{\bar a b}:=(\det\ms)\ms^{\bar a b}$ is $2$--$1$ and, 
conversely, the mapping 
$g^{\bar a b}\mapsto \ms^{\bar a b}:=\scale_g^{\,-1}g^{\bar ab}$ 
picks a preferred sign for $\ms^{\bar a b}$ but, otherwise, the same 
conclusions hold.
\end{prop}
\begin{proof} Assume first that $g_{b\bar c}$ is a compatible $(2,1)$-symplectic
Hermitian metric, i.e.~its canonical connection $\nabla$ is contained in
$[\nabla]$. Then $\ms^{\bar a b}= g^{\bar a b}\scale_g^{\,-1}$ is a real
section of $T^{0,1}M\otimes T^{1,0}M(-1,-1)$, which
satisfies~\eqref{metriequ1c} for $\nabla$ with $X^c=0$ and
$Y^{\bar{c}}=\overline{X^c}=0$.  Hence, $\ms^{\bar b c}$ is in the kernel of
$D^\cV$, and $\det\ms=\scale_g^{\,n+1}\scale_g^{\,-n}=\scale_g$.

Conversely, suppose that $\ms^{\bar b c}\in\Gamma(T^{0,1}M\otimes
T^{1,0}M(-1,-1))$ is a real nondegenerate section satisfying~\eqref{metriequ1c}
for some connection $\nabla\in[\nabla]$ with $X^b\in\Gamma(T^{1,0}M(-1,-1))$
and $Y^{\bar{b}}=X^{\bar b}\in \Gamma(T^{0,1}M(-1,-1))$. Since $\ms^{\bar a
  b}$ is nondegenerate, there is a unique $1$-form $\Upsilon_b$ such that
$\ms^{\bar a b}\Upsilon_b= X^{\bar a}$.  Let us denote by
$\hat\nabla\in[\nabla]$ the connection obtained by c-projectively changing
$\nabla$ via $\Upsilon_b$. Then we deduce form (\ref{metrichange}) that
$\hat\nabla_a\ms^{\bar b c}=\hat\nabla_{\bar{a}}\ms^{\bar b c}=0$. Since
$\epst_{ab\cdots c}$ is parallel for any connection in the c-projective class,
$\det\ms$ is parallel for $\hat\nabla$. Hence, $g^{\bar b c} =\ms^{\bar b
  c}\det\ms$ is a real nondegenerate section of $T^{0,1}M\otimes T^{1,0}M$
that is parallel for $\hat\nabla$, i.e.~its inverse $g_{b\bar{c}}$ is a
$(2,1)$-symplectic Hermitian metric whose canonical connection is
$\hat\nabla\in[\nabla]$.
\end{proof}

The real vector bundle $\cV$ can be realised naturally in two alternative ways
as follows. First, let us view $\T$ as a real vector bundle $\T_\R$ equipped
with a complex structure $J^\T$ (thus, equivalently,
$\T_\R\otimes\C\cong\T\oplus\overline\T$).  Then we can identify $\cV$ as the
$J^\T$-invariant elements in $S^2\T_\R$. However, since $J^\T$ induces an
isomorphism between $J^\T$-invariant elements in $S^2\T_\R$ and such elements
in $\Wedge^2\T_\R$, cf.~(\ref{Kaehlerform}), we may, secondly, realise $\cV$
as the latter. Realised as the bundle of $J^\T$-invariant elements in
$S^2\T_\R$ we can, for any choice of connection in the c-projective class,
identify an element of $\cV$ with a triple
\begin{equation*}
\begin{pmatrix} \ms^{\beta \gamma} \\ X^\beta\\ \rho
\end{pmatrix},\text{ where }
\left\{\begin{array}{l} 
\ms^{\beta\gamma}\in S^2TM\otimes\cE_\R(-1,-1)
\text{ with }
J_{\delta}{}^\beta J_{\epsilon}{}^\gamma\ms^{\delta\epsilon}=\ms^{\beta\gamma}\\
X^\beta\in TM\otimes\cE_\R(-1,-1)\\
\rho\in \cE_\R(-1,-1).
\end{array}\right.
\end{equation*}
In this picture the tractor connection becomes
\begin{equation}\label{Metritractor3}
\nabla^\cV_\alpha\begin{pmatrix} \ms^{\beta \gamma}\\ X^\beta\\ \rho
\end{pmatrix}=\begin{pmatrix}
\nabla_\alpha\ms^{\beta \gamma}+\delta_{\alpha}{}^{(\beta}X^{\gamma)}
+J_{\alpha}{}^{(\beta}J_{\epsilon}{}^{\gamma)}X^\epsilon\\ 
\nabla_{\alpha}X^\beta+\rho\delta_{\alpha}{}^{\beta}
-\Rho_{\alpha \gamma}\ms^{\beta \gamma}
\\ \nabla_{\alpha}\rho-\Rho_{\alpha\beta}X^{\beta}
\end{pmatrix}.
\end{equation}
The formulae (\ref{Metritractor1}) and (\ref{Metritractor2}) may be recovered
from~(\ref{Metritractor3}) by natural projection:
\begin{equation*}
\nabla_a^{\cV_\C}=\Pi_a^\alpha\nabla_\alpha^\cV,\quad
\nabla_{\bar{a}}^{\cV_\C}=
\overline{\Pi}{}_{\bar{a}}^\alpha\nabla_\alpha^\cV,\quad
\ms^{\bar b c}=
\overline{\Pi}{}_\beta^{\bar b} \Pi_\gamma^{c}\ms^{\beta\gamma},\quad
X^b=\Pi_\beta^bX^\beta,\quad
Y^{\bar{b}}=\overline{\Pi}{}_\beta^{\bar{b}}X^\beta,
\end{equation*}
so that, for example, 
\begin{align*}
\Pi_a^\alpha\Pi_b^\beta(\nabla_{\alpha}X^\beta+\rho\delta_{\alpha}{}^{\beta}
-\Rho_{\alpha \gamma}\ms^{\beta \gamma})&=
\nabla_aX^b+\rho\delta_a{}^b-\Rho_{a\gamma}\ms^{\gamma b}\\
&=\nabla_aX^b+\rho\delta_a{}^b-\Rho_{a\bar{c}}\ms^{\bar c b},
\end{align*}
as in (\ref{Metritractor1}). To pass explicitly to the second (skew) viewpoint
on $\cV$ described above, one can write
$\Phi^{\beta\gamma}=J_\alpha{}^\gamma\ms^{\alpha\beta}$ and
$Y^\beta=J_\alpha{}^\beta X^\alpha$. Then, for any choice of connection in the
c-projective class, an element of $\cV$ may alternatively be
identified with a triple
\begin{equation*}
\begin{pmatrix} \Phi^{\beta \gamma} \\ Y^\beta\\ \rho
\end{pmatrix},\text{ where } 
\left\{\begin{array}{l} 
\Phi^{\beta\gamma}\in \Wedge^2TM\otimes\cE_\R(-1,-1)
\text{ with } 
J_{\delta}{}^\beta J_{\epsilon}{}^\gamma\Phi^{\delta\epsilon}=
\Phi^{\beta\gamma}\\
Y^\beta\in TM\otimes\cE_\R(-1,-1)\\
\rho\in\cE_\R(-1,-1).
\end{array}\right.
\end{equation*}
The tractor connection becomes
\begin{equation}\label{Metritractor4}
\nabla^\cV_\alpha\begin{pmatrix} \Phi^{\beta \gamma}\\ Y^\beta\\ \rho
\end{pmatrix}=\begin{pmatrix}
\nabla_\alpha\Phi^{\beta \gamma}+\delta_{\alpha}{}^{[\beta}Y^{\gamma]}
+J_{\alpha}{}^{[\beta}J_{\epsilon}{}^{\gamma]}Y^\epsilon\\ 
\nabla_{\alpha}Y^\beta+\rho J_{\alpha}{}^{\beta}
+\Rho_{\alpha \gamma}\Phi^{\beta \gamma}
\\ \nabla_{\alpha}\rho+\Rho_{\alpha\beta}J_\gamma{}^\beta Y^\gamma
\end{pmatrix}.
\end{equation}
The formulae (\ref{Metritractor1}) and (\ref{Metritractor2}) are again
projections of~(\ref{Metritractor4}):
\begin{equation*}
\ms^{\bar b c}=
i\overline{\Pi}{}_\beta^{\bar b} \Pi_\gamma^{c}\Phi^{\beta\gamma}\quad
X^b=-i\Pi_\beta^bY^\beta\quad
Y^{\bar{b}}=i\overline{\Pi}{}_\beta^{\bar{b}}Y^\beta.
\end{equation*}

\subsection{The metrisability equation and mobility}\label{sec:met-mob}
Let $(M,J,[\nabla])$ be an almost c-projective manifold. By
Proposition~\ref{compKaehler}, solutions to the metrisability problem on $M$,
i.e.~compatible $(2,1)$-symplectic metrics up to sign, correspond bijectively
to nondegenerate solutions $\ms$ of the equation $D^\cV\ms=0$. We refer to this
equation as the \emph{metrisability equation} on $(M,J,[\nabla]$). It may be
written explicitly in several ways.

First, viewing $\cV$ as the real part of $\cV_\C$, $\ms^{\bar b c}$ satisfies,
by~\eqref{metriequ1c}, the conjugate equations:
\begin{equation} \label{metriequ1}
\nabla_a\ms^{\bar b c}+\delta_{a}{}^c X^{\bar b}=0\qquad\text{and}\qquad
\nabla_{\bar{a}}\ms^{\bar b c}+\delta_{\bar{a}}{}^{\bar{b}} X^c = 0
\end{equation}
for some (and hence any) connection $\nabla\in[\nabla]$ and some section $X^a$
of $T^{1,0}M\otimes\cE_\R(-1,-1)$ with conjugate $X^{\bar a}$. In the
alternative realisation~\eqref{Metritractor3} of $\cV$, the metrisability
equation for $J$-invariant sections $\ms^{\alpha\beta}$ of
$S^2TM\otimes\cE_\R(-1,-1)$ is
\begin{equation}\label{this_is_D_in_real_money}
\nabla_\alpha\ms^{\beta \gamma}+\delta_{\alpha}{}^{(\beta}X^{\gamma)}
+J_{\alpha}{}^{(\beta}J_{\epsilon}{}^{\gamma)}X^\epsilon=0
\end{equation}
for some section $X^\alpha$ of $TM\otimes\cE_\R(-1,-1)$. Similarly, using the
realisation~\eqref{Metritractor4} of $\cV$, the metrisability equation for
$J$-invariant sections $\Phi^{\alpha\beta}$ of $\Wedge^2TM\otimes\cE_\R(-1,-1)$ is
\begin{equation}\label{Ham2vector}
\nabla_\alpha\Phi^{\beta\gamma}+\delta_{\alpha}{}^{[\beta}Y^{\gamma]}
+J_{\alpha}{}^{[\beta}J_{\epsilon}{}^{\gamma]}Y^\epsilon=0
\end{equation}
for some section $Y^\alpha$ of $TM\otimes\cE_\R(-1,-1)$. 

\begin{defin} The \emph{\textup(degree of\textup) mobility} of an almost
c-projective manifold, is the dimension of the space
\[
\mob_c[\nabla]:=\ker D^\cV=\left\{\ms^{\alpha\beta} \left|\begin{array}{l}
J_\gamma{}^\alpha J_\epsilon{}^\beta\ms^{\gamma\epsilon}
=\ms^{\alpha\beta}\\ \nabla_\alpha\ms^{\beta\gamma}+\delta_\alpha{}^{(\beta}X^{\gamma)}
+J_\alpha{}^{(\beta}J_\epsilon{}^{\gamma)}X^\epsilon=0\mbox{ for some }X^\alpha
\end{array}\right.
\!\!\!\right\}
\]
of solutions to the metrisability equation.  
\end{defin}
In the sequel, the notion of mobility will only be of interest to us when the
metrisability equation has a nondegenerate solution. Then $(M,J,[\nabla])$ has
mobility $\geq 1$, and the mobility is the dimension of the space of
compatible $(2,1)$-symplectic metrics. For any $(2,1)$-symplectic Hermitian
metric $g$ on a complex manifold $(M,J)$, the mobility of the c-projective
class $[\nabla]$ of its canonical connection $\nabla$ is $\geq 1$, and will be
called the \emph{mobility of $g$}. If such a metric $g$ has mobility one,
i.e.~the constant multiples of $g$ are the only metrics compatible with
c-projective class $[\nabla]$, then most natural questions about the geometry
of the c-projective manifold $(M,J,[\nabla])$ can be reformulated as questions
about the Hermitian manifold $(M,J,g)$.  For example the c-projective vector
fields of $(M,J,[\nabla])$ are Killing or homothetic vector fields for $g$.
Hence, roughly speaking, there is essentially no difference between the
geometry of the Hermitian manifold $(M,J,g)$ and the geometry of the
c-projective manifold $(M,J,[\nabla])$.

We will therefore typically assume in the sequel that $(M,J,g)$, or rather, its
c-projective class $(M,J,[\nabla])$, has mobility $\geq 2$, and hence admits
compatible metrics $\tilde g$ that are not proportional to~$g$; we then say
$g$ and $\tilde g$ are \emph{c-projectively equivalent}. Although all metrics
in a given c-projective class are on the same footing, it will often be
convenient to fix a \emph{background metric} $g$, corresponding to a
nondegenerate solution $\ms$ of~\eqref{metriequ1}. Then any section $\sms$ of
$T^{0,1}M\otimes T^{1,0}M(-1,-1)$ may be written
\begin{equation*}
\sms^{\bar a c}=\ms^{\bar a b} A_b{}^c
\end{equation*}
for uniquely determined $A_b{}^c$---explicitly, we have:
\begin{equation*}
A^{\bar{a}b}=(\det\ms)\sms^{\bar{a}b}\quad\mbox{and}\quad
A_a{}^b=g_{a\bar{c}}A^{\bar{c}b}.
\end{equation*}
Since $\ms$ and $\sms$ are real, $A_b{}^c$ is $g$-Hermitian (i.e.\ the
isomorphism $T^{0,1}M\to\Omega^{1,0}$ induced by $g$ intertwines the transpose
of $A_{a}{}^{b}$ with its conjugate):
\begin{equation*}
\overline{A_a{}^b}=A^{\bar b}{}_{\bar a}:=g^{\bar bd} A_d{}^cg_{c\bar a}.
\end{equation*}
Using the canonical connection $\nabla$ of $g$, the metrisability
equation~\eqref{metriequ1} for $\sms$ may be rewritten as an equation for
$A_a{}^b$, which we call the \emph{mobility equation}:
\begin{equation}\label{metri-mob}
\nabla_a A_b{}^c =- \delta_a{}^c \hv_b, \quad\text{or\quad (equivalently)}
\qquad \nabla_{\bar c} A_a{}^b =- g_{a\bar c} \hv^b,
\end{equation}
where $\hv^b=\Pi^b_\beta\hv^\beta$ with $\hv^\beta$ real, and
$\hv_b=\Pi_b^\beta \hv_\beta= \Pi_b^\beta g_{\beta\alpha}\hv^\alpha= g_{b\bar
  a}\hv^{\bar a}$ with $\hv^{\bar a}=\Pi^{\bar a}_\alpha\hv^\alpha$.  Taking a
trace gives $\hv_c=\nabla_c\hp$ and $\hv_{\bar c} =\nabla_{\bar c} \hp$, with
$\hp=-A_{a}{}^a=-A_{\bar a}{}^{\bar a}$ real. The metric $g$ itself
corresponds to the solution $A_a{}^b = \delta_a{}^b$ of~\eqref{metri-mob},
with $\hv_c=0$.

Since the background metric $g$ trivialises the bundles $\cE(\ell,\ell)$ by
$\nabla$-parallel sections $\scale_g^{\;\ell}=(\det\ms)^\ell$, we shall often
assume these bundles are trivial. We may also raise and lower indices using $g$
to obtain further equivalent forms of the mobility equations:
\begin{align}\label{mob-raised}
\nabla_a A^{\bar{b}c}&=-\delta_a{}^{c}\hv^{\bar b}\qquad
\text{or}\qquad\nabla_{\bar{a}}A^{\bar{b}c}=-\delta_{\bar{a}}{}^{\bar{b}}\hv^c,\\
\label{mob-lowered}
\nabla_a A_{b\bar c}&=-g_{a\bar c}\hv_b\qquad\text{or}\qquad 
\nabla_{\bar a} A_{b\bar c}=-g_{b\bar a}\hv_{\bar c}.
\end{align}

Like the metrisability equation, the mobility equation can be rewritten in
explicitly real terms. If we let $\sms^{\alpha\gamma}=\ms^{\alpha\beta}
A_\beta{}^\gamma$ and raise indices using $g$, then the metrisability
equation~(\ref{this_is_D_in_real_money}) maybe rewritten as a mobility equation
for the unweighted tensor $A^{\alpha\beta}\in\Gamma(S^2_J(TM))$:
\begin{equation}\label{mobility_equation1}
\nabla_\alpha A^{\beta \gamma}=-\delta_{\alpha}{}^{(\beta}\hv^{\gamma)}
-J_{\alpha}{}^{(\beta}J_{\delta}{}^{\gamma)}\hv^\delta.
\end{equation}
We thus have that
\begin{equation}\label{Lambda}
\hv_\alpha=\nabla_\alpha \hp\qquad\text{where}\qquad
\hp=-\tfrac{1}{2}A_{\beta}{}^{\beta}.
\end{equation}
Tracing back through the identifications, note that
\begin{equation} \label{eqL}
A^{\alpha\beta}=\Bigl(\frac{\vol(\tilde g)}{\vol(g)}\Bigr)^{1/(n+1)}
\tilde g^{\alpha \beta},
\end{equation}
where $\tilde g^{\alpha \beta}=(\det\sms)\sms^{\alpha \beta}$ is the inverse
metric induced by $\sms^{\alpha \beta}$.

We may, of course, also lower indices to obtain:
\begin{equation}\label{mobility_equation2}
\nabla_\alpha A_{\beta\gamma}=-g_{\alpha(\beta}\hv_{\gamma)}
+\Omega_{\alpha(\beta}J_{\gamma)}{}^\delta \hv_\delta.
\end{equation}
This is the form of the mobility equation used in~\cite{DM,Sinjukov}
and~\cite[Equation~(3)]{FKMR} to study c-projectively equivalent K\"ahler
metrics. This is a special case of Proposition~\ref{compKaehler}, in which we
suppose that there is a \bps/K\"ahler metric in our c-projective class and we
ask about other \bps/K\"ahler metrics in the same c-projective class.

Finally, we may rewrite~\eqref{Ham2vector} as a mobility equation with respect
to a background $(2,1)$-symplectic metric $g$ with fundamental $2$-form $\Kf$
and canonical connection $\nabla$. Trivialising $\cE(1,1)$ and lowering indices
using $g$, we obtain
\begin{equation*}
\nabla_\alpha\Phi_{\beta\gamma}+g_{\alpha[\beta}Y_{\gamma]}
-\Kf_{\alpha[\beta}J_{\gamma]}{}^{\delta}Y_\delta=0
\end{equation*}
for a $2$-form $\Phi_{\alpha\beta}$. In the integrable case (i.e.~when $g$ is
\bps/K\"ahler) this is the equation for \emph{Hamiltonian $2$-forms} in the
terminology of~\cite{ACG}. We extend this terminology to the
$(2,1)$-symplectic setting and refer to its c-projectively invariant
version~\eqref{Ham2vector} as the equation for \emph{Hamiltonian $2$-vectors}
$\Phi^{\alpha\beta}$ on an almost c-projective manifold.

\begin{rem}\label{rem:wbf} If $g$ is a K\"ahler metric, then applying the
contracted differential Bianchi identity $g^{e\bar b} \nabla_{[e}R_{a]\bar b c
  \bar d}=0$ to the Bochner curvature decomposition~\eqref{BochnerDecomp}, we
deduce that if the Bochner curvature is coclosed, i.e.~$g^{e\bar b}\nabla_e
\Bt_{a\bar b c \bar d} =0$, then $A_{c\bar d}:=(n+2)\Xi_{c\bar d} +\Lambda
g_{c\bar d}$ satisfies the mobility equation in the form~\eqref{mob-lowered}.
Equivalently, the corresponding $J$-invariant $2$-form, which is a modification
of the Ricci form, is a Hamiltonian $2$-form.  This was one of the motivations
for the introduction of Hamiltonian $2$-forms in~\cite{ACG}, and is explored
further in~\cite{ApostolovIV}.
\end{rem}

\begin{rem}\label{rem:proj-analogue} Many concepts and results in c-projective
geometry have analogues in real projective differential geometry. We recall
that on a smooth manifold $M$ of dimension $m\geq 2$, a (real) projective
structure is a class $[\nabla]$ of projectively equivalent affine connections,
cf.~\eqref{projchange}. It is shown in~\cite{EM} that the operator
\begin{equation}\label{eq:proj-metrisability}
\Gamma(M,S^2TM(-2))\ni\ms^{\beta\gamma}\mapsto(\nabla_\alpha\ms^{\beta\gamma})_\circ,
\end{equation}
where $S^2TM(-2)$ denotes the bundle of contravariant symmetric tensors of
projective weight~$-2$ and $\circ$ denotes the trace-free part, is
projectively invariant (it is a first BGG operator) and that,
when $n$ is even and otherwise up to sign, nondegenerate
solutions are in bijection with compatible \bps/Riemannian metrics,
i.e.~metrics whose Levi-Civita connection is in the projective class
$[\nabla]$.  We define the \emph{mobility} of $[\nabla]$, or of any compatible
\bps/Riemannian metric, to be the dimension of this space
\[
\mob[\nabla]:=\{\ms^{\beta\gamma}\in\Gamma(M,S^2TM) \,|\,
\nabla_\alpha\ms^{\beta\gamma}=\delta_\alpha{}^\beta\mu^\gamma+\delta_a{}^\gamma\mu^\beta
\mbox{ for some }\mu^\alpha\}
\]
of solutions to this projective \emph{metrisability} or \emph{mobility
equation}, where we reserve the latter term for the case that the projective
structure admits a compatible metric.
\end{rem}

\subsection{Prolongation of the metrisability equation}
\label{metrisability_prolonged}

Suppose $(M,J,[\nabla])$ is an almost c-projective manifold and let us prolong
the invariant system of differential equations on sections $\ms^{\bar b c}$ of
$T^{0,1}M\otimes T^{1,0}M(-1,-1)$ given by (\ref{metriequ1c}). We have already
observed that (\ref{metriequ1c}) implies that
\begin{equation}\label{XYeq}
X^b=-\tfrac{1}{n}\nabla_{\bar a}\ms^{\bar a b}\quad\quad\text{and}
\quad\quad Y^{\bar b}=-\tfrac{1}{n}\nabla_a\ms^{\bar b a}.
\end{equation}
Moreover, we immediately deduce from (\ref{metriequ1c}) that 
\begin{align}
(\nabla_a\nabla_b-\nabla_b\nabla_a)\ms^{\bar c d}
+T_{ab}{}^{\bar e}\nabla_{\bar e}\ms^{\bar c d}
=2\delta_{[a}{}^d\nabla_{b]}Y^{\bar c}-T_{ab}{}^{\bar c} X^d\label{c1}\\
(\nabla_{\bar a}\nabla_{\bar b}-\nabla_{\bar b}\nabla_{\bar a})\ms^{\bar c d}
+T_{\bar a\bar b}{}^{e}\nabla_{e}\ms^{\bar c d}
=2\delta_{[\bar a}{}^{\bar c}\nabla_{\bar b]}X^d-T_{\bar a\bar b}{}^{d}Y^{\bar c}.\label{c2}
\end{align}
The left hand sides of equations (\ref{c1}) and (\ref{c2}) equal
\begin{align}
R_{ab}{}^{d}{}_{e}\ms^{\bar c e}&+R_{ab}{}^{\bar c}{}_{\bar e}\ms^{\bar e d}
+2\Rho_{[ab]}\ms^{\bar c d}
-\tfrac{1}{n+1}(\nabla_{\bar e} T_{ab}{}^{\bar e})\ms^{\bar c d}\nonumber\\
&=W_{ab}{}^d{}_e\ms^{\bar c e}+2\delta_{[a}{}^d\Rho_{b]e}\ms^{\bar c e}
+(\nabla_{\bar e} T_{ab}{}^{\bar c})\ms^{\bar e d}
-\tfrac{1}{n+1}(\nabla_{\bar e} T_{ab}{}^{\bar e})\ms^{\bar c d}\label{d1}\\
R_{\bar a\bar b}{}^{\bar c}{}_{\bar e}\ms^{\bar ed}&+R_{\bar a\bar b}{}^{d}{}_{e}\ms^{\bar ce}
+2\Rho_{[\bar a\bar b]}\ms^{\bar c d}
-\tfrac{1}{n+1}(\nabla_{e}T_{\bar a\bar b}{}^e)\ms^{\bar c d}\nonumber\\
&=W_{\bar a\bar b}{}^{\bar c}{}_{\bar e}\ms^{\bar e d}
+2\delta_{[\bar a}{}^{\bar c}\Rho_{\bar b]\bar e}\ms^{\bar ed}
+(\nabla_{e}T_{\bar a\bar b}{}^{d})\ms^{\bar c e}
-\tfrac{1}{n+1}(\nabla_{e}T_{\bar a\bar b}{}^e)\ms^{\bar c d}\label{d2},
\end{align}
where we have used Theorem~\ref{rosetta} to rewrite the curvature tensors
$R_{ab}{}^c{}_d$, $R_{\bar a\bar b}{}^{\bar c}{}_{\bar d}$, $R_{\bar a\bar b}{}^c{}_d$,
and $R_{ab}{}^{\bar c}{}_{\bar d}$. We conclude from \eqref{c1} and \eqref{d1},
taking a trace with respect to $a$ and $d$, and from \eqref{c2} and \eqref{d2},
taking a trace with respect to $\bar a$ and $\bar c$, that
\begin{equation}\label{e1}
\nabla_b Y^{\bar c}=\Rho_{be}\ms^{\bar c e}+\tfrac{1}{n}U_{b}{}^{\bar c}\qquad\qquad
\nabla_{\bar b} X^d=\Rho_{\bar b\bar e}\ms^{\bar e d}+\tfrac{1}{n}V_{\bar b}{}^d,
\end{equation}
where
\begin{align}
&U_{b}{}^{\bar c}:=\tfrac{n}{n-1} T_{ab}{}^{\bar c} X^a
+\tfrac{n}{n-1}(\nabla_{\bar e} T_{ab}{}^{\bar c})\ms^{\bar e a}
-\tfrac{n}{(n+1)(n-1)}(\nabla_{\bar e} T_{ab}{}^{\bar e})\ms^{\bar c a} \label{e4}\\
&V_{\bar b}{}^d:=\tfrac{n}{n-1} T_{\bar a\bar b}{}^{d} Y^{\bar a}
+\tfrac{n}{n-1}(\nabla_{e} T_{\bar a\bar b}{}^{d})\ms^{\bar a e}
-\tfrac{n}{(n+1)(n-1)}(\nabla_{e} T_{\bar a\bar b}{}^{e})\ms^{\bar a d},\label{e5}
\end{align}
depend linearly on $\ms^{\bar b c}$ and on $X^a$ respectively $Y^{\bar a}$.

\begin{rem} Suppose $J$ is integrable. Then the equations \eqref{e1} imply
$\nabla_b Y^{\bar c}=\Rho_{be}\ms^{\bar c e}$ and $\nabla_{\bar b} X^d
=\Rho_{\bar b\bar e}\ms^{\bar e d}$. Hence, in this case, the equalities between
(\ref{c1}) and (\ref{d1}) and between (\ref{c2}) and (\ref{d2}) show that
\begin{equation}\label{e3}
W_{ab}{}^d{}_e\ms^{\bar c e}\equiv 0\quad
\quad W_{\bar a\bar b}{}^{\bar c}{}_{\bar e}\ms^{\bar e d}\equiv 0.
\end{equation}
If $\ms^{\bar a b}$ is a nondegenerate solution of (\ref{metriequ1c}), then
(\ref{e3}) implies that $W_{ab}{}^c{}_d$ and its conjugate are identically
zero, which confirms again Proposition~\ref{kaehlerharmcurv} for $2n\geq 6$.
\end{rem}
Now consider
\begin{equation}\label{f1}
(\nabla_a\nabla_{\bar b}-\nabla_{\bar b}\nabla_a)\ms^{\bar c d}=
R_{a\bar b}{}^{d}{}_e\ms^{\bar c e}
+R_{a\bar b}{}^{\bar c}{}_{\bar e}\ms^{\bar e d}
+\Rho_{a\bar b}\ms^{\bar c d}-\Rho_{\bar b a}\ms^{\bar c d}.
\end{equation}
By Equation (\ref{metriequ1c}) and Theorem~\ref{rosetta} we may rewrite
(\ref{f1}) as
\begin{equation}\label{f11}
-\delta_{\bar b}{}^{\bar c}\nabla_aX^d
+\delta_{a}{}^d\nabla_{\bar b} Y^{\bar c}=
W_{a\bar b}{}^d{}_e\ms^{\bar c e}
+W_{a\bar b}{}^{\bar c}{}_{\bar e}\ms^{\bar e d}
+\delta_{a}{}^d\Rho_{\bar b e}\ms^{\bar c e}-
\delta_{\bar b}{}^{\bar c}\Rho_{a\bar e}\ms^{\bar e d}.
\end{equation}
Taking the trace in (\ref{f11}) with respect to $\bar{b}$ and $\bar{c}$ shows
that
\begin{equation}\label{f2}
\nabla_a X^d=\Rho_{a\bar e}\ms^{\bar e d}
-\tfrac{1}{n}\delta_{a}{}^d(\Rho_{\bar b e}\ms^{\bar b e}
-\nabla_{\bar b} Y^{\bar b})-\tfrac{1}{n}W_{a\bar b}{}^d{}_e\ms^{\bar b e}
\end{equation}
and with respect $a$ and $d$ that
\begin{equation}\label{f3}
\nabla_{\bar{b}}Y^{\bar{c}}=\Rho_{\bar{b}e}\ms^{\bar c e}
-\tfrac{1}{n}\delta_{\bar{b}}{}^{\bar{c}}(\Rho_{a\bar{e}}\ms^{\bar e a}
-\nabla_aX^d)+\tfrac{1}{n}W_{a\bar{b}}{}^{\bar{c}}{}_{\bar{e}}\ms^{\bar e a}.
\end{equation}
As the contraction of (\ref{f2}) with respect to $a$ and $d$ and the
contraction of (\ref{f3}) with respect $\bar{b}$ and $\bar{c}$ must lead to the
same result, we see that
\begin{equation}\label{f4}
\tfrac{1}{n}(\Rho_{\bar b e}\ms^{\bar b e}-\nabla_{\bar b} Y^{\bar b})=
\tfrac{1}{n}(\Rho_{a\bar e}\ms^{\bar e a}-\nabla_aX^a),
\end{equation}
which we denote by $\rho\in\Gamma(\cE(-1,-1))$.  Inserting (\ref{XYeq}) into
(\ref{f4}) therefore implies that
\begin{equation}\label{f5}
\rho=\tfrac{1}{n^2}
(\nabla_{\bar a}\nabla_b\ms^{\bar a b}+n\Rho_{\bar a b}\ms^{\bar a b})=
\tfrac{1}{n^2}
(\nabla_{b}\nabla_{\bar a}\ms^{\bar a b}+n\Rho_{a \bar b}\ms^{\bar b a}).
\end{equation}
By Theorem~\ref{rosetta} we have
\begin{align}\label{g1}
(\nabla_a\nabla_{\bar b}-\nabla_{\bar b}\nabla_a)X^c
&=R_{a\bar b}{}^c{}_d X^d+\Rho_{a\bar b}X^c-\Rho_{\bar ba}X^c\\
&=W_{a\bar b}{}^c{}_dX^d+\delta_{a}{}^c\Rho_{\bar b d}X^d
+\Rho_{a\bar b}X^c\nonumber.
\end{align}
Inserting the second equation of (\ref{e1}) and (\ref{f2}) into the left hand
side of (\ref{g1}) one computes that
\begin{equation} \label{g2}
(\nabla_a\nabla_{\bar b}-\nabla_{\bar b}\nabla_a)X^a
=n\nabla_{\bar b}\rho+\Rho_{a\bar b}X^a-n\Rho_{\bar b\bar e}Y^{\bar e}
+C_{a\bar b\bar e}\ms^{\bar e a}+Z_{\bar b},
\end{equation}
with
\begin{align}\label{t1}
Z_{\bar b}:=&\tfrac{n}{(n+1)(n-1)}(\nabla_{a}T_{\bar e\bar b}{}^a)Y^{\bar e}
+\tfrac{1}{n-1}T_{\bar e\bar b}{}^a\Rho_{ad}\ms^{\bar ed}
+\tfrac{1}{n(n-1)}T_{\bar e\bar b}{}^aU_a{}^{\bar e}\nonumber\\
&-\tfrac{1}{(n+1)(n-1)}(\nabla_a\nabla_dT_{\bar e\bar b}{}^d)\ms^{\bar e a}
+\tfrac{1}{(n-1)}(\nabla_a\nabla_dT_{\bar e\bar b}{}^a)\ms^{\bar e d},
\end{align}
where we have used \eqref{metriequ1c}, \eqref{e4} and that $W_{a\bar b}{}^a{}_d$
is zero. Note again that $Z_{\bar b}$ depends linearly on $\ms^{\bar a b}$,
$X^{a}$ and $Y^{\bar a}$.  {From}~(\ref{g1}), the expression (\ref{g2}) must
be equal to $n\Rho_{\bar b d}X^d+\Rho_{a\bar b}X^a$, which implies that
\begin{equation}\label{g3}
\nabla_{\bar b}\rho=\Rho_{\bar b a}X^a+\Rho_{\bar b\bar e}Y^{\bar e}
-\tfrac{1}{n}C_{a\bar b\bar e}\ms^{\bar e a}-\tfrac{1}{n} Z_{\bar b}.
\end{equation}
Rewriting $(\nabla_{a}\nabla_{\bar b}-\nabla_{\bar b}\nabla_{a})Y^{\bar c}$ 
analogously shows immediately that
\begin{equation}\label{g4}
\nabla_a\rho=\Rho_{a d}X^{d}+\Rho_{a\bar e}Y^{\bar e}
+\tfrac{1}{n}C_{a \bar b d}\ms^{\bar b d}+\tfrac{1}{n} Q_a,
\end{equation}
where 
\begin{align}\label{t2}
Q_{a}:=&\tfrac{n}{(n+1)(n-1)}(\nabla_{\bar e}T_{da}{}^{\bar e})X^{d}
+\tfrac{1}{n-1}T_{da}{}^{\bar b}\Rho_{\bar b\bar e}\ms^{\bar ed}
+\tfrac{1}{n(n-1)}T_{d a}{}^{\bar b}V_{\bar b}{}^{d}\nonumber\\
&-\tfrac{1}{(n+1)(n-1)}(\nabla_{\bar b}\nabla_{\bar e}T_{da}{}^{\bar e})\ms^{\bar b d}
+\tfrac{1}{(n-1)}(\nabla_{\bar b}\nabla_{\bar e}T_{da}{}^{\bar b})\ms^{\bar e d},
\end{align}
depends linearly on $\ms^{\bar a b}$, $X^{a}$ and $Y^{\bar a}$.
In summary, we have proved the following.
\begin{thm}\label{MetriProlongation}
Suppose $(M,J,[\nabla])$ is an almost c-projective manifold. The canonical
projection $\pi\colon\cV_\C:=\T\otimes \overline\T\to T^{0,1}M\otimes
T^{1,0}M(-1,-1)$ induces a bijection between sections of $\cV_\C$ that are
parallel for the linear connection
\begin{equation}\label{Metriprol1}
\nabla^{\cV_\C}_a\begin{pmatrix} \ms^{\bar b c}\\ X^b\enskip|\enskip Y^{\bar{b}}\\
\rho \end{pmatrix} +\frac{1}{n}\begin{pmatrix} 0\\
W_{a\bar d}{}^b{}_c\ms^{\bar d c}\enskip|\enskip -U_{a}{}^{\bar b}\\ 
\;-C_{a \bar b c}\ms^{\bar b c}-Q_a
\end{pmatrix}
\end{equation}
and
\begin{equation}\label{Metriprol2}
\nabla^{\cV_\C}_{\bar a}\begin{pmatrix} \ms^{\bar b c}\\
X^b\enskip|\enskip Y^{\bar{b}}\\ \rho\end{pmatrix}
+\frac{1}{n}\begin{pmatrix} 0\\ -V_{\bar a}{}^b\enskip|\enskip  
W_{\bar a c}{}^{\bar b}{}_{\bar d}\ms^{\bar d c}\\ 
\;-C_{\bar ac\bar b}\ms^{\bar b c}-Z_a
\end{pmatrix}
\end{equation}
and elements in the kernel of $D^{\cV_\C}$, where $U_{a}{}^{\bar b}$, $V_{\bar
  a}{}^b$, $Q_a$ and $Z_a$ are defined as in \eqref{e4}, \eqref{e5},
\eqref{t1} and \eqref{t2}.  The inverse of this bijection is induced by a
differential operator $L\colon T^{0,1}M\otimes T^{1,0}M(-1,-1)\to
\cV_\C$, which for a choice of connection $\nabla\in[\nabla]$ can be written
as
\begin{equation*}
L\colon \ms^{\bar b c}\mapsto\begin{pmatrix}\ms^{\bar b c}\\
-\frac{1}{n}\nabla_{\bar a}\ms^{\bar a b}
\enskip|\enskip{-\frac{1}{n}}\nabla_a\ms^{\bar b a}\\ 
\frac{1}{n^2}
(\nabla_{\bar a}\nabla_b\ms^{\bar a b}+n\Rho_{\bar a b}\ms^{\bar a b})
\end{pmatrix}.
\end{equation*}
If $J$ is integrable, $W_{a\bar b}{}^{c}{}_{d}=H_{a\bar b}{}^{c}{}_{d}$ and
$W_{\bar a b}{}^{\bar c}{}_{\bar d}=H_{\bar a b}{}^{\bar c}{}_{\bar d}$
\textup(by Theorem~\ref{rosetta}\textup) and $U_{a}{}^{\bar b}$,
$V_{\bar a}{}^b$, $Q_a$ and $Z_{\bar a}$ are identically zero.
\end{thm}

Let now $\cV$ be the real form of the vector bundle~$\cV_\C$, as defined in
the previous section. Obviously, the connection in
Theorem~\ref{MetriProlongation} preserves $\cV$ and therefore
Proposition~\ref{compKaehler} and Theorem~\ref{MetriProlongation} imply that:

\begin{cor}\label{MetriCorollary} Suppose $(M,J,[\nabla])$ is an almost
c-projective manifold of dimension $2n\geq4$. Then, up to sign, 
there exists a bijection
between compatible $(2,1)$-symplectic Hermitian metrics and sections $s$ of
$\cV$ that satisfy
\begin{itemize}
\item $\pi(s)\equiv\ms^{\bar b c}$ is nondegenerate 
\item $s$ is parallel for the connection given by {\rm(\ref{Metriprol1})} 
and~{\rm(\ref{Metriprol2})}. 
\end{itemize}
Note, that since $s$ is a real section, it is covariant constant for 
{\rm(\ref{Metriprol1})} if and only if it is covariant constant 
for~{\rm(\ref{Metriprol2})}.
\end{cor}

Suppose $s$ is a section of $\cV$ that is parallel for the tractor
connection. Then $\pi(s)\equiv\ms^{\bar b c}$ is still in the kernel of
$D^\cV$ and hence Theorem~\ref{MetriProlongation} implies that $s$ is also
parallel for the connection given by (\ref{Metriprol1})
and~(\ref{Metriprol2}), i.e.~$\pi(s)\equiv\ms^{\bar a b}$ must satisfy
$W_{a\bar{d}}{}^b{}_c\ms^{\bar d c}=0$,
$W_{\bar{a}c}{}^{\bar{b}}{}_{\bar{d}}\ms^{\bar d c}=0$, $U_{a}{}^{\bar b}=0$,
$V_{\bar a}{}^{b}=0$, $C_{a\bar{b}c}\ms^{\bar b c}+Q_a=0$ and
$C_{\bar{a}b\bar{c}}\ms^{\bar c b}+Z_{\bar a}=0$. The following proposition
gives a geometric interpretation of parallel sections of the tractor
connection and hence of so-called \emph{normal} solutions of the first
BGG operator~$D^\cV$ in the terminology of \cite{CGH}.

\begin{prop}\label{normal_sol_ms} 
Suppose $(M,J,[\nabla])$ is an almost c-projective manifold of dimension
$2n\geq4$.  Then, if $n$ is even, there is a 
bijection between sections $s$ of $\cV$ such that
\begin{itemize}
\item $\pi(s)\equiv\ms^{\bar b c}$ is nondegenerate
\item $s$ is parallel for the tractor connection $\nabla^\cV$ 
on~$\cV$
\end{itemize}
and compatible $(2,1)$-symplectic metrics $g$ satisfying the generalised
Einstein condition\textup:
\begin{equation}\label{generalised_Einstein}
\Ric_{ab}=0\quad \text{ and }\quad
\Ric_{a\bar b}=kg_{a\bar b} \text{ for some constant } k\in\R, 
\end{equation}
where $\Ric$ is the Ricci tensor of the canonical connection of $g$. If $J$ is
integrable, then \eqref{generalised_Einstein} simply characterises
\bps/K\"ahler--Einstein metrics. If $n$ is odd, the same conclusions are valid 
up to sign.
\end{prop}
\begin{proof}
Suppose $s\in\Gamma(\cV)$ is parallel for the tractor connection $\nabla^\cV$
and that $\pi(s)\equiv\ms^{\bar b c}\in\ker D^{\cV}$ is nondegenerate.  Then
Proposition \ref{compKaehler} implies that the inverse of $g^{\bar a b}
\equiv\ms^{\bar ab}\det\ms$ is a compatible $(2,1)$-symplectic Hermitian
metric. Now let $\nabla\in[\nabla]$ be the canonical connection of $g_{a\bar
  b}$.  With respect to the splitting of $\cV$ determined by $\nabla$ the
section $s$ corresponds to the section
\begin{equation}\label{s_in_splitting}
\begin{pmatrix} \ms^{\bar b c}\\ X^b\enskip|\enskip X^{\bar b}\\ \rho \end{pmatrix}
=\begin{pmatrix} \ms^{\bar b c}\\ 0\enskip|\enskip 0\\
\frac{1}{n}\Rho_{a\bar b}\ms^{\bar b a}\end{pmatrix}.
\end{equation}
{From} $\nabla^\cV s=0$ it therefore follows on the one hand that
$\Rho_{ac}\ms^{\bar b c}=0$, which implies $\Rho_{ac}=0$ by
the nondegeneracy of $\ms^{\bar bc}$.  Since
$\Ric_{ab}=(n-1)\Rho_{ab}+2\Rho_{[ab]}$, we see that
the first condition of \eqref{generalised_Einstein} holds for $g$.  On the
other hand, we deduce from $\nabla^\cV s=0$ that
$\Rho_{a\bar{c}}\ms^{\bar c b}=\rho\delta_{a}{}^b$ and
$\nabla_a\rho=\nabla_{\bar a}\rho=0$.  Since
$\Rho_{a\bar{c}}=\frac{1}{n+1}\Ric_{a\bar{b}}$, we conclude
that
\[
\Ric_{a\bar c}g^{\bar c b}=(n+1)\rho (\det\ms)\delta_{a}{}^b.
\]
Hence $g_{a \bar b}$ satisfies also the second condition of
\eqref{generalised_Einstein}.

Conversely, suppose $g_{a\bar b}$ is a compatible $(2,1)$-symplectic Hermitian
metric satisfying \eqref{generalised_Einstein}.  Let us write
$s\in\Gamma(\cV)$ for the corresponding parallel section of the prolongation
connection given by (\ref{Metriprol1}) and (\ref{Metriprol2}). With respect to
the splitting of $\cV$ determined by the canonical connection $\nabla$ of
$g_{a \bar b}$, the section $s$ is again given by \eqref{s_in_splitting},
where $\ms^{\bar a b}\det(\ms)=g^{\bar a b}$.  By assumption we have
\begin{equation}\label{condition_1}
\Ric_{ab}=0,
\end{equation}
which is equivalent to $\Rho_{ab}=0$, and also that
\begin{equation}\label{Einsteinscale}
\Rho_{a\bar b}=\tfrac{1}{n+1}\Ric_{a\bar b}=\tfrac{k}{n+1}g_{a\bar b}
\end{equation}
for some constant $k$. Moreover, \eqref{condition_1} yields 
\begin{equation}\label{deriv_torsion_Ricci}
0=\Ric_{ab}=-g^{\bar cd}\nabla_{\bar c} T_{da}{}^{\bar f}g_{b\bar f}\qquad\qquad
0=\Ric_{[ab]}=\tfrac{1}{2}\nabla_{\bar c} T_{ab}{}^{\bar c}
\end{equation} 
which shows immediately that (with respect to $\nabla$) in (\ref{Metriprol1})
and (\ref{Metriprol2}) we have
\[
U_{a}{}^{\bar b}=0\qquad V_{\bar a}{}^{b}=\overline{U_{a}{}^{\bar b}}=0\qquad
Q_a=0\qquad Z_{\bar a}=\overline{Q_{\bar a}}=0.
\]
Hence, to prove that $s$ is parallel for $\nabla^\cV$ it remains to show that
$W_{a\bar{d}}{}^b{}_cg^{\bar d c}$ and $C_{a\bar{b}c}g^{\bar b c}$ (or
equivalently their conjugates) are identically zero.  {From}
Theorem~\ref{rosetta} and (\ref{Einsteinscale}) we obtain
\begin{align}\label{Einsteincurv}
W_{a\bar{d}}{}^b{}_cg^{\bar d c}=
R_{a\bar{d}}{}^b{}_cg^{\bar d c}-\delta_{a}{}^b\Rho_{\bar{d}c}g^{\bar d c}
-\Rho_{\bar{d}a}g^{\bar{d}b}
=R_{a\bar{d}}{}^b{}_cg^{\bar d c}-k\delta_{a}{}^b.
\end{align}
Therefore, if we lower the $b$ index in (\ref{Einsteincurv}) with the metric,
we obtain
\[
W_{a\bar{d}\bar{b}c}g^{\bar d c} =R_{a\bar{d}\bar{b}c}g^{\bar d c}-kg_{a\bar{b}}.
\]
Since $\nabla$ preserves $g$, the tensors $R_{a\bar{d}\bar{b}c}
=-R_{\bar{d}a\bar bc}$ and $-R_{a\bar{d}c\bar{b}}$ coincide.  Hence,
$R_{a\bar{d}\bar{b}c}g^{\bar d c}=R_{\bar{d}ac\bar{b}}g^{\bar d c}
=\Ric_{a\bar{b}}=kg_{a\bar{b}}$, which shows that (\ref{Einsteincurv})
vanishes identically. {From} \eqref{condition_1} and \eqref{Einsteinscale} it
follows immediately that $C_{a\bar{b}c}=\nabla_a\Rho_{\bar bc}-\nabla_{\bar b}
\Rho_{ac}$ vanishes identically, which completes the proof.
\end{proof}

\begin{rem}\label{rem:metric-cone} As observed in
Section~\ref{metrisability}, $\cV_\C= \T\otimes\bar\T$, and sections of $\cV$
may be viewed as Hermitian forms on $\T^*$.  This has an interpretation in
terms of the construction of the complex affine cone $\pi_\cC\colon \cC\to M$
described in Section~\ref{cone_construction}: by
Lemma~\ref{tangent_bundle_cone}, a Hermitian form on $\T^*$ pulls back to a
Hermitian form on $T^*\cC$.  If this form is nondegenerate, its inverse
defines a Hermitian metric on $\cC$.  Further, if the section of $\cV$ is
parallel with respect to a connection on $\cV$ induced by a connection on
$\T$, then the latter connection induces a metric connection on $\cC$.

In particular, if we have a compatible metric satisfying the generalised
Einstein condition of Proposition~\ref{normal_sol_ms}, then it generically
induces a metric on $\cC$ which is parallel for the connection $\nabla^\cC$
induced by the tractor connection on $\T$.
\end{rem}

\subsection{The c-projective Hessian}\label{cpHessianSection}

Let us consider the dual $\cW$ of the tractor bundle $\cV$ of an almost
c-projective manifold. Its complexification is given by $\cW_\C=
\T^*\otimes \overline\T{}^*$, which admits a filtration
\[
\cW_\C=\cW_\C^{-1}\supset \cW_\C^{0}\supset \cW_\C^1,
\]
such that for any connection $\nabla\in[\nabla]$ we can write an element of
$\cW_\C$ as
\begin{equation*}
\begin{pmatrix} \sigma \\
\mu_b\enskip|\enskip \lambda_{\bar{b}}\\ \zeta_{b\bar{c}}
\end{pmatrix},\text{ where } 
\left\{\begin{array}{l}
\sigma\in \cE(1,1),\\
\mu_b\in \Wedge^{1,0}M(1,1),\quad \lambda_{\bar{b}}\in \Wedge^{0,1}M(1,1),\\ 
\zeta_{b\bar{c}}\in \Wedge^{1,1}M(1,1),
\end{array}\right.
\end{equation*}
and the tractor connection as
\begin{equation}\label{Hessiantractor1}
\nabla^{\cW_\C}_a\begin{pmatrix} \sigma\\
\mu_b\enskip|\enskip \lambda_{\bar{b}}\\ \zeta_{b\bar c} \end{pmatrix}
=\begin{pmatrix} \nabla_a\sigma-\mu_{a}\\ 
\nabla_a\mu_b+\Rho_{ab}\sigma\enskip|\enskip
\nabla_a\lambda_{\bar{b}}+\Rho_{a\bar b}\sigma-\zeta_{a\bar b}\\ 
\nabla_a\zeta_{b\bar c}+\Rho_{a\bar c}\mu_{b}+\Rho_{ab}\lambda_{\bar c}
\end{pmatrix}
\end{equation}
and
\begin{equation}\label{Hessiantractor2}
\nabla^{\cW_\C}_{\bar a}\begin{pmatrix} \sigma\\
\mu_b\enskip|\enskip \lambda_{\bar{b}}\\ \zeta_{b\bar c} \end{pmatrix}
=\begin{pmatrix} \nabla_{\bar a}\sigma-\lambda_{\bar a}\\ 
\nabla_{\bar a}\mu_b+\Rho_{\bar a b}\sigma-\zeta_{b\bar a}\enskip|\enskip
\nabla_{\bar a}\lambda_{\bar{b}}+\Rho_{\bar a\bar b}\sigma\\
\nabla_{\bar a}\zeta_{b\bar c}+\Rho_{\bar a\bar c}\mu_{b}+\Rho_{\bar ab}\lambda_{\bar c}
\end{pmatrix}.
\end{equation}
The first BGG operator associated to $\cW_\C$ or $\cW$ is a
c-projectively invariant operator of order two, which we call the
\emph{c-projective Hessian}.  It can be written as
\begin{equation}\label{cprojHessian}
D^\cW\colon \cE(1,1)\to\!\!
\begin{matrix} S^2\Wedge^{1,0}M(1,1)
\oplus
S^2\Wedge^{0,1}M(1,1)
\end{matrix}
\end{equation}
\begin{equation*}
D^\cW\sigma=(\nabla_ {(a}\nabla_{b)}\sigma+\Rho_{(ab)}\sigma\,,
\, \nabla_ {(\bar a}\nabla_{\bar b)}\sigma+\Rho_{(\bar a\bar b)}\sigma),
\end{equation*}
or alternatively as
\begin{equation} \label{cprojHessian3}
D^\cW\sigma=\nabla_{(\alpha}\nabla_{\beta)}\sigma+\Rho_{(\alpha\beta)}\sigma-
J_{(\alpha}{}^{\gamma}J_{\beta)}{}^{\delta}(\nabla_\gamma\nabla_\delta\sigma
+\Rho_{\gamma\delta}\sigma),
\end{equation}
for any connection $\nabla\in[\nabla]$. The reader might easily verify the
c-projective invariance of $D^\cW$ directly using 
Proposition~\ref{changeform}, the identities~(\ref{changes_on_densities}), 
and the formulae for the change of Rho tensor in Corollary \ref{Rho_changes}.
The following Proposition gives a geometric interpretation of nonvanishing
real solutions $\sigma=\bar\sigma\in\Gamma(\cE(1,1))$ of the invariant
overdetermined system $D^{\cW}\sigma=0$.

\begin{prop}\label{J_invariant_Ricci} Let $(M,J,[\nabla])$ be an almost
c-projective manifold and $\sigma\in\Gamma(\cE(1,1))$ a real nowhere vanishing
section.  Then $D^\cW\sigma=0$ if and only if the Ricci tensor of
the special connection $\nabla^{\sigma}\in[\nabla]$ associated to $\sigma$
satisfies $\Ric_{(ab)}=0$. In particular, if $J$ is integrable, then
$D^\cW\sigma=0$ if and only if the Ricci tensor of $\nabla^{\sigma}$
satisfies $\Ric_{ab}=0$, i.e.~the Ricci tensor is symmetric and $J$-invariant.
\end{prop}
\begin{proof}
Let $\sigma=\bar \sigma\in\Gamma(\cE(1,1))$ be nowhere vanishing. Recall
that the Ricci tensor of the special connection $\nabla^{\sigma}$ associated to
$\sigma$ satisfies
$$\Ric_{\bar a b}=\Ric_{b\bar a}\quad\quad\quad
\Ric_{[ab]}=\tfrac{1}{2}\nabla_{\bar c}^{\sigma}T_{ab}{}^{\bar c}.$$
With respect to $\nabla^{\sigma}$ the equation $D^{\cW_\C}\sigma=0$
reduces to
\[
\Rho_{(ab)}\sigma=0\quad\quad\quad\Rho_{(\bar a\bar b)}\sigma=0,
\]
i.e.~to $\Ric_{(ab)}=\frac{1}{n-1}\Rho_{(ab)}=0$ and $\Ric_{(\bar a\bar
  b)}=\frac{1}{n-1}\Rho_{(\bar a\bar b)}=0$, since $\sigma$ is nonvanishing.
\end{proof}
It follows immediately that if a c-projective manifold $(M,J,[\nabla])$
admits a compatible \bps/K\"ahler metric $g$, then $\scale_g=\vol(g)^{-1/(n+1)}
\in\Gamma(\cE(1,1))$ satisfies $D^\cW\scale_g=0$. By
Proposition~\ref{compKaehler}, $\scale_g=\det\ms$, where $\ms$ is the
nondegenerate solution of the metrisability equation corresponding to $g$.
This observation continues to hold without the nondegeneracy assumption.
\begin{prop}\label{det_of_metrisability} 
Let $(M,J, [\nabla])$ be a c-projective manifold and suppose that
$\ms^{\bar b c}\in\Gamma(T^{0,1}M\otimes T^{1,0}M(-1,-1))$ is a real section
satisfying~\eqref{metriequ1}. Then
$\sigma\equiv\det\ms\in\Gamma(\cE(1,1))$ is a real section in the
kernel of the c-projective Hessian \textup(which might be identically
zero\textup).
\end{prop}
\begin{proof} Let $U\subset M$ be the open subset (possibly empty),
where $\sigma$ is nowhere vanishing or equivalently where $\ms^{\bar b c}$ is
invertible.  By Proposition \ref{compKaehler} the section $\ms^{\bar b c}
(\det\ms)\in\Gamma(T^{0,1}M\otimes T^{1,0}M)$ defines the inverse of a
compatible \bps/K\"ahler metric on $U$ and its Levi-Civita connection on
$U$ is $\nabla^{\sigma}$.  Since the Ricci tensor of a \bps/K\"ahler metric
is $J$-invariant~\eqref{KaehlerCurv}, i.e.~$\Ric_{ab}=\Ric_{\bar a\bar b}=0$, we
deduce from Proposition \ref{J_invariant_Ricci} that $\sigma$ satisfies
$D^{\cW}\sigma=0$ on $U$ whence, by continuity, on~$\overline{U}$. 
Since $\sigma$ vanishes identically on the open set $M\setminus \overline{U}$,
we obtain that $D^\cW\sigma$ is identically zero on all of $M$.
\end{proof}

\begin{rem} For an almost c-projective manifold admitting a
compatible $(2,1)$-symplectic metric $g$, the section
$\scale_g\in\Gamma(\cE(1,1))$ is in the kernel of the c-projective Hessian, if
the Ricci tensor of the canonical connection $\nabla$ of $g$ satisfies
\begin{equation*}
\Ric_{(ab)}=-\nabla_{\bar c} T^{\bar c}{}_{(ab)}
=\tfrac{1}{4}\nabla_{\bar c} N^{\bar c}{}_{(ab)}=0,
\end{equation*}
where we use $g$ to raise and lower indices. It is well known that nearly
K\"ahler manifolds can be characterised as $(2,1)$-symplectic manifolds such
that $T_{abc}$ is totally skew (see e.g.\,\cite{Kob}). It then follows
straightforwardly from the identities \eqref{21Curv1a}--\eqref{21Curv1b} that
the canonical connection of a nearly K\"ahler manifold preserves its torsion,
i.e.~$\nabla T=-\frac{1}{4}\nabla N=0$ (see \cite{Kir,Nagy}), and
$R_{ab\bar c d}$ vanishes identically. Hence, Proposition
\ref{det_of_metrisability} extends to the nearly K\"ahler setting.
\end{rem}

\subsection{Prolongation of the c-projective Hessian}
\label{Section_Prolong_Hessian}

The c-projective Hessian will play a crucial role in the sequel. We therefore
prolong the associated equation. Suppose $\sigma\in\Gamma(\cE(1,1))$ is in the
kernel of the c-projective Hessian:
\begin{equation}\label{cprojHessian2}
\nabla_ {(a}\nabla_{b)}\sigma+\Rho_{(ab)}\sigma=0\quad\quad
\nabla_ {(\bar a}\nabla_{\bar b)}\sigma+\Rho_{(\bar a\bar b)}\sigma=0,
\end{equation}
Then we deduce from \eqref{another_curvature_on_densities} that
\eqref{cprojHessian2} is equivalent to
\begin{align}
\nabla_{a}\nabla_{b}\sigma+\Rho_{ab}\sigma
=\nabla_{[a}\nabla_{b]}\sigma+\Rho_{[ab]}\sigma
=\tfrac{1}{2(n+1)}(\nabla_{\bar c} T_{ab}{}^{\bar c})\sigma
-\tfrac{1}{2}T_{ab}{}^{\bar c}\nabla_{\bar c}\sigma\label{Phi}\\
\nabla_{\bar a}\nabla_{\bar b}\sigma+\Rho_{\bar a\bar b}\sigma
=\nabla_{[\bar a}\nabla_{\bar b]}\sigma+\Rho_{[\bar a\bar b]}\sigma
=\tfrac{1}{2(n+1)}(\nabla_{c} T_{\bar a\bar b}{}^{c})\sigma
-\tfrac{1}{2}T_{\bar a\bar b}{}^{c}\nabla_{c}\sigma\label{Psi},
\end{align}
where we abbreviate the left-hand sides by $\Phi_{ab}$ respectively
$\Psi_{\bar a \bar b}$, which depend linearly on $\sigma$ and on
$\lambda_{\bar a}:= \nabla_{\bar a}\sigma$ respectively $\mu_a:=
\nabla_a\sigma$.  {From} (\ref{curvature_on_densities}) we moreover deduce that
\[
\nabla_a\lambda_{\bar b}+\Rho_{a\bar b}\sigma
=\nabla_a\nabla_{\bar b}\sigma+\Rho_{a\bar b}\sigma
=\nabla_{\bar b}\nabla_a\sigma+\Rho_{\bar b a}\sigma
=\nabla_{\bar b} \mu_a+\Rho_{\bar b a}\sigma,
\]
which we shall denote by $\zeta_{a\bar b}\in \Wedge^{1,1}M(1,1)$.
Consequently, we have
\begin{align}\label{eq1Hessian}
&\nabla_a\nabla_{\bar c} \mu_b-\nabla_{\bar c}\nabla_a\mu_b=\\
&\nabla_a\zeta_{b\bar c}-(\nabla_{a}\Rho_{\bar c b})\sigma-\Rho_{\bar c b}\mu_a
+ (\nabla_{\bar c}\Rho_{ab})\sigma+\Rho_{ab}\lambda_{\bar c}+\Sigma_{ab\bar c},
\nonumber
\end{align}
where 
\begin{align}\label{Sigma}
\Sigma_{ab\bar c}:=&
-\tfrac{1}{2(n+1)}((\nabla_{\bar c}\nabla_{\bar d} T_{ab}{}^{\bar d})\sigma
+(\nabla_{\bar d} T_{ab}{}^{\bar d})\lambda_{\bar c})\\
&+\tfrac{1}{2}((\nabla_{\bar c}T_{ab}{}^{\bar d})\lambda_{\bar d}
-T_{ab}{}^{\bar d}\Rho_{\bar c\bar d}+T_{ab}{}^{\bar d}\Psi_{\bar c\bar d}) \nonumber
\end{align}
depends linearly on $\sigma$, $\mu_a$ and $\lambda_{\bar a}$. {From} Proposition
\ref{rosetta} and the identity (\ref{curvature_on_densities}) we obtain that
the expression (\ref{eq1Hessian}) must be also equal to
\begin{align}
\nabla_a\nabla_{\bar c} \mu_b-\nabla_{\bar c}\nabla_a\mu_b=
-W_{a\bar c}{}^d{}_b\mu_d-\Rho_{\bar c b}\mu_a-\Rho_{a\bar c}\mu_b,
\end{align}
which shows that
\begin{align}
\nabla_a\zeta_{b\bar c}=-\Rho_{ab}\lambda_{\bar c}
-\Rho_{a\bar c}\mu_b-W_{a\bar c}{}^{d}{}_b\mu_d+C_{a\bar c b}\sigma-\Sigma_{ab\bar c}.
\end{align}
Similarly, one shows that
\begin{align}
\nabla_{\bar a}\zeta_{b\bar c}=-\Rho_{\bar a b}\lambda_{\bar c}-\Rho_{\bar a\bar c}\mu_b-
W_{\bar ab}{}^{\bar d}{}_{\bar c}\lambda_{\bar d}+C_{\bar ab\bar c}\sigma-\Xi_{\bar a b\bar c},
\end{align}
where
\begin{align}\label{Xi}
\Xi_{\bar a b\bar c}:=&-\tfrac{1}{2(n+1)}((\nabla_{b}\nabla_dT_{\bar a\bar c}{}^d)\sigma
+(\nabla_dT_{\bar a\bar c}{}^d)\mu_b)\\
&+\tfrac{1}{2}((\nabla_{b}T_{\bar a\bar c}{}^d)\mu_d-T_{\bar a\bar c}{}^d\Rho_{bd}\sigma
+T_{\bar a\bar c}{}^d\Phi_{bd})\nonumber
\end{align}
depends linearly on $\sigma$, $\mu_a$ and $\lambda_{\bar a}$.  In summary, we
have shown the following theorem:
\begin{thm}\label{HessianProlongation} 
Suppose $(M, J, [\nabla])$ is a c-projective manifold. Then the canonical
projection $\pi\colon \cW_\C\to\cE(1,1)$ induces a bijection
between sections of $\cW_\C$ that are parallel for the linear
connection
\begin{align}\label{Hessian_prolongation}
\nabla^{\cW_\C}_a\begin{pmatrix} \sigma\\
\mu_b\enskip|\enskip \lambda_{\bar{b}}\\ \zeta_{b\bar c}\end{pmatrix}
&+\begin{pmatrix} 0\\  -\Phi_{ab}\enskip|\enskip 0\\ 
W_{a\bar c}{}^d{}_b\mu_d-C_{a\bar cb}\sigma+\Sigma_{ab\bar c}
\end{pmatrix}\\
\label{Hessian_prolongation2}
\nabla^{\cW_\C}_{\bar a}\begin{pmatrix} \sigma\\
\mu_b\enskip|\enskip \lambda_{\bar{b}}\\ \zeta_{b\bar c}\end{pmatrix}
&+\begin{pmatrix}  0\\  0\enskip|\enskip -\Psi_{\bar a\bar b}\\
W_{\bar a b}{}^{\bar d}{}_{\bar c} \lambda_{\bar d}
-C_{\bar ab\bar c}\sigma+\Xi_{\bar a b\bar c}
\end{pmatrix}.
\end{align}
and sections $\sigma\in\Gamma(\cE(1,1))$ in the kernel of the c-projective
Hessian, where $\Phi_{ab}$, $\Psi_{\bar a\bar b}$, $\Sigma_{ab\bar c}$ and
$\Xi_{\bar a b\bar c}$ are defined as in \eqref{Phi}, \eqref{Psi},
\eqref{Sigma} and \eqref{Xi}.  The inverse of this bijection is induced by a
linear differential operator $L$, which, for a choice of connection
$\nabla\in[\nabla]$, can be written as
\begin{align*}
L\colon \cE(1,1)&\to \cW_\C\\
L(\sigma)&=\begin{pmatrix} \sigma \\ \nabla_a\sigma \enskip|\enskip
\nabla_{\bar{a}}\sigma\\
\nabla_a\nabla_{\bar b}\sigma+\Rho_{a\bar b}\sigma \end{pmatrix}.
\end{align*}
\end{thm}
The following Proposition characterises normal solutions of
$D^\cW(\sigma)=0$, i.e.~real sections $\sigma=\bar\sigma\in
\Gamma(\cE(1,1))$ in the kernel of the c-projective Hessian that in
addition satisfy:
\begin{align}\label{normalsolHessian}
\Phi_{ab}&=0 & \Psi_{\bar a\bar b}&=0\\
W_{a\bar c}{}^{d}{}_b\nabla_d\sigma-C_{a\bar c b}\sigma&+\Sigma_{ab\bar c}=0
&W_{\bar ab}{}^{\bar d}{}_{\bar c}\nabla_{\bar d}\sigma-C_{\bar ab\bar c}\sigma
&+\Xi_{\bar a b\bar c}=0\label{normalsolHessian2},
\end{align}
where $\Phi$, $\Psi$, $\Sigma$ and $\Xi$ depend linearly on $\sigma$ and
$\nabla\sigma$.

\begin{prop} Let $(M, J,[\nabla])$ be an almost c-projective manifold and
suppose that $\sigma\in \Gamma(\cE(1,1))$ is a real nowhere vanishing section
in the kernel of the c-projective Hessian. Then $\sigma$ satisfies
\textup{(\ref{normalsolHessian})} if and only if the Ricci tensor
$\Ric_{\alpha\beta}$ of the special connection $\nabla^{\sigma}\in[\nabla]$
corresponding to $\sigma$ satisfies
\[
\Ric_{ab}=0\quad\quad\text{and}\quad\quad
\nabla_a^{\sigma}\Ric_{b\bar c}=0=\nabla_{\bar a}^{\sigma}\Ric_{b\bar c}.
\] 
If the Ricci tensor $\Ric_{b\bar c}=\Ric_{\bar c b}$ is, in addition, nondegenerate,
then it defines a $(2,1)$-symplectic Hermitian metric satisfying the
generalised Einstein condition \eqref{generalised_Einstein} with canonical
connection $\nabla^{\sigma}$.
\end{prop}
\begin{proof}
Let $\sigma=\bar\sigma\in\Gamma(\cE(1,1))$ be a real nowhere 
vanishing section in the kernel of (\ref{cprojHessian}). With respect to the 
special connection $\nabla^{\sigma}\in[\nabla]$ corresponding to $\sigma$, the
equations (\ref{normalsolHessian}) reduce to
\begin{align*}
0&=\tfrac{1}{2(n+1)}\nabla^\sigma_{\bar c} T_{ab}{}^{\bar c}
=\tfrac{1}{n+1}\Ric_{[ab]}=\Rho_{[ab]}\\
0&=\tfrac{1}{2(n+1)}\nabla^\sigma_{c} T_{\bar a\bar b}{}^{c}
=\tfrac{1}{n+1}\Ric_{[\bar a\bar b]}=\Rho_{[\bar a\bar b]},
\end{align*}
which, since $\sigma$ is in the kernel of the c-projective Hessian, is
equivalent to $\Ric_{ab}=0=\Ric_{\bar a\bar b}$. If these equations are
satisfied, is follows immediately that also $\Sigma_{ab\bar c}$ and $\Xi_{\bar
  a b\bar c}$ are identically zero (with respect to $\nabla^\sigma$) and that
the equations (\ref{normalsolHessian2}) reduces to
\begin{align*}
C_{a\bar c b}\sigma&=(\nabla_a^{\sigma}\Rho_{\bar c b})\sigma
=(\nabla_a^{\sigma}\Rho_{b\bar c})\sigma
=\tfrac{1}{n+1}(\nabla_a^{\sigma}\Ric_{b\bar c})\sigma=0\\
C_{\bar ab\bar c}\sigma&=(\nabla_{\bar a}^{\sigma}\Rho_{b \bar c})\sigma
=\tfrac{1}{n+1}(\nabla_{\bar a}^{\sigma}\Ric_{b\bar c})\sigma=0
\end{align*}
which proves the claim, since $\sigma$ is nowhere vanishing.
\end{proof}

\section{Metrisability, conserved quantities and integrability}
\label{sec:integrability}

In this section we investigate the implications of mobility $\geq 2$ for the
geodesic flow of a \bps/K\"ahler manifold $(M,J,g)$: we show that any metric
$\tilde g$ c-projectively equivalent, but not homothetic, to $g$ gives rise to
families of commuting linear and quadratic integrals for the geodesic flow of
$g$, and characterise when this implies integrability of the flow.

\subsection{Conserved quantities for the geodesic flow}

For any smooth manifold $M$, the total space of its cotangent bundle $p\colon
T^*M\to M$ has a canonical exact symplectic structure $d\Theta$, where
$\Theta\colon T T^*M\to \R$ is the \emph{tautological $1$-form} defined by
$\Theta_\alpha(X)=\alpha(Tp(X))$.  The Poisson bracket of smooth functions on
$T^*M$ preserves the subalgebra
\begin{equation*}
C^\infty_{\mathrm{pol}}(T^*M,\R)\cong\bigoplus_{k\geq 0} C^\infty(M,S^kTM)
\end{equation*}
of functions which are polynomial on the fibres of $p$, where a symmetric
tensor $Q$ of valence $(k,0)$, i.e.~a section of $S^kTM$, is identified with
the function $\alpha\mapsto Q(\alpha,\ldots,\alpha)$ on $T^*M$ (which is
homogeneous of degree $k$ on each fibre of $p$). The induced bracket
\begin{equation*}
\{\cdot,\cdot\}\colon C^\infty(M,S^jTM)\times C^\infty(M,S^k TM)
\to C^\infty(S^{j+k-1}TM)
\end{equation*}
on symmetric multivectors is sometimes called the (symmetric)
\emph{Schouten--Nijenhuis} bracket. It may be computed using any torsion-free
connection $\nabla$ on $TM$ as
\begin{equation}\label{Schouten_Nijenhuis}
\{Q,R\}^{\alpha\cdots\epsilon}= 
j\,Q^{\zeta(\alpha\cdots\beta}\nabla_\zeta R^{\gamma\delta\cdots\epsilon)} 
- k\,R^{\zeta(\delta\cdots\epsilon}\nabla_\zeta Q^{\alpha\cdots\beta\gamma)}.
\end{equation}
When $j=1$ and $Q$ is a vector field, $\{Q,R\}$ is just the Lie derivative
$\cL_Q R$.

Now suppose $g$ is a \bps/Riemannian metric on $M$. Then the inverse metric
$g^{\alpha\beta}$ induces a function on $T^*M$ which is quadratic on each
fibre. The flow of the corresponding Hamiltonian vector field on $T^*M$ is the
image of the geodesic flow on $TM$ under the vector bundle isomorphism $TM\to
T^*M$ defined by $g$.

\begin{defin} A smooth function $I\colon TM\to \R$ on a
\bps/Riemannian manifold $(M,g)$ is called an \emph{integral of the geodesic
  flow} (or an \emph{integral}) of $g$, if for any affinely parametrised
geodesic $\gamma$, the function $s\mapsto I(\dot\gamma(s))$ is constant.
\end{defin}
The interpretation of the geodesic flow as a Hamiltonian flow on $T^*M$ allows
us to describe integrals as functions on $T^*M$.
\begin{prop} $Q\colon T^*M\to \R$ defines an integral $I$ of the geodesic flow
of $g$ if and only if it is a conserved quantity for $g^{\alpha\beta}$
i.e.~has vanishing Poisson bracket with $g^{\alpha\beta}$.
\end{prop}
We shall only consider integrals defined by $Q\in
C^\infty_{\mathrm{pol}}(T^*M,\R)$.  Without loss of generality, we may assume
such an integral is homogeneous, hence given by a symmetric tensor
$Q^{\alpha\cdots\gamma}\in C^\infty(M,S^kTM)$. Using the Levi-Civita
connection of $g$ to compute the Schouten--Nijenhuis bracket, we obtain
\begin{equation*}
\{g,Q\}^{\beta\gamma\cdots\epsilon}= 2\, g^{\alpha(\beta}\nabla_\alpha
Q^{\gamma\cdots\epsilon)},
\end{equation*}
which is obtained from $\nabla_{(\alpha}Q_{\gamma\cdots\epsilon)}$ by raising
all indices (using $g$) and multiplying by $2$. When $k=1$, $\{g,Q\}=0$ if and
only if $Q^\alpha$ is a Killing vector field. Thus we recover Clairaut's
Theorem, that Killing vector fields define integrals of the geodesic flow.
More generally, a \emph{symmetric Killing tensor} of valence $(0,\ell)$ on a
\bps/Riemannian manifold $(M, g)$ is a tensor $H_{\alpha\beta\cdots \delta}\in
S^\ell T^*M$ that satisfies
\begin{equation}\label{Killingtensor}
\nabla_{(\alpha} H_{\beta\gamma\cdots\epsilon)}=0, 
\end{equation}
where $\ell\geq 1$ can be any integer and $\nabla$ is the Levi-Civita
connection of $g$.
\begin{cor} $Q^{\alpha\cdots\gamma}\in C^\infty(M,S^kTM)$ defines an integral
of the geodesic flow of $g$ if and only if $Q_{\alpha\cdots\gamma}$ is a
symmetric Killing tensor of $g$.
\end{cor}

\subsection{Holomorphic Killing fields}\label{KillingFields}

Let $(M,J,g)$ be a \bps/K\"ahler manifold with Levi-Civita connection $\nabla$
and K\"ahler form $\Kf_{\alpha\beta}=J_{\alpha}{}^\gamma g_{\gamma\beta}$.
\begin{defin}
A vector field $X$ on $(M, J, g)$ is called a \emph{holomorphic Killing
  field} if it preserves the complex structure $J$ and the metric $g$,
i.e.~$\cL_X J=0$ and $\cL_X g=0$.
\end{defin}

\noindent In terms of the Levi-Civita connection $\nabla$ the defining
properties of a holomorphic Killing field can be rewritten as:
\begin{equation}\label{holomorphic_Killing}
\nabla_\alpha X^\beta=-J_{\alpha}{}^\gamma 
J_{\delta}{}^\beta\nabla_{\gamma} X^{\delta}
\quad \text{ and }\quad
\nabla_\alpha X_\beta+\nabla_\beta X_\alpha=0.
\end{equation} 
It follows immediately from the definition of a holomorphic Killing field $X$
that $X$ also preserves the K\"ahler form, which means that $\cL_X\Kf=
d(i_X\Kf)=0$ or equivalently
\begin{equation}\label{symplectic_field}
\nabla_\alpha(\Kf_{\gamma\beta}X^\gamma)
-\nabla_\beta(\Kf_{\gamma\alpha}X^\gamma)=0.
\end{equation}
In particular, this equation is satisfied if there exists a smooth function
$f\colon M\to \R$ such that $-i_X\Kf=d f$, i.e.~$\Kf_{\alpha\gamma}X^\gamma
=\nabla_\alpha f$, or, using the Poisson structure $\Kf^{\alpha\beta}$,
\begin{equation}\label{Killingpotential}
X^{\beta}=\Kf^{\alpha\beta}\nabla_\alpha f= J_{\alpha}{}^\beta\nabla^\alpha f,
\end{equation}
in which case $X$ is said to be the \emph{symplectic gradient} of $f$.

\begin{prop}\label{PoissonCommute} If $X$ and $Y$ are symplectic gradients
of functions $f$ and $h$, then $\cL_X h=0$ if and only if $\cL_Y f =0$ if and
only if $\Kf^{\alpha\beta}(\nabla_\alpha f)(\nabla_\beta h)=0$ if and only if
$\Kf_{\alpha\beta} X^\alpha Y^\beta=0$. These equivalent conditions imply that
$X$ and $Y$ commute\textup: $[X,Y]=0$.
\end{prop}
\begin{proof} $i_X dh=-i_X(i_Y\Kf)=i_Y(i_X\Kf)=-i_Y df$ and so the
equivalences are trivial. Now $\cL_X h = 0$ implies $0=\cL_X dh=-\cL_X
(i_Y\Kf)=-i_{[X,Y]}\Kf$, since $\cL_X\Kf=0$. Hence $[X,Y]=0$, since
$\Kf$ is nondegenerate.
\end{proof}
In this situation, $X$ and $Y$ have \emph{isotropic} span with respect to
$\Kf$, and they are said to \emph{Poisson commute}, since $f$ and $h$ have
vanishing Poisson bracket.

We now return to holomorphic Killing fields.
\begin{prop}\label{KillingPot} Let $f\colon M\to \R$ be a smooth function.
Then the symplectic gradient $X^{\beta}=\Kf^{\alpha\beta}\nabla_\alpha f$ is a
holomorphic Killing field if and only if the Hessian $\nabla^2f$ is
$J$-invariant, i.e.
\begin{equation}\label{holom_Killing3}
\nabla_a\nabla_bf=0=\nabla_{\bar a}\nabla_{\bar b}f.
\end{equation}
\end{prop}
\begin{proof} Since any two equations of
\eqref{holomorphic_Killing} and \eqref{symplectic_field} imply the third, we
deduce that a vector field of the form $X^{\beta}
=\Kf^{\alpha\beta}\nabla_\alpha f$ is a
holomorphic Killing field if and only if
\begin{equation}\label{holom_Killing1}
\nabla_\alpha J_{\beta}{}^\gamma\nabla_\gamma f+
\nabla_\beta J_{\alpha}{}^\gamma\nabla_\gamma f=0
\end{equation}
or equivalently
\begin{equation}\label{holom_Killing2}
\nabla_\alpha\nabla_\beta f
=J_\alpha{}^\gamma J_{\beta}{}^\delta \nabla_\gamma\nabla_\delta f,
\end{equation}
which is equivalent to~\eqref{holom_Killing3}.
\end{proof}
We call $f$ in this case a \emph{Killing potential} or a \emph{Hamiltonian}
for the holomorphic Killing field $X$. Note that a holomorphic Killing field
always admits such a potential locally (and on any open subset $U$ with
$H^1(U,\R)=0$).

Suppose now that $g$ is a compatible \bps/K\"ahler metric on a c-projective
manifold $(M, J, [\nabla])$. Then we may write any real section
$\sigma\in\Gamma(\cE(1,1))$ as $\sigma = h \scale_g$ for some function $h\colon
M\to\R$, where $\scale_g$ is the trivialisation of $\cE(1,1)$ determined by
$g$.

\begin{prop}\label{HessKillingPot} Let $(M, J, [\nabla])$ be a c-projective
manifold and $h\in C^\infty(M,\R)$.
\begin{enumerate}
\item If $\scale_g$ is the \textup(real\textup) trivialisation of $\cE(1,1)$
corresponding to a compatible metric~$g$, then $\sigma=h\scale_g$ is in the
kernel of the c-projective Hessian $D^\cW\sigma=0$ if and only if
$h$ is a Killing potential with respect to $(g,J)$.

\item If $g$ and $\tilde g$ are compatible metrics whose corresponding
trivialisations of $\cE(1,1)$ are related by $\scale_{\tilde g}
=e^{-f}\scale_g$, then $h$ is a Killing potential with respect to $(g,J)$ if
and only if $e^fh$ is a Killing potential with respect to $(\tilde g,J)$.
\end{enumerate}
\end{prop}
\begin{proof} For the first part, compute $D^\cW\sigma$ using the
Levi-Civita connection $\nabla^g$. Since $\scale_g$ is parallel, and the Ricci
tensor of $g$ is $J$-invariant, $D^\cW\sigma=0$ if and only if the
$J$-invariant part of the Hessian of $h$ is zero, and
Proposition~\ref{KillingPot} applies. The second part follows from the first.
\end{proof}

These observations may be generalised to (possible degenerate) solutions $\ms$
of the metrisability equation. Given any $J$-invariant section
$\ms^{\alpha\beta}$ of $S^2 TM\otimes\cE_\R(-1,-1)$ and any section $\sigma$ of
$\cE_\R(1,1)$, we define vector fields $\hv(\ms,\sigma)$ and $K(\ms,\sigma)$ by
\begin{align}
\hv^\gamma(\ms,\sigma)&=\ms^{\alpha\gamma}\nabla_{\alpha}\sigma
-\tfrac{1}{n}\sigma\nabla_{\alpha}\ms^{\alpha\gamma}\\
K^\beta(\ms,\sigma)=J_{\gamma}{}^\beta \hv^\gamma(\ms,\sigma)&=
\Phi^{\alpha\beta}\nabla_{\alpha}\sigma
-\tfrac{1}{n}\sigma\nabla_{\alpha}\Phi^{\alpha\beta},
\end{align}
where $\Phi^{\alpha\beta}=J_\gamma{}^\beta\ms^{\alpha\gamma}$.
\begin{prop}\label{HolKillingPairing} $\hv(\ms,\sigma)$ and $K(\ms,\sigma)$
are c-projectively invariant, and if $\ms$ is a nondegenerate solution of the
metrisability equation corresponding to a metric $g$ and $\sigma=h\det\ms$ is
in the kernel of the c-projective Hessian, then $\hv(\ms,\sigma)$ is
holomorphic, and $K(\ms,\sigma)$ is the holomorphic Killing field of $g$ with
Killing potential $h$.
\end{prop}
\begin{proof} For a c-projectively equivalent connection
$\hat\nabla\in[\nabla]$, we have
\begin{equation*}
\ms^{\alpha\gamma}\hat\nabla_{\alpha}\sigma
-\tfrac{1}{n}\sigma\hat\nabla_{\alpha}\ms^{\alpha\gamma}
=\ms^{\alpha\gamma}\nabla_{\alpha}\sigma+\ms^{\alpha\gamma}\Upsilon_{\alpha}\sigma
-\tfrac{1}{n}\sigma\nabla_{\alpha}\ms^{\alpha\gamma}
-\ms^{\alpha\gamma}\Upsilon_{\alpha}\sigma
\end{equation*}
and the $\Upsilon$ terms cancel, showing that $\hv(\ms,\sigma)$---and hence
also $K(\ms,\sigma)$---is independent of the choice of $\nabla\in[\nabla]$.

Now if $\ms$ is nondegenerate, corresponding to a compatible metric $g$
with $\scale_g=\det\ms$, we use $\nabla^g$ to compute
\begin{equation*}
K^\beta(\ms,\sigma)=\Phi^{\alpha\beta}\nabla^g_{\alpha}(h\scale_g)
=\Kf^{\alpha\beta}\nabla_\alpha h,
\end{equation*}
which is the holomorphic Killing field associated to $h$.
\end{proof}

\begin{rem} Suppose that $(M, J, [\nabla])$ is an almost c-projective manifold
and consider the tensor product
\[
\cV_\C\otimes\cW_\C=\T^*\otimes\overline\T^*\otimes\T\otimes\overline\T.
\]
Since $\T^*\otimes \T=\A M\oplus\cE(0,0)$, there is a natural projection
\begin{equation}
\Pi\colon\cV_\C\otimes\cW_\C\to\!\! 
\begin{matrix}
\A M\oplus \cE(0,0)\\
\bigoplus\\
\overline{\A M}\oplus \cE(0,0)
\end{matrix}
\end{equation}
or equivalently a natural projection
\begin{equation}
\Pi\colon \cV\otimes\cW\to
\A M\oplus\cE(0,0).
\end{equation}
Hence, the results in~\cite{CD} imply that there are two invariant bilinear
differential operators
\begin{gather}
\hv\colon T^{0,1}M\otimes T^{1,0}M(-1,-1)\times \cE(1,1)
\to T^{1,0}M\oplus T^{0,1}M\\
c\colon T^{0,1}M\otimes T^{1,0}M(-1,-1)\times \cE(1,1)
\to \cE(0,0)\oplus \cE(0,0),
\end{gather}
which are constructed as follows. Consider the two differential operators
$L\colon T^{0,1}M\otimes T^{1,0}M(-1,-1))\to \cV_\C$ and $L\colon
\cE(1,1))\to\cW_\C$ from Theorem~\ref{MetriProlongation} respectively
\ref{HessianProlongation}.  Recall that in terms of a connection
$\nabla\in[\nabla]$ they can be written as
\begin{equation*}
L(\ms^{\bar b c})=
\begin{pmatrix}\ms^{\bar b c}\\-\frac{1}{n}\nabla_{\bar a}\ms^{\bar a b}
\enskip|\enskip{-\frac{1}{n}}\nabla_a\ms^{\bar b a}\\ 
\frac{1}{n}(\nabla_{\bar a}\nabla_b\ms^{\bar a b}
+\Rho_{\bar a b}\ms^{\bar a b})
\end{pmatrix}\quad\quad
L(\sigma)=\begin{pmatrix} \sigma \\ \nabla_a\sigma \enskip|
\enskip \nabla_{\bar{a}}\sigma\\ \nabla_a\nabla_{\bar b}\sigma
+\Rho_{a\bar b}\sigma
\end{pmatrix}.
\end{equation*}
Then $\hv$ and $c$ are defined as the projections to $T_\C M=\A_\C M/
\A^0_\C M$ respectively to $\cE(0,0)\oplus \cE(0,0)$ of
$\Pi(L(\ms^{\bar b c})\otimes L(\sigma))$.

In particular, for a choice of connection $\nabla\in[\nabla]$, the invariant
differential operator $\hv$ is given by
\begin{equation}\label{diffpairing}
\hv(\ms,\sigma) =\bigl(
\ms^{\bar b c}\nabla_{\bar b}\sigma
-\tfrac{1}{n}\sigma\nabla_{\bar b}\ms^{\bar b c}
\enskip\ | \enskip
\ms^{\bar b c}\nabla_{c}\sigma-\tfrac{1}{n}\sigma\nabla_{c}\ms^{\bar b c}\bigr).
\end{equation}
Note that if $\ms^{\bar b c}$ and $\sigma$ are real sections, the two
components of (\ref{diffpairing}) are conjugate to each other. In this case
we may identify $\ms^{\bar b c}$ with a $J$-invariant section
$\ms^{\beta\gamma}$ of $S^2TM$.
\end{rem}

\subsection{Hermitian symmetric Killing tensors}

Suppose $(M, J)$ is an almost complex manifold and $k\geq 1$.  Then we call a
symmetric tensor $H_{\alpha\beta\cdots\epsilon}\in \Gamma(S^{2k}T^*M)$
\emph{Hermitian}, if it satisfies
\begin{equation}\label{symmHermitian}
J_{(\alpha}{}^{\beta} H_{\beta\gamma\cdots \epsilon)}=0.
\end{equation}
Since, by definition,
$H_{\beta\gamma\cdots\epsilon}=H_{(\beta\gamma\cdots \epsilon)}$,
equation (\ref{symmHermitian}) is equivalent to
\begin{equation}\label{symmHermitian2}
J_{\alpha}{}^\beta H_{\beta\gamma\cdots \epsilon}+ 
J_{\gamma}^{\beta} H_{\alpha\beta\cdots
\epsilon}+\cdots +J_{\epsilon}^{\beta} H_{\alpha \gamma\cdots\beta}=0.
\end{equation}
Viewing a symmetric tensor $H$ of valence $(0,2k)$ as an element in
$S^{2k}T^*M\otimes\C=S^{2k}\Wedge^1$ via complexification, we can use the
projectors from Section \ref{almostcomplexmanifolds} to decompose $H$ into
components according to the decomposition of $S^{2k}\Wedge^1$ into irreducible
vector bundles:
\begin{equation}\label{S2k}
S^{2k} \Wedge^{1}
=\bigoplus _{j=0}^{2k} S^{2k-j}\Wedge^{1,0}\otimes S^{j}\Wedge^{0,1}.
\end{equation}
Since this decomposition is in particular invariant under the action of $J$,
all the components of a tensor $H_{\alpha\beta\cdots\gamma\delta}\in
S^{2k}\Wedge^1$ that satisfies (\ref{symmHermitian2}) must independently
satisfy (\ref{symmHermitian2}). If $H_{ab\cdots d\,\bar e\bar f\cdots \bar h}$
is a section of $S^{2k-j}\Wedge^{1,0}\otimes S^{j}\Wedge^{0,1}$ that
satisfies (\ref{symmHermitian2}), then this equation says that
$2(k-j)iH=0$, which implies that $H\equiv 0$ unless $j=k$.  We conclude
that Hermitian symmetric tensors of valence $(0,2k)$ can be viewed as real
sections of the vector bundle
\[
S^k\Wedge^{1,0}\otimes S^k\Wedge^{0,1},
\]
which is the complexification of the vector bundle that consists of those
elements in $S^{2k}T^*M$ that satisfy (\ref{symmHermitian}).
\begin{rem}
Note that, if $H$ is a symmetric tensor of valence $(0,2k+1)$ satisfying
(\ref{symmHermitian}), then the above reasoning immediately implies that
$H\equiv0$. The same arguments apply, mutatis mutandis, to symmetric tensors
$Q^{\alpha\beta\cdots\epsilon}$ of valence $(2k,0)$, and to weighted tensors
of valence $(0,2k)$ and $(2k,0)$.
\end{rem}

Suppose now $(M,J, g)$ is a \bps/K\"ahler manifold and
$\ell=2k$ is even, then we can restrict equation~\eqref{Killingtensor} for
symmetric Killing tensors of valence $(0,2k)$ to Hermitian tensors.  If we
complexify~\eqref{Killingtensor}, we obtain the following system of
differential equations on tensors $H\in \Gamma(S^k\Wedge^{1,0}\otimes
S^k\Wedge^{0,1})$:
\begin{equation}\label{HermitianKilling}
\nabla_{(a} H_{bc\cdots d) \bar e\bar f\cdots\bar h}=0 \qquad\text{and}\qquad
\nabla_{(\bar a} H_{|bc\cdots d| \bar e\bar f\cdots \bar g)}=0,
\end{equation}
where $|\cdots|$ means that one does not symmetrise over these indices. Real
solutions of (\ref{HermitianKilling}) thereby correspond to Hermitian symmetric
Killing tensors of valence $(0,2k)$ and obviously for real solutions
the two equations of (\ref{HermitianKilling}) are conjugates of each other.
The following proposition shows that (suitably interpreted) the Killing
equation for Hermitian symmetric tensors of valence $(0,2k)$ is c-projectively
invariant.

\begin{prop}\label{HermitianKillingInvariance}
Suppose $(M, J, [\nabla])$ is an almost c-projective manifold of dimension
$2n\geq 4$.  If $H_{ab\cdots d\,\bar e\bar f\cdots\bar h}\in
\Gamma(S^k\Wedge^{1,0}\otimes S^k\Wedge^{0,1}(2k,2k))$ satisfies
\begin{equation}\label{HermitianKilling2}
\nabla_{(a} H_{bc\cdots d) \bar e\bar f\cdots\bar h}=0 \quad\quad
\nabla_{(\bar a} H_{|bc\cdots d| \bar e\bar f\cdots\bar h)}=0,
\end{equation}
for some connection in $\nabla\in[\nabla]$, then it does so for any other
connection in the c-projective class.
 \end{prop}
\begin{proof}
Suppose $H\in\Gamma(S^k\Wedge^{1,0}\otimes S^k\Wedge^{0,1}(2k,2k))$ satisfies
(\ref{HermitianKilling2}) for some connection $\nabla\in[\nabla]$ and let
$\hat\nabla\in[\nabla]$ be another connection in the c-projective class. Then
it follows from Proposition \ref{changeform} and Corollary \ref{changedensity}
that
\begin{align*}
\hat\nabla_aH_{b\cdots d\bar e\cdots \bar h}
&=\nabla_aH_{b\cdots d\bar e\cdots \bar h}
-k\Upsilon_aH_{b\cdots d\bar e\cdots\bar h}
-\Upsilon_bH_{a\cdots d\bar e\cdots\bar h}
-\cdots-\Upsilon_d H_{b\cdots a\bar e\cdots\bar h}\\
&\qquad+2k\Upsilon_aH_{b\cdots d\bar e\cdots\bar h}\\
&=\nabla_a H_{b\cdots d\bar e\cdots\bar h}
+k \Upsilon_aH_{b\cdots d\bar e\cdots\bar h}
-\Upsilon_bH_{a\cdots d\bar e\cdots\bar h}-\cdots
-\Upsilon_dH_{b\cdots a\bar e\cdots\bar h}.
\end{align*}
Since $\nabla_{(a} H_{b\cdots d) \bar e\cdots \bar h}=0$ by assumption and 
\begin{equation*}
\Upsilon_{(a}H_{b\cdots d)\bar e\cdots \bar h}
=\tfrac{1}{k+1}(\Upsilon_a H_{b\cdots d\bar e\cdots\bar h}
+\Upsilon_b H_{a\cdots d\bar e\cdots\bar h}+\cdots
+\Upsilon_d H_{b\cdots a\bar e\cdots\bar h}),
\end{equation*}
we conclude that the symmetrisation over the unbarred indices on the right
hand side is zero, which proves that the first equation of
(\ref{HermitianKilling2}) is independent of the connection. Analogous
reasoning shows that this is also true for the second equation of
(\ref{HermitianKilling2}).
\end{proof}

We refer to solutions of the c-projectively invariant equation
\eqref{HermitianKilling2} as \emph{c-projective Hermitian symmetric Killing
  tensors of valence $(0,2k)$}.

\begin{cor}\label{MetricKillingTensor} Suppose $(M,J,[\nabla])$ is a metrisable
c-projective manifold with compatible \bps/K\"ahler metric $g$. Then a real
section $H\in\Gamma(S^k\Wedge^{1,0}\otimes S^k\Wedge^{0,1})$ is a Hermitian
symmetric Killing tensor of $g$ (i.e.~a solution of~\eqref{HermitianKilling}
with respect to $\nabla^g$) if and only if $\scale_g^{\,2k}H$ is a
c-projective Hermitian symmetric Killing tensor.  In particular, in this case,
if $\tilde g$ is another compatible \bps/K\"ahler metric, then $e^{2kf}H$ is a
Hermitian symmetric Killing tensor of $\tilde g$, where $f$ is given by
$\scale_{\tilde g}= e^{-f} \scale_g$.
\end{cor}

The differential equation~\eqref{HermitianKilling2} gives rise to a
c-projectively invariant operator, which is the first BGG operator
\begin{equation}\label{HK-BGGoperator}\begin{array}{lll}
&\weight{\!\!-2}0 \pweight {k{+}1}000 k000\\[-16pt] 
\weight 00 \pweight k000 k000
\raisebox{14pt}{$\begin{array}c
\nearrow\\ \searrow \end{array}$}\\[-16pt]
&\weight 0{\!\!-2} \pweight k000 {k{+}1}000
\end{array}\end{equation}
corresponding to the tractor bundle $\cW$, where $\cW$ is the Cartan product
of $k$ copies of $\Wedge^2\T^*$ and $k$ copies of $\Wedge^2\overline{\T}{}^*$.
As for the BGG operators discussed in previous sections, this implies
(see~\cite{BCEG,HSSS,neusser}), that there is a linear connection on $\cW$
whose parallel sections are in bijection to solution
of~\eqref{HermitianKilling2}.  Hence, the dimension of the solution space is
bounded by the rank of $\cW$.

\begin{prop} Suppose $(M, J, g)$ is a \bps/K\"ahler manifold of
dimension $2n\geq 4$ and let $k\geq 1$ be an integer.  Then the space of
Hermitian symmetric Killing tensors of valence $(0,2k)$ of $(M, J, g)$ has
dimension at most
\begin{equation}
\left(\frac{(k+1)(k+2)^2\cdots(k+(n-1))^2(k+n)}{(n-1)!n!}\right)^2.
\end{equation}
\end{prop}

In the sequel, we shall be interested in the case $k=1$, where any compatible
metric $g$ defines a c-projective Hermitian symmetric Killing tensor $H_{b\bar
  a}= \scale_g^{\,2} g_{b\bar a}$ by Corollary~\ref{MetricKillingTensor}.
This has the following c-projectively invariant formulation.

\begin{prop}\label{HKTensors} Let $(M,J,[\nabla])$ be an almost c-projective
manifold and let $\ms^{\bar a b}$ be a real section of $T^{0,1}M\otimes
T^{1,0}M(-1,-1)$. Then
\begin{equation}\label{eta-->H}
H_{b\bar a} := \tfrac1{(n-1)!}\,\bar\epst_{\bar a\bar c\cdots \bar e}\,
\epst_{bd\cdots f}\,\ms^{\bar c d}\cdots \ms^{\bar e f}
\end{equation}
is a real section of $\Wedge^{1,0}\otimes\Wedge^{0,1}(2,2)$ with $H_{b\bar a}
\ms^{\bar ac}=\sigma\delta_b{}^c$, where $\sigma=\det\ms$. If $\ms^{\bar ab}$
satisfies~\eqref{metriequ1c} for some $X^d,Y^{\bar c}$ \textup(depending on
$\nabla$\textup) then $H_{b\bar a}$ is a c-projective Hermitian symmetric
Killing tensor and
\begin{equation}\label{eq:H-deriv}
\ms^{\bar cd}\nabla_d H_{b\bar a} = Y^{\bar e}H_{b\bar e}\delta_{\bar a}{}^{\bar c}-
Y^{\bar c} H_{b\bar a}, \qquad
\ms^{\bar c d}\nabla_{\bar c} H_{b\bar a} = X^e H_{e \bar a}\delta_b{}^d
- X^d H_{b\bar a}.
\end{equation}
\end{prop}
\begin{proof} The first statement is straightforward.  For the rest, suppose
first that $\ms^{\bar a b}$ is nondegenerate, hence parallel with respect to
some connection $\hat\nabla$ in $[\nabla]$, related to $\nabla$ by $\Upsilon$
with $\Upsilon_b=H_{b\bar a}Y^{\bar a}$ and $\Upsilon_{\bar a}=H_{b\bar a}
X^b$.  Then $H_{b\bar a}$ is parallel with respect to $\hat\nabla$, hence a
Hermitian symmetric Killing tensor, and equation~\eqref{eq:H-deriv} follows by
rewriting this condition in terms of $\nabla$.  At each point, these are
statements about the $1$-jet of $H$, which depends polynomially on the $1$-jet
of $\ms$.  They hold when the $0$-jet of $\ms$ is invertible (at a given
point, hence in a neighbourhood of that point), hence in general by
continuity.
\end{proof}

\subsection{Metrisability pencils, Killing fields and Killing tensors}
\label{Canonical_Killing}

Suppose we have two (real) linearly independent solutions $\ms^{\bar a b}$ and
$\sms^{\bar a b}$ of the metrisability equation~\eqref{metriequ1}.  Since the
metrisability equation is linear, the one parameter family
\begin{equation}\label{pencil}
\sms^{\bar a b}(t) := \sms^{\bar a b}-t\ms^{\bar a b}
\end{equation}
also satisfies~\eqref{metriequ1}, and we refer to such a family as a
\emph{pencil} of solutions of the metrisability equation, or
\emph{metrisability pencil} for short.

By Proposition~\ref{det_of_metrisability}, the determinant
\begin{equation}\label{det_pencil}
\tilde\sigma(t):=\det\sms(t)
\end{equation}
of the pencil~\eqref{pencil} lies in the kernel of the c-projective Hessian
for all $t\in\R$ (as does $\sigma:=\det\ms$).  If $\sms(t)$ is degenerate for
all $t$, then $\tilde\sigma(t)$ is identically zero.  Otherwise, we may
assume, at least locally:
\begin{cond}\label{Nondegen} $\ms$ is nondegenerate, i.e.~$\sigma=\det\ms$
is nonvanishing, and hence $g^{\alpha\beta}=(\det\ms)\ms^{\alpha\beta}$ is
inverse to a compatible metric $g$.
\end{cond}
Assuming Condition~\ref{Nondegen}, we may write $\sms^{\bar a c}=\ms^{\bar a
  b} A_b{}^c$ as in Section~\ref{sec:met-mob}, where the $(g,J)$-Hermitian
metric $A$ satisfies~\eqref{metri-mob}. Setting
$A_a{}^b(t):=A_a{}^b-t\delta_a{}^b$, we have
\begin{equation*}
\sms^{\bar a c}(t)=\ms^{\bar a b} A_b{}^c(t)
\quad\text{and}\quad \tilde\sigma(t)=(\det\ms)(\det A(t)).
\end{equation*}
Thus $\tilde\sigma(t)$ is essentially the characteristic polynomial $\det
A(t)$ of $A_a{}^b$, regarded as a complex linear endomorphism of the complex
bundle~$T^{1,0}M$.

\begin{rem}\label{rem:pencil} A pencil is another name for a projective line:
if we make a projective change $s=(at+b)/(ct+d)$ of parameter (with $ad-bc\neq
0$) then the pencil may be rewritten, up to overall scale, as $a\sms+b\ms - s
(c\sms+d\ms) = (\sms - t \ms) (ad-bc)/(ct+d)$. Assuming that $c\sms+d\ms$ is
nondegenerate, the rescaled and reparameterized pencil is thus
$(c\sms+d\ms)(\tilde A - s \Id)$, where $\tilde A= (c A +
d\Id)^{-1}(a A+ b\Id)$.
\end{rem}

We next set $H_{b\bar a}:=\frac{1}{(n-1)!}\,\bar\epst_{\bar a\bar c\cdots\bar e}
\, \epst_{bd\cdots f}\, \ms^{\bar c d}\cdots \ms^{\bar e f}$ as in~\eqref{eta-->H}
and introduce
\begin{equation}\label{adj_pencil}
\wt H_{b\bar a}(t) := \tfrac{1}{(n-1)!}\,\bar\epst_{\bar a\bar c\cdots \bar e}\,
\epst_{bd\cdots f}\, \sms^{\bar c d}(t)\cdots \sms^{\bar e f}(t)= 
(\mathrm{adj}\, A(t))_b{}^c H_{c\bar a},
\end{equation}
where $\mathrm{adj}\, B$ denotes the endomorphism adjugate to $B$, with
$B\,\mathrm{adj}\, B=(\det B) I$.

Proposition~\ref{HKTensors} implies that for all $t\in\R$, $\wt H_{b \bar
  a}(t)$ is a c-projective Hermitian symmetric Killing tensor of
$(M,J,[\nabla])$. Hence for any $s\in\R$ with $\sms(s)$ nondegenerate,
$\tilde\sigma(s)^{-2}\wt H_{b\bar a}(t)$ defines a family of Hermitian
symmetric Killing tensors for the corresponding metric.

Similarly, Proposition~\ref{HolKillingPairing} implies that if $\sms(s)$ is
nondegenerate (for $s\in\R$), then for all $t\in\R$,
$K(\sms(s),\tilde\sigma(t))$ is a holomorphic Killing field with respect to
the corresponding metric (hence an \emph{inessential} c-projective
vector field). Now observe that
\begin{equation*}
K(\sms(s),\tilde\sigma(t))=K(\sms(t)+(t-s)\ms,\tilde\sigma(t))
=(t-s)K(\ms,\tilde\sigma(t)),
\end{equation*}
since $K$ is bilinear and $K(\sms(t),\tilde\sigma(t))=0$.  By continuity,
the vector fields
\begin{equation}\label{def_canonical_Killing}
\wt K(t):= K(\ms,\tilde\sigma(t)), \qquad\text{i.e.}\quad
\wt K^{\beta}(t)=\Kf^{\alpha\beta}\nabla_\alpha \det A(t),
\end{equation}
which are holomorphic Killing fields with respect to $g$, preserve $\sms(s)$
for all $s,t\in\R$, i.e.~$\cL_{\wt K(t)}\sms(s)=0$, and hence also
$\cL_{\wt K(t)}\wt H(s)=0=\cL_{\wt K(t)}\tilde\sigma(s)$. Thus
$\wt K(t)$ preserves the Killing potential $\det A(s)$ of $\wt K(s)$
with respect to $g$, so Proposition~\ref{PoissonCommute} implies that $\wt
K(s)$ and $\wt K(t)$ Poisson-commute. We summarise what we have proven as
follows.

\begin{thm}\label{Commuting_Killing_objects}
Let $(M, J, [\nabla])$ be a c-projective manifold with metrisability
solutions $\ms$ and $\sms$ corresponding to compatible \bps/K\"ahler metric
metrics $g$ and $\tilde g$ that are not homothetic. Let $\sms(t)$ be the
corresponding metrisability pencil~\eqref{pencil}.
\begin{enumerate}
\item The vector fields $\wt K(t):t\in\R$ defined
  by~\eqref{det_pencil}--\eqref{def_canonical_Killing} are Poisson-commuting
  holomorphic Killing fields with respect to $g$ and $\tilde g$.
\item The tensors $\wt H(t):t\in\R$ defined
  by~\eqref{det_pencil}--\eqref{adj_pencil} are c-projective Hermitian
  symmetric Killing tensors, invariant with respect to $\wt K(s)$ for any
  $s\in\R$. In particular, by Corollary \ref{MetricKillingTensor}, they induce
  Hermitian symmetric Killing tensors of $g$ respectively $\tilde g$ (by
  tensoring with $\tau_g^{-2}$ respectively $\tau_{\tilde g}^{-2}$).
\end{enumerate}
\end{thm}
We call the vector fields $\wt K(t)$ and tensor densities $\wt H(t)$ the
\emph{canonical Killing fields} and \emph{canonical Killing tensors}
(respectively) for the pair $(g,\tilde g)$; the former are Killing vector
fields with respect to any nondegenerate metric in the family~\eqref{pencil},
and the latter give rise to symmetric Killing tensor fields for any such
metric by tensoring with the corresponding trivialisation of $\mathcal
E(-2,-2)$.

Since the canonical Killing fields $\wt K(t)$
are holomorphic with $[\wt K(t),\wt K(s)]=0$ for all $s,t\in\R$, we also
have $[\wt K(t),J \wt K(s)]=0$. Since $J$ is integrable, $J\wt K(t)$
are also holomorphic vector fields, and $[J \wt K(t), J \wt K(s)]=0$ for
all $s,t\in\R$.

The fact that for all $t\in \R$, $\wt K(t)$ is a holomorphic Killing field
means equivalently (by linearity) that the coefficients of $\wt K(t)$ are
holomorphic Killing fields, whose Killing potentials with respect $g$ are the
coefficients of the characteristic polynomial $\det A(t)$. Up to scale, the
nontrivial coefficients of $\det A(t)$ can be written
\begin{equation}\label{canonicalpotentials}
\tilde\sigma_1:= A_b{}^b,\quad \tilde\sigma_2:= A_b{}^{[b}A_c{}^{c]},\quad
\ldots\quad \tilde\sigma_n:=  A_b{}^{[b}A_c{}^c\cdots A_d{}^{d]},
\end{equation}
which are real-valued because $A$ is Hermitian with respect to $g$. Raising
an index in~\eqref{metri-mob}, we have
\begin{equation}\label{metri-mob-vec}
\nabla^{\bar c} A_a{}^b = -g^{\bar cb} \hv_a,\quad\text{or equivalently,}\quad 
\nabla^c A_a{}^b = -\delta_a{}^c \hv^b.
\end{equation}
Hence, applying the $(1,0)$-gradient operator $\nabla^a=\Pi^a_\alpha
g^{\alpha\beta} \nabla_\beta$ to the canonical potentials
in~\eqref{canonicalpotentials}, we obtain (up to sign) holomorphic vector
fields
\begin{equation}\label{hvfields}
\hv_{(1)}^{\,a}:=\hv^a,\quad \hv_{(2)}^{\,a}:= 2 \hv^{[a}A_b{}^{b]},
\quad\ldots\quad \hv_{(n)}^{\,a}:=n \hv^{[a}A_b{}^bA_c{}^c\cdots A_d{}^{d]}
\end{equation}
whose imaginary parts are (up to scale) the coefficients of $\wt K(t)$.  In
particular, $\Pi^\alpha_a \hv^a= \frac12 (\hv^\alpha-i K^\alpha)$, where the
holomorphic Killing field $K^\alpha=J_{\beta}{}^\alpha \hv^\beta$ is the
leading coefficient of $\wt K(t)$. In general, the coefficients satisfy the
recursive relation
\begin{equation}\label{hvf-rec}
\hv_{{(k+1)}}^{\,a}= A_b{}^a \hv_{(k)}^{\,b}  + \tilde\sigma_1 \hv^a.
\end{equation}
\begin{prop}\label{open_dense_Killing_vector_fields} 
Let $g$ and $\tilde g$ be compatible metrics on $(M,J,[\nabla])$
related by a \textup(real\textup) solution $A$ of~\eqref{metri-mob}. Then
there is an integer $\ell$, with $0\leq\ell\leq n$, such that
$\hv_{(1)}^{\,a},\ldots, \hv_{(\ell)}^{\,a}$ are linearly independent on a dense
open subset of $M$, and $\dim\spann \smash{\wt K}(t)\leq \ell$ on $M$.
\end{prop}
\begin{proof} Suppose for some $p\in M$ and $1\leq k\leq n$, $\hv_{(k)}^{\,a}$
is a linear combination of $\hv_{(1)}^{\,a},\ldots, \hv_{(k-1)}^{\,a}$ at $p$.
Then $A_b{}^a \hv_{(k)}^{\,b}$ is a linear combination of
$A_b{}^a \hv_{(1)}^{\,b},\ldots, A_b{}^a\hv_{(k-1)}^{\,a}$, hence of
$\hv_{(1)}^{\,a},\ldots, \hv_{(k)}^{\,a}$
by~\eqref{hvf-rec}. Applying~\eqref{hvf-rec} once more, we see that
$\hv_{(k+1)}^{\,a}$ is a linear combination of $\hv_{(1)}^{\,a},\ldots,
\hv_{(k)}^{\,a}$.  Hence at each $p\in M$, $\dim\spann \{\hv_{(1)}^{\,a},\ldots,
\hv_{(n)}^{\,a}\}=\dim\spann \wt K(t)$ is the largest
integer $\ell$ such that $\hv_{(1)}^{\,a},\ldots, \hv_{(\ell)}^{\,a}$ are
linearly independent at $p$.  However, for any integer $k$,
$\hv_{(1)}^{\,a},\ldots, \hv_{(k)}^{\,a}$ are linearly dependent if and only if
the holomorphic $k$-vector $\hv_{(1)}^{\,[a}\hv_{(2)}^{\vphantom[\,b}\cdots
\hv_{(k)}^{\,e]}$ is zero. Hence the set where $\hv_{(1)}^{\,a},\ldots,
\hv_{(k)}^{\,a}$ are linearly independent is empty (for $k>\ell$) or dense (for
$k\leq \ell$). The result follows.
\end{proof}
Following~\cite{ACG}, the integer $\ell$ of this Proposition will be called
the \emph{order} of the pencil.

\begin{prop}\label{source_of_things_awesome} Let $g$ and $\tilde g$ be
compatible metrics on $(M,J,[\nabla])$ related by a \textup(real\textup)
solution $A$ of~\eqref{metri-mob}.  Then the endomorphisms $\nabla\hv$ and
$A$ commute, i.e.
\begin{equation}\label{endo_commute}
A_a{}^c\nabla_c\hv^b-A_c{}^b\nabla_a\hv^c=0.
\end{equation}
\end{prop}
\begin{proof} We first give a proof using
Theorem~\ref{Commuting_Killing_objects}, which implies that $K^\alpha$ is a
holomorphic Killing field with respect to both $g$ and $\tilde g$. It follows
that $\cL_K A =0$. However, $\nabla_K A =0$ (since $\hv^{\bar c}\nabla_{\bar
  c} A_a{}^b = -\hv_a \hv^b=\hv^c\nabla_{c} A_a{}^b$) and so $[\nabla K,A]=0$.
Equation~\eqref{endo_commute} is obtained by taking $(1,0)$-parts.

We now give a more direct proof of~\eqref{endo_commute}, starting from the
observation that
\begin{equation*}
\nabla_a \hv_{\bar{b}}=-\nabla_a\nabla_{\bar{b}}A_c{}^c
=-\nabla_{\bar{b}}\nabla_a\bar{A}^{\bar{c}}{}_{\bar{c}}
=\nabla_{\bar{b}}\bar \hv_a,
\end{equation*}
(i.e.~$\nabla_a \hv_{\bar{b}}$ is real). Now expand
$(\nabla_a\nabla_{\bar{b}}-\nabla_{\bar{b}}\nabla_a)A^{\bar{c}d}$ by curvature
and also by using~\eqref{mob-raised} to obtain
\begin{equation}\label{expand_two_ways}
R_{a\bar{b}}{}^{\bar{c}}{}_{\bar{e}}A^{\bar{e}d}
+R_{a\bar{b}}{}^d{}_eA^{\bar{c}e}=\delta_a{}^d\nabla_{\bar{b}}\bar\hv^{\bar{c}}
-\delta_{\bar{b}}{}^{\bar{c}}\nabla_a\hv^d.
\end{equation}
Transvect with $A^{\bar{b}}{}_{\bar{c}}$ to conclude that 
\begin{equation*}
A^{\bar{b}c}R_{a\bar{b}c\bar{e}}A^{\bar{e}}{}_{\bar{d}}
+A^{\bar{b}}{}_{\bar{c}}A^{\bar{c}e}R_{a\bar{b}\bar{d}e}
=g_{a\bar{d}}A^{\bar{b}}{}_{\bar{c}}\nabla_{\bar{b}}\bar\hv^{\bar{c}}
-A^{\bar{b}}{}_{\bar{b}}\nabla_a \hv_{\bar{d}}
\end{equation*}
and hence that 
\begin{equation}\label{another_reality}
A^{\bar{b}c}R_{a\bar{b}c\bar{e}}A^{\bar{e}}{}_{\bar{d}}
-A^{\bar{c}b}R_{b\bar{d}e\bar{c}}A_a{}^e=0
\end{equation}
(i.e.~$A^{\bar{b}c}R_{a\bar{b}c\bar{e}}A^{\bar{e}}{}_{\bar{d}}$ is real). Now
transvect (\ref{expand_two_ways}) with $A_f{}^a\delta_{\bar{c}}{}^{\bar{b}}$ to
conclude that
\[
-A_f{}^a\Ric_{a\bar{e}}A^{\bar{e}}{}_{\bar{d}}+A_f{}^aR_{a\bar{c}\bar{d}e}A^{\bar{c}e}=
A_{f\bar{d}}\nabla_{\bar{c}}\bar\hv^{\bar{c}}-nA_f{}^a\nabla_a \hv_{\bar{d}}
\]
and hence that
\[
A_a{}^bR_{b\bar{c}\bar{d}e}A^{\bar{c}e}-A^{\bar{b}}{}_{\bar{d}}R_{\bar{b}ca\bar{e}}A^{\bar{e}c}=
n(A^{\bar{b}}{}_{\bar{d}}\nabla_{\bar{b}}\bar\hv_a-A_a{}^b\nabla_b \hv_{\bar{d}}).
\]
But from (\ref{another_reality}) the left hand side vanishes whence
\[
A_a{}^b\nabla_b\hv^d=A^{\bar{b}d}\nabla_{\bar{b}}\bar\hv_a
=A^{\bar{b}d}\nabla_a\hv_{\bar{b}}=A_b{}^d\nabla_a\hv^b,
\]
as required.
\end{proof}

Note that~\eqref{endo_commute} can be used to provide an alternative proof
that the holomorphic vector fields~\eqref{hvfields} commute. If $V^a$ and
$W^b$ are two such fields, we must show that
\begin{equation}\label{key_to_commutation}
0=[V,W]^b=V^a\nabla_a W^b-W^a\nabla_a V^b.
\end{equation}
Let us take $V^a=\hv^a$ and $W^b=2\hv^{[b}A_c{}^{c]}$.
Then~\eqref{metri-mob} yields
\begin{align*}
\nabla_a W^b &=(\nabla_a\hv^b)A_c{}^c+\hv^b\nabla_aA_c{}^c
-(\nabla_a\hv^c)A_c{}^b-\hv^c\nabla_aA_c{}^b\\
&=(\nabla_a\hv^b)A_c{}^c-\hv^b\bar\hv_a
-(\nabla_a\hv^c)A_c{}^b+\delta_a{}^b\hv^c\bar\hv_c
\end{align*}
whence $V^a\nabla_a
W^b=\hv^a(\nabla_a\hv^b)A_c{}^c-\hv^a(\nabla_a\hv^c)A_c{}^b$ and
\begin{equation*}
V^a\nabla_a W^b-W^a\nabla_a V^b=\hv^a(A_a{}^c\nabla_c\hv^b-A_c{}^b\nabla_a\hv^c).
\end{equation*}
Similar computations show that all the fields in~\eqref{hvfields} commute.

\begin{rem}
It is interesting to compare Proposition~\ref{source_of_things_awesome} with
what happens in the real projective setting. The mobility
equations~\eqref{metri-mob} are replaced by
\[
\nabla^\alpha A_\beta{}^\gamma=-\delta_\beta{}^\alpha \hv^\gamma-g^{\alpha\gamma}\hv_\beta
\]
and the development runs in parallel. These equations control the existence of
another metric in the projective class other than the assumed background
metric and, from the coefficients of the characteristic polynomial
of~$A_\alpha{}^\beta$, a solution gives rise to $n$ canonically defined
potentials for $n$ canonically defined vector fields. These are counterparts
to the fields (\ref{hvfields}) and, as such, need not be Killing.
Nevertheless, they commute and to see this it is necessary to employ the
alternative reasoning that we encountered near the end of the proof just
given. The key observation, like~(\ref{endo_commute}), is that the
endomorphisms $A_\alpha{}^\beta$ and $\nabla_\alpha \hv^\beta$ commute and its
proof follows exactly the course just given.
\end{rem}

\subsection{Conserved quantities on c-projective manifolds}

On a c-projective manifold $(M,J,[\nabla])$, the construction of an integral
from a compatible metric and Killing tensor has a c-projectively invariant
formulation: in particular, given a nondegenerate solution $\ms^{\alpha\beta}$
of the metrisability equation, and a c-projective Hermitian symmetric Killing
tensor $H_{\gamma\delta}$ of valence $(0,2)$, the Hermitian $(2,0)$-tensor
$\ms^{\alpha\gamma}\ms^{\beta\delta}H_{\gamma\delta}$ defines an integral of
the geodesic flow of the metric corresponding to $\ms$; if $H_{\gamma\delta}$
is associated to $\ms^{\alpha\beta}$ by Proposition~\ref{HKTensors}, then this
integral is the Hamiltonian associated to the inverse metric
$g^{\alpha\beta}=(\det \ms)\, \ms^{\alpha\beta}$.

\begin{defin} Let $(M,J,[\nabla])$ be a c-projective manifold, and let
$\sms^{\alpha\beta}(t):= \sms^{\alpha\beta}-t\ms^{\alpha\beta}$ be a metrisability pencil
satisfying Condition~\ref{Nondegen}, so that $g^{\alpha\beta}=
(\det\ms)\ms^{\alpha\beta}$ is inverse to a (nondegenerate) compatible metric $g$,
and we may write $\sms^{\bar a c}(t)= \ms^{\bar a b} A_b{}^c(t)$. Then the
\emph{linear} and \emph{quadratic integrals} $L_t,I_t\colon TM\to\R$ of (the
geodesic flow of) $g$ are defined by $L_t(X):=g(\wt K(t),X)=
g_{\alpha\beta}\wt K^\alpha(t) X^\beta$ and $I_t(X):= \tau_g^{-2}\wt H(t)(X,X)
= \tau_g^{-2}\wt H_{\alpha\beta}(t) X^\alpha X^\beta$, where $\wt K(t)$ and $\wt H(t)$
are the holomorphic Killing fields and c-projective Hermitian symmetric Killing
tensors associated to $g$ by Theorem~\ref{Commuting_Killing_objects}.
\end{defin}

\begin{prop} \label{Integrals-in-involution} The integrals $I_t,L_t$ of $g$
\textup(for all $t\in\R$\textup) mutually commute under the Poisson bracket
on $TM$ induced by $g$.
\end{prop}
\begin{proof} Since, by Theorem~\ref{Commuting_Killing_objects}, the canonical
Killing fields commute and Lie preserve the canonical Killing tensors, it
remains only to show that for any $s,t\in \R$ the quadratic integrals $I_s$
and $I_t$ commute.  The Hermitian symmetric tensors of valence $(2,0)$
corresponding to $I_t$ are $Q^{\bar a b}_{(t)}:=\ms^{\bar a d}\ms^{\bar c
  b}H_{d\bar c}^{(t)}$, where we write $H_{d\bar c}^{(t)}$ instead of $\wt
H_{d\bar c}(t)$.  It thus suffices to show that for all $s\neq t\in\R$,
$Q^{\bar a b}_{(s)}$ and $Q^{\bar a b}_{(t)}$ have vanishing
Schouten--Nijenhuis bracket, which (in barred and unbarred indices) means
\begin{equation*}
\left (\nabla_b Q^{d(\bar a}_{(t)}\right)Q^{\bar c)b}_{(s)}
- \left(\nabla_b Q^{d(\bar a}_{(s)} \right)Q^{\bar c) b}_{(t)}=0,\qquad
\left(\nabla_{\bar a} Q^{\bar c(b}_{(t)}\right)Q^{d)\bar a}_{(s)}
-\left( \nabla_{\bar a} Q^{\bar c(b}_{(s)}\right)Q^{d)\bar a}_{(t)}=0.
\end{equation*}
We prove the first equation (the second is analogous); taking for $\nabla$ the
Levi-Civita connection of $g$ (so $\nabla\ms=0$), this reduces to:
\begin{equation*}
\ms^{\bar c f}\left(\nabla_f H^{(t)}_{\bar a (b}\right) H^{(s)}_{d)\bar c}
-\ms^{\bar c f}\left(\nabla_f H^{(s)}_{\bar a (b}\right) H^{(t)}_{d)\bar c}=0.
\end{equation*}
The key trick is to multiply the left hand side by $s-t$ and observe that
$(s-t)\ms^{\bar c f}=\sms^{\bar c f}(t)-\sms^{\bar cf}(s)$.  Now using
equation~\eqref{eq:H-deriv} for $\sms^{\bar a b}(t)$ and $H_{d\bar c}^{(t)}$,
we obtain
\begin{equation*}
\sms^{\bar c f}(t)\left(\nabla_f H^{(t)}_{\bar a (b}\right) H^{(s)}_{d)\bar c}=
Y^{\bar e}\delta_{\bar a}{}^{\bar c} H_{\bar e (b}^{(t)}H_{d)\bar c}^{(s)}-
Y^{\bar c} H_{\bar a (b}^{(t)}H_{d)\bar c}^{(s)}
=Y^{\bar c} H_{\bar c (b}^{(t)}H_{d)\bar a}^{(s)}-Y^{\bar c} H_{\bar a (b}^{(t)}H_{d)\bar c}^{(s)}
\end{equation*}
for some $Y^{\bar a}$. The same reasoning applies to $\sms^{\bar a b}(s)$
and $H_{d\bar c}^{(s)}$ with the same vector field $Y^{\bar a}$ to obtain the
same expression with $s$ and $t$ interchanged. These two expressions sum to zero
and hence
\begin{equation*}
\left(\sms^{\bar c f}(t)-\sms^{\bar c f}(s)\right)\left(
\left(\nabla_f H^{(t)}_{\bar a (b}\right) H^{(s)}_{d)\bar c}
-\left(\nabla_f H^{(s)}_{\bar a (b}\right) H^{(t)}_{d)\bar c}\right)=
- \tilde\sigma(s)\nabla^{\vphantom{y}}_{(d} H^{(s)}_{b)\bar a}
- \tilde\sigma(t)\nabla^{\vphantom{y}}_{(d} H^{(t)}_{b)\bar a}
\end{equation*}
which vanishes because $\wt H(s)$ and $\wt H(t)$ are c-projective Hermitian
symmetric Killing tensors.
\end{proof}

We now discuss the question how many of the functions $L_t$ and $I_t$
($t\in\R$) are \emph{functionally independent} on $TM$, i.e.~have linearly
independent differentials. Since $TM$ has dimension $4n$, and the functions
$L_t$ and $I_t$ mutually commute (i.e.~they span an Abelian subalgebra under
the Poisson bracket induced by $g$), at most $2n$ of these functions can be
functionally dependent at each point of $TM$. If equality holds on the fibres
of $TM$ over a dense open subset of $M$, the geodesic flow of $g$ is said to be
\emph{integrable}.

Since $A_{a}{}^{b}(t)= A_{a}{}^{b}-t\delta_a{}^b$, integrability turns out to
be related to the spectral theory of the field $A_{a}{}^{b}$ of endomorphisms
of $T^{1,0}M$. In particular, using the trivialisation of $\cE(1,1)$
determined by $g$, the determinant $\tilde\sigma(t):=\det\sms(t)$ becomes the
characteristic polynomial $\chi_A(t):=\det A(t)$ of $A_{a}{}^{b}$.  Since $A$
is Hermitian, the coefficients of $\chi_A(t)$ are smooth real-valued functions
on $M$. Any complex-valued function $\mu$ on an open subset $U\subseteq M$ has
an associated \emph{algebraic multiplicity} $m_\mu\colon U\to \N$, where
$m_\mu(p)$ is the multiplicity of $\mu(p)$ as a root of $\chi_A(t)$ at $p\in
U$, or equivalently the rank of the generalised $\mu(p)$-eigenspace of
$A_a{}^b$ in $T^{1,0}_p M$; additionally, its \emph{geometric multiplicity}
$d_\mu(p)$ is the dimension of the $\mu(p)$-eigenspace of $A_a{}^b$ in
$T^{1,0}_p M$, and its \emph{index} $h_\mu(p)$ is the multiplicity of $\mu(p)$
in the minimal polynomial of $A_a{}^b$ at $p$.

\begin{rem}\label{nilpotent} If $\mu\colon U\to\C$ is smooth with $m_\mu$
constant on $U$, then the restriction of $A_a{}^b-\mu\delta_a{}^b$ to the
generalised $\mu$-eigendistribution, defines, using an arbitrary local frame
of this distribution, a family of nilpotent $m_\mu\times m_\mu$ matrices
$N$. There are only finitely many conjugacy classes of such matrices,
parametrised by partitions of $m_\mu$: we can either use the \emph{Segre
  characteristics}, which are the sizes of the Jordan blocks of $N$, or the
dual partition by the \emph{Weyr characteristics} $\dim\ker N^k-\dim\ker
N^{k-1}$, $k\in\Z^+$.  The index $h_\mu$ is the first Segre characteristic
(i.e.~the size of the largest Jordan block), while the geometric multiplicity
is the first Weyr characteristic $\dim\ker N$ (so $\max\{d_\mu,h_\mu\}\leq
m_\mu$ with equality if and only if $d_\mu=1$ or $h_\mu=1$). The unique (hence
dense) open orbit consists of nilpotent matrices of degree $m_\mu$, whose
Jordan normal forms have a single Jordan block, and in general, the orbit
closures stratify the nilpotent matrices (with the partial ordering of strata
corresponding to the dominance ordering of partitions). Thus $N$ maps a
sufficiently small neighbourhood of a point $p\in U$ into a unique minimal
stratum, and for generic $p\in U$, this stratum is the orbit closure of
$N(p)$. In other words, the type of $N$ may be assumed constant in
neighbourhood of any point in a dense open subset of $U$.
\end{rem}

The general theory of families of matrices is considerably simplified here by
Proposition~\ref{open_dense_Killing_vector_fields} and the following two
lemmas.

\begin{lem}\label{eigenvalues} Suppose $U$ is an open subset of $M$
and $T^{1,0}U=E\oplus F$ where $E$ and $F$ are smooth $A$-invariant subbundles
over $U$ such that the restriction of $A$ to $E$ has a single Jordan block
with smooth eigenvalue $\mu\colon U\to\C$.  Then the gradient of $\mu$ is a
section over $U$ of $E\oplus F^\perp\subseteq T^{1,0}U\oplus T^{0,1}U =
TU\otimes\C$, where $F^\perp$ denotes the subspace of $T^{0,1}U$ orthogonal
to $F$ with respect to $g$.
\end{lem}
\begin{proof} In a neighbourhood of any point in $U$, we may choose a
frame $Z^a(1),\ldots,Z^a(m)$ of $E$ such that $A$ is in Jordan normal form on
$E$. We identify $E^*$ with the annihilator of $F$ in $\Omega^{1,0}$ and let
$Z_a(1),\ldots,Z_a(m)$ be the dual frame (with $Z_a(i)Z^a(j)=\delta_{ij}$).
Then the transpose of $A$ is in Jordan normal form with respect to this dual
frame in reverse order: for $k=1,\ldots m$ we thus have
\begin{align}\label{jordan-cond}
(A_a{}^b - \mu\,\delta_a{}^b) Z^a(k) &= Z^b(k-1),\\
(A_a{}^b - \mu\,\delta_a{}^b) Z_b(k) &= Z_b(k+1),
\label{jordan-dual}
\end{align}
where $Z^b(-1)=0=Z_b(m+1)$. By~\eqref{metri-mob}, the $(0,1)$-derivative
of~\eqref{jordan-cond} yields
\begin{equation}\label{T0}
(g_{a\bar c}\hv^b+ \delta_a{}^b\nabla_{\bar c}\mu)Z^a(k)
= (A_a{}^b - \mu\,\delta_a{}^b) \nabla_{\bar c} Z^a(k) - \nabla_{\bar c}Z^b(k-1),
\end{equation}
which we may contract with $Z_b(k)$, using~\eqref{jordan-dual}, to obtain
\begin{equation*}
-\hv^b Z_b(k) g_{a\bar c} Z^a(k)  + \nabla_{\bar c}\mu
= Z_b(k+1) \nabla_{\bar c} Z^b(k) - Z_b(k) \nabla_{\bar c}Z^b(k-1).
\end{equation*}
Summing from $k=1$ to $m$, the right hand side sums to zero, and hence
\begin{equation*}
m\nabla^a \mu =  \sum_{k=1}^m \hv^b Z_b(k) Z^a(k)
\end{equation*}
so the $(1,0)$-gradient of $\mu$ is a linear combination of $Z^a(1),\ldots,
Z^a(m)$, hence a section of $E$. Since $A$ is Hermitian, its restriction to
$\overline F^\perp$ also has a single Jordan block, with eigenvalue $\bar\mu$,
and so $\nabla^{\bar a}\mu=\overline{\nabla^a\bar\mu}$ belongs to $F^\perp$ by
the same argument.
\end{proof}

It follows that if there is more than one Jordan block with eigenvalue $\mu$,
then $\mu$ is constant---equivalently, all nonconstant eigenvalues of $A$ have
geometric multiplicity one. In fact, a stronger result holds.

\begin{lem}\label{alg-mult} Let $\mu$ be a smooth function on $M$ and
let $U\subset M$ be a nonempty open subset on which $\mu$ has constant
algebraic multiplicity $m$. If $\mu$ is constant and $M$ is connected, then
$\mu$ has algebraic multiplicity $m_\mu\geq m$ on $M$.  Conversely, if $m\geq
2$ then $\mu$ is locally constant on $U$.
\end{lem}
\begin{proof} Since $A$ is Hermitian, $\chi_A(t)$ has real
coefficients, $\bar\mu$ is an eigenvalue of $A$ with the same algebraic
multiplicity as $\mu$, and $\nabla^{\bar a}\mu= \overline{\nabla^a\bar\mu}$.
By assumption $\chi_A(t)=(t-\mu)^m q(t)$ on $U$, where $q$ is smooth with
$q(\mu)$ nonvanishing. Since $\chi_A(t)$ is a Killing potential for $g$, its
gradient $\nabla^a \chi_A(t)$ is a holomorphic vector field on $M$ for all
$t$.

Suppose first that $\mu$ is constant: we shall show by induction on $k$ that
if $m\geq k$ then $m_\mu\geq k$ on $M$, which is trivially true for $k=0$. So
suppose that $m\geq k+1$ and $m_\mu\geq k$, so that $p(t)=\chi_A(t)/(t-\mu)^k$
is a polynomial in $t$. Since $\nabla^a p(\mu)$ is holomorphic on $M$ and
vanishing on $U$, it vanishes on $M$, since $M$ is connected. Similarly for
$\overline{p(\mu)}$, so that $p(\mu)$ is locally constant on $M$, hence zero,
since $M$ is connected and $p(\mu)$ vanishes on $U$. Thus $m_\mu\geq k+1$ as
required.

For the second part, the $(m-2)$nd derivative in $t$ of $\chi^a(t)$ is also a
Killing potential, which may be written
\[
\nabla^a \chi^{(m-2)}_A(t)=m! (t-\mu) q(t) \nabla^a\mu + (t-\mu)^2 X^a(t)
\]
for some polynomial of vector fields $X^a(t)$. Applying $\nabla^b=g^{b\bar c}
\nabla_{\bar c}$ and evaluating at $t=\mu$ yields $\nabla^a\mu\nabla^b\mu=0$,
i.e.~$\nabla^a\mu=0$. Replacing $\mu$ by $\bar\mu$, we deduce that $\mu$ is
locally constant on $U$.
\end{proof}
In contrast, in the analogous real projective theory of geodesically
equivalent pseudo-Riemannian metrics, Jordan blocks with nonconstant
eigenvalues can occur: see~\cite{BM}.

In order to apply the above lemmas at a point $p\in M$, we need $p$ to be
stable for $A$ in the following sense. First, we need to suppose that the
number of distinct eigenvalues of $A_{a}{}^{b}$ is constant on some
neighbourhood of $p$.  This condition on $p$ is clearly open, and it is also
dense: if the number of distinct eigenvalues is not constant near $p$, then
there are points arbitrarily close to $p$ where the number of distinct
eigenvalues is larger; repeating this argument, there are points arbitrarily
close to $p$ where the number of distinct eigenvalues is locally maximal,
hence locally constant.  Now, on the dense open set where this condition
holds, the eigenvalues of $A_{a}{}^{b}$ are smoothly defined, and their
algebraic multiplicities are locally constant (since they are all upper
semi-continuous). Now a point $p$ in this dense open set is \emph{stable} if
in addition the Jordan type (Segre or Weyr characteristics) of each
generalised eigenspace of $A_{a}{}^{b}$ is constant on a neighbourhood of $p$.
The stable points are open and dense by Remark~\ref{nilpotent}.

\begin{defin}\label{def:regular} We say $p\in M$ is a \emph{regular point} for
the pencil $\sms(t)=\ms^{\bar ab}(A_{a}{}^{b}-t\delta_a{}^b)$ if it is stable
for $A_{a}{}^{b}$, and for each smooth eigenvalue $\mu$ on an open
neighbourhood of $p$, either $d\mu_p\neq 0$ or $\mu$ is constant on an open
neighbourhood of $p$.
\end{defin}
Equivalently, the regular points are the open subset of the stable points
where the rank of the span of the canonical Killing fields associated to the
pencil is maximal, i.e.~equal to the order $\ell$. Consequently, by
Proposition~\ref{open_dense_Killing_vector_fields}, the regular points form a
dense open subset of $M$.

\begin{cor}\label{cor:eigenvectors} Let $\mu$ be a smooth eigenvalue of $A$
over the set of stable points. Then
\begin{equation}\label{eigen_vectors}
A_{a}{}^b\nabla_b\mu=\mu\nabla_a\mu\quad \text{ and }\quad
A_{\bar a}{}^{\bar b}\nabla_{\bar b}\mu=\mu\nabla_{\bar a}\mu.
\end{equation}
If $\mu$ is constant, its algebraic multiplicity is constant on the set of
regular points.
\end{cor}
Indeed, where $\mu$ has algebraic multiplicity $m_\mu=1$, Lemma
\ref{eigenvalues} implies that $\nabla_a\mu$ generates the eigenspace of
$\mu$, whereas where $m_\mu\geq 2$, Lemma~\ref{alg-mult} implies that $\mu$ is
locally constant, and hence equations~\eqref{eigen_vectors} are trivially
satisfied. Furthermore, it implies that the algebraic multiplicities of the
constant eigenvalues are upper semi-continuous on $M$, hence constant on the
connected set of regular points.

\begin{thm} \label{thm:regular} Let $(M, J, [\nabla])$ be a c-projective
manifold that admits \bps/K\"ahler metrics $g_{a\bar b}$ and ${\tilde g}_{a\bar b}$
associated to linearly independent solutions $\ms^{\bar a c}$ and
$\sms^{\bar a c}=\ms^{\bar a b} A_b{}^c$ of the metrisability
equation~\eqref{metriequ1}.
\begin{enumerate}
\item The number of functionally independent linear integrals $L_s$ is equal to
  the number of nonconstant eigenvalues of $A$ at any regular point of $M$.
\item The number of functionally independent quadratic integrals $I_t$ is
  equal to the degree of the minimal polynomial of $A$ at any stable point of
  $M$.
\item The integrals $I_t$ are functionally independent from the integrals $L_s$.
\end{enumerate}
\end{thm}
\begin{proof} Integrals of the form $L_s$ or $I_t$ are functionally independent
near $X\in T_p M$ if their derivatives are linearly independent at $X$. Since
$L_s$ is linear along the fibres of $TM\to M$, the restriction of $d L_s$ to
$T_X(T_p M)\cong T_p M$ is $g(\wt K(s),\cdot)$ (at $p$).  Similarly, $I_t$ is
quadratic along fibres, and the restriction of $d I_t$ to $T_X(T_p M)\cong T_p
M$ is ${\scale_g}^{-2}\wt H(t)(X,\cdot)$ with $\scale_g=\det\ms$.  Hence, for
generic $X\in T_p M$, the quadratic integrals $I_t$ are functionally
independent from the linear integrals $L_s$, and the number of functionally
independent linear, respectively quadratic, integrals is at least the
dimension of the span of $\wt K(s)$ at $p$, respectively the dimension of the
span of $\wt H(t)$ at $p$.

The geodesic flow preserves the integrals and therefore the property of the
integrals to be functionally independent. Since any two points of $M$ can be
connected by a piecewise geodesic curve, it suffices to compute the dimensions
of these spaces at a regular point of $p$, where the dimensions of the spans of
$\wt K(s)$ and $\wt H(t)$ are maximal.

At such a point, the number of linearly independent Killing vector fields $\wt
K(s)$ is the number of nonconstant eigenvalues of $A$, so it remains to
compute the number of linearly independent Killing tensors $\wt H(t)$. For
this, recall that $\wt H_{b\bar a}(t)=(\mathrm{adj}\, A(t))_b{}^c H_{c\bar
  a}$, with $\mathrm{adj}\, A(t)=A(t)^{-1}\det A(t)$. Now write $A(t)$ in
Jordan canonical form: on an $h\times h$ Jordan block with eigenvalue $\mu$,
$(t-\mu)^h A(t)^{-1}$ is a polynomial of degree $h-1$ in $t$ with $h$ linearly
independent coefficients. Hence on the generalised $\mu$-eigenspace,
$(t-\mu)^{m_\mu}A(t)^{-1}$ is a polynomial in $t$ with $h_\mu$ linearly
independent coefficients, where $m_\mu$ is the geometric multiplicity of
$\mu$, and $h_\mu$ the index (the multiplicity of $\mu$ in the minimal
polynomial, i.e.~the size of the largest Jordan block). It follows readily
that the dimension of the span of $\mathrm{adj}\, A(t)$ is the degree of the
minimal polynomial of $A$.
\end{proof}

\subsection{The local complex torus action}

For a c-projective manifold $M^{2n}$ admitting a metrisability pencil with no
constant eigenvalues, Theorem~\ref{thm:regular} shows that any metric in the
pencil is integrable, i.e.~its geodesic flow admits $2n$ functionally
independent integrals. Furthermore $n$ of the independent integrals are
linear, inducing Hamiltonian Killing vector fields. Hence if $M$ is compact,
it is toric (i.e.~has an isometric Hamiltonian $n$-torus action).

When the pencil has constant eigenvalues, there are only $\ell$ independent
linear integrals, where $\ell$ is the \emph{order} of the pencil (the number
of nonconstant eigenvalues), and at most $n$ independent quadratic integrals.
In this case the flows of the Hamiltonian Killing vector fields $\wt K(t)$
generate a foliation of $M$ whose generic leaves are $\ell$-dimensional.  If
$M$ is compact, one can prove (see~\cite{ACG}) that these leaves are the
orbits of an isometric Hamiltonian action of an $\ell$-torus $U(1)^\ell$, and
it is convenient to assume this locally. The complexified action, generated by
the commuting holomorphic vector fields $\wt K(t)$ and $J\wt K(t)$, is
then a local holomorphic action of $(\C^\times)^\ell$, and the leaves of the
foliation, which are locally $J$-invariant submanifolds with generic dimension
$2\ell$, will be called \emph{complex orbits}.

\begin{lem}\label{totallygeodesic} The complex orbits through regular
points are totally geodesic and their tangent spaces are $A$-invariant.  The
c-projectively equivalent metrics $g$ and $\tilde g$ restrict to nondegenerate
c-projectively equivalent metrics \textup(with respect to the induced complex
structure\textup) on any regular complex orbit $\cO^c$.  The metrisability
pencil $\sms(t)$ restricts to a metrisability pencil of order $\ell$ on
$\cO^c$, using $g$ and its restriction to trivialise $\cE_\R(1,1)$, and then
$\tilde g$ is a constant multiple of the metric induced by $\sms$.
\end{lem}  
\begin{proof} Since the complex orbit of any regular point contains only
regular points and the tangent space to the orbit is spanned by holomorphic
vector fields, it suffices to prove that the $(1,0)$-tangent spaces to regular
points are closed under the Levi-Civita connection $\nabla$ of $g$. These
tangent spaces are spanned by eigenvectors $Z^a$ of $A_a{}^b$ with nonconstant
eigenvalues $\mu$, and by differentiating the eigenvector equation
using~\eqref{metri-mob}, as in the proof of Lemma~\ref{eigenvalues}, we obtain
\begin{equation*}
(A_a{}^b - \mu\,\delta_a{}^b) \nabla_{c} Z^a=
(-\hv_a\delta_c{}^b + ,\delta_a{}^b\nabla_c \mu) Z^a
= -(\hv_a Z^a)\delta_c{}^b + (\nabla_c \mu) Z^b.
\end{equation*}
Clearly if we contract the right hand side with a $(1,0)$-vector $X^c$
tangent to a complex orbit, we obtain another such vector. Hence $\nabla_X Z$
is a $(1,0)$-vector tangent to the complex orbit as required: the complex
orbits are thus totally geodesic.

The tangent spaces to a regular complex orbit $\cO^c$ are clearly $J$-invariant
and $A$-invariant, so that $g$ induces a K\"ahler metric $\cO^c$, with a
metrisability pencil spanned by the restrictions of $\ms$ and $\sms$, where we
use $g$ and its restriction to trivialise $\cE_\R(1,1)$. Since $\sms=\ms\circ
A$, and the generalised eigenspaces of $A$ which are not tangent to $\cO^c$
have constant eigenvalues, the metric induced by the restriction of $\sms$
is a constant multiple of the restriction of $\tilde g$.
\end{proof}

Also of interest is the local $(\R^+)^\ell$ action whose local orbits are the
leaves of the foliation generated by the vector fields $J\wt K(t)$, which
will be called \emph{real orbits}.

\begin{lem} \label{totallygeodesicsmall} The real orbits through regular
points are totally geodesic, and their tangent spaces are $A$-invariant and
generated by the gradients of the nonconstant eigenvalues of $A$.  The
c-projectively equivalent metrics $g$ and $\tilde g$ restrict to nondegenerate
projectively equivalent metrics on any regular real orbit $\cO$, and the
restriction of $A$ is a constant multiple of the $(1,1)$-tensor
$\bigl(\tfrac{\vol(\tilde g|_\cO)}{\vol(\tilde g|_\cO)}\bigr)^{1/(\ell+1)}
(\tilde g|_\cO)^{-1} g|_\cO$.
\end{lem}
\begin{proof} At a regular point $p$, $X\in T_pM$ is tangent to the real orbit
through $p$ if and only if it is tangent to the complex orbit through $p$ and
orthogonal to the Killing vector fields $\wt K(t)$ at $p$. Since both
properties are preserved along geodesics, the real orbits are totally geodesic
with respect to $g$ (hence also $\tilde g$).

Let $\cO^c$ be the complex orbit through the regular real orbit $\cO$, so that
$g$ and $\tilde g$ restrict to c-projectively equivalent K\"ahler metrics on
$\cO^c$. Furthermore $(\vol(\tilde g|_{\cO^c}))^{1/2(\ell+1)} \tilde
g^{-1}|_{\cO^c}$ is a constant multiple of $(\vol(g|_{\cO^c}))^{1/2(\ell+1)}
g^{-1}\circ A |_{\cO^c}$. The tangent spaces to $\cO$ are generated by the
vector fields $\nabla^a \mu$, for nonconstant eigenvalues $\mu$, which are
mutually orthogonal and non-null. Hence $T_p\cO^c$ is the orthogonal direct sum
of $T_p\cO$ and $JT_p\cO$ (with respect to both $g$ and $\tilde g$).  Hence
$g$ and $\tilde g$ restrict to nondegenerate metrics on $\cO$ and $A$
restricts to a constant multiple of $\bigl(\tfrac{\vol(\tilde g|_\cO)}
{\vol(\tilde g|_\cO)}\bigr)^{1/(\ell+1)} (\tilde g|_\cO)^{-1} g|_\cO$.  The
Levi-Civita connections of $g$ and $\tilde g$ on $\cO^c$ are related
by~\eqref{cprojchange} for some $1$-form $\Upsilon_\alpha$. If we now restrict
to $\cO$ (which is totally geodesic in $\cO^c$), it follows that the induced
Levi-Civita connections $\nabla$ and $\wt\nabla$ are related by
\[
\wt\nabla_\alpha X^\gamma -\nabla_\alpha X^\gamma
= \tfrac12 (\Upsilon_\alpha\delta_\beta{}^\gamma+\delta_\alpha{}^\gamma
\Upsilon_\beta),
\]
i.e.~the metrics on $\cO$ are projectively equivalent.
\end{proof}    

\subsection{Local classification}\label{ss:locclass}

Let $(M,J,[\nabla])$ be a c-projective $2n$-manifold admitting two compatible
non-homothetic \bps/K\"ahler metrics, and hence a pencil of solutions of the
metrisability equation of order $0\leq \ell\leq n$.
Lemma~\ref{totallygeodesicsmall} shows that the real orbits yield a foliation
of the set $M^0$ of regular points which is transverse and orthogonal to the
common level sets of the nonconstant eigenvalues $\xi_1,\ldots\xi_\ell$ of $A$;
these are also the levels of the elementary symmetric functions
$\sigma_1,\ldots \sigma_\ell$ of $\xi_1,\ldots \xi_\ell$, which are
Hamiltonians for Killing vector fields generating the local isometric
Hamiltonian $\ell$-torus action on $M^0$.  Indeed, on $M^0$,
$\chi_A(t)=\chi_{\mathrm c}(t) \chi_{\mathrm{nc}}(t)$, where
$\chi_{\mathrm{nc}}(t)=\prod_{i=1}^\ell(t-\xi_i)=\sum_{r=0}^\ell (-1)^r\sigma_r
t^{\ell-r}$, $\chi_c(t)$ has constant coefficients, and $\sigma_0=1$. The leaf
space $S$ of the foliation of $M^0$ by the complex orbits may then be
identified with the K\"ahler quotient of $M^0$ by this local $\ell$-torus
action.

It is convenient to write $\chi_c(t)=\prod_u \rho_u(t)^{m_u}$, where
$\rho_u(t)$ are the distinct irreducible real factors (with $\deg \rho_u=1$ or
$2$) and $m_u$ their multiplicities. Then if $S$ is a manifold, its universal
cover is a product of complex manifolds $S_u$ of (real) dimension
$2m_u\deg\rho_u$.

These observations lead to a local classification of \bps/K\"ahler metrics
which belong to a metrisability pencil (i.e.~admit a c-projectively equivalent
metric, or equivalently, a Hamiltonian $2$-form), which was obtained
in~\cite{ACG} in the K\"ahler case, and in~\cite{BMR} for general \bps/K\"ahler
metrics. We state the result as follows.

\begin{thm}\label{thm:locclass} Let $(M,J,[\nabla])$ be a
c-projective $2n$-manifold, and suppose that $g$ is a \bps/K\"ahler metric in
a metrisability pencil of order $\ell$, which we may write as $\sms^{\bar
  ab}(t)=\ms^{\bar ab} (A_{a}{}^{b}-t\delta_a{}^b)$, where $\ms^{\bar ab}$
corresponds to $g$.  Then on any open subset of $M^0$ for which the leaf space
of the complex orbits is a manifold $S$, we may write\textup:
\begin{align}\label{eq:locclass}
g&=\sum_u g_u(\chi_{\mathrm{nc}}(A_S)\cdot,\cdot)
+\sum_{i=1}^\ell\frac{\Delta_j}{\Theta_j(\xi_j)} d\xi_j^2
+\sum_{j=1}^\ell \frac{\Theta_j(\xi_j)}{\Delta_j}\Bigl(\sum_{r=1}^\ell
\sigma_{r-1}(\hat\xi_j)\theta_r\Bigr)^2,\\
\omega&=\sum_u \omega_u(\chi_{\mathrm{nc}}(A_S)\cdot,\cdot)
+\sum_{r=1}^\ell d\sigma_r\wedge \theta_r,\quad\text{with}\quad
d\theta_r=\sum_u (-1)^r\omega_u (A_S^{\ell-r}\cdot,\cdot),\\
J d \xi _j &= \frac{\Theta_j (\xi_j)}{\Delta_j} \,
\sum_{r=1}^\ell \sigma_{r-1} (\hat{\xi}_j) \,\theta_r,\qquad
J\theta_r = (-1)^r \,\sum_{j=1}^\ell \frac{\xi_j^{\ell-r}}{\Theta_j(\xi_j)}
d\xi_j.
\end{align}
The ingredients appearing here are as follows, where we lift objects on $S$ to
$M$ by identifying the horizontal distribution $\ker(d\xi_1,\ldots
d\xi_\ell,\theta_1,\ldots \theta_\ell)$ with the pullback of $TS$.
\begin{itemize}
\item $\xi_1,\ldots \xi_\ell$ are the nonconstant roots of $A$, which are
smooth complex-valued functions on $M^0$, functionally independent over $\R$,
such that for any $j\in\{1,\ldots\ell\}$, $\overline \xi_j =\xi_k$ for some
\textup(necessarily unique\textup) $k$.
\item $\chi_{\mathrm{nc}}(t)=\prod_{i=1}^\ell(t-\xi_i)=\sum_{r=0}^\ell
  (-1)^r\sigma_r t^{\ell-r}$, $\sigma_{r-1} (\hat{\xi} _j)$ is the $(r-1)$st
  elementary symmetric function of $\{\xi_k:k\neq j\}$, and
  $\Delta_j=\prod_{k\neq j}(\xi_j-\xi_k)$.
\item For $j\in\{1,\ldots \ell\}$, $\Theta_j$ is a smooth nonvanishing complex
function on the image of $\xi_j$ such that if $\overline\xi_j =\xi_k$ then
$\overline\Theta_j=\Theta_k$.
\item For each distinct irreducible real factor $\rho_u$ of $\chi_c$, the
metric $g_u$ is induced by a \bps/K\"ahler metric on the factor $S_u$ of the
universal cover of $S$.
\item $A_S$ is a parallel Hermitian endomorphism with respect to the local
  product metric $\sum_u g_u$ on $S$, preserving the distributions induced by
  $TS_u$, on which it has characteristic polynomial $\rho_u(t)^{m_u}$.
\end{itemize}
Any such \bps/K\"ahler metric admits a metrisability pencil of order $\ell$,
with
\begin{equation*}
A = A_S + \sum_{i=1}^\ell \xi_i \Bigl(d\xi_i \otimes\frac{\partial}{\partial\xi_i}
+ J d\xi_i \otimes J\frac{\partial}{\partial\xi_i}\Bigr).
\end{equation*}
\end{thm}
In other words $(M,g,J,\omega)$ is locally a bundle over a product $S$ of
\bps/K\"ahler whose fibres (the complex orbits) are totally geodesic toric
\bps/K\"ahler manifolds of a special kind, called ``orthotoric''. The proof
in~\cite{ACG} proceeds by establishing the orthotoric property of the fibres
and the special structure of the base $S$. In contrast, the proof
in~\cite{BMR} relies upon the observation (generalising
Lemma~\ref{totallygeodesicsmall}) that the local quotient of $(M,g)$ by the
real isometric $\ell$-torus action admits a projectively equivalent metric:
the first two sums in~\eqref{eq:locclass} are the general form of such a
metric when the nonconstant eigenvalues of the projective pencil have
algebraic multiplicity one.

In the Riemannian case, the expression~\eqref{eq:locclass} provides a complete
local description of the metric: locally, we may assume $S=\prod_u S_u$ is
product of open subsets $S_u\subseteq\R^{2m_u}$, and then $A_S$ is a constant
multiple of the identity on each factor. In the pseudo-Riemannian, it remains
only to describe explicitly the parallel Hermitian endomorphism $A_S$ on
$S=\prod_u S_u$, for which we refer to~\cite{Boubel}.

\begin{rem}\label{proj-change} In order to understand the compatible metrics
corresponding to the general element $\sms-t\ms$ of the metrisability pencil,
it is convenient to make a projective change $s=(at+b)/(ct+d)$ of parameter, as
in Remark~\ref{rem:pencil}.  The metric corresponding to $c\sms+d\ms$
(assuming this is nondegenerate) must have the same form~\eqref{eq:locclass}
as $g$, with respect to the coordinates $\tilde\xi_j = (a\xi_j +
b)/(c\xi_j+d)$, and with $A$ replaced by $\tilde A= (c A+ d)^{-1}(a A + b)$.
We find in particular that the new functions $\tilde\Theta_j$ are related to
the old functions by $\tilde\Theta_j(s) (ct+d)^{\ell+1} =
(ad-bc)^{\ell+1}\Theta_j(t)$---in other words they transform like polynomials
of degree $\ell+1$ (sections of $\cO(\ell+1)$ over the projective parameter
line).
\end{rem}

\begin{rem}\label{CHSC-fibre} It is straightforward to show that the
restriction of the metric~\eqref{eq:locclass} to any complex orbit (a totally
geodesic integral submanifold of
$\partial_{\xi_j},J\partial_{\xi_j}:j\in\{1,\ldots \ell\}$) has constant
holomorphic sectional curvature if and only if each $\Theta_j(t)$ is a
polynomial independent of $j$, of degree at most $\ell+1$: the curvature
computations in~\cite{ACG} extend readily to the \bps/K\"ahler case. If we
write $\Theta_j(t)=\Theta(t):=\sum_{r=-1}^\ell a_{r} t^{\ell-r}$, then the
complex orbits have constant holomorphic sectional curvature $B=\frac14
a_{-1}$.

Following~\cite{ACG}, we may introduce holomorphic coordinates $u_r+it_r$ on
the complex orbits by writing $\theta_r=d t_r+\alpha_r$ and $Jd u_r=d t_r$ for
$r\in\{1,\ldots \ell\}$, where $\alpha_r$ are pullbacks of $1$-forms on $S$.
Thus
\begin{equation*}
J d u_r= -\alpha_r
-(-1)^r\sum_{j=1}^\ell\frac{\xi_j^{\ell-r}}{\Theta_j(\xi_j)}J d\xi_j
\end{equation*}
where $d\alpha_r = \sum_u (-1)^r \omega_u(A_S^{\ell-r}\cdot,\cdot)$, and these
formulae extend to any $r\leq \ell$. For $r\geq 1$, $d J d u_r=0$, whereas $dJ
d u_0 = -\omega$ and $d J d u_{-1}= \phi+\sigma_1\omega$, where
$\phi=g(JA\cdot,\cdot)$.

In particular, if $\Theta_j(t)=\Theta(t)$, then
\begin{equation*}
\sum_{r=-1}^\ell (-1)^r a_r (J d u_r+\alpha_r)=-J d\sigma_1
\end{equation*}
and hence
\begin{equation*}
d J d \sigma_1= a_{-1} (\phi+\sigma_1\omega) + a_0 \omega - \sum_u
\omega_u(\Theta(A)\cdot,\cdot).
\end{equation*}
However, $\sigma_1$ differs from $\mathrm{trace}\, A = A_a{}^a$ by an additive
constant, so $d\sigma_1 = - \Lam$ and hence $d J d\sigma_1 = -2\nabla J \Lam$,
i.e.
\begin{equation}\label{eq:lc-lam}
2\nabla\Lambda = (a_{-1} +a_0\sigma_1) g + a_{-1} g(A\cdot,\cdot)-
\sum_u g_u(\Theta(A)\cdot,\cdot).
\end{equation}
\end{rem}

\section{Metric c-projective structures and nullity}\label{metricstructures}

Henceforth, we assume that $(M,J)$ is a complex manifold (i.e.~with $J$
integrable) of real dimension $2n\geq 4$, equipped with a \emph{metric
  c-projective structure}, i.e.~a c-projective structure~$[\nabla]$ containing
the Levi-Civita connection of a \bps/K\"ahler metric $g$, which we denote by
$\nabla^g$, or $\nabla$ if $g$ is understood. We may also consider a metric
c-projective structure as an equivalence class $[g]$ of \bps/K\"ahler metrics
on $(M,J)$ having the same $J$-planar curves.

By Proposition~\ref{compKaehler}, the map sending a metric $g\in[g]$ to
$\ms=\scale_g^{\,-1} g^{-1}$ embeds $[g]$ into $\mob_c=\mob_c[\nabla]$ as an
open subset of the nondegenerate solutions to the metrisability
equation~\eqref{this_is_D_in_real_money}. We refer to $\dim\mob_c$ as the
mobility of $g$ for any $g\in[g]$, cf.~Section~\ref{sec:met-mob}, and we are
interested in the case that $\dim\mob_c\geq 2$. In
Section~\ref{sec:integrability}, we obtained some consequences of this
assumption for the geodesic flow of $g$ on $M$.  We now turn to the relationship
between mobility and curvature.

As explained in Section~\ref{metrisability_prolonged}, $\mob_c$ may be
identified with the space of parallel sections of the real tractor bundle
$\cV$ with respect to the \emph{prolongation
  connection}~\eqref{Metriprol1}--\eqref{Metriprol2}. However,
in~\cite[Theorem~5]{FKMR}, it was shown that if $\dim\mob_c\geq 3$, then
$\mob_c$ may also be identified with the space of parallel sections of $\cV$
with respect to the connection
\[
\nabla_\alpha\begin{pmatrix} A^{\beta \gamma}\\ \hv^\beta\\ \rho
\end{pmatrix}=\begin{pmatrix}
\nabla_\alpha A^{\beta\gamma}+\delta_{\alpha}{}^{(\beta}\hv^{\gamma)}
+J_{\alpha}{}^{(\beta}J_{\epsilon}{}^{\gamma)}\hv^\epsilon\\ 
\nabla_{\alpha}\hv^\beta+\rho\delta_{\alpha}{}^{\beta}
-2Bg_{\alpha \gamma} A^{\beta \gamma}\\ 
\nabla_{\alpha}\rho-2Bg_{\alpha\beta}\hv^{\beta}
\end{pmatrix}
\]
for some uniquely determined constant~$B$. In this section we explore this
phenomenon, and its implications for the curvature of $M$. First, as a
warm-up, we consider the analogous situation in real projective geometry.

\subsection{Metric projective geometry and projective nullity}

A \emph{metric projective structure} on a smooth manifold $M$ of dimension
$n\geq 2$ is a projective structure $[\nabla]$ containing the Levi-Civita
connection of a \bps/Riemannian metric, or (which amounts to the same thing)
an equivalence class $[g]$ of \bps/Riemannian metrics with the same geodesic
curves.  As in the c-projective case (see Section~\ref{metrisability} and
Remark~\ref{rem:proj-analogue}), up to sign, $[g]$ embeds into the space
$\mob=\mob[\nabla]$ of solutions to the projective metrisability
equation~\eqref{eq:proj-metrisability} as the open subset of nondegenerate
solutions. 

A metric projective structure has mobility $\dim\mob\geq 1$, and we are
interested in the case that $\dim\mob\geq 2$. However, it is shown
in~\cite{KM} that, on a connected projective manifold $(M,[\nabla])$ with
mobility $\dim\mob\geq 3$, there is a constant $B$ such that solutions
$\A^{\beta\gamma}$ of the mobility equations may be identified with parallel
sections for the connection
\begin{equation}\label{funni_mobility}
\nabla_\alpha\begin{pmatrix} A^{\beta \gamma}\\ \mu^\beta\\ \rho
\end{pmatrix}=\begin{pmatrix}
\nabla_\alpha A^{\beta\gamma}+2\delta_\alpha{}^{(\beta}\mu^{\gamma)}\\ 
\nabla_{\alpha}\mu^\beta+\rho\delta_{\alpha}{}^{\beta}
-B g_{\alpha \gamma} A^{\beta\gamma}\\ 
\nabla_{\alpha}\rho-2B g_{\alpha\beta}\mu^{\beta}
\end{pmatrix}
\end{equation}
on the tractor bundle associated to the metrisability equation. This
connection is the main tool used in~\cite{FM} to determine all possible values
of the mobility of an $n$-dimensional simply-connected Lorentzian manifold.

This result is an example of a general phenomenon: in metric projective
geometry, solutions to first BGG equations are often in bijection with
parallel sections of tractor bundles for a much simpler (albeit somewhat
mysterious) connection than the prolongation connection. We illustrate this
with a toy example. The operator
\[
\Gamma(TM(-1))\ni\ts^\beta\mapsto(\nabla_\alpha\ts^\beta)_\circ
=\nabla_\alpha\ts^\beta-\tfrac1n\delta_\alpha{}^\beta\nabla_\gamma\ts^\gamma
\]
is projectively invariant, where $TM(-1)$ denotes the bundle of vector fields
of projective weight $-1$; its kernel consists of solutions to the
\emph{concircularity equation}
\begin{equation}\label{concircular}
(\nabla_\alpha\ts^\beta)_\circ=0,
\end{equation}
called \emph{concircular vector fields}. This equation is especially congenial
in that its prolongation connection coincides with the Cartan
connection. Indeed, following~\cite{beg}, for any solution $\ts^\alpha$
of~\eqref{concircular} there is a unique function $\rho$ (of projective weight
$-1$) such that $\nabla_\alpha\ts^\beta=-\delta_\alpha{}^\beta\rho$, namely
$\rho=-\tfrac1n\nabla_\gamma\ts^\gamma$. We then have
\begin{equation}\label{toy-prolong}
R_{\alpha\beta}{}^\gamma{}_\delta \ts^\delta=
(\nabla_\alpha\nabla_\beta-\nabla_\beta\nabla_\alpha)\ts^\gamma=
\delta_\alpha{}^\gamma\nabla_\beta\rho
-\delta_\beta{}^\gamma\nabla_\alpha\rho,
\end{equation}
and tracing over $\alpha\gamma$ yields
$\Ric_{\beta\delta}\ts^\delta=(n-1)\nabla_\beta\rho$. We conclude that
$\ts^\alpha$ lifts uniquely to parallel section of the standard tractor bundle
for the connection
\begin{equation}\label{tractor_connection}
\nabla_\alpha\begin{pmatrix} \ts^\beta\\ \rho \end{pmatrix}
=\begin{pmatrix} \nabla_\alpha\ts^\beta+\delta_\alpha{}^\beta\rho\\ 
\nabla_\alpha\rho-\Rho_{\alpha\beta}\ts^\beta \end{pmatrix}
\end{equation}
induced by the (normal) Cartan connection, where
$\Rho_{\alpha\beta}\equiv\frac1{n-1}\Ric_{\alpha\beta}$.

The simpler connection arising in the metric projective case is described as
follows.

\begin{thm} \label{funni_toy}
Let $(M,[\nabla])$ be a metric projective manifold, and for any $p\in M$, let
$N_p$ be the dimension of the span at $p$ of the local solutions
of~\eqref{concircular}. Then for any metric $g$ with Levi-Civita connection
$\nabla^g\in[\nabla]$, there is a function $B$ on $M$, which is uniquely
determined and smooth where $N_p\geq 1$, such that every concircular vector
field lifts uniquely to a parallel section of the standard tractor bundle for
the connection
\begin{equation}\label{funni_connection_toy}
\nabla_\alpha\begin{pmatrix} \ts^\beta\\ \rho
\end{pmatrix}=\begin{pmatrix}
\nabla^g_\alpha\ts^\beta+\delta_\alpha{}^\beta\rho\\ 
\nabla^g_\alpha\rho-Bg_{\alpha\beta}\ts^\beta
\end{pmatrix}.
\end{equation}
Moreover $B$ is locally constant on the open set where $N_p\geq 2$, which is
empty or dense in each connected component of $M$. If $M$ is connected and $B$
is locally constant on a dense open set, it may be assumed constant on $M$.
\end{thm}
\begin{proof} We take $\nabla=\nabla^g$ and use $g$ to raise and lower indices.
Suppose that
\begin{equation*}
\nabla_\alpha\ts^\beta+\delta_\alpha{}^\beta\rho=0\qquad\text{and}\qquad
\nabla_\alpha\wt\ts^\beta+\delta_\alpha{}^\beta\wt\rho=0
\end{equation*}
for solutions $\ts^\alpha,\wt\ts^\alpha$ of~\eqref{concircular}.
Then~\eqref{toy-prolong} implies that
\begin{equation}\label{stronger-key}
\begin{split}
R^g_{\alpha\beta\gamma\delta} \ts^\delta=
  g_{\alpha\gamma}\nabla_\beta\rho-g_{\beta\gamma}\nabla_\alpha\rho
\qquad&\text{and}\qquad
R^g_{\alpha\beta\gamma\delta}\wt\ts^\delta=
  g_{\alpha\gamma}\nabla_\beta\wt\rho-g_{\beta\gamma}\nabla_\alpha\wt\rho,\\
\text{and so}\qquad 2\wt\ts_{[\alpha}\nabla_{\beta]}\rho=
R^g_{\alpha\beta\gamma\delta}\wt\ts^\gamma\ts^\delta&=
-R^g_{\alpha\beta\gamma\delta}\ts^\gamma\wt\ts^\delta
=2\ts_{[\beta}\nabla_{\alpha]}\wt\rho.
\end{split}\end{equation}
In particular, $\ts_{[\alpha}\nabla_{\beta]}\rho=0$ and so there is a unique
smooth function $B$ on the open set where $\ts^\alpha\neq 0$ such that
\begin{equation}\label{concirc-B}
\nabla_\alpha\rho - B g_{\alpha\beta}\ts^\beta=0
\end{equation}
on $M$ for any extension of $B$ over the zero-set of $\ts^\alpha$ (since
$\nabla_\alpha\rho$ also vanishes there). Equation~\eqref{stronger-key} now
implies that any two concircular vector fields have the same function $B$
where both functions are determined. Thus $B$ is uniquely determined and
smooth where $N_p\geq 1$. Differentiating~\eqref{concirc-B} on the open set
where $B$ is smooth gives
\[
\nabla_\beta\nabla_\alpha\rho-\ts_\alpha\nabla_\beta B +Bg_{\alpha\beta}\rho=0,
\]
and so $\ts_{[\alpha}\nabla_{\beta]}B=0$. Hence $\nabla_\alpha B=0$ on the
open set where $N_p\geq 2$. This subset is empty or dense in each component of
$M$, since two solutions of~\eqref{concircular} that are pointwise linearly
dependent on an open set are linearly dependent on that open set.

It remains to show that if $M$ is connected and $B$ is locally constant on a
dense open subset $U$ (which could be disconnected), then it may be assumed
constant. To see this, we use only that
\[
\Rho_{\beta\gamma}\ts^\gamma=B\ts_\beta\quad\mbox{and} \quad\nabla_\alpha B=0
\]
on $U$, for then we may differentiate once more to conclude that
\[
(\nabla_\alpha\Rho_{\beta\gamma})\ts^\gamma
+\Rho_{\beta\gamma}\nabla_\alpha\ts^\gamma
=B\nabla_\alpha\ts_\beta
\]
and hence that
\[
(\nabla_\alpha\Rho_{\beta\gamma})\ts^\gamma-\Rho_{\alpha\beta}\rho
=-Bg_{\alpha\beta}\rho
\]
on~$U$. Tracing over $\alpha\beta$ yields
\[
(\nabla_\alpha\Rho^\alpha{}_\gamma)\ts^\gamma-\Rho_\alpha{}^\alpha\rho=-nB\rho
\]
and hence that
\begin{equation}\label{extendB}
\begin{pmatrix}\Rho^\alpha{}_\beta&0\\
-\frac1n\nabla_\alpha\Rho^\alpha{}_\beta&\frac1n\Rho_\alpha{}^\alpha 
\end{pmatrix}
\begin{pmatrix}\ts^\beta\\ \rho\end{pmatrix}=
B\begin{pmatrix}\ts^\alpha\\ \rho\end{pmatrix}
\end{equation}
on $U$. Although this equation was derived on~$U$, it is a valid stipulation
everywhere on~$M$. Moreover, the tractor
\[
\begin{pmatrix}\ts^\beta\\ \rho\end{pmatrix}
\]
is nowhere vanishing on~$M$ (else in~\eqref{tractor_connection}, the vector
field $\ts^\beta$ would vanish identically). {From} this point of view, we see
that $B$ extends as a smooth function on~$M$. Finally, since $B$ is locally
constant on~$U$, it is locally constant and hence constant on~$M$.
\end{proof}
The connection~\eqref{funni_connection_toy} of Theorem~\ref{funni_toy} differs
from the tractor connection~\eqref{tractor_connection} by the
endomorphism-valued $1$-form
\[
\begin{pmatrix} 0 &0\\ \Rho_{\alpha\beta}-Bg_{\alpha\beta}&0\end{pmatrix}\colon
X^\alpha\otimes \begin{pmatrix}\ts^\beta\\ \rho\end{pmatrix}\mapsto
\begin{pmatrix} 0\\ 
(\Rho_{\alpha\beta}-Bg_{\alpha\beta})X^\alpha\ts^\beta\end{pmatrix}
\]
The connections agree on the flat model.  Specifically, on the unit sphere we
have
\[
R_{\alpha\beta\gamma\delta}
=g_{\alpha\gamma}g_{\beta\delta}-g_{\beta\gamma}g_{\alpha\delta}\quad
\mbox{whence}\quad\Rho_{\alpha\beta}=g_{\alpha\beta},
\]
so that the connections coincide with $B=1$.

The proof of Theorem~\ref{funni_toy} may be broken down into two steps. First,
one shows that the connection~\eqref{funni_connection_toy} has the required
lifting property for some function $B$, which may only be uniquely determined
and smooth on an open set.  Secondly, one establishes sufficient regularity to
determine the connection globally on $M$ (in this case, with $B$ constant). In
the literature, the second step has often been carried out by probing $M$ with
geodesics.  In the above proof we advocate an alternative line of argument
that we believe to be simpler and more generally applicable.

\begin{rem}\label{rem:proj-mob} For example, we may apply the same technique to
the mobility equations~\eqref{funni_mobility}, where the replacement
for \eqref{extendB} has the form
\[
\begin{pmatrix} R&0&0\\ \nabla R&R&0\\
\nabla\nabla R+R\bowtie R&\nabla R& R \end{pmatrix}
\begin{pmatrix} A^{\beta\gamma}\\
\mu^\beta\\ \rho\end{pmatrix}=
B\left[\begin{pmatrix} A^{\beta\gamma}\\
\mu^\beta\\ \rho\end{pmatrix}
-\frac{A_\delta{}^\delta}n \begin{pmatrix}g^{\beta\gamma}\\
0\\ \frac1n\Rho_\gamma{}^\gamma\end{pmatrix}\right].
\]
As above, this is sufficient to show that $B$ is constant if it is locally
constant on a dense open set. One striking difference between this case and
Theorem~\ref{funni_toy}, however, is that the
connection~\eqref{funni_connection_toy} actually has the same covariant
constant sections as does the standard Cartan or prolongation
connection~\eqref{tractor_connection}.  For the mobility equations, however,
not only is the resulting connection~\eqref{funni_mobility} different from the
prolongation connection~\cite{EM} but also their covariant constant sections
are generally different. Nevertheless, all solutions of the mobility equations
lift uniquely as covariant constant sections with respect to either of these
connections (and this is their crucial property).
\end{rem}

We next seek to elucidate the first step in the proof of
Theorem~\ref{funni_toy}.  Here we observe that the key
equations~\eqref{stronger-key} used to establish the uniqueness of $B$ may be
viewed as a characterisation of $B$ in terms of the curvature $R^g$ of $g$,
namely that
\[
R^g_{\alpha\beta\gamma\delta}\ts^\delta=B(g_{\alpha\gamma}\ts_\beta
-g_{\beta\gamma}\ts_\alpha).
\]
This motivates the introduction of some terminology, following
Gray~\cite{Gray}.

\begin{defin}\label{def_nullity} Let $(M,g)$ be a \bps/Riemannian manifold
and suppose that the tensor $R_{\alpha\beta\gamma\delta}$ has the symmetries
of the Riemannian curvature of $g$. Then a \emph{nullity vector} of $R$ at
$p\in M$ is a tangent vector $v^\alpha\in T_pM$ with
$R_{\alpha\beta\gamma\delta}v^\delta=0$, and the \emph{nullity space} of $R$
at $p$ is the set of such nullity vectors. We say $R$ \emph{has nullity} at
$p$ if the nullity space is nonzero, i.e.~the nullity index is positive.
\end{defin}
In particular, if $R^g$ is the Riemannian curvature of $g$, then at each
$p\in M$, there is at most one scalar $B\in\R$ such that
$R^B_{\alpha\beta\gamma\delta}:=R^g_{\alpha\beta\gamma\delta}
-B(g_{\alpha\gamma}g_{\beta\delta}-g_{\beta\gamma}g_{\alpha\delta})$ has
nullity at $p$.  Indeed if $v^\alpha$ and $\wt v^\alpha$ are nullity
vectors for $R^B$ and $R^{\smash{\wt B}}$ respectively then
\[
0=(B-\wt B) (g_{\alpha\gamma}g_{\beta\delta}-g_{\beta\gamma}g_{\alpha\delta})
\wt v^\beta v^\delta
=(B-\wt B) (v_\beta\wt v^\beta g_{\alpha\gamma}-v_\alpha\wt v_\gamma),
\]
which implies that $B=\wt B$ unless $v^\alpha$ or $\wt v^\alpha$ are zero.

\begin{defin}\label{proj-nullity} Let $(M,g)$ be a \bps/Riemannian manifold.
Then the (\emph{projective}) \emph{nullity distribution} of $g$ is the union
of the nullity spaces of $R^B_{\alpha\beta\gamma\delta}$ over $B\in\R$ and
$p\in M$.  We say that $g$ \emph{has \textup(projective\textup) nullity} at
$p\in M$ if there is a nonzero $v^\alpha\in T_p M$ in the nullity distribution
of $g$, i.e.
\begin{equation}\label{nullity}
\bigl(R^g_{\alpha\beta\gamma\delta}
-B(g_{\alpha\gamma}g_{\beta\delta}-g_{\beta\gamma}g_{\alpha\delta})\bigr)
v^\delta=0,
\end{equation}
for some $B\in\R$, uniquely determined by $p$.
\end{defin}
The definition of $B$ is reminiscent of an eigenvalue; indeed, the
$\alpha\gamma$ trace of~\eqref{nullity} is
\[
\Rho_\alpha{}^\beta v^\alpha=B v^\beta,
\]
so $B$ is an eigenvalue of the endomorphism $\Rho_\alpha{}^\beta$. On the
other hand the trace-free part of~\eqref{nullity} provides a projectively
invariant characterisation, using the projective Weyl tensor
$P_{\alpha\beta}{}^\gamma{}_\delta:= R_{\alpha\beta}{}^\gamma{}_\delta -
\delta_\alpha{}^\gamma \Rho_{\beta\delta} +
\delta_\beta{}^\gamma\Rho_{\alpha\delta}$, as follows (cf.~\cite{GM}).
\begin{prop} Let $(M,g)$ be a \bps/Riemannian manifold of dimension $n\geq2$,
and let $v^\delta\in T_pM$ be nonzero. Then the following statements are
equivalent\textup:
\begin{enumerate}
\item $v^\delta$ is a projective nullity vector at $p$
\item there exists $B\in\R$ such that $P_{\alpha\beta}{}^\gamma{}_\delta v^\beta
=(\Rho_{\alpha\delta}-Bg_{\alpha\delta})v^\gamma$
\item $P_{\alpha\beta}{}^\gamma{}_\delta v^\delta=0$.
\end{enumerate}
\end{prop}
\begin{proof} (1)$\Rightarrow$(2). Since $\Rho_{\alpha\beta}v^\beta=
Bg_{\alpha\beta} v^\beta$, 
$R_{\alpha\beta\gamma\delta}=R_{\gamma\delta\alpha\beta}$ and
$\Rho_{\alpha\beta}=\Rho_{\beta\alpha}$, we have
\begin{align*} P_{\alpha\beta\gamma\delta}v^\beta
&= R_{\gamma\delta\alpha\beta} v^\beta
-g_{\alpha\gamma}\Rho_{\beta\delta}v^\beta
+g_{\beta\gamma}\Rho_{\alpha\delta}v^\beta\\
&=B(g_{\alpha\gamma}g_{\beta\delta}-g_{\beta\gamma}g_{\alpha\delta})v^\beta
-Bg_{\alpha\gamma}g_{\beta\delta} v^\beta +g_{\beta\gamma}\Rho_{\alpha\delta}v^\beta
=(\Rho_{\alpha\delta}-Bg_{\alpha\delta})g_{\beta\gamma}v^\beta,
\end{align*}
and (2) follows by raising the index $\gamma$.

(2)$\Rightarrow$(3). Since $P_{[\alpha\beta}{}^\gamma{}_{\delta]}=0$, which follows
easily from $R_{[\alpha\beta}{}^\gamma{}_{\delta]}=0$,
$$P_{\alpha\beta}{}^\gamma{}_\delta v^\delta=(P_{\alpha\delta}{}^\gamma{}_\beta
-P_{\beta\delta}{}^\gamma{}_\alpha)v^\delta,$$
which vanishes by (2), since $\Rho_{\alpha\beta}-Bg_{\alpha\beta}$ 
is symmetric in $\alpha\beta$.

(3)$\Rightarrow$(1). Observe that 
$0=R_{\alpha\beta\gamma\delta}v^\gamma v^\delta
=(g_{\alpha\gamma}\Rho_{\beta\delta}-g_{\beta\gamma}\Rho_{\alpha\delta})
v^\gamma v^\delta =v_{[\alpha}\Rho_{\beta]\delta}v^\delta$. Hence there exists
$B\in\R$ such that $\Rho_{\beta\delta}v^\delta=Bg_{\beta\delta}v^\delta$, 
and hence
$$R_{\alpha\beta\gamma\delta}v^\delta=
(g_{\alpha\gamma}\Rho_{\beta\delta}-g_{\beta\gamma}\Rho_{\alpha\delta})v^\delta
=B(g_{\alpha\gamma}g_{\beta\delta}-g_{\beta\gamma}g_{\alpha\delta})v^\delta,$$
i.e.~$v^\delta$ is in the projective nullity at $p$.
\end{proof}

In particular, this shows that the projective nullity distribution is a
metric projective invariant, as is the expression
$\Rho_{\alpha\beta}-Bg_{\alpha\beta}$ wherever there is projective nullity,
hence so is the special tractor connection~\eqref{funni_connection_toy}.  The
above argument for this fact is given in~\cite{GM}, where the special
connection on the standard tractor bundle is also discussed.

\begin{rem}
In the $2$-dimensional case, all metrics have nullity at all points and $B$ is
the Gau{\ss}ian curvature. On the unit $n$-sphere the nullity distribution is
the tangent bundle and $B\equiv 1$. Condition (\ref{nullity}) may be written
as
\[
C_{\alpha\beta\gamma\delta}v^\delta
=\bigl(\tfrac1{(n-1)(n-2)}\Scal-\tfrac{1}{n-2}B\bigr)
(g_{\alpha\gamma}v_\beta-g_{\beta\gamma}v_\alpha)
-\tfrac1{n-2}(\Ric_{\alpha\gamma}v_\beta-\Ric_{\beta\gamma}v_\alpha)
\]
where $C_{\alpha\beta\gamma\delta}$ is conformal Weyl curvature tensor. For a
Riemannian metric, we may orthogonally diagonalise the Ricci tensor to see
that if $C_{\alpha\beta\gamma\delta}=0$ (as it is in three dimensions or in
the conformally flat case in higher dimensions) then
$R_{\alpha\beta\gamma\delta}$ has nullity if and only if all but possibly one
of the eigenvalues of $\Rho_\alpha{}^\beta$ coalesce with $B$ being the
possible exception.  So in the three-dimensional Riemannian case
$R_{\alpha\beta\gamma\delta}$ has nullity if and only if the discriminant of
the characteristic polynomial of $\Rho_\alpha{}^\beta$ vanishes:
\begin{multline*}
(\Rho_\alpha{}^\alpha)^6
-9(\Rho_\alpha{}^\alpha)^4(\Rho_\beta{}^\gamma \Rho_\gamma{}^\beta)
+21(\Rho_\alpha{}^\alpha)^2(\Rho_\beta{}^\gamma \Rho_\gamma{}^\beta)^2
-3(\Rho_\beta{}^\gamma \Rho_\gamma{}^\beta)^3\\
+8(\Rho_\alpha{}^\alpha)^3
(\Rho_\delta{}^\epsilon\Rho_\epsilon{}^\gamma \Rho_\gamma{}^\delta)
-36(\Rho_\alpha{}^\alpha)(\Rho_\beta{}^\gamma \Rho_\gamma{}^\beta)
(\Rho_\delta{}^\epsilon \Rho_\epsilon{}^\zeta \Rho_\zeta{}^\delta)
+18(\Rho_\delta{}^\epsilon \Rho_\epsilon{}^\gamma \Rho_\gamma{}^\delta)^2=0.
\end{multline*}
Indeed, in three dimensions (where $R_{\alpha\beta\gamma\delta}$ is determined
by $\Rho_{\alpha\beta}$) it is also the case in Lorentzian signature that
$R_{\alpha\beta\gamma\delta}$ has nullity if and only if $\Rho_\alpha{}^\beta$
is diagonalisable with eigenvalues distributed in this manner. In any case, in
three dimensions it follows that $B$ is a continuous function and is smooth
except perhaps at points where $\Rho_{\alpha\beta}$ is a multiple
of~$g_{\alpha\beta}$. In the four-dimensional Riemannian case, one can check
that if $R_{\alpha\beta\gamma\delta}$ has nullity and the eigenvalues of
$\Rho_\alpha{}^\beta$ are $B,\lambda_2,\lambda_3,\lambda_4$, then
\[
I\equiv C_{\alpha\beta\gamma\delta}C^{\alpha\beta\gamma\delta} =6\bigl(
(\lambda_2-\lambda_3)^2+(\lambda_3-\lambda_4)^2+(\lambda_4-\lambda_2)^2 \bigr)
\]
and if this expression is nonzero, then
\[
B=\tfrac1 4 \Rho_\alpha{}^\alpha+
\tfrac1{4I} \bigl(C_{\alpha\beta}{}^{\gamma\delta}C_{\gamma\delta}{}^{\epsilon\zeta}
C_{\epsilon\zeta}{}^{\alpha\beta}
-18C_{\alpha\beta}{}^{\gamma\delta}\Rho_\gamma{}^\alpha\Rho_\delta{}^\beta\bigr).
\]
It follows that $B$ is smooth on $\{I\neq0\}$ whilst on $\{I=0\}$ three of the
four eigenvalues of $\Rho_\alpha{}^\beta$ merge as above and $B$ is the odd one
out unless $\Rho_{\alpha\beta}\propto g_{\alpha\beta}$. Therefore, as in three
dimensions, it follows that $B$ extends as a continuous function that is smooth
except where $\Rho_{\alpha\beta}$ is a multiple of~$g_{\alpha\beta}$. We
anticipate similar behaviour in general but, for the moment, the regularity of
$B$ remains unknown.
\end{rem}

\subsection{C-projective nullity}

We return now to metric c-projective geometry, where we seek to develop
analogous interconnections between curvature and special tractor connections
to those in the metric projective case. In order to do this, we first develop
a notion of c-projective nullity for \bps/K\"ahler metrics, modelled on the
curvature of complex projective space~\eqref{FS_curvature} in the same way
that projective nullity for \bps/Riemannian metrics is modelled on the
curvature of the unit sphere.

We suppose therefore that $(M,J,g)$ is a \bps/K\"ahler manifold with $\nabla$
the Levi-Civita connection of $g_{\alpha\beta}$ and
$\Kf_{\alpha\beta}=J_{\alpha}{}^\gamma g_{\gamma\beta}$ the K\"ahler form.
Further let us write
\begin{equation}\label{holomorphic_sectional_curvature}
S_{\alpha\beta\gamma\delta}\equiv g_{\alpha\gamma}g_{\beta\delta}-g_{\beta\gamma}g_{\alpha\delta}
+\Omega_{\alpha\gamma}\Omega_{\beta\delta}
-\Omega_{\beta\gamma}\Omega_{\alpha\delta}
+2\Omega_{\alpha\beta}\Omega_{\gamma\delta}
\end{equation}
for the K\"ahler curvature tensor of constant sectional holomorphic curvature
$4$.  As in the \bps/Riemannian case, at each $p\in M$, there is at most one
scalar $B\in\R$ such that $G^B_{\alpha\beta\gamma\delta}:=
R_{\alpha\beta\gamma\delta} - B S_{\alpha\beta\gamma\delta}$ has nullity at
$p$.  Indeed, if $v^\alpha$ and $\wt v^\alpha$ are nullity vectors for
$G^B_{\alpha\beta\gamma\delta}$ and $G^{\wt B}_{\alpha\beta\gamma\delta}$
respectively then
\begin{align*}
0&=(B-\wt{B})S_{\alpha\beta\gamma\delta} \wt{v}^\beta v^\delta\\
&= (B-\wt{B})(v_\beta\wt v^\beta g_{\alpha\gamma}-v_\alpha\wt v_\gamma
+J_\beta{}^\delta v_\delta\wt v^\beta \Omega_{\alpha\gamma}
+J_\alpha{}^\delta v_\delta J_\gamma{}^\beta\wt v_\beta
+2 J_\gamma{}^\delta v_\delta J_\alpha{}^\beta\wt{v}_\beta)
\end{align*}
which implies that $B=\wt B$ unless $v^\alpha$ or $\wt v^\alpha$ are
zero.  By analogy with Definition~\ref{proj-nullity}, and again following
Gray~\cite{Gray} (who used the term ``holomorphic constancy''), we therefore
define c-projective nullity as follows.

\begin{defin}\label{cproj-nullity} The (\emph{c-projective})
\emph{nullity distribution} $\Ns$ of a \bps/K\"ahler manifold $(M,J,g)$ is the
union of the nullity spaces of $G^B_{\alpha\beta\gamma\delta}$ over $B\in\R$
and $p\in M$, and for each $p\in M$, we write $\Ns_p$ for the
(\emph{c-projective}) \emph{nullity space} $\Ns\cap T_p M$. We say that
$(J,g)$ \emph{has \textup(c-projective\textup) nullity} at $p\in M$ if $\Ns_p$
is nonzero, i.e.
\begin{equation}\label{c-nullity}
\bigl(R_{\alpha\beta\gamma\delta}- B S_{\alpha\beta\gamma\delta}\big ) v^\delta=0,
\end{equation}
for some $B\in\R$, uniquely determined by $p$, and some nonzero $v^\alpha\in
T_p M$.
\end{defin}
Thus $\Ns_p$ is the kernel of the linear map
\[
v^\delta\mapsto G^B_{\alpha\beta\gamma\delta}v^\delta,
\]
for some $B\in\R$ depending on $p$.  Let us remark that, since $G=G^B$ has the
symmetries of the curvature tensor of a K\"ahler metric, $\Ns_p$ is a
$J$-invariant subspace of $T_pM$ (i.e.~$v^\delta\in\Ns_p$ implies
$J_{\alpha}{}^\delta v^\alpha\in \Ns_p$), hence is even dimensional.

Bearing in mind the discussion of Section~\ref{Kaehlersection}, we may write
(\ref{c-nullity}) in barred and unbarred indices. We find that
\begin{equation}\label{c-nullity_bis}
\begin{split}
&\bigl(R_{a\bar{b}c\bar{d}}+
2B(g_{a\bar{b}}g_{c\bar{d}}+g_{c\bar{b}}g_{a\bar{d}})\bigr)v^{\bar{d}}=0\\
&\bigl(R_{a\bar{b}\bar c d}-
2B(g_{a\bar{c}}g_{d\bar{b}}+g_{a\bar{b}}g_{d\bar{c}})\bigr)v^{d}=0.
\end{split}\end{equation}
As in the projective case, tracing (\ref{c-nullity}) over $\alpha\gamma$
yields an eigenvalue equation
\begin{equation}\label{nullity_leads_eigenvector_of_Rho}
\Ric^\beta{}_\delta v^\delta=2(n+1)Bv^\beta,\enskip\mbox{equivalently}\enskip
\Rho^\beta{}_\delta v^\delta=2Bv^\beta,
\end{equation}
since $\Rho_{\alpha\beta}=\frac{1}{n+1}\Ric_{\alpha\beta}$ by (\ref{RhoTensor})
and~(\ref{KaehlerCurv}). This can equivalently be expressed in barred and
unbarred indices as
\begin{equation}\label{nullity_leads_eigenvector_of_Rho_2}
\Rho^b{}_d{}v^d=2Bv^b, \qquad\text{or as}\qquad
\Rho^{\bar{b}}{}_{\bar{d}}v^{\bar{d}}=2Bv^{\bar{b}}.
\end{equation}
Of course, we may derive \eqref{nullity_leads_eigenvector_of_Rho_2} also
directly by tracing the second equation of~\eqref{c-nullity_bis}, respectively
its conjugate, with respect to $a\bar c$, respectively~$\bar a c$. Further, 
note that the symmetries of the Ricci tensor of a \bps/K\"ahler metric show
that \eqref{nullity_leads_eigenvector_of_Rho_2} can be also equivalently
written as $\Rho_d{}^bv^d=2Bv^b$, respectively
$\Rho_{\bar{d}}{}^{\bar{b}}v^{\bar{d}}=2Bv^{\bar{b}}.$

Now assume that \eqref{c-nullity_bis} is satisfied and decompose $R_{a\bar b}
{}^c{}_d$ according to \eqref{full_curvature_decomposition} as
$$R_{a\bar{b}}{}^c{}_d
=H_{a\bar{b}}{}^c{}_d+\delta_a{}^c\Rho_{\bar{b}d}+\delta_d{}^c\Rho_{\bar{b}a}.$$
Then equation \eqref{nullity_leads_eigenvector_of_Rho_2} implies
\[
H_{a\bar{b}}{}^c{}_dv^{\bar{b}}=(R_{a\bar{b}}{}^c{}_d
-\delta_a{}^c\Rho_{\bar{b}d}-\delta_d{}^c\Rho_{\bar{b}a})v^{\bar{b}}
=2B(\delta_a{}^cv_d+\delta_d{}^cv_a)
-(\delta_a{}^c\Rho_{\bar{b}d}+\delta_d{}^c\Rho_{\bar{b}a})v^{\bar{b}}=0.
\]
Furthermore, 
\begin{align*}
H_{a\bar{b}}{}^c{}_dv^d &= (R_{a\bar{b}}{}^c{}_d
-\delta_a{}^c\Rho_{\bar{b}d}-\delta_d{}^c\Rho_{\bar{b}a})v^d\\
&= (2B(g_{a\bar{b}}v^c+\delta_a{}^cv_{\bar{b}})
-2B\delta_a{}^cv_{\bar{b}}-\Rho_{\bar{b}a}v^c)\\
&=(2Bg_{a\bar{b}}-\Rho_{\bar{b}a})v^c,
\end{align*}
which implies $H_{a\bar{b}}{}^c{}_dv^av^d=0$. It fact these two conditions are
also sufficient for nullity.

\begin{prop} \label{c-projective_invariant_characterisation_of_nullity}
Let $(M,J, g)$ be a \bps/K\"ahler manifold of dimension $2n\geq4$, and let
$v^d\in T_p^{1,0}M\cong T_pM$ be a nonzero tangent vector. Then the following
statements are equivalent\textup:
\begin{enumerate}
\item $v^d\in\Ns_p$
\item \label{c-nullity_mixed} there exists $B\in\R$ such that
  $H_{a\bar{b}}{}^c{}_dv^d=(2Bg_{a\bar{b}}-\Rho_{a\bar b})v^c$
\item \label{c-nullity_H} $H_{a\bar{b}}{}^c{}_dv^av^d=0\quad
\text{and}\quad H_{a\bar{b}}{}^c{}_dv^{\bar{b}}=0$, 
\end{enumerate}
where, as in~\eqref{full_curvature_decomposition}, $H_{a\bar{b}}{}^c{}_d$ is
the trace-free part of~$R_{a\bar{b}}{}^c{}_d\equiv-g^{\bar{e}c}R_{a\bar{b}d\bar{e}}$.
\end{prop}
\begin{proof} 
We have just observed that (1) implies (2) and (3). Note, moreover that taking
the trace with respect to $a$ and $c$ in (2) gives
\eqref{nullity_leads_eigenvector_of_Rho_2}, which shows immediately that (2)
implies (1). Hence, it remains to show that (3) implies (1). If
(\ref{c-nullity_H}) holds, then
\[
R_{a\bar{b}}{}^c{}_dv^av^d=
(\delta_a{}^c\Rho_{\bar{b}d}+\delta_d{}^c\Rho_{\bar{b}a})v^av^d
\enskip\Rightarrow\enskip R_{a\bar{b}\bar{c}d}v^av^d
=2v_{\bar{c}}\Rho_{\bar{b}d}v^d
\]
so $0=v_{[\bar{c}}\Rho_{\bar{b}]d}v^d$ and we conclude that
$\Rho_{\bar{c}d}v^d=2\bar{B}v_{\bar{c}}$, for some constant~$\bar{B}$.
Substituting the conjugate conclusion $\Rho_{\bar{d}c}v^{\bar{d}}=2Bv_c$ into
$R_{a\bar{b}}{}^c{}_dv^{\bar{b}}$ gives
\[
R_{a\bar{b}}{}^c{}_dv^{\bar{b}}=
(H_{a\bar{b}}{}^c{}_d+\delta_a{}^c\Rho_{\bar{b}d}
+\delta_d{}^c\Rho_{\bar{b}a})v^{\bar{b}}
=2B\delta_a{}^cv_d+2B\delta_d{}^cv_a
\]
which, after lowering the index $c$ is equivalent to~(\ref{c-nullity_bis}), as
required. (Note that $B$ is necessarily real, since $R_{a\bar b}{}^c{}_d$ and
$S_{a\bar b}{}^c{}_d$ are real tensors.)
\end{proof} 
\begin{cor}\label{invariant-nullity} At any $p\in M$, the nullity distribution
of $\Ns_p$ is a metric c-projective invariant, i.e.~the same for
c-projectively equivalent \bps/K\"ahler metrics $g_{a\bar{b}}$ and $\tilde
g_{a\bar b}$.  Furthermore, if $\Ns_p$ is nonzero, and $B,\wt B\in \R$ are the
corresponding scalars in the definition of $\Ns_p$ with respect to $g,\tilde
g$ respectively, then $\wt\Rho_{a\bar b}-2\wt B \tilde g_{a\bar b}=\Rho_{a\bar
  b}-2B g_{a\bar b}$.
\end{cor}
\begin{proof}
By Proposition~\ref{rosetta}, criterion~\eqref{c-nullity_H} of
Proposition~\ref{c-projective_invariant_characterisation_of_nullity} is
c-projectively invariant.  In fact, by Proposition~\ref{kaehlerharmcurv},
$H_{a\bar{b}}{}^c{}_d$ is precisely the harmonic curvature of the underlying
c-projective structure. The last part follows immediately from
criterion~\eqref{c-nullity_mixed}.
\end{proof}
\begin{rem}\label{rem:c-projective_invariant_characterisation_of_nullity}
For later use, we apply the projectors of Section~\ref{almostcomplexmanifolds}
to reformulate the equivalent conditions of Proposition
\ref{c-projective_invariant_characterisation_of_nullity} directly in terms of
$v^\delta\in T_p M$ as follows:
\begin{enumerate}
\item $v^\delta\in\Ns_p$
\item there exists a constant $B\in\R$ such that
$H_{\alpha\beta}{}^\gamma{}_\delta v^\delta =(J_{\alpha}{}^\epsilon \Rho_{\epsilon\beta}
-2B\Omega_{\alpha\beta})J_{\delta}{}^\gamma v^\delta$
\item $H_{\alpha\beta}{}^\gamma{}_\delta v^\alpha v^\delta=0$ and
$(H_{\alpha\beta}{}^\gamma{}_\delta
+J_{\beta}{}^\epsilon J_{\zeta}{}^\gamma H_{\alpha\epsilon}{}^\zeta{}_\delta )v^\beta=0$,
\end{enumerate}
where $H_{\alpha\beta}{}^\gamma{}_\delta\equiv
R_{\alpha\beta}{}^\gamma{}_\delta-\delta_{[\alpha}{}^\gamma\Rho_{\beta]\delta}
+J_{[\alpha}{}^\gamma\Rho_{\beta]\zeta}J_\delta{}^\zeta
+J_{\alpha}{}^\zeta\Rho_{\beta\zeta}J_\delta{}^{\gamma}$ and
$\Rho_{\beta\delta}\equiv \frac{1}{n+1}R_{\alpha\beta}{}^\alpha{}_\delta$.
\end{rem}

\begin{prop}\label{non-null-nullity} Let $(M,J,g)$ be a \bps/K\"ahler
manifold of dimension $2n\geq4$, and $B$ a smooth function on an open subset
$U$. Then for any \textup(real\textup) vector field $v$ in the nullity of
$G=G^B$ on $U$, if $v$ is non-null at $p\in U$, then $d B=0$ there.
\end{prop}
\begin{proof} The differential Bianchi identity $\nabla_{[a}R_{b]\bar c d}{}^e=0$
on $U$ implies that
\begin{equation*}
\nabla_{[a} G_{b] \bar c d}{}^e=2(\nabla_{[a} B)g_{b]\bar
  c}\delta_d{}^e+2(\nabla_{[a}B)\delta_{b]}{}^eg_{d\bar c}.
\end{equation*}
Since $v^a$ and $v^{\bar a}$ belong to the nullity of $G$, we may contract
with $v^{\bar c} v^d$ to obtain
\begin{equation*}
0=2(\nabla_{[a} B)g_{b]\bar
  c}v^{\bar c} v^e+2(\nabla_{[a}B)\delta_{b]}{}^e g_{d\bar c}v^dv^{\bar c}.
\end{equation*}
A further contraction with $v_e=g_{e{\bar f}} v^{\bar f}$ yields $(\nabla_{[a}
  B) v_{b]} g_{d\bar c} v^dv^{\bar c}=0$, so if $v$ is non-null at $p$,
$(\nabla_{[a} B) v_{b]}=0$ there; hence $(\nabla_{[a}B)\delta_{b]}{}^e=0$,
which implies $\nabla_a B=0$, i.e.~$d B=0$.
\end{proof}

\subsection{Mobility, nullity, and the special tractor connection}
\label{mob-null}

Our aim in this section is to show that, under certain conditions, the
solutions of the mobility equation~\eqref{mob-raised} on a \bps/K\"ahler
manifold $(M, J, g)$ lift uniquely to parallel sections of
$\cV\subseteq\cV_\C$ for the special tractor connection:
\begin{equation}\label{new_funni_formula}\begin{split}
\nabla^{\cV_\C}_a
\begin{pmatrix} A^{\bar b c}\\ \hv^b\enskip|\enskip \hv^{\bar{b}}\\
\mu \end{pmatrix}
&=\begin{pmatrix} \nabla_aA^{\bar b c}+\delta_{a}{}^c\hv^{\bar b}\\ 
\nabla_a\hv^b+\mu\delta_{a}{}^b-2BA_{a}{}^{b}\enskip|\enskip\nabla_a\hv^{\bar{b}}\\
\nabla_a\mu-2B\hv_{a} \end{pmatrix}\\
\nabla^{\cV_\C}_{\bar a}\begin{pmatrix} A^{\bar b c}\\
\hv^b\enskip|\enskip \hv^{\bar{b}}\\ \mu
\end{pmatrix}
&=\begin{pmatrix} 
\nabla_{\bar a}A^{\bar b c}+\delta_{\bar a}{}^{\bar b}\hv^{c}\\ 
\nabla_{\bar a}\hv^b\enskip|\enskip
\nabla_{\bar a}\hv^{\bar{b}}+\mu\delta_{\bar a}{}^{\bar b}-2BA^{\bar b}{}_{\bar a}\\
\nabla_{\bar a}\mu-2B\hv_{\bar a} \end{pmatrix},
\end{split}\end{equation}
where $\nabla$ is the Levi-Civita connection for~$g_{a\bar{b}}$ and $B$ is a
smooth function on $M$.  Here $\cV_\C$ is identified via $g_{a\bar b}$ with a
direct sum of unweighted tensor bundles, and we write the connection in barred
and unbarred indices, so that for sections of $\cV\subseteq\cV_\C$, the two
lines of~\eqref{new_funni_formula} are conjugate.

\begin{rem}\label{rem:lifts} By Theorem~\ref{MetriProlongation}, we know
already that any solution $A^{\bar a b}$ of the mobility
equation~\eqref{mob-raised} lifts uniquely to a parallel section of $\cV$ for
the more complicated prolongation
connection~\eqref{Metriprol1}--\eqref{Metriprol2}. If it also lifts to a
parallel section for~\eqref{new_funni_formula}, then
(cf.~Remark~\ref{rem:proj-mob}) the two lifts may differ, albeit only in the
last component. More precisely the last component $\mu$ of the parallel lift
for the special tractor connection is given by
$\mu=\mu'-\frac{1}{n}(\Rho_{a\bar b}-2Bg_{a\bar b})A^{\bar a b}$, where $\mu'$
is the last component of the parallel lift for the prolongation connection.
Note that if the metric $g^{\bar a b}$ itself lifts to a parallel section
for~\eqref{new_funni_formula}, then $B$ must be locally constant.
\end{rem}
In~\cite[Theorem~5]{FKMR}, it is shown that if the mobility of $(M,g,J)$ is at
least three, then there is a constant $B$ such that all solutions of the
mobility equation lift uniquely to parallel sections of $\cV$
for~\eqref{new_funni_formula}. Before developing this, and related results, it
will be useful to establish some basic properties of special tractor
connections~\eqref{new_funni_formula} and their parallel sections.  Throughout
this section we set, for a given function $B$,
\begin{equation}\label{G-def}
G_{a\bar{b}c\bar{d}}:= R_{a\bar{b}c\bar{d}}
+2B(g_{a\bar{b}}g_{c\bar{d}}+g_{c\bar{b}}g_{a\bar{d}}).
\end{equation}
The equations satisfied by parallel sections of~\eqref{new_funni_formula} are
\begin{subequations}\label{funni_system}
\begin{align}
\nabla_aA_b{}^c&=-\delta_a{}^c\hv_b\label{first_line}\\
\nabla_a\hv_{\bar c}&=-\mu g_{a\bar{c}}+2BA_{a\bar{c}}\quad\text{and}\quad
\nabla_a\hv_{b}=0,\label{second_line}\\
\nabla_a\mu&=2B\hv_a \label{third_line}
\end{align} 
\end{subequations}
and their complex conjugates. Of course, the first line is simply the mobility
equation. In particular, from~\eqref{f1}--\eqref{f11} (and the symmetry of
$\Rho_{b\bar c}$) we have
\begin{equation}\label{mob1-int}\begin{split}
g_{d\bar b}\nabla_a\hv_{\bar c}-g_{a\bar c}\nabla_{\bar b} \hv_d
&=R_{a\bar be\bar c} A_d{}^e-R_{a\bar b d\bar e}A^{\bar e}{}_{\bar c}\\
&=W_{a\bar be\bar c} A_d{}^e-W_{a\bar b d\bar e}A^{\bar e}{}_{\bar c}
-g_{a\bar c}\Rho_{e\bar b}A_d{}^e+g_{d \bar b}\Rho_{a\bar e} A^{\bar e}{}_{\bar c},
\end{split}\end{equation}
where $W_{a\bar be\bar c}=H_{a\bar be\bar c}$, since $J$ is integrable.

\begin{lem}\label{useful} If $A_b{}^c$ and $\hv_b$ satisfy
\eqref{first_line}--\eqref{second_line} for smooth functions $B,\mu$, then
\begin{align}\label{commuting}
G_{a\bar{b}c}{}^eA_e{}^d&=G_{a\bar{b}e}{}^dA_c{}^e\\
\label{crucial_identity}
G_{a\bar b c}{}^d\hv_d&=-g_{c\bar b}(\nabla_a\mu-2B\hv_a)+2(\nabla_aB)A_{c\bar b}.
\end{align}
\end{lem}
\begin{proof} We substitute~\eqref{second_line} into~\eqref{mob1-int} to obtain
\begin{align*}
R_{a\bar be\bar c} A_d{}^e-R_{a\bar b d\bar e}A^{\bar e}{}_{\bar c}
&=g_{a\bar c}(\mu g_{d\bar b}-2BA_{d\bar b})-g_{d\bar b}(\mu g_{a\bar c}-2BA_{a\bar c})\\
&=2B(g_{d\bar b}g_{a\bar e}A^{\bar e}{}_{\bar c}-g_{a\bar c}g_{e\bar b}A_d{}^e),
\end{align*}
and~\eqref{commuting} follows from~\eqref{G-def}. To
obtain~\eqref{crucial_identity}, we apply~\eqref{G-def} instead to the
identity
\begin{align*}
R_{a\bar{b}c}{}^d\hv_d&=(\nabla_a\nabla_{\bar{b}}-\nabla_{\bar{b}}\nabla_a)\hv_c
=\nabla_a(-\mu g_{c\bar b}+2BA_{c\bar b})\\
&=2(\nabla_aB)A_{c\bar b}-2B g_{a\bar b}\hv_c-g_{c\bar b}(\nabla_a\mu). \qedhere
\end{align*}
\end{proof}
In Theorem~\ref{new_funni}, we will show that if $(M,g,J)$ of mobility $\geq
2$ has c-projective nullity, then all solutions of the mobility equation lift
uniquely to parallel sections of $\cV$ for~\eqref{new_funni_formula}, where
$B$ is characterised by nullity.  First we establish the following converse
and regularity result.

\begin{thm}\label{Lambda_in_nullity}
Let $(M,J,g)$ be a connected \bps/K\"ahler manifold with a non-parallel
solution $A^{\bar ab}$ of the mobility equation, which lifts, over a dense
open subset $U$ of $M$, to a real parallel section $(A^{\bar a b},\hv^a,\mu)$
for~\eqref{new_funni_formula} with $B$ locally constant. Then\textup:
\begin{enumerate}
\item $B$ is constant and $(A^{\bar a b},\hv^a,\mu)$ extends to a parallel
  section over $M$\textup;
\item $G_{a\bar b c}{}^d\hv_d=0$, and hence $(J,g)$ has c-projective
  nullity on the dense open subset where $\hv^a$ is nonzero.
\end{enumerate}
\end{thm}
\begin{proof} As noted in Remark~\ref{rem:lifts}, Theorem~\ref{MetriProlongation}
provides a real section $(A^{\bar a b}, \hv^a, \mu')$ of $\cV$ (defined on all
of $M$) which is parallel for the connection given by \eqref{Metriprol1} and
\eqref{Metriprol2}. On $U$ we compute, using \eqref{second_line} and
\eqref{mob1-int}, that
\begin{equation}
R_{a\bar b}{}^d{}_{f}A^{\bar c f}+R_{a\bar b}{}^{\bar c}{}_{\bar e}A^{\bar e d}
=2B(\delta_{a}{}^dA^{\bar c}{}_{\bar b}-\delta_{\bar b}{}^{\bar c} A_{a}{}^d),
\end{equation}
which implies
\begin{equation}\label{RA_Eq2}
\tfrac{1}{n}(\Ric_{\bar b f} A^{\bar c f}+R_{a\bar b}{}^{\bar c}{}_{\bar e} A^{\bar e a})
=2B(A^{\bar c}{}_{\bar b}-\tfrac{1}{n}\delta_{\bar b}{}^{\bar c} A_{a}{}^a).
\end{equation}
Applying $\nabla_{\bar d}$ to \eqref{RA_Eq2} and taking the trace with respect
to $\bar d$ and $\bar c$ shows that
\begin{equation}\label{RA_Eq3}
\tfrac{1}{(1-n)(n+1)}\bigl((\nabla_{\bar d} \Ric_{\bar b f})A^{\bar df}
+(\nabla_{\bar d} R_{a\bar b}{}^{\bar d}{}_{\bar e})A^{\bar e a}
+ (1-n)\Ric_{\bar b f}\hv^f\bigr)=2B\hv_{\bar b}.
\end{equation}
Recall that
\[
-\mu\delta_{a}{}^b+2BA_{a}{}^b=\nabla_a\hv^b
=-\mu'\delta_{a}{}^b+\Rho_{a\bar c}A^{\bar cb}-\tfrac1n H_{a\bar c}{}^b{}_{d} A^{\bar cd}
\]
and that $\Rho_{a\bar b}=\frac{1}{n+1}\Ric_{a\bar b}$.  Hence, applying
$\nabla_{h}$ to \eqref{RA_Eq3} and taking trace shows, together with the
identities~\eqref{RA_Eq2} and~\eqref{RA_Eq3}, that we have an identity of the
form
\begin{equation}\label{RA_Eq4}
\begin{pmatrix}
R&0&0\\
\nabla R&R&0\\
\nabla\nabla R+R\bowtie R&\nabla R& R
\end{pmatrix} \begin{pmatrix}A^{\bar ab}\\
\hv^a\\ \mu'\end{pmatrix} = 2B\left[\begin{pmatrix} A^{\bar ab}\\
\hv^a\\ \mu'\end{pmatrix}-\frac{A_c{}^c}n\begin{pmatrix}g^{\bar ab}\\
0\\ \frac1n\Rho_d{}^d\end{pmatrix}\right].
\end{equation}
Since the left-hand side of \eqref{RA_Eq4} is defined on all of $M$ and
$(A^{\bar ab}, \hv^a, \mu')$ is a nowhere vanishing section on $M$, the
identity \eqref{RA_Eq4} can be used to extend $B$ smoothly as a function to
all of $M$. Since $B$ is locally constant on $U$ and $M$ is connected, $B$ is
actually a constant and $(A^{\bar ab}, \hv^a,\mu)$ extends smoothly to a
parallel section of the connection \eqref{new_funni_formula} on all of $M$.

The second part is immediate from~\eqref{crucial_identity} with $\nabla_a B=0$
and $\nabla_a\mu=2B\hv_a$.
\end{proof}
\begin{rem}\label{cproj-inv;B=0} When $(J,g)$ has c-projective
nullity, $\Rho_{a\bar b}-2B g_{a\bar b}$ (with $B$ given by ~\eqref{c-nullity}) 
is a metric c-projective invariant by Corollary~\ref{invariant-nullity}, and hence the
connection~\eqref{new_funni_formula} is metric c-projectively invariant. In
particular, by Theorem \ref{Lambda_in_nullity}, the connection is metric c-projectively 
invariant if $B$ is constant and it admits a parallel section with $\hv_a$ nonzero.

On the other hand, if the connection~\eqref{new_funni_formula} admits a
parallel section with $\hv_a=0$, then~\eqref{second_line} shows that $B=0$
unless the corresponding solution $A_{a\bar b}$ of the mobility equation is a
(necessarily locally constant) multiple of $g_{a\bar b}$. Thus a parallel
solution of the mobility equation which is not a multiple of $g$ lifts to a
parallel section for~\eqref{new_funni_formula} if and only if $B=0$.
\end{rem}

\begin{thm}\label{funni_mobility_3} Let $(M,J,[\nabla])$ be a connected metric
c-projective manifold of dimension $2n\geq 4$ arising from a \bps/K\"ahler
metric~$g$ with mobility $\geq 3$. Then either $(J,g)$ has c-projective
nullity on a dense open subset $U\subseteq M$, with $B$ constant
in~\eqref{c-nullity}, or $2n\geq 6$ and all metrics c-projectively equivalent
to $g$ are affinely equivalent to $g$ \textup(i.e.~have the same Levi-Civita
connection\textup).
\end{thm} 

To prove this theorem, we use a couple of lemmas, the first of which is a
purely algebraic (pointwise) result.
\begin{lem}\label{Algebraic_Lemma}
Suppose that $R_{a\bar b c\bar d}$ is a tensor which has K\"ahler
symmetries~\eqref{Kurvature} with respect to $g_{a\bar b}$.  Let $A_{a\bar
  b}$, $\wt A_{a\bar b}$, $\hv_{a\bar b}$ and $\wt\hv_{a\bar b}$
be \textup(real\textup) tensors that satisfy
\begin{align}\label{RA_Identity1}
&R_{a\bar b\bar c e} A_{d}{}^e+R_{a\bar b d\bar f}A^{\bar f}{}_{\bar c}
= g_{a\bar c}\hv_{d\bar b}-g_{d\bar b}\hv_{a\bar c}\\
\label{RA_Identity2}
&R_{a\bar b\bar c e} \wt A_{d}{}^e+R_{a\bar b d\bar f}\wt A^{\bar f}{}_{\bar c}
=g_{a\bar c}\wt\hv_{d\bar b}-g_{d\bar b}\wt\hv_{a\bar c}.
\end{align}
If $A_{a \bar b}$, $\wt A_{a\bar b}$ and $g_{a\bar b}$ are linearly
independent, then $\hv_{a\bar b}$, respectively $\wt\hv_{a\bar b}$, are linear
combinations of $g_{a\bar b}$ and $A_{a \bar b}$, respectively $g_{a\bar b}$
and $\wt A_{a\bar b}$, with the same second coefficient.
\end{lem}
\begin{proof}
Note first that these equations remain unchanged if we add scalar multiplies
of $g_{a\bar b}$ to the tensors $A_{a\bar b}$, $\wt A_{a\bar b }$, $\hv_{a\bar
  b}$ and $\wt\hv_{a\bar b}$.  Hence, we can assume without loss of generality
that the trace of these tensors vanishes.  We then have to show that
$\hv_{a\bar b}$ and $\wt\hv_{a\bar b}$ are a common scalar multiple of
$A_{a\bar b}$ and $\wt A_{a\bar b}$ respectively.

{From} equation \eqref{RA_Identity1} it follows immediately that
\begin{multline}\label{RA_Identity3}
\wt A_{a}{}^h(R_{h\bar b\bar c e} A_{d}{}^e+R_{h\bar b d\bar f} A^{\bar f}{}_{\bar c})
-\wt A^{\bar i}{}_{\bar b}(R_{a\bar i\bar c e} A_{d}{}^e
+R_{a\bar i d\bar f} A^{\bar f}{}_{\bar c})\\
=\wt A_{a\bar c}\hv_{d\bar b} +\wt A_{d\bar b}\hv_{a\bar c}
-g_{a\bar c} \wt A^{\bar i}{}_{\bar b}\hv_{d\bar i}
-g_{d\bar b} \wt A_{a}{}^h\hv_{h\bar c}.
\end{multline}
By the symmetries \eqref{Kurvature} of $R_{a\bar bc\bar d}$, the left-hand side
of identity \eqref{RA_Identity3} equals
\begin{align}\label{RA_Identity4}
(\wt A_a{}^hR_{h\bar b\bar ce}&-\wt A^{\bar i}{}_{\bar b}R_{a\bar i\bar ce})A_d{}^e
+(\wt A_a{}^hR_{h\bar bd\bar f}-\wt A^{\bar i}{}_{\bar b}R_{a\bar id\bar f})
A^{\bar f}{}_{\bar c}\nonumber\\
&=(\wt A_a{}^hR_{h\bar b\bar ce}+\wt A^{\bar i}{}_{\bar b}R_{e\bar ca\bar i})A_d{}^e
 -(\wt A_a{}^hR_{d\bar f\bar bh}+\wt A^{\bar i}{}_{\bar b}R_{d\bar fa\bar i})
A^{\bar f}{}_{\bar c}\nonumber\\
&=A_{a\bar c}\wt \hv_{d\bar b} + A_{d\bar b}\wt \hv_{a\bar c}
-g_{a\bar c} A^{\bar i}{}_{\bar b}\wt\hv_{d\bar i}
-g_{d\bar b} A_{a}{}^h \wt \hv_{h\bar c}
\end{align}
where the last equality follows from \eqref{RA_Identity3}.  {From}
\eqref{RA_Identity3} and \eqref{RA_Identity4} we therefore obtain
\begin{equation}\label{tau_Identity}
g_{a\bar c}\tau_{d\bar b}+g_{d\bar b}\tau_{a\bar c}=A_{d\bar b}\wt\hv_{a\bar c}
+A_{a\bar c}\wt\hv_{d\bar b}-\wt A_{d\bar b}\hv_{a\bar c}
-\wt A_{a\bar c}\hv_{d\bar b},
\end{equation}
where $\tau_{a\bar b}=A_a{}^e\wt\hv_{e\bar b}-\wt A^{\bar f}{}_{\bar b}
\hv_{a\bar f}$. Taking the trace with respect to $a$ and $c$ yields
\begin{equation}
n\tau_{d\bar b}+g_{d\bar b}\tau_{a}{}^a=0,
\end{equation}
which shows that $\tau_{a\bar b}=0$. Therefore, we conclude from
\eqref{tau_Identity} that
\begin{equation*}
A_{d\bar b}\wt\hv_{a\bar c}+A_{a\bar c}\wt\hv_{d\bar b}
=\wt A_{d\bar b}\hv_{a\bar c}+\wt A_{a\bar c}\hv_{d\bar b}.
\end{equation*}
Since any nonzero tensor of this form determines its factors up to scale, and
$A_{a\bar b}$ and $\wt A_{a\bar b}$ are linearly independent, we conclude that
$\hv_{a\bar b}$ and $\wt\hv_{a\bar b}$ are the same multiple of $A_{a\bar b}$
and $\wt A_{a\bar b}$ respectively.
\end{proof}

We next relate linear dependence to pointwise linear dependence.

\begin{lem}\label{Pointwise} Let $(M,J,g)$ be a connected
\bps/K\"ahler $2n$-manifold $(n\geq 2)$ and let $A_a{}^b$ be a solution
of~\eqref{metri-mob} such that $\wt A_a{}^b := p \delta_a{}^b + q A_a{}^b$ is
also a solution for real functions $p$ and $q$. Then $p$ and $q$ are constant
or $A_a{}^b=\xi\delta_a{}^b$ for constant $\xi$.
\end{lem}
\begin{proof} By assumption, we have $\nabla_a A_b{}^c = -\delta_a{}^c \hv_b$
and $\nabla_a \wt A_b{}^c = -\delta_a{}^c \wt\hv_b$, hence
\begin{equation*}
\nabla_a p\,\delta_b{}^c+\nabla_a q\, A_b{}^c = -(\wt\hv_b-q\hv_b)\delta_a{}^c
\end{equation*}
If $\nabla_a q=0$, it follows easily that $p$ and $q$ are locally constant
hence constant. Otherwise, contracting this expression with a nonzero tangent
vector $X^a$ in the kernel of $\nabla_a q$, we deduce that $\wt\hv_b=q\hv_b$
and $\nabla_a p = -\xi \nabla_a q$ for some function $\xi$.  Thus
$\nabla_a q\, (A_b{}^c-\xi\delta_b{}^c)=0$. If
$A_b{}^c=\xi\delta_b{}^c$, it follows from what we have already proven
that $\xi$ is constant.  Otherwise, we deduce that $p$ and $q$ are
constant.
\end{proof}
\begin{proof}[Proof of Theorem~\textup{\ref{funni_mobility_3}}]
Suppose that $A_{a\bar b}$ and $\wt A_{a\bar b}$ are nondegenerate solutions of
the mobility equation such that $g_{a\bar b}, A_{a\bar b}$ and $\wt A_{a\bar
  b}$ are linearly independent.  At each point of $M$,~\eqref{mob1-int} implies
that $A_{a\bar b}$ and $\wt A_{a\bar b}$
satisfy~\eqref{RA_Identity1}--\eqref{RA_Identity2}, with $\hv_{a\bar
  c}=\nabla_a \hv_{\bar c}$ and $\wt\hv_{a\bar c}=\nabla_a\wt\hv_{\bar c}$.  By
Lemma~\ref{Pointwise}, $g_{a\bar b}, A_{a\bar b}$ and $\wt A_{a\bar b}$ are
pointwise linearly independent on a dense open set $U'$, and hence, on $U'$,
Lemma~\ref{Algebraic_Lemma} implies that $A_{a\bar b}$ and $\wt A_{a\bar b}$
lift to smooth solutions $(A_{a\bar b},\hv_a,\mu)$ and $(\wt A_{a\bar b},
\wt\hv_a,\tilde\mu)$ of~\eqref{first_line}--\eqref{second_line} for the same
smooth function $B$. Thus we may apply Lemma~\ref{useful}.

The trace-free parts of $A_a{}^b$ and $\wt A_a{}^b$ are pointwise linearly
independent on $U'$, hence if $n=2$, their common centraliser at each $p\in U'$
consists only of multiples of the identity. By~\eqref{commuting}, $G_{a\bar
  bc}{}^d$ is a multiple $\alpha_{a\bar b}$ of $\delta_a{}^d$, hence zero,
since $G_{a\bar bc}{}^d=G_{c\bar ba}{}^d$. Thus $g$ has constant holomorphic
sectional curvature, which proves the theorem for $2n=4$.

To prove the theorem for $2n\geq 6$, we substitute~\eqref{crucial_identity}
into $G_{a\bar b c}{}^d=G_{c\bar b a}{}^d$ to obtain
\begin{equation*}
g_{c\bar b}(\nabla_a\mu-2B\hv_a)-2(\nabla_aB)A_{c\bar b}=
g_{a\bar b}(\nabla_c\mu-2B\hv_c)-2(\nabla_cB)A_{a\bar b}.
\end{equation*}
If we contract this equation with a vector $Y^c$ in the kernel of $\nabla_c
B$, then since $n\geq 3$, we obtain a degenerate Hermitian form on the left
hand side, equal to a multiple of $g_{a\bar b}$. Hence both sides vanish,
i.e.~$Y^c$ is in the kernel of $\nabla_c\mu-2B\hv_c$ and we have
\begin{equation*}
A_{c\bar b} Y^c = \xi g_{c\bar b} Y^c\qquad\text{and}\qquad 
\nabla_a\mu-2B\hv_a = 2\xi (\nabla_aB)
\end{equation*}
for some function $\xi$ on $U'$. Hence~\eqref{crucial_identity} now reads
\begin{equation*}
G_{a\bar b c}{}^d\hv_d=2(\nabla_aB) (A_{c\bar b}-\xi g_{c\bar b})
= 2(\nabla_cB) (A_{a\bar b}-\xi g_{a\bar b}).
\end{equation*}
If $\nabla_c B$ is nonzero on an open subset of $U'$, it follows that
$A_a{}^b-\xi \delta_a{}^b$ has (complex) rank at most one there, with image
spanned by $\nabla^a B$ and kernel containing the kernel of $\nabla_aB$. Since
the same holds for $\wt A_a{}^b-\tilde \xi\delta_a{}^b$ for some function
$\tilde\xi$, we have that $\delta_a{}^b, A_a{}^b$ and $\wt A_a{}^b$ are
linearly dependent, a contradiction. Hence $\nabla_c B$ is identically zero on
$U'$, i.e.~$B$ is locally constant.  The result now follows from
Theorem~\ref{Lambda_in_nullity}.
\end{proof}

\begin{rem}\label{fkmr} The above proof shows (for mobility $\geq 3$) that any
solution of the mobility equation~\eqref{mob-raised} lifts to a parallel
section for~\eqref{new_funni_formula}, where $B$ is given
by~\eqref{c-nullity}, unless all solutions are parallel (i.e.~affine
equivalent to $g$). However, in the latter case, any solution
of~\eqref{mob-raised} lifts to a parallel section
for~\eqref{new_funni_formula} with $B=0$ (cf.~Remark~\ref{cproj-inv;B=0}).
This establishes \cite[Theorem~5]{FKMR}; the next result may be seen as a
strengthening of this theorem in which c-projective nullity is brought to the
fore, cf.~also~\cite[Theorem~2]{CMR}.
\end{rem}

\begin{thm}\label{new_funni} Let $(M, J, g)$ be a connected \bps/K\"ahler
manifold admitting a solution of the mobility equation that is not a constant
multiple of $g$.  Assume that there is a dense open subset $U\subseteq M$ on
which $(J,g)$ has c-projective nullity and denote by $B$ the function
in~\eqref{c-nullity}.  Then the following hold\textup:
\begin{itemize}
\item $B$ is constant
\item any solution $A^{\bar a b}$ of the mobility equation lifts uniquely to a
  section of $\cV$ which is parallel for the special tractor
  connection~\eqref{new_funni_formula}.
\end{itemize}
\end{thm}

We divide the proof of Theorem~\ref{new_funni} into several propositions. 

\begin{prop}\label{new_funni_A} Under the assumptions of
Theorem~\textup{\ref{new_funni}}, there is a dense open subset $U'\subseteq U$
on which $B$ is smooth, and for any solution $A_{b\bar c}$ of the mobility
equation\textup:
\begin{enumerate}
\item there is a smooth real-valued function $\mu$ on $U'$ such
  that~\eqref{second_line} holds, and if $B$ is locally constant
  then~\eqref{third_line} also holds on $U'$\textup;
\item for any vector $v^a$ in the nullity distribution of $(J,g)$,
\begin{equation}\label{crucial_identity_2}
2(\nabla_aB)A_{c\bar b}v^c+v_{\bar b}(2B\hv_a-\nabla_a\mu)=0.
\end{equation}
In particular, if $v^a$ is not in any eigenspace of $A_a{}^b$ then $B$ is
locally constant.
\end{enumerate}
\end{prop}
\begin{proof} To see that (1) holds for $A_{b\bar c}$, first recall that $\hv^a$,
given by~\eqref{first_line}, is holomorphic. Next, by assumption, at any $p\in
U$ there is a nonzero tangent vector $v^{\bar b}$ such that $v^{\bar
  b}G_{a\bar{b}c\bar{d}}=0$.  Hence, by equation~\eqref{mob1-int}, on $U$ we
have
\begin{equation}\label{RA_Eq} \begin{split}
-v_c\nabla_a\hv_{\bar{d}}+g_{a\bar{d}}v^{\bar{b}}\nabla_{\bar{b}}\hv_{c}
&=v^{\bar{b}}R_{a\bar{b}c\bar{e}}A^{\bar{e}}{}_{\bar{d}}
-v^{\bar{b}}R_{a\bar{b}e\bar{d}}A_c{}^e\\
&=-2B(g_{a\bar{e}}v_c+g_{c\bar{e}}v_a)A^{\bar{e}}{}_{\bar{d}}
+2B(g_{a\bar{d}}v_e+g_{e\bar{d}}v_a)A_c{}^e\\
&=2B(g_{a\bar{d}}v^{\bar{b}}A_{c\bar{b}}-v_cA_{a\bar{d}})
\end{split} \end{equation}
and so
\[
v_cV_{a\bar{d}}-g_{a\bar{d}}v^{\bar{b}}\overline{V}_{\bar{b}c}=0,
\]
where $V_{a\bar{b}}\equiv \nabla_a\hv_{\bar{b}}-2BA_{a\bar{b}}$.
As~$v^{\bar{b}}\neq0$ on $U$, it follows that $V_{a\bar{b}}$ is pure trace,
i.e.~the second equation of (\ref{second_line}) holds pointwise. By
assumption, there is a dense open subset $U'\subseteq U$ on which $g_{a\bar
  b}$ and $A_{a\bar b}$ are pointwise linearly independent for some solution
$A_{a\bar b}$, from which it follows that $B$ is a smooth real-valued function
on $U'$. Hence \eqref{first_line}--\eqref{second_line} hold on $U'$ for any
solution, with $\mu$ smooth on $U'$.

By Lemma~\ref{useful}, any solution satisfies~\eqref{crucial_identity}, which
implies \eqref{crucial_identity_2}. Now if $\nabla_a B=0$ it follows
immediately from the existence of nullity that~\eqref{third_line} holds on
$U'$.
\end{proof}
Proposition~\ref{new_funni_A} and Theorem~\ref{Lambda_in_nullity} have the
following immediate consequence.
\begin{cor} Theorem~\textup{\ref{new_funni}} holds unless the nullity
distribution is contained in an eigendistribution of every solution of
the mobility equations.
\end{cor}
It remains to show that Theorem~\ref{new_funni} also holds when the nullity
distribution is contained in an eigendistribution of every solution of the
mobility equations, and for this it suffices to show that $\nabla_aB=0$ on a
dense open set.  Suppose then that $v^b$ is a nonzero nullity vector such that
$A_{a}{}^bv^a=\xi v^b$ for some smooth function $\xi$, so that
$\nabla_a\mu-2B\hv_a=2\xi \nabla_a B$ by~\eqref{crucial_identity_2} and
hence~\eqref{crucial_identity} reads
\begin{equation}\label{crucial3}
G_{a\bar b c}{}^d\hv_d=2(\nabla_aB) (A_{c\bar b}-\xi g_{c\bar b})
= 2(\nabla_cB) (A_{a\bar b}-\xi g_{a\bar b})
\end{equation}
as in the proof of Theorem~\textup{\ref{funni_mobility_3}}. Since $v^a$ is an
eigenvector of $A_a{}^b$ with eigenvalue $\xi$, $v^{\bar b}$ is an eigenvector
of $A^{\bar a}{}_{\bar b}$ with eigenvalue $\bar\xi$. However $v^{\bar b}$ is
in the nullity of $G$, so the contraction of~\eqref{crucial3} with $v^{\bar
  b}$ yields $(\nabla_cB)(A_{a\bar b}-\xi g_{a\bar b})v^{\bar b}=0$. If we now
combine these observations with Proposition~\ref{non-null-nullity}, we obtain
that either $\nabla_a B=0$ on a dense open set (and we are done) or there is
an open set on which $\wt A_a{}^b:=A_a{}^b-\xi \delta_a{}^b$ has (complex)
rank at most one, $v^a$ is a null vector in its kernel, and $\xi$ is real.
Hence the generalised $\xi$-eigenspace of $A_a{}^b$ is nondegenerate, and so
has (complex) dimension at least two, which implies that $\xi$ is locally
constant by Lemma~\ref{alg-mult}.  Now $\wt A_a{}^b$ is a rank one solution of
the mobility equation with a nonzero (but null) nullity vector in its kernel,
and so Theorem~\ref{new_funni} is a consequence of the following proposition.

\begin{prop}\label{rank1_solutions}
Suppose $(M,J,g)$ is a connected \bps/K\"ahler manifold of dimension $2n\geq
4$ admitting a non-parallel solution of the mobility equation $A_{a\bar b}$ such
that $A_{a}{}^b$ is a complex endomorphism of rank $1$.  Assume that $g$ has
nullity on some dense open set $U\subseteq M$ and that there is a nonzero
vector in the nullity distribution that is in the kernel of $A_{a}{}^b$. Then
the function $B$ defined as in \eqref{c-nullity} is a constant and the
conclusions of Theorem~\textup{\ref{new_funni}} hold.
\end{prop}

Before we give a proof of Proposition \ref{rank1_solutions} we collect some
crucial information about solutions of the mobility equation of rank $1$:

\begin{lem}\label{properties_of_rank1_solution}
Suppose $(M,J,g)$ is a connected \bps/K\"ahler manifold of dimension $2n\geq
4$ admitting a non-parallel solution of the mobility equation $A_{a\bar b}$ such
that $A_{a}{}^b$ is a complex endomorphism of rank $1$.  Assume that $g$ has
nullity on some dense open set $U\subseteq M$ and that there is a nonzero
vector in the nullity distribution that is in the kernel of
$A_{a}{}^b$. Denote by $B$ the function defined as in \eqref{c-nullity}
and let $\hv_a=\nabla_a\hp$ with $\hp=-A_a{}^a$. Then the following holds on a
dense open subset $U'\subseteq U$\textup:
\begin{enumerate}
\item the triple $(A_a{}^b,\hv^a,\mu)$ satisfies system~\eqref{funni_system}
  \textup(and its conjugate\textup) for some smooth nonvanishing real-valued
  function $\mu$\textup;
\item $A_{a\bar b}=\mu^{-1}\hv_a\hv_{\bar b}$ and $-\mu\hp=\hv_a\hv^a$\textup;
\item $\nabla_aB$ is proportional to $\hv_a$, and at any $x\in U'$ either
$\nabla_aB=0$ or the nullity space of $g$ at $x$ lies in the kernel of
$A_{a}{}^b$.
\end{enumerate}
\end{lem}
\begin{proof} Statement (1) follows immediately from~\eqref{Lambda},
Proposition~\ref{new_funni_A} and the existence of a nullity vector in the
kernel of $A_{a}{}^b$. Since $A_{a}{}^b$ has rank $1$, its nonzero eigenvalue
is $-\hp$, and $\hv^a$ is a nonzero section of the corresponding eigenspace by
Corollary~\ref{cor:eigenvectors}.  Thus on the dense open subset $U'\subseteq
U$ where $\hv^a\neq 0$, $A_{a}{}^b=\xi \hv_a\hv^{b}$, with
$\xi=-\hp/(\hv_a\hv^a)$, and differentiating this identity
using~\eqref{funni_system} yields
\[
(\nabla_a\xi+2B\xi^2\hv_a)\hv_{\bar c}=(\xi\mu-1)g_{a\bar c}.
\]
Since the left hand side is simple and $g_{a\bar b}$ nondegenerate both sides
must vanish, which shows that $\xi=\mu^{-1}$, and hence (2) holds. The
identity~\eqref{crucial_identity} may now be written
\[
G_{a\bar b c}{}^d\hv_d=2(\nabla_aB)A_{c\bar b}=2\mu^{-1}(\nabla_aB)\hv_c\hv_{\bar b}
\]
This immediately implies the second statement of (3), while the first
statement follows from the symmetry of $G_{a\bar b c}{}^d$ in $a$ and $c$.
\end{proof}

\begin{proof}[Proof of Proposition~\textup{\ref{rank1_solutions}}]
We have already noted that to prove Proposition \ref{rank1_solutions} it
suffices to show that $B$ is locally constant. By
Lemma~\ref{properties_of_rank1_solution}, $A_{\alpha}{}^\beta$ is of the form
\begin{equation}\label{form_of_A_real}
A_{\alpha}{}^\beta=\tfrac{1}{2\mu}(\hv_\alpha\hv^{\beta}
+\Omega_{\gamma\alpha}{}\hv^\gamma J_{\delta}{}^\beta\hv^\delta).
\end{equation}
Let us write $D\subset TM$ for the distribution defined by the kernel of
$A_{\alpha}{}^{\beta}$ and
\[
P_{\alpha}{}^\beta=\delta_{\alpha}{}^\beta
-\tfrac{1}{\hv_\gamma\hv^\gamma}
(\hv_\alpha\hv^\beta+\Omega_{\delta\alpha}{}\hv^\delta J_{\zeta}{}^\beta\hv^\zeta)
=\delta_{\alpha}{}^\beta
+\tfrac{1}{2\mu\hp}
(\hv_\alpha\hv^\beta+\Omega_{\delta\alpha}{}\hv^\delta J_{\zeta}{}^\beta\hv^\zeta)
\]
for the orthogonal projection $P\colon TM\to D$, where we use
Lemma~\ref{properties_of_rank1_solution}(2) to rewrite
$\hv_\gamma\hv^\gamma=2\hv_c\hv^c=-2\hp\mu$. Note that $(J,g)$ induces by
restriction a complex structure $J^D$ and a $J^D$-invariant metric $g_D$ on
$D$. The projection $P$ also determines a linear connection on $D$ by
\[
\nabla^D_\alpha X^{\beta}=P_{\gamma}{}^\beta\nabla_{\alpha}X^\gamma,
\quad\mathrm{for}\quad X\in \Gamma(D)
\]
which preserves this Hermitian structure on $D$. Since $\hv$ and $J\hv$
commute and preserve $D$, $\cL_\hv P=0$ and $\Gamma(D)$ is generated by
sections commuting with $\hv$ and $J\hv$, which are called \emph{basic}.  For
any basic $X\in \Gamma(D)$, Lemma~\ref{properties_of_rank1_solution}(1)
implies
\begin{equation}\label{X_Lambda}
\nabla_\hv X=\nabla_X\hv=-\mu X
\end{equation}
and hence for any other basic element $Y\in\Gamma(D)$ we compute
\[
(\cL_\hv g_D)(X,Y)=\cL_\hv (g_D(X,Y))=\nabla_\hv (g_D(X,Y))=-2\mu g_D(X,Y).
\]
Since $\cL_\hv\hp=\nabla_\hv\hp=g(\hv,\hv)=-2\hp\mu$, it follows that $\cL_\hv
(\hp^{-1} g_D) = 0$.  Let us now regard $\nabla^D$ as a \emph{partial
  connection} on $D$, i.e.~an operator $\nabla^D\colon \Gamma(D)\to
\Gamma(D^*\otimes D)$. Since $\nabla^D_X\hp=\nabla_X\hp=0$ for any
$X\in\Gamma(D)$, the partial connection $\nabla^D$ preserves $\hp^{-1}
g_D$. Furthermore, its \emph{partial torsion}, given by
\[
\nabla_X^DY-\nabla_Y^DX-P([X,Y])\quad \text{ for }\quad  X,Y\in\Gamma(D),
\]
vanishes.  It follows that $\cL_\hv\nabla^D$ is a section of $D^*\otimes
D^*\otimes D\cong D^*\otimes D^*\otimes D^*$, which is symmetric in the first
two entries and skew in the last two entries, which implies it vanishes
identically. We conclude that $\cL_\hv R^D=0$, where the \emph{horizontal
  curvature} $R^D$ of $\nabla^D$ is defined, for $X,Y,Z\in \Gamma(D)$,
by
\begin{equation*}
R^D(X,Y)(Z)=\nabla_X^D\nabla_Y^DZ-\nabla_Y^D\nabla_X^DZ-\nabla_{P([X,Y])}^D Z.
\end{equation*}
For basic $X,Y,Z\in \Gamma(D)$ we compute via~\eqref{X_Lambda} that
\begin{align*}
\nabla_X^D\nabla_Y^DZ&=P(\nabla_X\nabla_YZ)
-\tfrac{1}{2\hp} g(\nabla_YZ,\hv)X
-\tfrac{1}{2\hp} g(\nabla_YZ,J\hv)JX\\
\nabla^D_{\nabla_X^DY}Z&=P(\nabla_{\nabla_X^Y}Z)
-\tfrac{1}{2\hp} g(\nabla_XY,\hv)Z
-\tfrac{1}{2\hp} g(\nabla_XY,J\hv)JZ.
\end{align*}
Using $g(Y,\hv)=0=g(Y, J\hv)$ for $Y\in \Gamma(D)$, we also obtain, for
$X\in\Gamma(D)$, that
\begin{equation*}
g(\nabla_XY, \hv)=\mu g(X, Y)\qquad\text{and}\qquad
g(\nabla_XY, J\hv)=\mu g(Y, JX),
\end{equation*}
from which we deduce, for (basic) $X,Y,Z\in \Gamma(D)$, that
\begin{equation}\label{RD2}
R^D(X,Y)Z=P(R(X,Y)Z)-\tfrac{\mu}{2\hp}S(X,Y)Z,
\end{equation}
where $S$ is the constant holomorphic sectional curvature tensor defined as
in~\eqref{holomorphic_sectional_curvature}.

Let us write $\Ric^{D}(Y,Z)=\mathrm{trace}(X\mapsto R^D(X, Y)Z)$ and
$\Ric^P(Y,Z)=\mathrm{trace}(X\mapsto P(R(X,Y)Z))$ for the Ricci-type
contractions of $R^D$ and $P(R(X,Y)Z)$. Via the inverse $g_D^{-1}$ of $g_D$,
we view $\Ric^D$ and $\Ric^P$ as endomorphism of $D$, from which viewpoint
equation \eqref{RD2} implies that they are related as follows:
\begin{equation}
\Ric^D=\Ric^P-\tfrac{n\mu}{\hp}\, \Id_D.
\end{equation}
By assumption, at each point of a dense open subset, there is a vector $V$ in
$D$ that lies in the nullity distribution of $g$. Inserting $V$ into equation
\eqref{RD2} yields
\begin{equation}\label{RD3}
R^D(X, V)Z=\bigl(B-\tfrac{\mu}{2\hp}\bigr)S(X,V)Z, 
\end{equation}
which implies that
\begin{equation}\label{RD4}
\hp\Ric^{D}(V)=2n\bigl(B\hp-\tfrac{\mu}{2}\bigr)V.
\end{equation}
Set $C:=B\hp-\frac{\mu}{2}$. By (1) and (3) of
Lemma~\ref{properties_of_rank1_solution} we see that $\nabla_XC=0$ for all
$X\in\Gamma(D)$ and that $\nabla_{J\hv}C=0$. Equation \eqref{RD4} shows that
$V$ is an eigenvector of $\hp\Ric^D$ with eigenvalue $2C$. Since $\cL_\hv
R^D=0$ and $\cL_\hv(\hp g^{-1}_D)=0$, it follows that $\cL_\hv(\hp\Ric^D)=0$,
and hence $\nabla_\hv C=\cL_\hv C=0$ as well. Thus $C$ is locally constant,
which implies that $B$ is locally constant by
Lemma~\ref{properties_of_rank1_solution}(1), and this completes the proof.
\end{proof}

\begin{rem}
Proposition~\ref{rank1_solutions} may alternatively be proved as follows.
Starting with the usual equations
\begin{equation}\label{usual_equations}\begin{array}{l}
\nabla_\alpha\hv_\beta=-\mu g_{\alpha\beta}+2BA_{\alpha\beta}\\[4pt]
\nabla_\alpha\mu=2B\hv_\alpha\\[4pt]
\nabla_\alpha\hp=\hv_\alpha
\end{array}
\quad\mbox{where}\quad\begin{array}{l}
\displaystyle 
A_{\alpha\beta}
=\frac{\hv_\alpha\hv_\beta+K_\alpha K_\beta}{2\mu}\\[4pt]
K^\alpha=J_\beta{}^\alpha\hv^\beta\\[4pt]
\hv_\alpha\hv^\alpha+2\hp\mu=0
\end{array}\end{equation}
we may consider the new metric
\begin{equation}\label{new_metric}\tilde g^{\alpha\beta}\equiv 
\hp g^{\alpha\beta}+(1+2\mu)A^{\alpha\beta}
\end{equation}
and verify from (\ref{usual_equations}) that
\begin{itemize}
\item $\cL_\hv\tilde g^{\alpha\beta}=0$,
\item for any $v^\beta$ such that $A_{\alpha\beta}v^\beta=0$, we have 
\begin{equation}\label{awesome_but_severe}
\wt\Ric_{\alpha\beta}v^\beta =\Ric_{\alpha\beta}v^\beta
-\bigl(2B+\frac{n\mu}\hp+\frac1{2\hp}\bigr)v_\alpha,
\end{equation}
where $\wt\Ric_{\alpha\beta}$ is the Ricci tensor of~$\tilde g^{\alpha\beta}$.
\end{itemize}
Hence, if $v^\beta$ is a nullity vector for $g_{\alpha\beta}$ so that in 
addition $\Ric_{\alpha\beta}v^\beta=2(n+1)B v_\alpha$, then 
\begin{equation}\label{insert_an_eigenvector}
\tilde g^{\alpha\gamma}\wt\Ric_{\gamma\beta}v^\beta
=\bigl(n(2B\hp-\mu)-\tfrac12\bigr)v^\alpha.
\end{equation}
Now, since $\cL_\hv(\tilde g^{\alpha\gamma}\wt\Ric_{\gamma\beta})=0$, it
follows that any eigenvalue of this endomorphism is preserved by the flow of
$\hv^\alpha$.  Therefore,
\[
0=\cL_\hv\bigl(n(2B\hp-\mu)-\tfrac12\bigr)=n\bigl(2\hp\cL_\hv B
+\hv^\alpha(2B\nabla_\alpha\hp-\nabla_\alpha\mu)\bigr)=2n\hp\cL_\hv B.
\]
But from (\ref{usual_equations}) we see that
$0=\nabla_{[\alpha}\nabla_{\beta]}\mu=2(\nabla_{[\alpha}B)\hv_{\beta]}$ whence
$\nabla_\alpha B=0$, as required.

The only drawback with this proof is that
verifying~(\ref{awesome_but_severe}), though straightforward, is
computationally severe, whereas the corresponding identity~\eqref{RD4} in the
previous proof is more easily established. The previous proof may be seen as a
limiting case of the reasoning just given. Specifically, for any constant
$c\neq0$, consider the metric
\begin{equation}\label{metric_with_parameter}\tilde g_{\alpha\beta}
\equiv\frac1\hp g_{\alpha\beta}
+\frac1{\hp^2}\bigl(1+\frac{c}\mu\bigr)A_{\alpha\beta} \quad\mbox{with
inverse}\quad \tilde g^{\alpha\beta}\equiv \hp
g^{\alpha\beta}+(1+\frac\mu{c})A^{\alpha\beta}\end{equation}
to arrive at
\begin{equation}\label{has_RD2_as_limit}
\wt \Ric_{\alpha\beta}v^\beta =\Ric_{\alpha\beta}v^\beta
-\bigl(2B+\frac{n\mu}\hp+\frac{c}{\hp}\bigr)v_\alpha
\end{equation}
instead of (\ref{awesome_but_severe}), an equation in which one can take a
sensible limit as $c\to 0$ essentially to arrive at~(\ref{RD4}) instead
of~(\ref{insert_an_eigenvector}). The metrics (\ref{metric_with_parameter}) and
their invariance $\cL_\hv\tilde g_{\alpha\beta}=0$ can also be 
recognised in the previous proof. More precisely, the first equation from 
(\ref{usual_equations}) can be expressed as 
$\cL_\hv g_{\alpha\beta}=-2\mu g_{\alpha\beta}+4BA_{\alpha\beta}$
or, more compactly, as 
$$\cL_\hv(\hp^{-1}g_{\alpha\beta})
=4\hp^{-1}BA_{\alpha\beta},$$
which implies, using our earlier terminology, that the metric 
$\hp^{-1}g_{\alpha\beta}$ restricted to $D$ is invariant under the flow 
of~$\hv^\alpha$. We also observed in the previous proof that orthogonal 
projection 
$P_\alpha{}^\beta=\delta_\alpha{}^\beta+\hp^{-1}A_\alpha{}^\beta$ onto $D$ 
is invariant under this flow. We are therefore led to invariance of the 
covariant quadratic form 
$$P_\alpha{}^\gamma P_\beta{}^\epsilon\hp^{-1}g_{\gamma\epsilon}
=\hp^{-1}(g_{\alpha\beta}+\hp^{-1}A_{\alpha\beta}),$$ which is the limit of
(\ref{metric_with_parameter}) as $c\to 0$ whilst the nondegenerate metric
$\tilde g_{\alpha\beta}$ is obtained by decreeing that the remaining vectors
$\hv^\alpha$ and $K^\alpha$ at each point be orthogonal to $D$ and each other
and satisfy $\tilde g_{\alpha\beta} \hv^\alpha\hv^\beta=\tilde
g_{\alpha\beta}K^\alpha K^\beta=2c$. The metric (\ref{new_metric}) is the case
that $\hv^\alpha$ and $K^\alpha$ are taken to be orthonormal. In any case, it
follows that $\cL_\hv\tilde g_{\alpha\beta}=0$. \qed
\end{rem}

\subsection{The standard tractor bundle for metric c-projective structures}

The metric theory of the standard tractor bundle $\T$ turns out to be rather
degenerate. For a metric c-projective structure $(M,J,[\nabla])$ induced by
the Levi-Civita connection $\nabla$ of a K\"ahler metric $g$, we have
$\Rho_{ab}=0$ and so the standard tractor connection
\eqref{tractorconna}--\eqref{tractorconnbara} is given by
\begin{equation}\label{tractorconntf}
\nabla^\T_a \begin{pmatrix}X^b\\ \rho \end{pmatrix}
=\begin{pmatrix}\nabla_aX^b+\rho\delta_a{}^b\\ \nabla_a\rho \end{pmatrix}
\qquad
\nabla^\T_{\bar a} \begin{pmatrix}X^b\\ \rho \end{pmatrix}
=\begin{pmatrix} \nabla_{\bar{a}}X^b\\ \nabla_{\bar{a}}\rho-\Rho_{\bar{a}b}X^b
\end{pmatrix}.
\end{equation}
The kernel $\ker D^\T$ of the first BGG operator~\eqref{StandardBGG} consists
of vector fields $X^b$ with c-projective weight~$(-1,0)$ which satisfy
\begin{equation}\label{suppose_that}
\nabla_aX^b+\rho\delta_a{}^b=0\qquad\nabla_{\bar{a}}X^b=0
\end{equation}
for some section $\rho$ of $\cE(-1,0)$; then $\rho=-\frac1 n \nabla_aX^a$,
and, setting the torsion to zero in Proposition~\ref{standard_prolong},
$(X^a,\rho)$ defines a parallel section for the tractor
connection~\eqref{tractorconntf}.  This is similar to the projective case,
with the following distinction: although the tensor $\ts^\alpha$ in
Theorem~\ref{funni_toy} is projectively weighted, the bundle $\cE(1)$ is
canonically trivialised by a choice of metric; here, in contrast, it is the
real line bundle $\cE(1,1)$ that enjoys such a trivialisation, and not the
complex line bundle $\cE(1,0)$.

However, taking care to use~\eqref{curvature_on_densities} (see also
Proposition~\ref{rosetta}), it follows that any solution
of~\eqref{suppose_that} satisfies
\begin{equation}\label{weighted_curvature} \begin{split}
-\delta_b{}^c\nabla_{\bar{a}}\rho
=(\nabla_{\bar{a}}\nabla_b-\nabla_b\nabla_{\bar{a}})X^c
&=R_{\bar{a}b}{}^c{}_dX^d+\Rho_{\bar{a}b}X^c\\
&=H_{\bar{a}b}{}^c{}_dX^d-\delta_b{}^c\Rho_{\bar{a}d}X^d,
\end{split}\end{equation}
where $H_{\bar{a}b}{}^c{}_d=H_{\bar{a}d}{}^c{}_b$ and
$H_{\bar{a}b}{}^b{}_d=0$.  We may rearrange this as
\begin{equation}\label{integrability}
\delta_b{}^c\nabla_{\bar{a}}\rho
=\delta_b{}^c\Rho_{\bar{a}d}X^d-H_{\bar{a}b}{}^c{}_dX^d;
\end{equation}
then the trace over $b$ and $c$ shows that $\nabla_{\bar{a}}\rho
=\Rho_{\bar{a}b}X^b$ (as in Proposition~\ref{standard_prolong}) and hence that
$H_{\bar{a}b}{}^c{}_d X^d=0$. Following the projective case
(Theorem~\ref{funni_toy}), we lower an index in~\eqref{weighted_curvature}
and~\eqref{integrability} to obtain
\begin{equation}\label{half_c-nullity}
R_{\bar{a}b\bar{c}d}X^d=-g_{b\bar{c}}\nabla_{\bar{a}}\rho
-\Rho_{\bar{a}b}X_{\bar{c}}.
\end{equation}
It follows that for any solutions $(X^a,\rho)$ and $(\wt X^a,\tilde \rho)$
of~\eqref{suppose_that},
\[
R_{\bar{a}b\bar{c}d}\wt X^b X^d
=-\wt X_{\bar{c}}\nabla_{\bar{a}}\rho-\Rho_{\bar{a}b}\wt X^b X_{\bar{c}}
=-\wt X_{\bar{c}}\nabla_{\bar{a}}\rho-X_{\bar{c}}\nabla_{\bar{a}}\tilde\rho
\]
and hence, by symmetry,
\[
X_{[\bar{a}}\nabla_{\bar{c}]}\tilde\rho=\wt X_{[\bar{c}}\nabla_{\bar{a}]}\rho.
\]
As in Theorem~\ref{funni_toy}, by first taking $\wt X=X$, we conclude that
there is a real function $B$, uniquely determined and smooth on the union of
the open sets where some solution $X^a$ of~\eqref{suppose_that} is nonzero,
such that for any solution $(X^a,\rho)$ of~\eqref{suppose_that},
\begin{equation}\label{only_possibility}
\nabla_{\bar{a}}\rho=\Rho_{\bar{a}b}X^b=2BX_{\bar{a}}.
\end{equation}

\begin{thm}\label{funny_connection} Let $(M,J,[\nabla])$ be a connected
c-projective manifold, where $\nabla$ preserves a \bps/K\"ahler
metric~$g_{a\bar{b}}$. Suppose that $\dim\ker D^\T\geq 2$. Then there is a
unique constant $B$ such that any element of the kernel of $D^\T$ lifts to a
parallel section of $\T$ for the connection
\begin{equation}\label{funny_connection_formula}
\nabla_a \begin{pmatrix}X^b\\ \rho \end{pmatrix}
=\begin{pmatrix}\nabla_aX^b+\rho\delta_a{}^b\\ \nabla_a\rho
\end{pmatrix}\qquad\quad
\nabla_{\bar a}\begin{pmatrix}X^b\\ \rho \end{pmatrix}
=\begin{pmatrix} \nabla_{\bar{a}}X^b\\ \nabla_{\bar{a}}\rho-2Bg_{\bar{a}b}X^b
\end{pmatrix}.
\end{equation}
\end{thm}
\begin{proof} By Proposition~\ref{standard_prolong}
and~\eqref{only_possibility}, it remains only to show that the smooth function
$B$ is actually a constant.  Differentiating the equation
$2Bg_{\bar{c}b}X^b=\nabla_{\bar{c}}\rho$ and using (\ref{suppose_that}) gives
\begin{equation}\label{what_gives}
2(\nabla_aB)X_{\bar{c}}-2Bg_{\bar{c}a}\rho
=\nabla_a\nabla_{\bar{c}}\rho\quad\mbox{and}\quad
2(\nabla_{\bar{a}}B)X_{\bar{c}}=\nabla_{\bar{a}}\nabla_{\bar{c}}\rho.
\end{equation}
With (\ref{another_curvature_on_densities}), the second equation of
(\ref{what_gives}) implies that $X_{[\bar{c}}\nabla_{\bar{a}]}B=0$. Where
there are two nonzero solutions $X^a$ and $\wt X^a$, it follows
from~\eqref{suppose_that} that the sections $X^a,\wt
X^b\in\Gamma(M,T^{1,0}(-1,0))$ and therefore $X^{[a}{\wt
    X}^{b]}\in\Gamma(M,T^{2,0}(-2,0))$ are holomorphic.  Consequently,
$U=\{X^{[a}\wt X^{b]}\neq0\}$ is the complement of an analytic subvariety and is
thus connected. On $U$ we also have $\wt X_{[\bar{c}}\nabla_{\bar{a}]}B=0$
whence $\nabla_{\bar{a}}B=0$, i.e.~$B$ is locally constant on $U$,
hence constant.
\end{proof}
By analogy with the projective case, one might now expect c-projective nullity
to appear. However, when we combine \eqref{half_c-nullity} and
\eqref{only_possibility}, we obtain
\[
\bigl(R_{\bar ab\bar cd}+\Rho_{\bar ab}g_{\bar cd}
+2Bg_{\bar cb}g_{\bar ad}\bigr)X^d=0,
\]
which is a halfway house on the way to~\eqref{c-nullity_bis}. Underlying this
degeneracy is the fact that $\T$ is associated to a holomorphic representation
of $\mathrm{SL}(n+1,\C)$.

Nevertheless the constant $B$ in Theorem~\ref{funny_connection} is generically
characterised by c-projective nullity in the following degenerate sense.

\begin{thm} Let $(M,J,g)$ be a connected \bps/K\"ahler manifold admitting
a non-parallel solution $X^a$ of~\eqref{suppose_that}. For any function $B$,
the following are equivalent\textup:
\begin{enumerate}
\item $B$ is characterised by c-projective nullity~\eqref{c-nullity_bis} on a
  dense open subset\textup;
\item $B$ is constant and $X^b$ lifts to a section of $\T$ parallel
  for~\eqref{funny_connection_formula}\textup;
\item $\Rho_{\bar a b}=2Bg_{\bar a b}$.
\end{enumerate}
In particular, $g$ is an Einstein metric, and the
connections~\eqref{tractorconntf} and \eqref{funny_connection_formula}
coincide.
\end{thm}
\begin{proof} (1)$\Rightarrow$(2). This follows from Theorem~\ref{new_funni}
because $X^{\bar a}\otimes X^{b}$ is a solution of the mobility equation which
is not a constant multiple of $g^{\bar a b}$, and by
contracting~\eqref{weighted_curvature} by a nullity vector $v^b$.

\smallbreak\noindent(2)$\Rightarrow$(3). The
identity~\eqref{curvature_on_densities} implies
\begin{equation*}
\Rho_{\bar a b}\rho
=(\nabla_{\bar a}\nabla_b-\nabla_b\nabla_{\bar a})\rho
=-2Bg_{\bar a c}\nabla_b X^{c}
=2B g_{\bar a b}\rho,
\end{equation*}
which establishes (3) on the dense open subset $\{\rho\neq 0\}$, hence
everywhere.
\smallbreak
\noindent(3)$\Rightarrow$(1). Since $\nabla_{\bar a}\rho
=\Rho_{\bar a d} X^d=2Bg_{\bar ad}X^d$, equation~\eqref{weighted_curvature} 
implies
\begin{equation}
-2Bg_{\bar ad}X^d\delta_b{}^c
=R_{\bar a b}{}^c{}_dX^d+\Rho_{\bar a b}X^c
=R_{\bar ab}{}^c{}_dX^d+2Bg_{\bar ab}X^c,
\end{equation}
and we deduce that $G_{\bar ab}{}^c{}_dX^d=0$. Hence, $X^d/\rho$ is a nullity
vector for $g$ on the dense open subset $\{\rho\neq 0\}$.
\end{proof}

\subsection{Special tractor connections and the complex cone}

Let $(M,J,[\nabla])$ be a metric c-projective structure. Then for any
compatible metric $g$ and any function $B$, there is a special tractor
connection on $\T$ defined by~\eqref{funny_connection_formula}. We first
observe that the induced connection on $\cV_\C=\overline\T\otimes\T$ is the
special tractor connection~\eqref{new_funni_formula} (for the given $g$ and
$B$). This can be seen easily by taking $A^{\bar b c}=\overline{X^b}X^c$,
$\hv^a=\bar\rho X^a$ and $\mu=\bar\rho\rho$ in~\eqref{new_funni_formula}.
Consequently, parallel sections for the special tractor connection on $\cV$
define parallel Hermitian forms on $\T^*$. This was used in~\cite{CMR} to
characterise, for K\"ahler manifolds $(M,J,g)$, the presence of nontrivial
parallel sections for~\eqref{new_funni_formula} in terms of the local
classification of~\cite{ACG} (see Section~\ref{ss:locclass}). Using the
extension of this classification \bps/K\"ahler manifolds~\cite{BMR}, together
with Remark~\ref{CHSC-fibre}, and Theorems~\ref{Lambda_in_nullity}
and~\ref{new_funni}, we have the following more general characterisation.

\begin{thm}\label{thm:special} Let $(M,J,g)$ be a connected \bps/K\"ahler
manifold admitting a solution $A$ of the mobility equation which is not
parallel \textup(i.e.~$\Lam\neq 0$\textup), and let $B$ be a function on
$M$. Then the following are equivalent.
\begin{enumerate}
\item For the given function $B$, $(J,g)$ has c-projective nullity on a dense
  open set.
\item Any solution to the mobility equation lifts to a
  global parallel section for~\eqref{new_funni_formula} with $B$ constant.
\item On a dense open subset of $M$, $A$ lifts to a parallel section
  for~\eqref{new_funni_formula} with $B$ locally constant.
\item $\Lam$ lies in the c-projective nullity of $(J,g)$ with constant $B$.
\item $B$ is constant and its c-projective nullity distribution contains the
  complex span \textup(tangent to the complex orbits\textup) of the Killing
  vector fields of the pencil $A+t\, \Id$.
\item Near any regular point, $(J,g)$ is given by~\eqref{eq:locclass}, where
  for all $j\in\{1,\ldots\ell\}$, $\Theta_j(t)=\Theta(t)$, a polynomial of
  degree $\ell+1$, with constant coefficients, leading coefficient $4B$, and
  divisible by any constant coefficient factor of the minimal polynomial of
  $A$.
\end{enumerate}
\end{thm}
\begin{proof} (1) $\Rightarrow$ (2) by
Theorem~\ref{new_funni}, (2) $\Rightarrow$ (3) trivially, and (3)
$\Rightarrow$ (4) by Theorem~\ref{Lambda_in_nullity}, and (4) $\Rightarrow$
(1) because $\Lam$ is nonzero on a dense open subset of $M$. Clearly (5)
$\Rightarrow$ (4), and conversely, (1)--(4) imply~\eqref{commuting},
i.e.~$G_{a\bar{b}c}{}^{d}= R_{a\bar{b}c}{}^{d}
+2B(g_{a\bar{b}}\delta_{c}{}^{d}+g_{c\bar{b}}\delta_{a}{}^{d})$ commutes with
$A_a{}^b$: therefore, since $G_{a\bar{b}c}{}^{d}\Lam_d=0$ and $\Lam^\alpha$ is
the sum of the gradients of the nonconstant eigenvalues $\xi_1,\ldots
\xi_\ell$ of $A_a{}^b$, which are sections of the corresponding eigenspaces,
it follows that $\mathrm{grad}_g \xi_j$ and $J\mathrm{grad}_g\xi_j$ are in the
nullity for $j\in\{1,\ldots \ell\}$; this is the tangent distribution to the
complex orbits.

Now (5) implies that the restriction of the metric to the complex orbits has
constant holomorphic sectional curvature, and hence the functions
$\Theta_j(t)$ are equal to a common polynomial $\Theta(t)$ of degree
$\ell+1$. Now comparing~\eqref{second_line} with~\eqref{eq:lc-lam}, we
conclude that $\Theta(t)$ has leading coefficient $a_{-1}=4B$, and that
$g_u(\Theta(A)\cdot,\cdot)=0$ for all irreducible constant coefficient factors
$\rho_u$ of $\chi_A$; thus all constant coefficient factors of the minimal
polynomial of $A$ are also factors of $\Theta$.

Conversely, given (6),~\eqref{eq:lc-lam} implies~\eqref{second_line} with $B$
constant, and hence, by equation~\eqref{crucial_identity} of
Lemma~\ref{useful}, and the Bianchi symmetry $G_{a\bar b c}{}^d=G_{c\bar ba}
{}^d$, we have
\begin{equation*}
Y_{\bar b}(\nabla_a\mu-2B\hv_a)=-G_{a\bar b c}{}^d\hv_dY^c=
g_{a\bar b}(\nabla_c\mu-2B\hv_c)Y^c
\end{equation*}
for any $(1,0)$-vector $Y^c$. Since the left hand side is degenerate, whereas
$g_{a\bar b}$ is nondegenerate, we conclude that $\nabla_a\mu=2B\hv_a$, thus
establishing (3).
\end{proof}

Secondly, we observe that by Lemma~\ref{tangent_bundle_cone}, the special
tractor connection on $\T$ induces a complex affine connection on the complex
affine cone $\pi_\cC\colon \cC\to M$ described in
Section~\ref{cone_construction}.  Combining these observations, as in
Remark~\ref{rem:metric-cone}, any solution $A^{\bar b c}$ of the mobility
equation which lifts to a parallel section of $\cV$ for the special tractor
connection on $\cV$, and is nondegenerate as a Hermitian form on $\T^*$,
induces a Hermitian metric on $\cC$ which is parallel for the induced complex
affine connection on $\cC$.

In particular, if $B$ is constant, then the metric $g$ itself induces a
parallel section of $\cV$ with $A^{\bar b c}=g^{\bar b c}$, $\hv^a=0$ and
$\mu=2B$, which is clearly nondegenerate on $\T$ if and only if $B\neq 0$.

\begin{prop}\label{p:spec-cone} Let $(M,J,[\nabla])$ be a metric c-projective
structure. Then any compatible metric $g$ and any real constant $B\neq 0$
induce a Hermitian metric on the complex affine cone $\cC\to M$ defined by the
c-projective structure.
\end{prop}
For $B>0$, this Hermitian metric is, up to scale, a metric cone $dr^2+r^2 \hat
g$ over a \bps/Sasakian metric $\hat g=g +(d\psi+\alpha)^2$ on a (local)
circle bundle over $(M,J,g)$ whose curvature $d\alpha$ is a multiple of the
K\"ahler form $\omega$ of $g$ (see e.g.~\cite{BG}). In the present context, as
observed by~\cite{MatRos,Mikes}, any solution of the mobility equation which
lifts to parallel section for the special tractor
connection~\eqref{new_funni_formula} on $\cV$ (with $B\neq 0$ constant)
induces a Hermitian $(0,2)$ tensor on $M$ which is parallel for the cone
metric of $g,B$ or $-g,-B$. Combining this with Theorem~\ref{thm:special}, we
have the following extension of one of the key results of~\cite{MatRos} to the
\bps/K\"ahler case.

\begin{thm} Let $(M, J, g)$ be a connected \bps/K\"ahler manifold admitting a
non-parallel solution of the mobility equation, and assume that there is a
dense open subset $U\subseteq M$ on which $(J,g)$ has c-projective nullity.
Then, perhaps after replacing $g$ by a c-projectively equivalent metric,
there is a nonzero constant $B$ such that solutions of the mobility equation
on $M$ are in bijection with parallel Hermitian $(0,2)$ tensors on the cone
$\cC$.
\end{thm}
\begin{proof} Under these assumptions, the equivalent conditions of
Theorem~\ref{thm:special} apply. Hence the constant $4B$ is the leading
coefficient of the common polynomial $\Theta_j(t)=\Theta(t)$ of degree
$\ell+1$ appearing in~\eqref{eq:locclass}, which therefore vanishes if and
only if $\Theta$ has a root at infinity. However, by Remark~\ref{proj-change},
$\Theta$ transforms as a polynomial of degree $\ell+1$ over projective changes
of pencil parameter, and since $\Theta$ is not identically zero, we may change
the pencil parameter so that $\infty$ is not a root, while keeping the
metrisability solution nondegenerate. We thus obtain a c-projectively
equivalent metric with nullity constant $B\neq 0$, and the rest follows from
Theorem~\ref{thm:special}(2), Proposition~\ref{p:spec-cone} and the subsequent
observations (above) applied to this metric.
\end{proof}

By Theorem~\ref{funni_mobility_3}, the hypotheses of this theorem are
satisfied when $(M,J,g)$ has mobility at least $3$ and the compatible metrics
are not all affinely equivalent. In the case that $g$ is positive definite,
this allowed V. Matveev and S. Rosemann~\cite{MatRos} to obtain the following
classification of the possible mobilities of K\"ahler metrics, using the
de Rham (or de Rham--Wu) decomposition of the cone.

\begin{thm} Let $(M, J, g)$ be a simply connected K\"ahler manifold of
dimension $2n\geq 4$ admitting a non-parallel solution of the mobility
equation. Then the mobility of $(g,J)$ has the form $k^2 + \ell$ for
$k,\ell\in\N$ with $0\leq k\leq n-1$, $1\leq \ell \leq (n+1-k)/2$ and
$(k,\ell)\neq (0,1)$, unless $(g,J)$ has constant holomorphic sectional
curvature. Furthermore, any such value arises in this way.
\end{thm}

\subsection{The c-projective Hessian, nullity, and the Tanno equation}

Let $(M,J,g)$ be a \bps/K\"ahler manifold of dimension $2n\geq 4$ and denote
by $\nabla$ the Levi-Civita connection of $g$.  Recall that for any solution
$A_{a\bar b}$ of the mobility equation of $g$ the gradient, respectively the
skew gradient, of the function $A_{a}{}^a=-\hp$ is a holomorphic vector field,
respectively a holomorphic Killing field, which is, by Corollary
\ref{HessKillingPot}, equivalent to the real section
$\sigma=\hp\vol(g)^{-\frac{1}{n+1}}\in\cE(1,1)$ being in the kernel of the
c-projective Hessian. With respect to the Levi-Civita connection and the
trivialisation $\vol(g)^{-\frac{1}{n+1}}$ of $\cE(1,1)$, the c-projective
Hessian equation reads as
\begin{equation}\label{Hessian_in_Tanno_section}
\nabla_a\nabla_b \hp=0, \quad\text{ respectively }\quad
\nabla_{\bar a}\nabla_{\bar b}\hp=0.
\end{equation}
In Section \ref{Section_Prolong_Hessian} we prolonged the c-projective Hessian
equation and have seen that any (real) solution of this equation lifts to a
unique section of the connection
\eqref{Hessian_prolongation}--\eqref{Hessian_prolongation2}, which shows that
\eqref{Hessian_in_Tanno_section} implies that the function $\hp=\bar{\hp}$
satisfies also the following system of equations
\begin{equation}\label{Invariant_Version_Tanno}
\begin{aligned}
&\nabla_a\nabla_b\nabla_c\hp=0\quad\quad\nabla_a\nabla_{\bar b}\nabla_c\hp
=-\Rho_{c\bar b}\nabla_a\hp- \Rho_{a\bar b}\nabla_c\hp
-H_{a\bar b}{}^{ d}{}_{c}\nabla_{d}\hp\\
&\nabla_{\bar{a}}\nabla_b\nabla_c\hp=0 \quad\quad
\nabla_{\bar a}\nabla_{\bar b}\nabla_c\hp
=-\Rho_{c\bar b}\nabla_{\bar a}\hp- \Rho_{c \bar a}\nabla_{\bar b}\hp
-H_{\bar ac}{}^{\bar d}{}_{\bar b}\nabla_{\bar d}\hp.
\end{aligned}\end{equation}
Suppose now that $g$ has mobility $\geq 2$ and that $g$ has nullity on a dense
open subset of $M$. Then Theorem~\ref{new_funni} shows that the function $B$
defined as in \eqref{c-nullity}, is actually constant and any solution
$A_{a\bar b}$ lifts to a (real) section of $\cV_\C$ for the connection
\eqref{new_funni_formula}. The connection \eqref{new_funni_formula} induces a
connection on the dual vector bundle $\cW_\C=\cV_\C^*$, which is given
by
\begin{align}\label{HessianFunny1}
\nabla^{\cW_\C}_a\begin{pmatrix} \hp\\
\mu_b\enskip|\enskip \nu_{\bar{b}}\\ \zeta_{b\bar c} \end{pmatrix}
&=\begin{pmatrix} \nabla_a\hp-\mu_{a}\\ 
\nabla_a\mu_b\enskip|\enskip
\nabla_a\nu_{\bar{b}}+2Bg_{a\bar b}\hp-\zeta_{a\bar b}\\ 
\nabla_a\zeta_{b\bar c}+2Bg_{a\bar c}\mu_{b}
\end{pmatrix}\\
\label{HessianFunny2}
\nabla^{\cW_\C}_{\bar a}\begin{pmatrix} \hp\\
\mu_b\enskip|\enskip \nu_{\bar{b}}\\ \zeta_{b\bar c} \end{pmatrix}
&=\begin{pmatrix} \nabla_{\bar a}\hp-\nu_{\bar a}\\ 
\nabla_{\bar a}\mu_b+2Bg_{\bar a b}\hp-\zeta_{b\bar a}\enskip|\enskip
\nabla_{\bar a}\nu_{\bar{b}}\\
\nabla_{\bar a}\zeta_{b\bar c}+2Bg_{\bar ab}\nu_{\bar c}
\end{pmatrix},
\end{align}
where $\cW_\C$ is identified via $g$ with a direct sum of unweighted
tensor bundles. By Proposition \ref{rosetta} the two equations on the
right-hand side of \eqref{Invariant_Version_Tanno} can be also written as
\begin{equation*}
\nabla_a\nabla_{\bar b}\nabla_c\hp=-R_{a\bar b}{}^{ d}{}_{c}\nabla_{d}\hp
\quad\text{and}\quad \nabla_{\bar a}\nabla_{\bar b}\nabla_c\hp
=-R_{\bar ac}{}^{\bar d}{}_{\bar b}\nabla_{\bar d}\hp.
\end{equation*}
Comparison of these equations with \eqref{HessianFunny1} and
\eqref{HessianFunny2}, shows immediately that any function $\hp$ satisfying
\eqref{Hessian_in_Tanno_section} lifts to a parallel section for the special
connection \eqref{HessianFunny1}--\eqref{HessianFunny2} if and only if the
gradient of $\hp$ lies in the nullity distribution of $g$.  Note that, by
Theorem~\ref{Lambda_in_nullity}, for any solution of the mobility equation
$A_{a\bar b}$ the function $A_{a}{}^a=-\hp$ has this property. Moreover, in
fact, the following Proposition holds:

\begin{prop}\label{Tanno_and_nullity} 
Let $0\neq B\in\R$ be some constant. Suppose $\hp$ be a smooth real-valued
function and write $\hv_a=\nabla_{a}\hp$ for its derivative.  Then the
following statements are equivalent\textup:
\begin{enumerate}
\item
\begin{equation*}
G_{a\bar b c}{}^d\hv_{d}=0\quad \text{ and } \quad\nabla_{a}\hv_{b}=0
\end{equation*}
\item $\hp$ lifts uniquely to a parallel section of the connection
  \eqref{HessianFunny1} and \eqref{HessianFunny2} \textup(and $B$ is
  characterised by \eqref{c-nullity}\textup)
\item  
\begin{equation}\label{Tanno_Eq}\begin{aligned}
&\nabla_a\nabla_b\hv_c=0\quad\text{ and }\quad
\nabla_a\nabla_{\bar b}\hv_c=-2B(\hv_ag_{c\bar b}+\hv_c g_{a\bar b})\\ 
&\nabla_{\bar a}\nabla _b\hv_c=0\quad\text{ and }\quad
\nabla_{\bar{a}}\nabla_{\bar b}\hv_c
=-2B(\hv_{\bar a}g_{c\bar b}+\hv_{\bar b} g_{c\bar a}),
\end{aligned}\end{equation}
\end{enumerate} 
where $G_{a\bar{b}c\bar{d}}\equiv
R_{a\bar{b}c\bar{d}}+2B(g_{a\bar{b}}g_{c\bar{d}}+g_{c\bar{b}}g_{a\bar{d}})$.
\end{prop}
\begin{proof}
In the discussion above we have already observed that $(1)$ is equivalent to
$(2)$. Let us now show that $(1)$ is equivalent to $(3)$, as shown also by 
Tanno~\cite[Prop.~10.3]{Tanno}.  
Suppose first that $(1)$ holds. Then, obviously also
the first two equations on the left-hand side of \eqref{Tanno_Eq} hold.
Moreover, we immediately deduce from $(1)$ that
\begin{equation}\label{Tanno_Eq2}
-2B(\hv_ag_{c\bar b}+\hv_c g_{a\bar b})=R_{a\bar b c}{}^d\hv_d
=\nabla_a\nabla_{\bar b}\hv_c,
\end{equation}
which shows that the first equation of the right-hand side of \eqref{Tanno_Eq}
holds.  The conjugate of \eqref{Tanno_Eq2} implies the second equation of the
right-hand side of \eqref{Tanno_Eq}, since $\nabla_{\bar a}\nabla_{\bar
  b}\nabla_c\hp=\nabla_{\bar a}\nabla_c\nabla_{\bar b}\hp$.
Conversely, suppose now that $(3)$ holds. Then, obviously the identity
\eqref{Tanno_Eq2} is satisfied, which shows that $G_{a\bar b
  c}{}^d\hv_{d}=0$. Hence, it remains to show that $\hv^a$ is a holomorphic
vector field. {From} \eqref{Tanno_Eq} we deduce that
\[
\nabla_a\nabla_b\nabla_{\bar c}\hv_d
=-2B((\nabla_a\hv_b)g_{c\bar d}+(\nabla_a\hv_d) g_{b\bar c}).
\]
Since $R_{ab}{}^c{}_d=0$ and $R_{ab}{}^{\bar c}{}_{\bar d}=0$, skewing in $a$
and $b$ yields
\[
0=-B((\nabla_a\hv_d) g_{b\bar c}-(\nabla_b\hv_d) g_{a\bar c})),
\]
which implies $0=-B(n-1)\nabla_a\hv_d$.
\end{proof}
\begin{rem}\label{Tanno_and_nullity_remark}
If a function $\hp$ satisfies $(1)$ for $B=0$, then this is still
equivalent to $(2)$, and the equivalent statements $(1)$ and $(2)$ imply
$(3)$, but the implication from $(3)$ to $(1)$ is not necessarily true.
\end{rem}

The system of equations \eqref{Tanno_Eq} can also be written as
\begin{equation}\label{Tanno_real}
\nabla_\alpha\nabla_\beta\nabla_\gamma \lambda
=-B(2\hv_{\alpha}g_{\beta\gamma}+g_{\alpha\beta}\hv_\gamma+g_{\alpha\gamma}\hv_\beta
-\Omega_{\alpha\beta} J_{\gamma}{}^\delta\hv_\delta
-\Omega_{\alpha\gamma} J_{\beta}{}^\delta\hv_\delta),
\end{equation}
where $\hv_\alpha=\nabla_\alpha\lambda$. Since on K\"ahler manifolds equation
\eqref{Tanno_real} (respectively \eqref{Tanno_Eq}) was intensively studied by
Tanno in \cite{Tanno}, we refer to this equation as the \emph{Tanno equation}.

\section{Global results}\label{sec:global}

We now turn to the global theory of \bps/K\"ahler manifolds of mobility at
least two. In Section~\ref{ss:locclass} we presented a local
classification~\cite{ACG,BMR} which shows that such a \bps/K\"ahler manifold
$(M,J,g)$ is locally a bundle of toric (in fact, ``orthotoric'') \bps/K\"ahler
manifolds over a local product $S$ of \bps/K\"ahler manifolds. When $M$ is
compact, and $g$ is positive definite, this is not far from being true
globally.

Indeed, several simplifications occur. First, any compact smooth orthotoric
K\"ahler $2\ell$-manifold is biholomorphic (though not necessarily isometric)
to $\CP^\ell$~\cite{ApostolovII}. Secondly, when $g$ is positive definite, the
Hermitian endomorphism $A_a{}^b$ is diagonalisable, and the local
classification~\eqref{eq:locclass} simplifies to give
\begin{align*}
g&=\sum_u \chi_{\mathrm{nc}}(\eta_u) g_u
+\sum_{i=1}^\ell\frac{\Delta_j}{\Theta_j(\xi_j)} d\xi_j^2
+\sum_{j=1}^\ell \frac{\Theta_j(\xi_j)}{\Delta_j}\Bigl(\sum_{r=1}^\ell
\sigma_{r-1}(\hat\xi_j)\theta_r\Bigr)^2,\\
\omega&=\sum_u \chi_{\mathrm{nc}}(\eta_u)\omega_u
+\sum_{r=1}^\ell d\sigma_r\wedge \theta_r,\quad\text{with}\quad
d\theta_r=\sum_u (-1)^r\eta_u^{\ell-r}\omega_u 
\end{align*}
where $\eta_u$ are the (real) constant eigenvalues of $A$, while the
nonconstant eigenvalues $\xi_j$, and functions $\Theta_j$, are all real
valued. Thirdly, the eigenvalues are globally ordered and do not cross.
However, if $\eta_u$ is a root of $\Theta_j$ (for some $u,j$) and $g_u$ is a
Fubini--Study metric on $\CP^{m_u}$, it is possible to have $\xi_j=\eta_u$
along a critical submanifold of $\xi_j$, in which case
$\chi_{\mathrm{nc}}(\eta_u)=0$ along that submanifold, and the corresponding
factor of the base manifold $S$ collapses. We thus have the following global
description~\cite{ApostolovII}.

\begin{thm} Let $(M,J,g)$ be a compact connected K\"ahler manifold admitting
a c-projectively equivalent K\"ahler metric that generates a metrisability
pencil $\ms\circ(A-t\Id)$ of order $\ell$. Then the blow-up of $M$, along the
subvarieties where nonconstant and constant eigenvalues of $A$ coincide, is a
toric $\CP^\ell$-bundle over a complex manifold $S$ covered by a product, over
the distinct constant eigenvalues, of complete K\"ahler manifolds with
integral K\"ahler classes.
\end{thm}

Conversely, such complex manifolds do admit K\"ahler metrics of mobility at
least two, but we refer to~\cite{ApostolovII} for a more precise description.
In particular, examples are plentiful, and have been used to construct
explicit extremal K\"ahler metrics~\cite{ApostolovIII}, including in
particular, weakly Bochner-flat K\"ahler metrics~\cite{ApostolovIV}.

As soon as we impose nullity---for instance, by requiring metrics of mobility
at least three---this plenitude disappears, even in the \bps/K\"ahler case: a
compact connected \bps/K\"ahler $2n$-manifold satisfying the equivalent
conditions of Theorem~\ref{thm:special} is isometric to $\CP^n$, equipped with
a constant multiple of the Fubini--Study metric~\cite{FKMR}, and this rigidity
result extends to compact orbifolds (see~\cite{CMR}).

In the remainder of this section, we focus on positive definite complete
K\"ahler metrics, and begin by showing, in
Section~\ref{Complete_Kaehler_metrics}, that the rigidity result for K\"ahler
metrics with nullity also obtains in this case. In Section~\ref{ssec:yano}, we
then discuss the group of c-projective transformations and the Yano--Obata
Conjecture. This was established for compact \bps/K\"ahler metrics
in~\cite{MR}, and here we show it also holds for complete K\"ahler metrics.

\subsection{Complete K\"ahler metrics with nullity}
\label{Complete_Kaehler_metrics}

In this section we prove the following.

\begin{thm} \label{mobility>2}  Let $g$ be a complete K\"ahler metric on a
connected complex manifold $(M, J)$ of real dimension $2n\geq4$ that has
nullity on a dense open subset of $M$. Then any complete K\"ahler metric on
$(M, J)$ that is c-projectively equivalent to $g$, is affinely equivalent to
$g$, unless there is a positive constant $c\in\R$ such that the K\"ahler
manifold $(M, J, c g)$ is isometric to $(\CP^n, J_{\mathrm{can}}, g_{FS})$.
\end{thm}  

Since, by Theorem~\ref{funni_mobility_3}, a connected K\"ahler manifold $(M,
J, g)$ of degree of mobility at least $3$ has nullity on a dense open set
unless all c-projectively equivalent metrics are affinely equivalent to $g$,
we obtain the following immediate corollary, which in the case of closed
K\"ahler manifolds was proved in \cite[Theorem 2]{FKMR}.

\begin{cor} \label{Cor_mobility>2} Let $g$ be a complete K\"ahler metric on a
connected complex manifold $(M, J)$ of real dimension $2n\geq4$ with mobility
at least $3$.  Then any complete K\"ahler metric on $(M, J)$ that is
c-projectively equivalent to $g$, is affinely equivalent to $g$, unless there
is a positive constant $c\in\R$ such that the K\"ahler manifold $(M, J, c g)$
is isometric to $(\CP^n, J_{\mathrm{can}}, g_{FS})$.
\end{cor}

\begin{rem} 
It is easy to construct a complete K\"ahler manifold $(M, J, g)$ of
nonconstant sectional holomorphic curvature with mobility $\geq 3$ such that
all complete K\"ahler metrics on $(M, J)$ are affinely equivalent to
$g$. Indeed, take the direct product
$$(M_1, g_1, J_1) \times (M_2, g_2, J_2) \times (M_3, g_3, J_3)$$
of three K\"ahler manifolds. It is again a K\"ahler manifold $(M,J,g)$ with
complex structure $J:=J_1+ J_2 + J_3$ and K\"ahler metric $g:=g_1+ g_2 + g_3$.
Obviously, $(M,J,g)$ has mobility $\geq 3$, since $c_1 g_1 + c_2 g_2 + c_3
g_3$ is again a K\"ahler metric on $(M,J)$ for any constants $c_1, c_2,
c_3>0$. Note also that all these K\"ahler metrics are affinely equivalent to
$g$ and that, if $(M_i, g_i, J_i)$ is complete for $i=1,2,3$ the metric $c_1
g_1 + c_2 g_2 + c_3 g_3$ is also complete for any constants $c_1, c_2, c_3>0$.
\end{rem}

\begin{rem}
Suppose that $\tilde g$ is \bps/K\"ahler metric that is compatible with
$[\nabla^g]$, then we may write
\begin{equation}\label{M1}
\vol(\tilde g)=e^{(n+1)f}\vol(g) \qquad \text{and}\qquad
\scale_{\tilde g}=e^{-f}\scale_g,
\end{equation}
where $f=\frac{1}{n+1}\log\bigl|\frac{\vol(\tilde g)}{\vol(g)}\bigr|$.  We
have seen in Section \ref{almost_c-projective} that the Levi-Civita
connections $\wt\nabla$ and $\nabla$ of $\tilde g$ and $g$ are related
by~\eqref{cprojchange} with $\Upsilon_\alpha=\nabla_\alpha f$.  For later use,
note that that $\wt \nabla_\alpha \tilde g_{\beta\gamma}=0$ implies that
the derivative $\nabla_\alpha f$ of $f$ satisfies the equation
\begin{equation}\label{LC}
\nabla_\alpha \tilde g_{\beta\gamma}=(\nabla_\alpha f)\tilde g_{\beta\gamma}+
\tilde g_{\alpha(\beta}\nabla_{\gamma)} f
- J_\alpha{}^\delta\tilde g_{\delta (\beta} J_{\gamma)}{}^\epsilon \nabla_{\epsilon}f.
\end{equation}
Thus $\nabla=\wt\nabla$, i.e.~$g$ and $\tilde g$ are affinely equivalent, if
and only if $\nabla_\alpha f=0$.
\end{rem}

In order to prove Theorem~\ref{mobility>2} suppose that $g$ is a complete
K\"ahler metric on a complex connected manifold $(M,J)$ of dimension $2n\geq4$
with nullity on a dense open set. Further we may assume that $g$ has mobility
at least $2$, since otherwise Theorem~\ref{mobility>2} is trivially satisfied.
Then Theorem~\ref{new_funni} implies that the function $B$ defined as in
\eqref{c-nullity_bis} is actually a constant. We shall see that for $B>0$, the
theory of the Tanno equation implies that $(M,J,g)$ has positive constant
holomorphic sectional curvature. For $B\leq 0$, we will show that any
complete K\"ahler metric that is c-projectively equivalent to $g$, is
necessarily affinely equivalent to $g$. The proof will make essential use of
the following two Lemmas. Recall for this purpose that a geodesic which is
orthogonal to a Killing vector field at one point is orthogonal to it at all
points.

\begin{lem}\label{ODE_Lemma} 
Suppose $(M,J,g)$ is a K\"ahler manifold of dimension $2n\geq 4$ with mobility
$\geq 2$ and nullity on a dense open subset of $M$. Let $\tilde g$ be a
K\"ahler metric that is c-projectively equivalent to $g$ and $f$ the function
defined as in \eqref{M1}, whose derivative
$\Upsilon_{\alpha}\equiv\nabla_\alpha f$ relates the Levi-Civita connections
of $\tilde g$ and $g$ as in \eqref{cprojchange}.  Consider a geodesic $c=c(t)$
that is orthogonal to the canonical Killing fields of the pair $(g,\tilde g)$,
defined as in Theorem~\textup{\ref{Commuting_Killing_objects}}.  Then, for the
constant $B\in\R$ defined as in Theorem~\textup{\ref{new_funni}}, the function
$f(t)=f(c(t))$ satisfies the following ordinary differential equation:
\begin{equation}\label{ODE_along_c}
\dddot f(t)= -4 B g(\dot c, \dot c) \dot f(t)
+ 3 \dot f(t)\ddot f(t)-(\dot f(t))^3.
\end{equation}
\end{lem}
\begin{proof} 
By Corollary~\ref{invariant-nullity} and identity \eqref{changeRho} we have
\begin{equation}\label{2nd_diff_f}
\nabla_{\alpha}\Upsilon_\beta=-2\wt B\tilde g_{\alpha\beta}+2Bg_{\alpha\beta}
+\tfrac{1}{2}(\Upsilon_\alpha\Upsilon_\beta
-J_{\alpha}{}^\gamma J_{\beta}{}^\delta\Upsilon_\gamma\Upsilon_\delta).
\end{equation}
Differentiating \eqref{2nd_diff_f} and inserting \eqref{LC} yields
\begin{align*}
\nabla_\alpha\nabla_\beta \Upsilon_\gamma&=\\&
-2\wt B(\Upsilon_\alpha\tilde g_{\beta\gamma}+\tilde g_{\alpha(\beta}\Upsilon_{\gamma)}
-J_{\alpha}{}^\delta\tilde g_{\delta(\beta}J_{\gamma)}{}^\epsilon\Upsilon_{\epsilon})
+\tfrac{1}{2}(\nabla_{\alpha}(\Upsilon_\beta\Upsilon_\gamma)
-\nabla_{\alpha}(J_{\beta}{}^\delta J_{\gamma}{}^\epsilon\Upsilon_\delta\Upsilon_\epsilon)).
\end{align*}
Substituting for $-2\wt B\tilde g$ the expression \eqref{2nd_diff_f} we
therefore obtain
\begin{align*}
\nabla_\alpha\nabla_\beta &\Upsilon_\gamma=
(\nabla_\beta\Upsilon_\gamma) \Upsilon_\alpha-2B g_{\beta\gamma}\Upsilon_\alpha
-\tfrac{1}{2}(\Upsilon_\beta\Upsilon_\gamma
-J_{\beta}{}^\delta J_{\gamma}{}^\epsilon\Upsilon_\delta\Upsilon_\epsilon)\Upsilon_\alpha\\
&+ (\nabla_\alpha\Upsilon_{(\beta})\Upsilon_{\gamma)}-2B g_{\alpha(\beta}\Upsilon_{\gamma)}
-\tfrac{1}{2}(\Upsilon_\alpha\Upsilon_{\beta}\Upsilon_{\gamma}
-J_{\alpha}{}^\delta\Upsilon_\delta\Upsilon_{\epsilon}J_{(\beta}{}^\epsilon\Upsilon_{\gamma)})\\
&- J_{\alpha}{}^\delta(\nabla_\delta\Upsilon_{(\beta})J_{\gamma)}{}^{\zeta}\Upsilon_{\zeta}
-2B g_{\delta(\beta}J_{\gamma)}{}^{\zeta}\Upsilon_{\zeta}
-\tfrac{1}{2}(\Upsilon_\delta\Upsilon_{(\beta} J_{\gamma)}{}^{\zeta}\Upsilon_{\zeta}
-J_{\delta}{}^\eta\Upsilon_\eta \Upsilon_\epsilon
J_{(\beta}{}^\epsilon J_{\gamma)}{}^{\zeta}\Upsilon_{\zeta}))\\
&+\tfrac{1}{2}(\nabla_{\alpha}(\Upsilon_\beta\Upsilon_\gamma)
-\nabla_{\alpha}(J_{\beta}{}^\delta J_{\gamma}{}^\epsilon\Upsilon_\delta\Upsilon_\epsilon)).
\end{align*}
Note that the determinant of the complex endomorphism $A_{a}{}^b$ relating $g$
and $\tilde g$ as in \eqref{eqL} is given by $e^{-f}$.  Hence the canonical
Killing field $\wt K^\beta(0)=J^{\alpha\beta}\nabla _{\alpha} \det
A=J^{\alpha\beta}\nabla _{\alpha} e^{-f}$, defined as in
\eqref{def_canonical_Killing}, is proportional to
$J^{\alpha\beta}\Upsilon_\alpha =J^{\alpha\beta}(\nabla_\alpha f)$ and we have
$\Upsilon_\beta J_{\alpha}{}^{\beta}\dot c^{\alpha}=0$. Thus, contracting the
above equation with $\dot c^{\alpha}\dot c^{\beta}\dot c^{\gamma}$ all terms
involving the complex structure $J$ disappear and we derive that $f(t)$
satisfies the desired ODE.
\end{proof}

\begin{lem}\label{Reparametrisation_Lemma} Let $g$ and $\tilde g$ be
c-projectively equivalent metrics on a complex manifold $(M, J)$ of dimension
$2n\geq 4$.  Let $c$ be a geodesic of $g$ that is orthogonal to all canonical
Killing fields of $(g,\tilde g)$ \textup(as defined in
Theorem~\textup{\ref{Commuting_Killing_objects}}\textup). Then there is a
reparametrisation $\phi$ such that $\tilde c(t)=c (\phi(t))$ is a geodesic of
$\tilde g$. The inverse $\tau$ of the reparametrisation $\phi$ satisfies the
formula
\begin{equation}\label{reparametrisation}
\frac{d}{dt} f(t)=\frac{d}{dt}\log\Bigl|\frac{d\tau}{dt}\Bigr|,
\end{equation}
where $f(t)=f(c(t))$ with $f$ defined as in \eqref{M1}.
\end{lem}
\begin{proof}
Let $c$ be a geodesic of $g$ that is orthogonal to the canonical Killing
fields associated to $(g,\tilde g)$. Then formula \eqref{cprojchange}
for the difference of the Levi-Civita connections $\wt\nabla$ and $\nabla$
implies
\begin{equation*}
\wt\nabla_{\dot c}\dot c
=\Upsilon(\dot c)\dot c-\Upsilon(J\dot c)J\dot c
=\Upsilon(\dot c)\dot c,
\end{equation*}
where the last identity follows from the fact that $\Upsilon_\beta
J^{\alpha\beta}$ is proportional to a canonical Killing field, as explained in
proof of Lemma \ref{ODE_Lemma}.  Hence, $\wt\nabla_{\dot c}\dot c$ is a
multiple of $\dot c$ and therefore there exists a reparametrisation $\phi$ of
$c$ such that $\tilde c(t)=c(\phi(t))$ is geodesic of $\tilde
g$. Differentiating $c(t)=\tilde c(\tau(t))$, where $\tau$ denotes the inverse
of $\phi$, gives $\dot c(t)=\dot\tau(t)\dot{\tilde c}(\tau(t))$ and hence
\begin{align*}
0=\nabla_{\dot c(t)} \dot c(t)
=\nabla_{\dot c(t)}(\dot\tau(t)\dot{\tilde c}(\tau(t)))
&=\dot\tau(t)^2\nabla_{\dot{\tilde c}(\tau(t))}{\dot{\tilde c}(\tau(t))}
+\ddot{\tau}(t){\dot{\tilde c}(\tau(t))}\\
&=\left (\ddot{\tau}(t)-\dot{\tau}(t)^2\Upsilon(\dot{\tilde c}(\tau(t)))\right)
\dot{\tilde c}(\tau(t)),
\end{align*}
since $\tilde c$ is a geodesic of $\tilde g$.  This implies that
$\ddot\tau(t)=\dot\tau(t)^2\Upsilon(\dot{\tilde c}(\tau(t)))=\dot
\tau(t)\Upsilon(\dot c(t)),$ which is equivalent to \eqref{reparametrisation},
since $\Upsilon_\alpha=\nabla_\alpha f$.
\end{proof}

\begin{proof}[Proof of Theorem~\textup{\ref{mobility>2}}] Let $(M,J,g)$ be a
connected complete K\"ahler manifold with nullity on a dense open set. We may
assume that $(M,J,g)$ has mobility at least $2$. Hence, by
Theorem~\ref{new_funni} any solution of the mobility equation lifts uniquely
to a parallel section of the connection
\eqref{new_funni_formula} with $B$ constant. This in turn, by Theorem
\ref{Lambda_in_nullity}, shows that for any solution $A^{\bar b c}$ of the
mobility equation of $g$ the function $\hp=-A_{a}{}^a$ satisfies (1) of
Proposition \ref{Tanno_and_nullity}.  Hence, by Proposition
\ref{Tanno_and_nullity} and Remark \ref{Tanno_and_nullity_remark} $\hp$
satisfies the Tanno equation \eqref{Tanno_Eq}.  Tanno showed in \cite{Tanno}
that on a complete connected K\"ahler manifold and for a constant $B>0$, the
existence of a nonconstant solution $\hp$ of the Tanno equation
\eqref{Tanno_Eq} implies that $(M,J,g)$ has positive constant holomorphic
sectional curvature, which in turn implies that $(M,J,g)$ is actually closed
and isometric to $(\CP^n, J_{\mathrm{can}}, cg_{FS})$ for some positive
constant $c$. Since any metric that is c-projectively but not affinely 
equivalent to $g$ gives rise to a non-parallel solution of the mobility
equation of $g$ and hence to a nonconstant solution of the Tanno equation, 
Theorem \ref{mobility>2} holds provided $B$ is positive.

It remains to consider the case that the constant $B$ defined as in
Theorem~\ref{new_funni} is nonpositive. Let $\tilde g$ be another complete
K\"ahler metric on $(M,J)$, which is c-projectively equivalent to $g$. Denote
by $f$ again the function defined as in \eqref{M1}, which has the property
that $\Upsilon_\alpha=\nabla_\alpha f$ relates the Levi-Civita connections of
$g$ and $\tilde g$ as in \eqref{cprojchange}.  We will show that $B\leq 0$ 
implies $\Upsilon_\alpha\equiv 0$, that is, $g$ and $\tilde g$ are necessarily 
affinely equivalent.

Note first that, since the canonical Killing fields associated to $(g,\tilde
g)$ are Killing for both metrics by Theorem~\ref{Commuting_Killing_objects}
and $f$ is constructed in a natural way only from the pair $(g,\tilde g)$, the
local flows of the canonical Killing fields preserve $f$. Hence, the canonical
Killing fields lie in the kernel of $\Upsilon$. To show that $\Upsilon$ is
identically zero, it therefore remains to show that $\Upsilon$ vanishes when
inserting vector fields orthogonal to the canonical Killing fields.

Consider a parametrised geodesic $c$ of $g$, which at one (and hence at all
points) is orthogonal to the canonical Killing fields.
Since $g$ and $\tilde g$ are complete, $c$ is defined for all times and $\tau$
from Lemma \ref{Reparametrisation_Lemma} is a diffeomorphism of $\R$. Without
loss of generality we assume that $\dot \tau$ is positive, otherwise replace
$t$ by $-t$.  By Lemma \ref{Reparametrisation_Lemma}, the function
$\tau\colon\R \to\R$ satisfies \eqref{reparametrisation}, which we
rewrite as
\begin{equation} \label{un2}
f(t) = \log(\dot\tau(t)) + \mathrm{const}_0 .
\end{equation} 
Now let us consider equation \eqref{ODE_along_c} and set $\dot
\tau(t)=(p(t))^{-1}$.  Substituting \eqref{un2} into \eqref{ODE_along_c}
yields
\begin{equation}
\dddot p = -4  B g(\dot c, \dot c) \dot p \label{un3a}.  
\end{equation}  
If $B=0$, the equation simplifies to $\dddot p=0$ and its general solution is
of the form $$p(t)=C_2 t^2+C_1 t+C_0,$$ where $C_i$ is a real constant for
$i=0,1,2$.  Hence, we get
\begin{equation} 
\tau(t) = \int_{t_0}^t \frac{d\xi }{C_2 \xi^2 + C_1 \xi +C_0}\ \ +\mathrm{const}.
\end{equation}  
If the polynomial $p(t)=C_2 t^2 + C_1 t + C_0$ has real roots (which is always
the case if $C_2=0$, $C_1\neq 0$), then the integral starts to be infinite in
finite time. If the polynomial has no real roots, but $C_2\neq 0$, the
function $\tau$ is bounded. Thus, the only possibility for $\tau $ to be a
diffeomorphism is $C_2=C_1=0$ implying $\dot\tau =\frac{1}{ C_0}$, which shows
that $f$ is constant along the geodesic $c$.

If $B<0$, the general solution of equation \eqref{un3a} is 
\begin{equation}\label{un10}
C + C_+e^{2\sqrt{-Bg(\dot c, \dot c)}\cdot t}+ C_-e^{-2\sqrt{-Bg(\dot c, \dot c)}\cdot  t},
\end{equation}
for real constants $C$, $C_+$ and $C_-$. Hence, $\tau$ is of the form 
\begin{equation}
\tau(t)  =  \int_{t_0}^t \frac{d\xi}{C + C_+e^{2\sqrt{-Bg(\dot c, \dot c)}\xi }
+ C_-e^{-2\sqrt{-Bg(\dot c, \dot c)}\xi}} \ \  + \mathrm{const}. \label{un7}
\end{equation}
If one of the constants $C_+, C_-$ is not zero, the integral \eqref{un7} is
bounded from one side, or starts to be infinite in finite time.  In both
cases, $\tau $ is not a diffeomorphism of $\R$. The only possibility for
$\tau\colon\R\to\R$ to be a diffeomorphism is when $C_+=C_-=0$, in which case
$\dot\tau$ is constant implying $f$ is constant along the geodesic $c$. Hence,
in both cases ($B=0$ and $B<0$) the one form $\Upsilon_\alpha=\nabla_\alpha f$
vanishes when inserting vector fields orthogonal to the canonical Killing
fields.
\end{proof}

\subsection{The Yano--Obata Conjecture for complete K\"ahler manifolds}
\label{ssec:yano}

For a K\"ahler manifold $(M, J, g)$ let us write $\mathrm{Isom}(J, g)$,
$\mathrm{Aff}(J, g)$ and $\mathrm{CProj}(J, g)$ for the group of complex
isometries, the group of complex affine transformations (i.e.~of complex
diffeomorphisms preserving the Levi-Civita connection) and the group of
c-projective transformations of $(M, J, g)$ respectively. By definition of
these groups we obtain the following inclusions
$$\mathrm{Isom}(J, g)\subseteq\mathrm{Aff}(J, g)\subseteq\mathrm{CProj}(J, g)$$
and consequently we also have
$$\mathrm{Isom}_0(J, g)\subseteq\mathrm{Aff}_0(J, g)
\subseteq\mathrm{CProj}_0(J, g),$$
where subscript $0$ denotes the connected component of the identity. 

Recall that Lie groups of affine transformations of complete Riemannian
manifolds are well understood; see for example \cite{lich:book}.  As explained
there, if a connected Lie group $G$ acts on a simply-connected complete
Riemannian manifold $(M^n, g)$ by affine transformations, then there exists a
Riemannian decomposition
\[
(M^n, g) = (M_1^{n_1} , g_1) \times (\R^{n_2}, g_{\mathrm{euc}})
\]
of $(M^n, g)$ into a direct product of a Riemannian manifold $(M_1^{n_1},
g_1)$ and a Euclidean space $(\R^{n_2}, g_{\mathrm{euc}})$ such that $G$ acts
componentwise. Specifically, it acts on $(M_1^{n_1} , g_1)$ by isometries and
on $(\R^{n_2}, g_{\mathrm{euc}})$ by affine transformations
(i.e.\,compositions of linear isomorphisms and parallel translations). Note
that this implies that for closed simply-connected Riemannian manifolds one
always has $\mathrm{Isom}_0(J, g) = \mathrm{Aff}_0(J, g)$, which holds in fact
for any closed (not necessarily simply-connected) Riemannian manifold; see
\cite[Theorem 4]{Yano1}.
If in addition $(M^n, g)$ is K\"ahler for a complex structure $J$, and $G$ is
a connected Lie group of complex affine transformations, then $(M_1^{n_1},
J_1, g_1)$ and $(\R^{n_2}, J_{\mathrm{can}}, g_{\mathrm{euc}})$ are also
K\"ahler; furthermore, $G$ acts on $(M_1^{n_1}, J_1, , g_1)$ by complex
isometries and on $(\R^{n_2},J_{\mathrm{can}},g_{\mathrm{euc}})$ by complex
affine transformations.

For the complex projective space $\CP^n$ equipped with its natural complex
structure $J_{\mathrm{can}}$ and the Fubini--Study metric $ g_{FS}$ we have
$\mathrm{Aff}(J_{\mathrm{can}}, g_{FS})\ne \mathrm{CProj}(J_{\mathrm{can}},
g_{FS})$.  To see this recall from the introduction that the $J$-planar curves
of $(\CP^n, J_{\mathrm{can}}, g_{FS})$ are precisely those smooth regular
curves that lie within complex lines.  Moreover, recall that any complex
linear isomorphism of $\C^{n+1}$ induces a complex transformation of $\CP^n$
that sends complex lines to complex lines and hence induces a c-projective
transformation of $(\CP^n, J_{\mathrm{can}}, g_{FS})$ by Proposition
\ref{CprojTransCharacterisation}. In fact, Proposition \ref{CprojTrans_CPn}
shows that all c-projective transformations of $(\CP^n, J_{\mathrm{can}},
g_{FS})$ arise in this way, i.e.~$\mathrm{CProj}(J_{\mathrm{can}}, g_{FS}) $
can be identified with the connected Lie group
$\textrm{PGL}(n+1,\C)\cong\textrm{PSL}(n+1,\C)$.  Note that an element in
$\textrm{GL}(n+1,\C)$ induces a complex isometry on $(\CP^n, J_{\mathrm{can}},
g_{FS})$ if and only if it is proportional to a unitary automorphism of
$\C^{n+1}$, which shows that $\textrm{Isom}(J_{\mathrm{can}}, g_{FS})$ can be
identified with the connected Lie group
$\textrm{PU}(n+1)=\textrm{U}(n+1)/\textrm{U}(1)$ of projective unitary
transformations.  For $(\CP^n, J_{\mathrm{can}}, g_{FS})$ we also
clearly have $\mathrm{Isom}(J_{\mathrm{can}},
g_{FS})=\mathrm{Aff}(J_{\mathrm{can}}, g_{FS})$.
 
In \cite{MR} in conjunction with \cite{FKMR} it was shown that any closed
connected K\"ahler manifold $(M, J, g)$ of dimension $2n\geq 4$ with the
property that $\mathrm{CProj}_0(J, g)$ contains $\mathrm{Isom}_0(J,
g)=\mathrm{Aff}_0(J, g)$ as a proper subgroup is actually isometric to
$(\CP^n, J, c g_{FS})$ for some positive constant $c$.  This rigidity result
answers affirmatively in the case of closed K\"ahler manifolds the so-called
\emph{Yano--Obata Conjecture}, which is a c-projective analogue of the
\emph{Projective} and \emph{Conformal Lichnerowicz--Obata Conjectures}; see
the introductions of the papers \cite{Matveev2007,MR} for a historical
overview.  Most recently the conjecture has been proved for closed
(pseudo-)K\"ahler manifolds of all signatures in \cite{BMR}.  We now show that
the Yano--Obata Conjecture also holds for complete connected K\"ahler manifolds.

\begin{thm}[Yano--Obata Conjecture] \label{thm:obata}
Let $(M, g, J)$ be a complete connected K\"ahler manifold of real dimension
$2n\geq 4$. Then $\mathrm{Aff}_0(g, J) = \mathrm{CProj}_0(g, J)$, unless $(M, g,
J)$ is actually compact and isometric to $(\CP^n, J_{\mathrm{can}}, c
g_{FS})$ for some positive constant $c\in \R$.
\end{thm}

\begin{rem} In the projective case, a stronger version of the analogous
Lichnerowicz--Obata result has recently been established~\cite{Matveev2016}: 
on a complete Riemannian manifold $(M,g)$ of dimension $n\geq 2$, the quotient
of the projective group $\mathrm{Proj}(g)$ by the affine group
$\mathrm{Aff}(g)$ has at most two elements unless $(M,g)$ has constant
positive sectional curvature. It would be natural to establish such a result
in the c-projective case.
\end{rem}

\subsection{The proof of the Yano--Obata Conjecture}\label{Proof_Yano_Obata}

Note that if the mobility of $(M, J, g)$ is $1$, then any metric $\tilde g$
that is c-projectively equivalent to $g$ is homothetic to $g$. In particular,
any c-projective transformation has to preserve the Levi-Civita connection of
$g$ and hence $\mathrm{Aff}_0(g, J) = \mathrm{CProj}_0(g, J)$ in this case.
On the other hand, since the pullback of a complete K\"ahler metric by a
c-projective transformation is again a complete K\"ahler metric,
Theorem~\ref{thm:obata} follows from Corollary~\ref{Cor_mobility>2} or
Theorem~\ref{mobility>2} in the case that $g$ has mobility $\geq 3$ or has
mobility $\geq 2$ and nullity on a dense open set.

For the rest of this section we will therefore assume that $(M,J,g)$ is a
connected complete K\"ahler manifold of dimension $2n\geq 4$ with mobility
$2$. To show that Theorem~\ref{thm:obata} holds in this case (which for closed
K\"ahler manifolds was proved in~\cite{MR}), let us write $\mathrm{Sol}(g)$
for the $2$-dimensional solution space of the mobility equation of $g$, which
we view as a linear subspace of the space of $J$-invariant sections in
$S^2TM(-1,-1)$.

Suppose now that $\mathrm{Aff}_0(g, J)$ does not coincide with
$\mathrm{CProj}_0(g, J)$. Then there exists a complete c-projective vector
field $V$ that is not affine. Since the flow $\Phi_t$ of $V$ acts on $(M,J,g)$
by c-projective transformations, for any $t\in \R$ and for any
$\ms\in\mathrm{Sol}(g)$ the pullback $\Phi_t^*\ms$ is an element of the vector
space $\mathrm{Sol}(g)$. Hence, the Lie derivative
$\cL_V\ms=\frac{d}{dt}|_{t=0}\Phi_t^*\ms$ can also be identified with an
element of $\mathrm{Sol}(g)$, which implies that $\cL_V$ induces a linear
endomorphism of $\mathrm{Sol}(g)$.  By the Jordan normal form, in a certain
basis $\ms, \sms\in \mathrm{Sol}(g)$, the linear endomorphism
$\cL_V\colon\mathrm{Sol}(g)\to\mathrm{Sol}(g)$ corresponds to a matrix of one
of the following three forms:
\begin{equation} \label{threecases} 
\left(\begin{array}{cll}a&   0  \\
                 0 & b\end{array}\right)  \hspace{3ex} 
                   \left( \begin{array}{cc}a& b  \\
                  -b& a \end{array}  \right)
                  \hspace{3ex}\left( \begin{array}{cc} a& 1\\
                  0& a\end{array}\right),
\end{equation}  
where $a,b\in\R$.  We will deal with these three cases separately. The last
two cases are easy and will be considered in Section \ref{twocases} The
challenging case is the first one, which will be treated in Section
\ref{thirdcase}.

\subsubsection{$\cL_V$ has complex-conjugate eigenvalues or a nontrivial
Jordan block} \label{twocases}

Suppose first that the endomorphism $\cL_V:\mathrm{Sol}(g)\to \mathrm{Sol}(g)$
has two complex-conjugated eigenvalues. Hence, in some basis $\ms, \sms
\in \mathrm{Sol}(g)$ the endomorphism $ \cL_V$ corresponds to a matrix of the
second type in \eqref{threecases}. If $b=0$, then $V$ acts by homotheties on
any element in $\mathrm {Sol}(g)$ and hence preserves in particular the
Levi-Civita connection of $g$, which contradicts our assumption that $V$ is
not affine. Therefore, we can assume that $b\neq 0$.  Then the evolution of
the solutions $\ms$, $\sms$ along the flow $\Phi_t$ of $V$ is given by
\[ \begin{array}{cll}\Phi^*_t  \ms
&=  e^{a t} \cos(b t) \ms + e^{a t} \sin(b t) \sms   \\
\Phi^*_t \sms & =  -e^{a t} \sin(b t)  \ms +e^{a t} \cos(b t)\sms
\end{array}.\]
Write $g^{-1}\vol(g)^{\frac{1}{n+1}}\in \mathrm{Sol}(g)$ as $c \ms + d \sms$
for some real constants $c$ and $d$. Then one has
\[ \begin{array}{cl} \Phi_t^* (c \ms + d\sms)
&=c( e^{a t} \cos(b t) \ms  +  e^{a t} \sin(b t) \sms )  \\
&+d( -e^{a t} \sin(b t) \ms  +  e^{a t} \cos(b t) \sms) \\
&= e^{at} \sqrt{c^2+d^2} (\cos(bt+\alpha) \ms + \sin(bt+\alpha) \sms),
\end{array}\]
where $\alpha= \arccos(\frac{c}{\sqrt{c^2 + d^2}})$.  Since $g$ is a
Riemannian metric, for any point $x\in M$ there exists a basis of $T_xM$ in
which $\ms$ and $\sms$ are diagonal matrices.  Hence, in this basis the
$i$-th entry of $ \Phi_t^* (c \ms + d\sms)$ is given by $e^{a t}
\sqrt{c^2 + d^2} ( \cos(b t + \alpha)e_i + \sin(b t + \alpha) \tilde e_i)$,
where $e_i$ and $\tilde e_i$ are the $i$-th diagonal entries of $\ms$ and
$\sms$. Therefore, we see that $ \Phi_t^* (c \ms + d\sms)$ is
degenerate for some $t$, which contradicts the fact that $c \ms +
d\sms=g^{-1}\vol(g)^{\frac{1}{n+1}}$ is nondegenerate. The obtained
contradiction shows that $\cL_V$ can not have two complex-conjugate
eigenvalues.

Suppose now that $\cL_V\colon\mathrm{Sol}(g)\to\mathrm{Sol}(g)$ is with
respect to some basis $\ms,\sms\in \mathrm{Sol}(g)$ a matrix of the
third type in \eqref{threecases}.  Then the evolution of $\ms$ and
$\sms$ along $\Phi_t$ is given by
\begin{align*}
\Phi^*_t \ms& = e^{a t} \ms +  t e^{a t} \sms   \\
\Phi^*_t \sms&=e^{a t}  \sms.
\end{align*}
We assume again that $g^{-1}\vol(g)^{\frac{1}{n+1}}=c \ms+ d
\sms$. Then we obtain
\[ \begin{array}{cl} \Phi_t^* (c \ms + d\sms) & =
 c( e^{a t}  \ms  +  e^{a t}  t \sms )  + 
  d( e^{a t}  \sms) \\ & 
   = e^{a t}  (   c \ms  +  (d + c t)\sms).
\end{array}\]
Hence, we see again that for $c\ne0$ there exists $t$ such that $\Phi_t^* (c
\ms + d \sms)$ is degenerate which contradicts the fact that $g$ is
nondegenerate.  Now, if $c=0$, then $\Phi_t$ acts by homotheties on $g$, which
contradicts our assumption that $V$ is not affine. Thus $\cL_V$ can also not
be of the third type in \eqref{threecases}.

\subsubsection{$\cL_V$ has two real eigenvalues} \label{thirdcase} 

Now we consider the remaining case, namely the one where $\cL_V$ has two
different real eigenvalues $a$ and $b$ (the case where the two eigenvalues are
equal was already excluded in the previous section).  Without loss of
generality, we can assume that at least one of the eigenvalues is positive,
since we can otherwise just replace $V$ by $-V$. Hence, we can assume without
loss of generality that $a>b$ and that $a>0$. Since $\phi_t^*\ms=e^{at}\ms$
and $\phi_t^*\sms=e^{bt}\sms$, we see that neither $\ms$ nor
$\sms$ can equal $g^{-1}\vol(g)^{\frac{1}{n+1}}$, since otherwise
$\phi_t$ acts by homotheties on $g$, which contradicts our assumption that $V$
is not affine. Hence, $g^{-1}\vol(g)^{\frac{1}{n+1}}=c\ms+d\sms$ for
constant $c,d\neq 0$. By rescaling $\ms$ and $\sms$, we can therefore
assume without loss of generality that
\[
g^{\alpha\beta}\vol(g)^{\frac{1}{n+1}}=\ms^{\alpha\beta}+\sms^{\alpha\beta}.
\]
Let us write $D^{\alpha\beta}=\ms^{\alpha\beta}\vol(g)^{-\frac{1}{n+1}}$ and
$\tilde D^{\alpha\beta}=\sms^{\alpha\beta}\vol(g)^{-\frac{1}{n+1}}$ such
that
\[
g^{\alpha\beta}=D^{\alpha\beta}+\tilde D^{\alpha\beta}.
\]

Note that for a K\"ahler manifold $(M,J,g)$ of mobility $2$, the dense open
subset of regular points (Definition~\ref{def:regular}) does not depend on the
choice of non-proportional c-projectively equivalent metrics from the
c-projective class of $g$. In particular, the set of regular points is
invariant under c-projective transformations and hence under the action of the
flow of a c-projective vector field.  Thus if we fix a regular point $x_0\in
M$ and consider the integral curve $\phi_t(x_0)$ of our complete
c-projective vector field $V$ through $x_0$, then there is an open
neighbourhood $U$ of the curve $\phi_t(x_0)$ in the set of regular points,
and, since $g$ is positive definite, a frame of $TU$, in which $g$ corresponds
to the identity matrix and $D$ and $\tilde D$ to diagonal matrices:
\begin{equation}\label{D,tildeD}
D= \left(\begin{array}{ccccc}
d_1&&&&\\ &d_1&&&\\&&\ddots&&\\&&&d_n&\\&&&&d_n
\end{array}\right), \quad \tilde D= \left(\begin{array}{ccccc}
\tilde d_1&&&&\\ &\tilde d_1&&&\\&&\ddots&&\\&&&\tilde d_n&\\&&&&\tilde d_n
\end{array}\right),
\end{equation}
where $d_i$ and $\tilde d_i$ are smooth real-valued functions on $U$ such that
$d_i+\tilde d_i=1$ for $i=1,\ldots,n$.  Then in the local frame the tensor
$A_t^{\alpha\beta}
=\phi_t^*(\ms^{\alpha\beta}+\sms^{\alpha\beta})\vol(g)^{-\frac{1}{n+1}}$
corresponds to the following diagonal matrix:
\begin{equation}\label{A_t}
A_t= \left(\begin{array}{ccccc}
e^{at}d_1+e^{bt}\tilde d_1&&&&\\ &e^{at}d_1+e^{bt}\tilde d_1&&&\\
&&\ddots&&\\&&&e^{at}d_n+e^{bt}\tilde d_n&\\&&&&e^{at}d_n+e^{bt}\tilde d_n
\end{array}\right).
\end{equation}
Since $g$ and $\phi_t^*g$ are positive definite, all diagonal entries of
\eqref{A_t} are positive for all $t\in\R$. Hence, $d_i+e^{(b-a)t}\tilde d_i>0$
respectively $e^{(a-b)t}d_i+\tilde d_i>0$ for all $t$ and taking the limit
$t\rightarrow\infty$ respectively $t\rightarrow -\infty$ shows that
$d_i,\tilde d_i\geq 0$ for all $i=1,\ldots,n$. Since $d_i+\tilde d_i=1$, we
conclude that
\[
0\leq d_i\leq 1\quad \text{ and }\quad 0\leq \tilde d_i\leq 1
\quad\quad \text{ for all } i=1,\ldots,n.
\]
Now consider the $(1,1)$-tensor field $D_{\alpha}{}^\beta
=g_{\alpha\gamma}\ms^{\gamma\beta}\vol(g)^{-\frac{1}{n+1}} $ and its pullback
$\phi_t^*(D_{\alpha}{}^\beta)
=\phi_t^{*}(g_{\alpha\gamma}\ms^{\gamma\beta}\vol(g)^{-\frac{1}{n+1}})$.
Since $g_{\alpha\beta}\vol(g)^{-\frac{1}{n+1}}$ is inverse to
$\ms^{\alpha\beta}+\sms^{\alpha\beta}$, we conclude that
$\phi_t^*(D_{\alpha}{}^\beta)$ is given by a block diagonal matrix whose
$i$-th block is given by the following $2\times 2$ matrix
\begin{equation}\label{pull_back_D}
\frac{e^{at}d_i}{e^{at}d_i+e^{bt}(1-d_i)}\Id_2.
\end{equation}
By definition $\phi_t$ acts on the endomorphism $D_{\alpha}{}^\beta$ as
$\phi_t^*(D_{\alpha}{}^\beta)=(T\phi_t)^{-1}\circ D_{\alpha}{}^\beta \circ
T\phi_t$, which implies that the eigenvalues of $\phi^*_t(D_{\alpha}{}^\beta)$
at a point $x\in M$ are the same as the eigenvalues of $D_{\alpha}{}^\beta$ at
$\phi_t(x)$. Therefore, it follows from \eqref{pull_back_D} that the only
possible constant eigenvalues of $D_{\alpha}{}^\beta$ on $U$ are $0$ and $1$
(i.e.~the only possible constant diagonal entries in \eqref{D,tildeD} are $0$
or $1$). Note that $d_i=0$ (respectively $d_i=1$) on some open set implies
that $\tilde d_i=1$ (respectively $\tilde d_i=0$).  Hence, the only possible
constant eigenvalues of $A_t$ defined as in \eqref{A_t} are $e^{at}$ and
$e^{bt}$.  Since $U$ consists of regular points, the distinct eigenvalues of
$A_t$ (for any fixed $t$) are smooth real-valued functions with constant
algebraic multiplicities on $U$.  Let us write $2m$, respectively $2\tilde m$,
for the multiplicity of the eigenvalues $e^{a}$ and $e^{b}$ of $A_1$ on
$U$. By Lemma \ref{alg-mult}, the number of distinct nonconstant eigenvalues
of $A_1$ is given by $n-m-\tilde m$, and $m,\tilde m$ are constant on the set
of regular points by Corollary~\ref{cor:eigenvectors}. We allow, of course,
that $m$, respectively $\tilde m$, are zero.

\begin{lem}\label{Lemma_Mobility_2_case}  
If at least one of the following two inequalities,
\begin{equation}\label{inequalities}
(n-\tilde m)a + (\tilde m+1)b\leq 0 \quad\text{and}\quad (m+1)a + (n-m)b\geq 0. 
 \end{equation} 
is not satisfied, then the vector field $\hv^\alpha$ given by the gradient of
$\hp=-\frac{1}{2}D_{\beta}{}^\beta$ lies in the nullity space of $M$ at $x_0$.
\end{lem}
\begin{proof}
Set $G_t:=\det_\R(A_t)^{-\frac{1}{2}} A_t^{-1}$ and note that
$\phi_t^*g=g(G_{t}\cdot,\cdot)$. We may assume that the first
$2\ell:=2n-2m-2\tilde m$ elements of $D$ are not constant (which is equivalent
to assuming that $d_i(x_0)\neq0,1$ for $i=1,\ldots,\ell$), the next $2m$
elements are equal to $1$, and the remaining $2\tilde m$ elements are zero on
$U$.  Then, we deduce from \eqref{A_t} that $G_t$ on $U$ is a block diagonal
matrix of block sizes $2\ell\times 2\ell$, $2m\times 2m$ and $2\tilde m\times
2\tilde m$ respectively, where the three blocks are given by
\begin{equation}\label{B_t}
\Psi(t)\left(\begin{array}{ccccc}
\frac{1}{d_1 e^{at} + (1-d_1) e^{b t}}&&&&\\
&\frac{1}{d_1 e^{at} + (1-d_1) e^{b t}}&&&\\&&\ddots&&\\
&&&\frac{1}{d_\ell e^{at} + (1-d_\ell) e^{b t}}&\\
&&&&\frac{1}{d_\ell e^{at} + (1-d_\ell) e^{b t}}
\end{array}\right)
\end{equation}
\begin{equation*}
\Psi(t)e^{-at}\Id_{2m},\quad\quad\text{respectively}\quad\quad
\Psi(t) e^{-bt} \Id_{2\tilde m}, 
\end{equation*}
where
\[
\Psi(t):=
e^{-amt}e^{-b\tilde m t}\prod_{i=1}^{\ell}\frac{1}{d_i e^{at} + (1-d_i) e^{b t}}.
\]
Let us write $\nu_1,\ldots,\nu_\ell$, $\nu$ and $\tilde\nu$ for the
eigenvalues of these respective diagonal matrices.  Note that their asymptotic
behaviour for $t\to +\infty $ respectively for $t \to -\infty$ is as follows
\begin{equation}\label{asymptotics_eigenvalues}
\begin{array}{cccc}
t\to+\infty & \nu_i(t)\sim \frac{e^{-((n-\tilde m+1)a + \tilde m b)t} }{d_i\prod d_j}
& \nu(t) \sim \frac{e^{-((n-\tilde m+1)a + \tilde m b)t} }{\prod d_j}
& \tilde \nu(t)\sim \frac{e^{-((n-\tilde m)a + (\tilde m+1) b)t} }{\prod d_j}\\ &&&\\
t\to-\infty & \nu_i(t)\sim \frac{e^{( ma + (n-m+1) b)t} }{(1-d_i)\prod (1-d_j)}
&   \nu(t) \sim \frac{e^{((m+1)  a + (n-m)b )t} }{\prod (1-d_j)}
&  \tilde \nu(t)\sim \frac{e^{( ma + (n-m+1) b)t} }{\prod (1- d_j)}.
\end{array}
\end{equation}
Let us now assume that \eqref{inequalities} is not satisfied. We can assume
without loss of generality that the first inequality of \eqref{inequalities}
is not satisfied, that is to say we can assume that $ (n-\tilde m)a + (\tilde
m+1) b>0$, since otherwise we can change the sign of $V$, which interchanges
the inequalities. Then it follows from \eqref{asymptotics_eigenvalues} that
all eigenvalues of $G_t$ decay exponentially as $t\to\infty$.  Consider now
the sequence $(\phi_k(x_0))_{k\in \Z_{\geq 0}}$. We claim that it is a Cauchy
sequence. Indeed, note that the distance $d(\phi_k(x_0),\phi_{k+1}(x_0))$
between $\phi_k(x_0)$
and $\phi_{k+1}(x_0)$ satisfies
\begin{align}\label{Cauchy_seq}
d(\phi_k(x_0),\phi_{k+1}(x_0))\leq&
\int_{0}^1\sqrt{g(V(\phi_{k+t}(x_0)),V(\phi_{k+t}(x_0)))}dt\\\nonumber
&=  \int_{0}^1\sqrt{(\phi_k^*g)(V(\phi_{t}(x_0)),V(\phi_{t}(x_0)))}dt.
\end{align}

Since $\phi_k^*g=g(G_{k}\cdot,\cdot)$ and all eigenvalues of $G_t$ decay
exponentially as $t\to\infty$, the inequality \eqref{Cauchy_seq} shows that
$d(\phi_k(x_0),\phi_{k+1}(x_0))$ decays geometrically as $k\to\infty$
(i.e.\,for sufficiently large $k$ we have $d(\phi_k(x_0),\phi_{k+1}(x_0))\le
\mathrm{const}\ q^k$ for some $q<1$). Hence, $(\phi_k(x_0))_{k\in \Z_{\geq
    0}}$ is a Cauchy sequence and completeness of $M$ implies that
$(\phi_k(x_0))_{k\in \Z_{\geq 0}}$ converges. We denote the limit of
$(\phi_k(x_0))_{k\in \Z_{\geq 0}}$ by $p\in M$.  Now consider the smooth
real-valued function $F$ on $M$ given by
\begin{equation*}
F= H_{\alpha\beta}{}^\gamma{}_\delta H_{\epsilon\zeta}{}^\theta{}_\eta
g_{\gamma\theta} g^{\alpha\epsilon} g^{\beta\zeta} g^{\delta\eta},
\end{equation*}
where $H_{\alpha\beta}{}^\gamma{}_\delta$ denotes the harmonic curvature of
$(M, J, [\nabla^g])$ defined as in Proposition \ref{kaehlerharmcurv},
respectively Proposition~\ref{rosetta}.  
Since $H_{\alpha\beta}{}^\gamma{}_\delta$ is
c-projectively invariant, we deduce that $(\phi_k^*F)(x_0)=F(\phi_k(x_0))$
equals
\begin{equation}\label{pull_back_F}
F(\phi_k(x_0))=(H_{\alpha\beta}{}^\gamma{}_\delta H_{\epsilon\zeta}{}^\theta{}_\eta\,
\phi_k^*g_{\gamma\theta}\, \phi^*_k g^{\alpha\epsilon}\,\phi_k^*g^{\beta\zeta}\,
\phi_k^*g^{\delta\eta})(x_0).
\end{equation}
Moreover, since $F$ is continuous, we have $\lim_{k\to\infty}F(\phi_k(x_0))=F(p)$. 

Since in the frame we are working the matrices corresponding to $g$ and $G_t$
are diagonal, we see that the function $F(\phi_t(x_0))$ is a sum of the form
\begin{equation}\label{sum}
\sum_{1\leq i,j,k,\ell\leq 2n}C(ijk\ell;t)
\left( H_{\alpha_i \alpha_j}{}^{\alpha_k}{}_{ \alpha_\ell}(x_0)  \right)^2, 
\end{equation} 
where the coefficient $C(ijk\ell;t)$ is the product of the $k$-th diagonal
entry and the reciprocals of the $i$-th, $j$-th and $\ell$-th diagonal entry
of the diagonal matrix that corresponds to $G_t$ (in our chosen frame). The
coefficients $C(ijk\ell;t)$ depend on $t$ and their asymptotic behaviour for
$t\to \pm\infty$ can be read off from \eqref{asymptotics_eigenvalues}. Note
moreover that all coefficients $C(ijk\ell;t)$ are positive.

We claim that, if at least one of the indices $i,j$ or $\ell$ is less or equal
than $2n-2m-2\tilde m$, then
$H_{\alpha_i\alpha_j}{}^{\alpha_k}{}_{\alpha_\ell}(x_0)$ vanishes. Indeed,
note that, by \eqref{asymptotics_eigenvalues}, $\phi^*_tg$ decays
exponentially at least as $e^{-((n-\tilde m+1)a + \tilde m b)t}$, which is up
to a constant the smallest eigenvalue of $G_t$, and $\phi^*_tg^{-1}$ goes
exponentially to infinity at least as $e^{((n-\tilde m)a + (\tilde m+1)b )t}$
as $t\to \infty$. Suppose now that at least one of the indices $i, j$ or
$\ell$ is less or equal than $2n-2m-2\tilde m$. Then we deduce that up to
multiplication by a positive constant $C(ijk\ell;t)$ behaves asymptotically as
$t\to\infty$ at least as
\[
e^{((n-\tilde m)a + (\tilde m+1)b )t}e^{((n-\tilde m)a
+ (\tilde m+1)b )t}e^{((n-\tilde m+1)a + \tilde m b)t}e^{-((n-\tilde m+1)a + \tilde m b)t}
=e^{2((n-\tilde m)a + (\tilde m+1)b )t}.
\]
Since $(n-\tilde m)a + (\tilde m+1)b >0$ by assumption, we therefore conclude
that the coefficient
\[
C(ijk\ell;t)\to\infty \quad\quad\text{ as } t\to\infty.
\]
Since all terms in the sum \eqref{sum} are nonnegative and the sequence
$F(\phi_\ell(x_0))$ converges, we therefore deduce that
$H_{\alpha_i\alpha_j}{}^{\alpha_k}{}_{\alpha_\ell}(x_0)=0$ provided that at
least one of the indices $i,j$ or $\ell$ is less or equal than $2n-2m-2\tilde
m$.

Observe now that $\hv^\alpha$ equals the negative of the sum of the gradients
of the distinct nonconstant eigenvalues $d_1,\ldots,d_{n-m-\tilde m}$ of
$D$. We therefore conclude that at $x_0$
\[
H_{\alpha\beta}{}^\gamma{}_\delta \hv^\alpha=0,\quad
H_{\alpha\beta}{}^\gamma{}_\delta \hv^\beta=0\quad\text{and}\quad
H_{\alpha\beta}{}^\gamma{}_\delta \hv^\delta=0.
\]
It follows that (3) of
Remark~\ref{rem:c-projective_invariant_characterisation_of_nullity} is
satisfied, which implies that the vector field $\hv^\alpha$ lies in the
nullity space of $M$ at $x_0$.
\end{proof}

Since $x_0$ is an arbitrary regular point and the conditions
\eqref{inequalities} do not depend on the choice of regular point, Lemma
\ref{Lemma_Mobility_2_case} shows that, if the inequalities
\eqref{inequalities} are not satisfied, then $(M,J,g)$ has nullity on the
dense open subset of regular points, since $\hv^\alpha$ does not vanish at
regular points (otherwise $V$ is necessarily affine, which contradicts our
assumption).  Hence, by Theorem~\ref{mobility>2}, we have established
Theorem~\ref{thm:obata}, except when the inequalities \eqref{inequalities}
are satisfied.

Therefore, assume now that the inequalities \eqref{inequalities}
are satisfied. Subtracting the second inequality from the first
shows that
$$(n-m-\tilde m-1)(a-b)\leq 0.$$
Since $a-b>0$ by assumption, we must have $n-m-\tilde m=0$ or $n-m-\tilde
m=1$. In the first case $\phi_t^*g$ is parallel for the Levi-Civita connection
of $g$ for all $t$ and hence $V$ is affine, which contradicts our
assumption. Therefore, we must have $n-m-\tilde m=1$.  Now substituting this
identity back into \eqref{inequalities} shows that $(m+1)a + (\tilde m
+1)b\leq 0$ and $(m+1)a + (\tilde m +1)b\geq 0$, which implies $(m+1)a =- (\tilde m
+1)b.$ Hence, we conclude that we must have
\begin{equation}\label{equality}
n-m- \tilde m=1\quad \text{ and }\quad (m+1)a = -(\tilde m+1)b.
\end{equation}
Therefore, locally around any regular point $D$ has precisely one nonconstant
eigenvalue, which we denote by $\rho$, and the constant eigenvalues $1$ and
$0$ with multiplicity $2m$, respectively $2\tilde m$.  Hence, for any regular
point $x_0\in M$ we can find an open neighbourhood $U$ of the curve
$\phi_t(x_0)$ in the set of regular points and a frame of $TU$ such that $D$
corresponds to a matrix of the form
\begin{equation}\label{possible_form_of_D}
D= \left(\begin{array}{cccc}
\rho&&&\\ &\rho&&\\&&\Id_{2m}&\\&&&0_{2\tilde m}
\end{array}\right),
\end{equation}
where $\rho$ is a smooth function on $U$ with $\rho(x)\neq 0,1$ for $x\in U$.
Note that we have $\hp=-\frac{1}{2} D_{\alpha}{}^\alpha=-(\rho+m)$ and
consequently
\begin{equation*}
\hv_{\alpha}=-\nabla_{\alpha} \rho,
\end{equation*}
which is nowhere vanishing on $U$.
Recall also that the pair $(D, \hv)$ satisfies the mobility equation  
\eqref{mobility_equation1}, and that by Corollary
\ref{cor:eigenvectors} $\hv^\alpha$ is an eigenvector of $D_{\alpha}{}^\beta$
with eigenvalue~$\rho$.  Hence, $\hv^\alpha$ and
$J_{\beta}{}^\alpha\hv^\beta$ form a basis for the eigenspace of
$D_{\alpha}{}^\beta$ corresponding to the eigenvalue $\rho$.

For later use, let us also remark that by \eqref{pull_back_D} the action of
the flow $\phi_t$ on $D$ preserves its block structure
\eqref{possible_form_of_D}. This implies, in particular, that
$\cL_V\hv=[V,\hv]$ is an eigenvector of $D$ with eigenvalue $\rho$.  Since
furthermore $\nabla_\alpha \rho=-\hv_{\alpha}$ vanishes in direction of all
vector fields orthogonal to $\hv^\alpha$ and $[\hv, J\hv]=0$, we therefore
conclude that $[V,J\hv]\cdot \rho=0$ implying that the vector field $[V,
  J\hv]=J[V,\hv]$ is actually a proportional to $J\hv$. Hence, the
c-projective vector field $V$ preserves the orthogonal projection from $TM$ to
the $1$-dimensional subspace spanned by $\hv^{\alpha}$, which is defined on
the dense open subset on which $\hv^\alpha$ is not vanishing (hence in
particular on $U$).

\begin{lem} \label{the_very_last_Lemma}
Assume $n-m-\tilde m= 1 \text{ and } (m+1)a = -(\tilde m +1)b$. Then in a
neighbourhood of any regular point there exists a real positive constant $B$
and a smooth real-valued function $\mu$ such that
\begin{equation}\label{second_line_satisfied}
\nabla_\alpha\hv_\beta=-\mu g_{\alpha\beta}+2B D_{\alpha\beta}
 \end{equation}
 \begin{equation}\label{third_line_satisfied}
 \nabla_\alpha \mu=2B\hv_\alpha.
 \end{equation}
\end{lem}

\begin{proof} 
Fix a regular point $x_0\in M$, an open neighbourhood $U$ of $\phi_t(x_0)$ 
in the set of regular points, and a frame of $TU$ with respect to
which $g$ corresponds to the identity and $D$ is of the form
\eqref{possible_form_of_D}.  Since $\hv^\alpha$ is an eigenvector of
$D_{\alpha}{}^\beta$ with eigenvalue $\rho$, we can furthermore assume without
loss of generality that $\hv^\alpha$ is proportional to the first vector of
our fixed local frame.  We restrict our considerations from now on to $U$.

Differentiating the equation $D_{\alpha}{}^\beta\hv^\alpha=\rho\hv^\beta$,
and substituting \eqref{mobility_equation1} and $\nabla_\alpha
\rho=-\hv_\alpha$, we obtain
\begin{equation}\label{differentiation_of_eigenvalue_equation}
(\rho\delta_{\gamma}{}^\beta-D_\gamma{}^\beta)\nabla_{\alpha}\hv_{\beta}
=-\tfrac{1}{2}(g_{\alpha\gamma}\hv_\beta\hv^\beta-\hv_\alpha\hv_\gamma
-J_{\alpha}{}^{\beta}\hv_\beta J_{\gamma}{}^\epsilon\hv_\epsilon).
\end{equation}
Since $D_{\alpha}{}^\beta$ commutes with $\nabla_{\alpha}\hv^\beta$ by
Proposition \ref{source_of_things_awesome}, the block diagonal form of
$D_{\alpha}{}^\beta$ implies that $\nabla_{\alpha}\hv^\beta$ has the same
block diagonal form. Equation \eqref{differentiation_of_eigenvalue_equation}
therefore shows that the second and the third block of
$\nabla_{\alpha}\hv^\alpha$ are proportional to the identity with coefficient
of proportionality $-\frac{\hv_\beta\hv^\beta}{2(\rho-1)}$ and
$-\frac{\hv_\beta\hv^\beta}{2\rho}$ respectively. The condition
\eqref{second_line_satisfied} therefore reduces to the following three
equation
\begin{equation}\label{equivalent_conditions_1}
(\nabla_{\alpha}\hv_\beta) \hv^\beta=(-\mu+2B\rho)\hv_\alpha
\quad\hv_\beta\hv^\beta=2(\mu-2B)(\rho-1) \quad\hv_\beta\hv^\beta=2\rho\mu,
\end{equation}
which can be equivalently rewritten as 
\begin{equation}\label{equivalent_conditions_2}
(\nabla_{\alpha}\hv_\beta) \hv^\beta=2B(2\rho-1)\hv_\alpha
\quad \hv_\alpha\hv^\alpha=-4B(\rho-1)\rho\quad \mu=-2B(\rho-1).
\end{equation}
First note that, 
\begin{equation} \label{deri} 
 \hv_\alpha V^\alpha= \tfrac{d}{dt}_{|t=0} \phi_t^*\hp =  
-\tfrac{d}{dt}_{|t=0} \frac{\rho e^{at}}{\rho e^{at} + (1-\rho)e^{b t}} =
 (b-a) \rho(1- \rho),
\end{equation}
which is nowhere vanishing on $U$. It follows that
\begin{equation} \label{pullback_Lambda_V} 
\phi_t^*(\hv_\alpha V^\alpha)
=\frac{(b-a)\rho(1-\rho)e^{at}e^{b t}}{(\rho e^{at} + (1-\rho) e^{b t})^2}.
\end{equation} 
Let us write $P$ for the orthogonal projection of $TU$ to the line subbundle
spanned by $\hv^\alpha$. We have already noticed that $V$ preserves this
projection, which implies in particular that $\cL_V P(V)=0$. Note also that by
\eqref{deri} the vector field $P(V)$ is nowhere vanishing on $U$. Hence, we
deduce from \eqref{B_t} that
\begin{equation}\label{pullback_of_norm_of _P(V)}
\phi^*_t(g(P(V),P(V)))= g(G_tP(V), P(V))
=\frac{ g(P(V),P(V))} { (\rho e^{at} + (1- \rho)e^{b t})^2 e^{(am + b \tilde m)t}}.
\end{equation}  
Since by assumption $(m+1)a+(\tilde m+1)b=0$, equations
\eqref{pullback_Lambda_V} and \eqref{pullback_of_norm_of _P(V)} show that the
function
\begin{equation}\label{constant_function}
t\mapsto\phi^*_t\left(\frac{g(\hv, V)}{g(P(V),P(V))}\right)
=\frac{(b-a)\rho(1-\rho)}{g(P(V),P(V))}
\end{equation}
is constant. Since $g(\hv,V)=g(\hv, P(V))$, we therefore obtain
\[
\phi_t^*(\hv)=\frac{\rho(1-\rho)(b-a)}{g(P(V), P(V))} \phi_t^*(P(V)),
\]
and hence \eqref{pullback_Lambda_V} implies 
\begin{equation} \label{deri1}
\phi_t^*(\hv_\alpha\hv^\alpha)=\frac{(b-a)^2\rho^2(1-\rho)^2e^{at}e^{bt}}
{(\rho e^{at} + (1-\rho) e^{bt})^2 g(P(V), P(V))}.
\end{equation}
Differentiating identity \eqref{deri1} gives 
\begin{equation}\label{length} (\nabla_\alpha\hv_\beta \hv^\beta) V^\alpha
= \tfrac{d}{ dt }_{|t=0} \Phi_t^*(\hv_\beta\hv^\beta)
= \frac{(b-a)^3\rho^2(1-\rho)^2(2\rho-1) }{g(P(V), P(V))}.
\end{equation} 
Since $(\nabla_{\alpha}\hv_\beta)\hv^\alpha=\frac{1}{2}\nabla_\alpha
(\hv_\beta\hv^\beta)$, we can rewrite the first condition of
\eqref{equivalent_conditions_2} as
\begin{equation*}
\nabla_\alpha (\hv_\beta\hv^\beta)= 4B(2\rho-1)\hv_\alpha,
\end{equation*}
and we conclude from \eqref{length} and \eqref{deri} that contracting the
above identity with $V^\alpha$ yields
\begin{equation}\label{formula_for_B}
B=\frac{1}{4} \frac{(b-a)^2 \rho(1-\rho)}{g(P(V), P(V))}.
\end{equation}  
By \eqref{deri1} the second condition of \eqref{equivalent_conditions_2} reads
\begin{equation}
\frac{(b-a)^2\rho^2(1-\rho)^2}{g(P(V), P(V))}=4B(1-\rho)\rho,
\end{equation}
which is of course equivalent to \eqref{formula_for_B}. Since the third
condition of \eqref{equivalent_conditions_2}, given by $\mu=-2B(\rho-1)$,
simply defines $\mu$ in terms of $\rho$ and $B$, we conclude that there exist
functions $\mu$ and $B$ on $U$ such that \eqref{second_line_satisfied} is
satisfied. Note that by \eqref{formula_for_B} the function $B$ is also
positive as required. It remains to show that $B$ is constant (in a
sufficiently small neighbourhood of $x_0$), which implies in particular that
$\mu=-2B(\rho-1)$ is smooth and that \eqref{third_line_satisfied} is
satisfied.

The formula \eqref{formula_for_B} for $B$ shows that $B$ is proportional with
a constant coefficient to \eqref{constant_function}. Hence, we have
$\nabla_VB=0$. To show that the derivative of $B$ also vanishes along vector
field transversal to $V$ consider the $2n-1$ dimensional submanifold of $U$
given by the level set of $\rho$
\begin{equation}
M_{y}:=\{x\in U: \rho(x)=\rho(y)\},
\end{equation}
where $y$ is some arbitrary point in $U$.  Since the derivative of $\rho$ is
nontrivial along the c-projective vector field $V$ by \eqref{deri}, $V$ is
transversal to $M_{y}$. We claim that the derivative of $B$ at $y$ vanishes
along all vectors in $T_yM_y$. Indeed, note that in view of
\eqref{formula_for_B} this is equivalent to the vanishing of the derivative of
$g(P(V), P(V))$ at $y$ along tangent vectors of $M_{y}$. Since $\hv$ and
$P(V)$ are proportional to each other by definition, we have
\begin{equation*}
g(P(V), P(V))=\frac{(g(P(V), \hv))^2}{g(\hv,\hv)}.
\end{equation*}
It is thus sufficient to show that at $y$ the derivative of $g(P(V), \hv)$,
respectively $g(\hv,\hv)$, vanishes along tangent vectors of $M_{y}$.  By
\eqref{deri} this follows immediately for the derivative of $g(P(V),
\hv)=g(V, \hv)$. Now consider $g(\hv,\hv)$ and let $W\in T_yM_{y}$.  Then
we compute
\begin{align*}
W^\alpha \nabla_\alpha(\hv_\beta\hv^\beta)
&=2W^\alpha(\nabla_{\alpha}\hv_\beta) \hv^\beta\\
&=-2\mu W^\alpha \hv_\alpha+4BW^\alpha D_{\alpha\beta} \hv^\beta
=-2\mu W^\alpha \hv_\alpha+4B\rho W^\alpha \hv_\alpha.
\end{align*}
Since $\hv_\alpha=-\nabla_\alpha \rho$, we see that $W^\alpha \hv_\alpha$
vanishes and consequently so does $W^\alpha
\nabla_\alpha(\hv_\beta\hv^\beta)$. Hence, the derivative of $B$ vanishes at
$y\in U$. Since $y\in U$ was an arbitrary point, we conclude that $\nabla B$
is identically zero on $U$, which completes the proof.
\end{proof}

We can now complete the proof of Theorem~\ref{thm:obata}. Under the assumption 
that \eqref{equality} holds, Lemma \ref{the_very_last_Lemma} shows that in an 
open neighbourhood of any
regular point there exists a positive constant $B$ (which a priori may depend
on the neighbourhood) and a function $\mu$ such that the triple $(D, \hv,
\mu)$ satisfies, in addition to the mobility equation, the equations
\eqref{second_line_satisfied} and \eqref{third_line_satisfied}. Since the set
of regular points is open and dense, Theorem~\ref{Lambda_in_nullity} implies
that $B$ is actually the same constant at all regular points and that the
equations \eqref{second_line_satisfied} and \eqref{third_line_satisfied} hold
on $M$ for some smooth function $\mu$.  Theorem~\ref{Lambda_in_nullity} also
implies that the function $\hp=-\frac{1}{2}D_{\alpha}{}^{\alpha}$ satisfies
the equivalent conditions of Proposition~\ref{Tanno_and_nullity} on $M$ for a 
positive constant $B$. Since $\hv_\alpha=\nabla_\alpha\hp$ is not identically
zero, Tanno's result~\cite{Tanno} completes the proof. 

\section{Outlook}\label{s:outlook}

There has been considerable activity in c-projective geometry since we began
work on this article in 2013. Despite this, there remain many open questions
and opportunity for further work. In this final section, we survey some of the
developments and opportunities which we have not discussed already in the
article.

\subsection{Metrisability and symmetry}

One of the main focuses of this article has been metrisable c-projective
structures and their mobility. However, in later sections, we restricted
attention to integrable complex structures (the torsion-free case). It would
be interesting to extend more of the theory to non-integrable structures with
a view to applications in quasi-K\"ahler geometry, including $4$-dimensional
almost K\"ahler geometry---here some partial results have been obtained
in~\cite{ACG0}.

Even in the integrable case, however, a basic question remains wide open:
when is a c-projective structure metrisable? One would like to provide a
complete c-projectively invariant obstruction, analogous to the obstruction
found for the $2$-dimensional real projective case in~\cite{BDE}.

Additional questions concern the symmetry algebra
$\mathfrak{cproj}(J,[\nabla])$ of infinitesimal automorphisms of an almost
c-projective $2n$-manifold $(M,J,[\nabla])$.  As with any parabolic
geometry~\cite{InfautCap,csbook}, the prolongation of the infinitesimal
automorphism equation (see Section~\ref{infinitesimal_automorphisms},
Proposition~\ref{Infautconn2}) shows that $\mathfrak{cproj}(J,[\nabla])$ is
finite dimensional, with its dimension bounded above---in this case, by
$2(n+1)^2-2$. This bound is attained only in the c-projectively flat case,
and, as shown in~\cite{KMT}, in the non-flat case, the dimension of
$\mathfrak{cproj}(J,[\nabla])$ is at most $2n^2 - 2n + 4 + 2\delta_{3,n}$
(the so-called ``submaximal dimension''). The determination of the possible
dimensions of $\mathfrak{cproj}(J,[\nabla])$ remains an open question.

The symmetry algebra $\mathfrak{cproj}(J,[\nabla])$ acts (by Lie derivative)
on the space of solutions to the metrisability equation, and for any
nondegenerate solution, corresponding to a compatible metric $g$, it has
subalgebras $\mathfrak{isom}(J,g)\subseteq\mathfrak{aff}(J,g)$ of holomorphic
Killing fields and infinitesimal complex affine transformations. Thus the
presence of compatible metrics constrains $\mathfrak{cproj}(J,[\nabla])$
further, the Yano--Obata theorems being global examples of this.  Even
locally, if $(M,J,g)$ admits an \emph{essential} c-projective vector field,
i.e., an element
$X\in\mathfrak{cproj}(J,[\nabla])\setminus\mathfrak{isom}(J,g)$, then $g$
must have mobility $\geq 2$, and if $X\notin\mathfrak{aff}(J,g)$, then there
are metrics c-projectively, but not affinely, equivalent to $g$. For example,
nontrivial c-projective vector fields with higher order zeros are essential
(because a Killing vector field is determined locally uniquely by its $1$-jet
at a point), and such \emph{strongly essential local flows} exist only on
c-projectively flat geometries~\cite{melnick}.

Constraints on the possible dimensions of $\mathfrak{cproj}(J,[\nabla])
/\mathfrak{isom}(J,g)$ for K\"ahler manifolds are given in~\cite{MatRos}, and
an explicit classification of $4$-dimensional \bps/K\"ahler metrics admitting
essential c-projective vector fields is given in~\cite{BMMR}.

An open question here is whether locally nonlinearizable c-projective vector
fields exist on nonflat c-projective geometries.

\subsection{Applications in K\"ahler geometry}

The original motivation for c-projective geometry~\cite{OtsTash} was to extend
methods of projective geometry to K\"ahler metrics (and this is one reason why
we have concentrated so much on the metrisability equation). Thus one expects
ideas from c-projective geometry to be useful in K\"ahler geometry, and indeed
important concepts in K\"ahler geometry have c-projective origins: for
instance, Hamiltonians for Killing vector fields form the kernel of the
c-projective Hessian.

Apostolov et al.~\cite{ACG,ApostolovII,ApostolovIII,ApostolovIV} use
c-projectively equivalent metrics (in the guise of Hamiltonian $2$-forms) to
study \emph{extremal K\"ahler metrics}, where the scalar curvature of a
K\"ahler metric lies in the kernel of its c-projective Hessian, and it would
be natural to consider extremal quasi-K\"ahler metrics in the same light.
K\"ahler--Ricci solitons (and generalisations) admitting c-projectively
equivalent metrics have also been studied in special
cases~\cite{ApostolovIV,L-TF,M-TF}, but the picture is far from complete.

A more recent development is the work of \v Cap and Gover~\cite{CG3}, which
extends previous work on projective compactification of Einstein
metrics~\cite{CG2} to K\"ahler (and quasi-K\"ahler) metrics using c-projective
geometry.

Let us now touch on prospective applications in the theory of finite
dimensional integrable systems.  As we have seen in
Section~\ref{sec:integrability}, the Killing equations for Hermitian
symmetric Killing tensors are c-projectively invariant. This suggests to
study these equations from a c-projective viewpoint. We expect that in this
way one may find interesting new examples of integrable systems, in
particular on closed K\"ahler or Hermitian manifolds (note that only few such
examples are known).  Moreover, the analogy between metric projective and
c-projective geometries suggests that ideas and constructions from the theory
of integrable geodesic flows on $n$-dimensional Riemannian manifolds could be
used in the construction and description of integrable geodesics flows on
$2n$-dimensional K\"ahler manifolds.  For Killing tensors of valence two this
approach is very close to the one in~\cite{Kiyo1997}, and we expect similar
applications for Killing tensors of higher valence.

\subsection{Projective parabolic geometries}

We noted in the introduction that there are many analogies between methods and
results in projective and c-projective geometry, and we have followed the
literature in exploiting this observation. We have already noted a partial
explanation for the similarities: both are Cartan geometries modelled on flag
varieties $G/P$, one a complex version of the other. However, the c-projective
metrisability equation for compatible \bps/K\"ahler metrics is not the
complexification of the corresponding projective metrisability equation.
Instead, both are first BGG operators for a $G$-representation with a
$1$-dimensional $P$-subrepresentation. These representations determine
projective embeddings of the model $G/P$, namely, the Veronese embedding of
$\RP^n$ as rank one symmetric matrices in the projective space of
$S^2\R^{n+1}$, and the analogous projective embedding of $\CP^n$ using rank
one Hermitian matrices.

Symmetric and Hermitian matrices are examples of Jordan algebras, and this
relation with projective geometry is well known (see e.g.~\cite{Bertram}),
which suggests to define a \emph{projective parabolic geometry} as one in
which the model has a projective embedding into a suitable Jordan algebra.
Apart from projective and c-projective geometry, the examples include
quaternionic geometry, conformal geometry, and the geometry associated to the
Cayley plane over the octonions. In his PhD thesis~\cite{Frost}, G. Frost has
shown the much of the metrisability theory of these geometries can be
developed in a unified framework. Further, in addition to being analogous,
projective parabolic geometries are closely interrelated. For instance,
S. Armstrong~\cite{Arm} uses cone constructions to realise quaternionic and
c-projective geometry as holonomy reductions of projective Cartan connections,
while the \emph{generalised Feix--Kaledin construction}~\cite{BC} shows how to
build quaternionic structures from c-projective structures, modelled on the
standard embedding of $\CP^n$ in $\mathbb{HP}^n$.

\end{document}